\documentclass[reqno,11pt]{amsart}
\usepackage{amssymb,amsmath,amsthm,}
  \usepackage{a4wide}
\usepackage{amscd}
\usepackage{amsfonts}
\usepackage{amssymb}
\usepackage{latexsym}
\usepackage{color}
\usepackage{esint}
\usepackage{graphicx}
\usepackage{float}
\graphicspath{{Figures/}}
\usepackage{graphicx}
 \usepackage[makeroom]{cancel}

\newtheorem{theorem}{Theorem}
\newtheorem*{mtheorem}{Main Theorem}
\newtheorem*{existenceThrm}{Existence Theorem of \cite{EGKM}}

\newtheorem*{definitionKD}{Definition of A Kinetic Distance}
\newtheorem{definition}{Definition}
\newtheorem{lemma}{Lemma}
\newtheorem{proposition}[theorem]{Proposition}
\newtheorem{remark}{Remark}
\let\a=\alpha
\let\e=\varepsilon

\let\p=\partial

\let\O=\Omega

\let\o=\omega

\numberwithin{equation}{section}

\let\hide\iffalse
\let\unhide\fi

\DeclareMathAlphabet{\mathpzc}{OT1}{pzc}{m}{it}

\newcommand{\R}{\mathbb{R}}

\renewcommand{\S}{\mathbb{S}}

\newcommand{\be}{\begin{equation}}
\newcommand{\bm}{\begin{multline}}
\newcommand{\ee}{\end{equation}}
\newcommand{\dd}{\mathrm{d}}

\newcommand{\xb}{x_{\mathbf{b}}}
\newcommand{\tb}{t_{\mathbf{b}}}

\newcommand{\X}{\mathbf{x}}
\newcommand{\V}{\mathbf{v}}

\newcommand{\Bes}{\begin{eqnarray*}}
\newcommand{\Ees}{\end{eqnarray*}}
\newcommand{\Be}{\begin{equation} }
\newcommand{\Ee}{\end{equation}}

\def\p{\partial}

\def\O{\Omega}
\def\R{\mathbb{R}}

\def\B{\begin{equation}}
\def\E{\end{equation}}
\def\BN{\begin{eqnarray*}}
\def\EN{\end{eqnarray*}}

\begin{document}

\title{Regularity of Stationary Boltzmann equation in Convex Domains}

 \author{Hongxu Chen}  \address{Department of Mathematics, The University of Wisconsin-Madison, Email: hchen463@wisc.edu}

  \author{Chanwoo Kim}
 \address{Department of Mathematics, The University of Wisconsin-Madison, Email: chanwoo.kim@wisc.edu}

\maketitle

\begin{abstract} Higher regularity estimate has been a challenging question for the Boltzmann equation in bounded domains. Indeed, it is well-known to have ``\textit{the non-existence of a second order derivative at the boundary}'' in \cite{GKTT} even for symmetric convex domains such as a disk or sphere. In this paper we answer this question in the affirmative by constructing the $C^{1,\beta}$ solutions away from the grazing boundary, for any $\beta<1$, to the stationary Boltzmann equation with the non-isothermal diffuse boundary condition in strictly convex domains, as long as a smooth wall temperature has small fluctuation pointwisely.
\end{abstract}

\tableofcontents

\section{Introduction} An interesting physical application of the kinetic theory is its mesoscopic description of the heat transfer of rarefied gas. The \textit{quantitative description of the stationary state} and a derivation of macroscopic models (as the Knudsen number $\mathpzc{Kn}\rightarrow \infty$) can be achieved through the famous steady Boltzmann equation:
\begin{equation}\label{BE_F}
v\cdot \nabla_x F=  \frac{1}{\mathpzc{Kn}}Q(F,F), \ \   (x,v) \in \O \times \R^3,
\end{equation}
where the hard sphere collision operator $Q(F,F)$ takes the form:
\begin{equation}\label{collision operator}
\begin{split}
   Q(F_1,F_2) & := Q_{\text{gain}}(F_1,F_2)-Q_{\text{loss}}(F_1,F_2) \\
     & =\int_{\mathbb{R}^3}\int_{\mathbb{S}^2}|(v-u) \cdot \o|
     \big[F_1(u')F_2(v')-F_1(u)F_2(v) \big]\dd \omega \dd u,
\end{split}
\end{equation}
with $u'=u+[(v-u)\cdot \omega]\omega$, $v'=v-[(v-u)\cdot \omega]\omega$ with $\omega \in \mathbb{S}^2$.

In particular when the gas interacts with a non-isothermal boundary it is well-known that the non-equilibrium steady states can be constructed by the Boltzmann equation (\ref{BE_F}). The kinetic description of the boundary interaction with the gas particles has been extensively investigated in various aspects (see \cite{Maxwell,CIP,Sone} and the references therein). In this paper we are interested in one of the basic and physical conditions, the so-called \textit{diffuse reflection boundary condition}, which takes into account an instantaneous thermal equilibrating with the non-constant wall temperature of a reflecting gas particle:
%
%
\begin{equation}\label{diffuseBC_F}
F(x,v)|_{n(x) \cdot v<0} = \ M_W(x,v)\int_{n(x)\cdot u>0} F(x,u)\{n(x)\cdot u\}\dd u
, \ \ x \in \p\O ,
\end{equation}
where the outward normal at the boundary $\p\O$ is denoted by $n(x)$. Here, we define the wall Maxwellian associated with the described wall temperature $T_W(x)$ at $x\in \partial\Omega$:
\begin{equation}\label{Wall Maxwellian}
M_W(x,v) =\sqrt{\frac{2\pi}{T_W}}M_{1,0,T_W}
:=\frac{1}{2\pi [T_W(x)]^2}e^{-\frac{|v|^2}{2T_W(x)}}.
\end{equation}

Recently a unique stationary solution of (\ref{BE_F}) with (\ref{diffuseBC_F}) in general bounded domains has been constructed in an $L^\infty$-space when the non-constant wall temperature is a small fluctuation around any constant temperature $T_0$ in \cite{EGKM} (See \cite{G} for the construction in convex domains).
Moreover, the authors prove that such non-equilibrium solutions are dynamically and asymptotically stable. We also refer relevant literatures \cite{EM} and the references therein for the PDE aspects of non-equilibrium steady states. As an important application of such construction the authors further derive the Fourier law (Navier-Stokes-Fourier system, more precisely) rigorously as $\mathpzc{Kn} \rightarrow \infty$ in \cite{EGKM A}. On the other hand, for each fixed finite Knudsen number $\mathpzc{Kn}$, they formulate a criterion of the Fourier law in mesoscopic level in \cite{EGKM}. Utilizing the available numeric results, they illustrate the violation of such a criterion, which demonstrates a deviation from the Fourier law for each fixed finite Knudsen number $\mathpzc{Kn}$.

\textit{Qualitatively} the kinetic and macroscopic descriptions of heat transfer are remarkably different in the presence of boundaries in particular. In the absence of fluid velocity flow, a macroscopic description via the Fourier law is given by the Laplace equation with suitable boundary condition, which enjoys analytic smoothness of the solutions. On the other hand, the kinetic description from the Boltzmann equation (\ref{BE_F}) possesses a boundary singularity intrinsically (\cite{Kim}). Such a drastic discrepancy comes from the convection effect $\mathpzc{Kn} v\cdot \nabla_x F$, which has small factor but non-zero for any finite Knudsen number $\mathpzc{Kn}>0$. Indeed, it is very interesting to study the \textit{quantitative} effect of such a convection term $\mathpzc{Kn} v\cdot \nabla_x F$ in the interaction of the boundary and collision process in the limiting process $\mathpzc{Kn}\rightarrow \infty$. Our work in this paper originates from this motivation.

As the first step toward this goal, in this paper we are looking for the smoothness of the stationary Boltzmann equation for fixed $\mathpzc{Kn} \sim 1$, comparable to the regularity of the corresponding (in a sense of $\mathpzc{Kn}\rightarrow \infty$) elliptic equation, for which \textit{the Schauder estimates} are available. More precisely the main purpose of this paper is to develop a robust and unified \textit{higher regularity} estimate in $C^{1,\beta}_x$ with the aid of weights for the stationary Boltzmann solutions to (\ref{BE_F}) with the diffuse reflection boundary condition (\ref{diffuseBC_F}) in the \textit{convex} domains. For this purpose we focus on the convex domain as a discontinuous singularity appears in the non-convex domain in general \cite{Kim,GKTTBV}.

In general convex domains, regularity estimates at most up to the first derivatives away from the so-called \textit{grazing set}
\Be\label{grazing}
\gamma_0:= \{(x,v) \in \p\O \times \R^3: n(x) \cdot v=0\}
\Ee
has been established in \cite{GKTT, yunbai1, yunbai2} for the nonlinear dynamical Boltzmann equation. The key idea of the approach is based on the so-called kinetic distance which is almost invariant along the characteristics. With the aid of such weight a generic singularity $\frac{1}{n(x) \cdot v}$ of the first order derivatives can be controlled. We refer \cite{Ikun} for the regularity of the stationary linear equation up to the first derivatives. \textit{However,} any higher regularity beyond the first order derivatives away from the boundary has been a challenging question. Apparently any second order derivatives estimate seems impossible due to the well-known \textit{``non-existence of second order spatial normal derivative at the boundary''} in \cite{GKTT} even in the convex domain, or even in symmetric domains. We note that the mechanism of such phenomenon is against the conventional effect of the collision in some sense, which will be described in Section 1.2.  Throughout this paper we will use the following notations: $f\lesssim  g \Leftrightarrow$ there exists $0<C<\infty$ such that $0 \leq f\leq Cg$; $f\thicksim g   \Leftrightarrow$ there exists $0<C<\infty$ such that $0 \leq \frac{f}{C}\leq g\leq Cf$; $f \ll g\Leftrightarrow$ there exists a small constant $c>0$ such that $0 \leq  f \leq  cg$;
$f = O(g) \Leftrightarrow$ $|f| \lesssim g$; 
$f = o(g) \Leftrightarrow$ $|f|\ll g.$

\subsection{Main Theorem} Throughout this paper we assume the domain is defined as $\O = \{x \in \R^3: \xi(x) <0\}$ via a $C^3$ function $\xi : \R^3 \rightarrow \R$. Equivalently we assume that for all $q\in \partial \Omega$, there exists a $C^3$ function $\eta_p$ and $0<\delta_1 \ll 1$, such that 
\Be\label{O_p}
\eta_q:
B_+(0; \delta_1)
\ni \mathbf{x}_q := (\mathbf{x}_{q,1},\mathbf{x}_{q,2},\mathbf{x}_{q,3})
 \rightarrow
\mathcal{O}_q:= \eta_q(
B_+(0; \delta_1))
 \ \textit{ is one-to-one and onto},
\Ee
and $\eta_q(\mathbf{x}_q)\in \partial \Omega$ if and only if $\mathbf{x}_{q,3}=0$ within the range of $\eta_q$. We refer to \cite{EGKM A} for the construction of such $\xi$ and $\eta_q$. We further assume that the domain is strictly convex in the following sense:
 \Be\label{convex}
\sum_{i,j=1}^3\zeta_i \zeta_j\p_i\p_j \xi(x) \gtrsim |\zeta|^2  \  \text{ for   all }   x \in \bar{\O}  \text{ and }  \zeta \in \R^3.
 \Ee
Without loss of generality we may assume that $\nabla \xi \neq 0$ near $\p\O$.

In order to control the generic singularity at the boundary we adopt the following weight of \cite{GKTT}:
 \begin{definition} For sufficiently small $0<\e\ll  \Vert \xi\Vert_{C^2}$, we define a \text{kinetic distance}:
\Be\label{kinetic_distance}
\begin{split}
\alpha(x,v)   : = \chi_\e ( \tilde{\alpha}(x,v)  ) ,  \ \
\tilde{\alpha}(x,v)   : = \sqrt{ |v \cdot \nabla_x \xi(x)|^2 - 2 \xi(x) (v \cdot \nabla_x^2 \xi(x) \cdot v)}, \ \ (x,v) \in \bar{\O} \times \R^3,
\end{split}
\Ee
where $\chi_a: [0,\infty) \rightarrow [0,\infty)$ stands for a non-decreasing smooth function such that
\Be\label{chi}
\chi_\a (s) = s \ for \ s \in [0, a ],  \ \chi_a (s) = 2a \ for \ s \in [ 4 a, \infty ], \ and \  | \chi^\prime_a(s) | \leq 1 \ for  \  \tau  \in [0,\infty).
\Ee
\end{definition}
We note that $\alpha\equiv 0$ on the grazing set $\gamma_0$. From a computation, we have $|v\cdot \nabla_x \alpha(x,v)| \leq    |v| \alpha(x,v)$, together with $\tau \chi_\e^\prime(\tau) \leq \chi_\e (\tau)$, this implies
\begin{align}
e^{-|v|s} \alpha(x-sv,v) \leq   \alpha(x,v) \leq e^{|v|s} \alpha(x-sv,v)
\ \text{as long as } x-sv \in \bar{\O}
. \label{Velocity_lemma}
\end{align}
 The definition of $\alpha, \tilde{\alpha}$ in~\eqref{kinetic_distance} implies 
\begin{equation}\label{n geq alpha}
  \tilde{\alpha}(x,v)\gtrsim\alpha(x,v).
\end{equation}


We extend the outward normal in the domain:
  \Be\label{normal}
n(x) := \chi^\prime _{\e/2}( \text{dist} (x, \p\O))  {\nabla \xi(x)}/{|\nabla \xi(x)|} \ \ \text{for all }
x \in \bar{\O}.
\Ee
In particular, we note that $n(x) \equiv 0$ when $\text{dist} (x, \p\O)\geq 2 \e$. In order to explore the ``better'' behavior of the tangential derivative versus the normal derivative we define a $G$-derivative (which is a matrix):
\begin{equation}\label{tang deri}
 \nabla_{\parallel}f(x)=G(x)\nabla_x f(x),
\end{equation}
where
\begin{equation}\label{G}
G(x):=\big(I-n(x)\otimes n(x) \big).
\end{equation}
Note that near the boundary, from~\eqref{normal} we have
\begin{equation}\label{tang times normal}
G(x)n(x)=0  \text{ for } \text{dist}(x,\partial \Omega)\leq \e/2. 
\end{equation}
From the definition of $n$ in (\ref{normal}), the $G$-derivative is actually a full derivative away from the boundary: if $\text{dist} (x, \p\O)\geq 2\e$, then $G(x) \nabla_x= \nabla_x$.

\begin{mtheorem}\label{main_theorem} Fix $\mathpzc{Kn}>0$. Assume the domain is convex (\ref{convex}) and the boundary is $C^3$. Suppose $\sup_{x \in \p\O}|T_W(x)-T_0|\ll  1$ for some constant $T_0>0$ and $T_W(x)  \in C^1(\p\O)$. 
For given $\mathfrak{m}>0$ we construct a unique solution
\Be\label{F}
F(x,v)=\mathfrak{m}M_{1,0,T_0}(v)+ \sqrt{M_{1,0,T_0}(v)} f(x,v) \geq0,
\Ee
 to the stationary Boltzmann equation (\ref{BE_F}) and the diffuse reflection boundary condition (\ref{diffuseBC_F}) such that $\iint_{\O\times \R^3} f \sqrt{M_{1,0,T_0}(v)}   =0$
, and 
 \Be\label{infty_bound}
 \Vert w f\Vert_\infty   \lesssim \| T_W-T_0\|_{L^\infty(\partial\Omega)} , \ \ \ w(v):=e^{\varrho|v|^2} \ \text{for some}  \ 0<\varrho< {1}/{4}.
 \Ee
Moreover, $f$ (and $F$) belongs to $C^1(\bar{\O} \times \R^3\backslash \gamma_0)$ locally and satisfies \begin{align}
 \| w_{\tilde{\theta}}(v) \alpha(x,v)  \nabla_x f (x,v)\|_{L^\infty(\O\times \R^3)}
& \lesssim \| T_W-T_0\|_{C^1(\p\O)},
\label{estF_n}
\\
 \|  w_{\tilde{\theta}/2}(v) |v|   \nabla_\parallel f (x,v)\|_{L^\infty(\O\times \R^3)}& \lesssim \| T_W-T_0\|_{C^1(\p\O)},
 \label{estF_tau}
 \\
  \|w_{\tilde{\theta}}(v)  |v|^2  \nabla_v f (x,v)\|  _{L^\infty(\O\times \R^3)}& \lesssim \| T_W-T_0\|_{C^1(\p\O)}, \label{estF_v}
\end{align}
where $w_{\tilde{\theta}}(v)=e^{\tilde{\theta}|v|^2}$ with $0<\tilde{\theta}\ll \varrho$.

If we further assume $T_W(x)\in C^2(\partial \Omega)$, then for any $0\leq \beta<1$, the solution $F(x,v)$ belongs to $C^{1,\beta}(\bar{\Omega}\times \mathbb{R}^3\backslash \gamma_0)$ locally. Moreover,
\begin{equation}\label{estF_C1beta}
 \sup_{x,y \in \Omega}\Big\Vert w_{\tilde{\theta}}(v)|v|^2\min\Big\{\frac{\alpha(x,v)}{|v|},\frac{\alpha(y,v)}{|v|}\Big\}^{2+\beta} \frac{\nabla_x f(x,v)-\nabla_x f(y,v)}{|x-y|^\beta}\Big\Vert_{L^\infty(\R^3_v)}   \lesssim \Vert T_W-T_0\Vert_{C^2(\p\O)},
\end{equation}
\begin{equation}\label{estF_C1betatang}
 \sup_{x,y \in \Omega}\Big\Vert w_{\tilde{\theta}/2}(v) |v|^2 \min\Big\{\frac{\alpha(x,v)}{|v|},\frac{\alpha(y,v)}{|v|}\Big\}^{1+\beta}\frac{|\nabla_\parallel f(x,v)-\nabla_\parallel f(y,v)|}{|x-y|^\beta}\Big\Vert_
{L^\infty(\R^3_v)} 
\lesssim \Vert T_W-T_0\Vert_{C^2(\p\O)}.
\end{equation} 
\begin{equation}\label{estF C1v betax}
 \sup_{x,y \in \Omega}\Big\Vert w_{\tilde{\theta}}(v)  |v| \min\Big\{\frac{\alpha(x,v)}{|v|},\frac{\alpha(y,v)}{|v|}\Big\}^{1+\beta} \frac{\nabla_v f(x,v)-\nabla_v f(y,v)}{|x-y|^\beta}\Big\Vert_{L^\infty(\R^3_v)}  \lesssim \Vert T_W-T_0\Vert_{C^2(\p\O)}.
\end{equation}
\end{mtheorem}
 \begin{remark}The unique solvability and the pointwise estimate has been established in \cite{EGKM}. We record the statement of the theorem in Section 2 for the sake of readers' convenience.\end{remark}

  \begin{remark} 
 The second estimate (\ref{estF_tau}) implies that any tangential spatial derivatives of $f(x,v)$ does not blow up near the grazing set. Also comparing the $C^{1,\beta}$ estimates~\eqref{estF_C1beta} and \eqref{estF_C1betatang}, the weight in the semi-norm of the tangential spatial derivative has a lower power in terms of $\alpha$ than the one for the normal derivative.
\end{remark}

  \begin{remark} Estimating differential quotient with respect to $v$ has some subtle (probably technical) issue, since the trajectory is not stable at $v=0$ near the boundary. Since our motivation of the paper is investigating the regularity in space we omit to discuss them. This issue (instability of the trajectory at $v=0$ in the H\"older norm estimate) will be discussed in the forthcoming paper \cite{KL_holder}.
\end{remark}

\subsection{Major Difficulties}In this section we illustrate the major difficulties, and in the next sections we will explain the key ideas and analytical development to overcome such obstacles. A generic feature of the boundary problem of the Boltzmann equation is a singularity of solutions, which originates mainly from 1) characteristics feature of the phase boundary $\p\O \times \R^3$ with respect to the transport operator (i.e. the phase boundary is always characteristic but not uniformly characteristic at the grazing set $\gamma_0$ of (\ref{grazing})), and 2) the mixing effect of the collision operator.

The effect of characteristics phase boundary can appear in several ways. Depending on the shape of the domain, the generic boundary singularity at $\gamma_0$ can propagate inside the domain and affect the global dynamics. Indeed, it has been proved in \cite{Kim} that any general non-convex domains admit smooth initial datum will produce the discontinuity for the Boltzmann solution in a stable manner, which propagate along the trajectory. Although such discontinuity-type singularities may stay near the boundary for the convex domains, its derivatives blow up near the grazing set. Actually it is not merely the effect of characteristics phase boundary but also the mixing effect of the collision operator: the mixing immediately produces a singular source term for the normal derivative at the boundary. In \cite{GKTT}, the authors quantify the rate of the blow-up with respect to the kinetic distance of (\ref{kinetic_distance}) and study the mixing effect by the collision operator in terms of the kinetic distance. As a result they establish the first order derivatives estimate for the dynamical Boltzmann equation. On the other hand, the kinetic distance produces \textit{a loss of moment} and they utilize a fast decay weight $e^{-C \langle v\rangle t}$ to recover such a loss. In other words, the success of the approach in \cite{GKTT} to the dynamic problem can be achieved in the space losing its exponential moment quickly (exponentially). Evidently utilizing such functional spaces is not possible in the stationary problem, which is one of the major difficulties to establish the main theorem.

The effect of the nonlinear collision operator is complex, in particular, within the interaction of the transport operator, which eventually restricts our regularity strictly below two derivatives in any $L^p$-space: the boundary singularity of Boltzmann solutions appears as $\frac{\p F}{\p n} \sim \frac{Q(F,F)}{n(x) \cdot v} \notin L^1_{loc}$, while the leading order term of any second order derivatives $\nabla_{x,v} \p_n$ contains 
 a factor of $Q(\nabla _x F, F)(\xb(x,v),v)$ at a backward exit position $\xb(x,v):= x-  \tb(x,v) v$ which is defined through a backward exit time $\tb$:
\Be\label{BET}
 \tb(x,v) :=   \sup \{ s>0:x-sv \in \O\}.
\Ee
Due to a lack of symmetry of $\frac{\p F}{\p n}$, in particular for the diffuse reflection boundary condition, any possibility of cancellation in the integration formula $Q(\frac{\p F}{\p n}, F)$ can be expelled generically in \cite{GKTT}. Then it follows that $|\frac{\p^2 F}{\p n^2} (x,v)| = \infty$ for all $v$ for some $x\in \p\O$. This singularity likely appears at all boundary points with all velocities then propagates along the trajectory inside the domain, and masses up all directional derivatives. Even strictly below the second derivatives estimate, at first glance it is not obvious that the similar failure is avoidable in our weighted $C^{1,\beta}$. Moreover, we encounter similar type of, but much more geometrically involved, terms associated with the diffuse reflection boundary condition intertwined with the transport operator. Such non-integrable singularities could barge in the higher order estimates, which is the other major difficulty of the proof.

\subsection{Regularizing via the mixing of the binary collision, transport, and diffuse reflection} To overcome such difficulties described in Section 1.2., we establish a novel and robust \textit{quantitative} estimate of regularization effect (in space and velocity) of the \textit{velocity mixing} via the diffuse reflection boundary condition (\ref{diffuseBC_F}) or/and the binary collision (\ref{BE_F}) intertwined with the \textit{transport operator}.

We demonstrate the scheme first for $\nabla_x F$, of which the most singular term comes from the boundary contribution such as
%
%
\Be\label{intro_boundary1}
  \nabla_x \xb(x,v)
\int_{n(\xb)\cdot v^1>0} \nabla_{\xb} F(\xb(x,v),v^1)  
|n(\xb(x,v))\cdot v^1|\dd v^1 .  \Ee
Upon using the transport operator once again, the contribution of the collision operator (ignoring the singularity of $Q$ for simplicity) can be viewed as
\Be\label{intro_boundary1_int}
 \nabla_x \xb \int_{n(\xb)\cdot v^1>0} \int^{0}_{\tb(\xb, v^1)} \int_{\mathbb{R}^3}  \nabla_x F(\xb
 -s v^1,u)
 |n(\xb)\cdot v^1|
  \dd u   \dd s
 \dd v^1 .
 \Ee
A key observation is that the $x$-derivative has a natural relation with the $v^1$-derivative as
\Be\label{x_to_v}
\nabla_x F(\xb-(t^1-s)v^1,u)= \frac{\nabla_{v^1} [F(\xb-(t^1-s)v^1,u) ]}{-(t^1-s)}.
\Ee
When $t^1-s$ has a positive lower bound, thanks to the $v^1$-integral from the diffuse reflection boundary condition, we are able to remove such a $v^1$-derivative completely from $F$. As a result of the integration by parts, the singularity of $\nabla_{v^1} \tb(\xb, v^1)$ occurs, which will be compensated by the boundary measure and thus we obtain a bound like $ \nabla_x \xb(x,v) \times  \| F \|_\infty$. When $t^1-s$ is small we use the so-called the nonlocal-to-local estimate and derive $O(|t^1-s|) \alpha(x,v)^{-1} \| \alpha \nabla_xF \|_\infty$. We will describe the nonlocal-to-local estimate and its application in detail at the next subsection.

On the other hand, the boundary contribution of (\ref{intro_boundary1}) upon applying the transport operator appears as
\Be\label{intro_boundary1_bdry}
 \nabla_x \xb(x,v)
\int_{n(\xb)\cdot v^1>0} \nabla_{\xb} F(\xb(x^1,v^1),v^1)  
|n(x^1)\cdot v^1|\dd v^1 ,
\Ee
where $x^1= \xb(x,v)$. The key idea is to convert $v^1$-integration to the integration in $(x^2,  \tb(x^1,v^1) )=(\xb(x^1,v^1), \tb(x^1,v^1) )$, while the change of variables produces a factor of the Jacobian as $\frac{|n(x^2)\cdot v^1|}{\tb(x^1,v^1)^3}$. Then we are able to move $\nabla_{\xb}$-derivative from $F$ via the integration by parts, while the derivative to the geometric components arise. Using the convexity and boundary measure crucially we are able to bound this amount by $ \nabla_x \xb(x,v) \times  \| F \|_\infty$.

\subsection{Higher regularity}For the higher regularity estimate in the weighted $C^{1,\beta}$-space, we 1) adopt the idea of Section 1.3 with stronger weight in $\alpha$, 2) crucially establish a ``better'' estimate for the tangential derivatives, 3) use the full range of the nonlocal-to-local estimate, and 4) carefully study the possibly harmful (which has been explained in the last paragraph of Section 1.2.) term $\frac{1}{|x-y|^\beta} \int^{\tb(x,v)}_{\tb(y,v)}Q(\nabla_x F,F) (x-sv,v) \dd s$.

By expressing $\frac{\nabla_x F(x,v)-\nabla_x F(y,v)}{|x-y|^\beta}$ along the trajectories (see~\eqref{nablatbx-nablatby}-\eqref{partial e-e} for the details), we notice that the difference is singular at least as
\Be\label{nabla xb H}
\frac{\nabla_x\xb(x,v)-\nabla_x\xb(y,v)}{|x-y|^\beta}
\int_{n(\xb)\cdot v^1>0} \nabla_{\xb} F(\xb ,v^1 )  
|n(\xb )\cdot v^1|\dd v^1
,
\Ee
 where the integration is bounded using the weighted $C^1$-estimate. By the mean value type estimate and the computation of $\nabla_x^2 \xb$, for $x(\tau)=\tau x+(1-\tau)y$, we derive that the difference quotient of $\nabla_x \xb$ is bounded by
\begin{equation}\label{difference of nabla xb}
\begin{split}
    |x-y|^{1-\beta} \int^1_0\frac{|v|^3 }{\alpha^3(x(\tau),v)}\dd \tau.
\end{split}
\end{equation}
We prove that $\alpha(x(\tau),v)\gtrsim \min\{\alpha(x,v),\alpha(y,v)\}$ for $|x-y|<\min\{\frac{\alpha(x,v)}{|v|},\frac{\alpha(y,v)}{|v|}\}$ in the convex domains, for which we use the weight of $ \min\{\frac{\alpha(x,v)}{|v|},\frac{\alpha(y,v)}{|v|}\}^{2+ \beta}$ for $\frac{\nabla_x F(x,v)-\nabla_x F(y,v)}{|x-y|^\beta}$. The convexity of the domain is crucial since any similar type of the bound is false for the non-convex domains in general.

Unfortunately this estimate with the weight of the power $2+\beta$ is too singular! In particular the difference quotient of $\nabla_x F$ contains
\Be\notag
\nabla_x \xb(x,v) \int
\frac{\nabla_{\xb} F(\xb(x,v),v^1)
-\nabla_{\xb} F(\xb(y,v),v^1)
}{|x-y|^\beta}
|n(\xb)\cdot v^1|\dd v^1,
\Ee
in which the control of the possible singularity of $|n(\xb) \cdot v^1|^{-(1+ \beta)}$ would be non-integrable. To overcome it, realizing that $\nabla_{\xb} F $ is $\nabla_\parallel F$, we establish an estimate of the difference quotient for the tangential derivatives $\frac{\nabla_\parallel F(x,v)-\nabla_\parallel F(y,v)}{|x-y|^\beta}$ with the weight $\min\{\frac{\alpha(x,v)}{|v|},\frac{\alpha(y,v)}{|v|}\}$ for a lower power than $2+\beta$. The optimal power is examined through (\ref{nabla xb H}), which turns out to be $1+\beta$.

To estimate the difference quotient with different weights, we first employ delicate splitting for the boundary integral and the time integral depending on how the trajectories from two different points hit the boundary. Then we adopt the idea of the scheme of Section 1.3 when $t^1-s$ has a positive lower bound. On the other hand, when $t^1-s$ is small we use the weight and derive
\Be\label{NLL_intro}
\Big\| \min\Big\{\frac{\alpha(x,v)}{|v|},\frac{\alpha(y,v)}{|v|}\Big\}^{\beta}
\nabla F \Big\|_\infty
\int _{\text{small interval}}\int  \frac{1}{ \min\Big\{\frac{\alpha(x-sv,u)}{|u|},\frac{\alpha(y-sv,u)}{|u|}\Big\}^{\beta}}\dd u  \dd s.
\Ee
The second author and collaborators studied a similar estimate of (\ref{NLL_intro}) in \cite{GKTT}. In this paper we elaborate the so-called nonlocal-to-local estimate, which consists of analytical and geometrical arguments: first we study the $u$-integration of the integrand and derive a gain of power such as, for $1<\beta<3$
\Be\label{gain_power_xi}
\frac{1}{\min \{ \xi(x-sv,u) , \xi(y-sv,u)  \}^{\frac{\beta-1}{2}}},
\Ee
where $\xi(x)$ can be understood as the distance from $x$ to the boundary. Second we employ $s \mapsto \xi(x-sv,u)$ with the Jacobian $\dd s = \frac{1}{|u \cdot \nabla \xi|}\dd \xi(x-sv,u)$ and recover a power of $\alpha$ as in the bound of $\xi$ through the geometric velocity lemma. We crucially utilize such a gain of $\alpha$ to extract a smallness in (\ref{NLL_intro}).

Lastly we discuss the possible harmful term $\frac{1}{|x-y|^\beta} \int^{\tb(x,v)}_{\tb(y,v)}Q(\nabla_x F,F) (x-sv,v) \dd s$. First we apply the $\alpha$-weighted bound for $\nabla_x F$ and then establish $Q(\nabla_x F,F) (x-sv,v) \sim \ln |\xi(x-sv)|$. Upon the time integration on $[\tb(y,v), \tb(x,v)]$ we derive a bound $\min\{\frac{\alpha(x,v)}{|v|},\frac{\alpha(y,v)}{|v|}\} \ln \big( \min\{\frac{\alpha(x,v)}{|v|},\frac{\alpha(y,v)}{|v|}\}\big)$. For $|x-y|<\min\{\frac{\alpha(x,v)}{|v|},\frac{\alpha(y,v)}{|v|}\}$, we realize the difference quotient is bounded. Of course such bound blows up if $\beta=1$.

%



 \hide
 One of the generic feature of boundary problem of the Boltzmann equation is a singularity of derivatives of solutions. For easy illustration let us consider a transport equation $v\cdot \nabla_x F=0$ in $(x,v) \in \O \times \R^3$ with an inflow boundary condition $F(x,v)|_{n(x) \cdot v<0}= G(x,v)$ at $x \in \p\O$. A solution is given by an explicit form as $F(x,v)= G(\xb(x,v), v)$, where a backward exit position $\xb:= x-  \tb(x,v) v$ is defined through a backward exit time $\tb$:
\Be\label{BET}
 \tb(x,v) :=   \sup \{ s>0:x-sv \in \O\}.
\Ee
A derivative $\nabla_{x,v} F(x,v)$ certainly inherits a singularity from $\nabla_x \tb(x,v) =  {n(\xb(x,v))}/{(n(\xb(x,v)) \cdot v)}$ and $\nabla_v \tb(x,v) =   -\tb (x,v) n(\xb(x,v))/{(n(\xb(x,v)) \cdot v)}$. For $\tb(x,v) \lesssim |n(\xb(x,v)) \cdot v|/|v|^2$ in a convex domain we expect to have (see (\ref{nabla_tbxb}))
\Be \label{simpleF_n}
n(x)\cdot \nabla_x F (x,v)  \sim |v|  |{n(\xb(x,v)) \cdot v}|^{-1} , \ \
  \nabla_v F(x,v) \sim  |v|^{-2}.
\Ee
 On the other hand, for any spatial derivative perpendicular to $n(x)$, we define
\begin{equation}\label{G}
G(x):=\big(I-n(x)\otimes n(x) \big),
\end{equation}
\begin{equation}\label{tang deri}
 \nabla_{\parallel}f(x)=G(x)\nabla_x f(x).
\end{equation}
Note that we always have
\begin{equation}\label{tang times normal}
G(x)n(x)=0.
\end{equation}
Such derivative is expected to behave as
\Be \label{simpleF_tau}
  \nabla_{\parallel} F (x,v) \sim {|n(x)- n(\xb(x,v))|} / |{n(\xb(x,v)) \cdot v}|  \sim  |v| ^{-1}.
\Ee
These computation clearly indicates a singularity of $n  \cdot \nabla_x F $ near the \textit{grazing set }
\Be\label{grazing}
\gamma_0:= \{(x,v) \in \p\O \times \R^3: n(x) \cdot v=0\}.
\Ee

A major goal of this paper 
is to achieve \textit{optimal} estimates of $\nabla_{x,v} F$, which are predicted by the derivatives of an explicit form of the inflow solution as (\ref{simpleF_n}) and (\ref{simpleF_tau}). To measure the singularity as in (\ref{simpleF_n}) we introduce the following definition.
\begin{definitionKD}
We assume that the domain is convex and the boundary is $C^3$ as in (\ref{convex}). For sufficiently small $0<\e\ll_\O 1$,
\Be\label{kinetic_distance}
\begin{split}
\alpha(x,v)   : = \chi_\e ( \tilde{\alpha}(x,v)  ) ,  \ \
\tilde{\alpha}(x,v)   : = \sqrt{ |v \cdot \nabla_x \xi(x)|^2 - 2 \xi(x) (v \cdot \nabla_x^2 \xi(x) \cdot v)}, \ \ (x,v) \in \bar{\O} \times \R^3.
\end{split}
\Ee
Here $\chi_a: [0,\infty) \rightarrow [0,\infty)$ stands for a non-decreasing smooth function such that
\Be\label{chi}
\chi_\a (s) = s \ for \ s \in [0, a ],  \ \chi_a (s) = 2a \ for \ s \in [ 4 a, \infty ], \ and \  | \chi^\prime_a(s) | \leq 1 \ for  \  s   \in [0,\infty).
\Ee
\end{definitionKD}

We note that $\alpha\equiv 0$ on the grazing set $\gamma_0$. From a computation, we have $|v\cdot \nabla_x \alpha(x,v)| \leq    |v| \alpha(x,v)$. This, together with $\tau \chi_\e^\prime(\tau) \leq \chi_\e (\tau)$, implies
\begin{align}
e^{-|v|s} \alpha(x-sv,v) \leq   \alpha(x,v) \leq e^{|v|s} \alpha(x-sv,v), \ \ 0\leq s \leq \tb(x,v). \label{Velocity_lemma}
\end{align}
\cancel{and thus} From the definition of $\alpha, \tilde{\alpha}$ in~ and $$
\begin{equation}\label{n geq alpha}
  |n(\xb(x,v))\cdot v| \gtrsim \tilde{\alpha}(x,v)\gtrsim\alpha(x,v).
\end{equation}

Since $\tb(x,v) |v|$ is bounded above by a diameter of the domain, $\alpha(x,v)$ is equivalent to $\alpha(\xb(x,v),v)$ as $(C_\O)^{-1} \alpha (x-sv,v)\leq \alpha(x,v)\leq C_\O \alpha (x-sv,v)$ for $C_\O\gg1 $. In particular this implies that $\alpha(x,v)$ is equivalent to $|n(\xb(x,v)) \cdot v|$ which appears in (\ref{simpleF_n}).


\begin{mtheorem}\label{main_theorem} Assume all the assumptions of \textbf{Existence Theorem of \cite{EGKM}}. We further assume that the domain is convex in (\ref{convex}) and the boundary is $C^3$ as in (\ref{O_p}). Suppose $T_W(x)  \in C^1(\p\O)$. 
Then the solution $F(x,v)$ of (\ref{F}) to (\ref{BE_F}) and (\ref{diffuseBC_F}) belongs to $C^1(\bar{\O} \times \R^3\backslash \gamma_0)$. Moreover
\begin{align}
 \| w_{\tilde{\theta}}(v) \alpha(x,v)  \nabla_x f (x,v)\|_{L^\infty(\O\times \R^3)}
& \lesssim \| T_W-T_0\|_{C^1},
\label{estF_n}
\\
 \| w_{\tilde{\theta}/2}(v)|v| \big(I-n(x)\otimes n(x)\big) \nabla_x f (x,v)\|_{L^\infty(\O\times \R^3)}& \lesssim \| T_W-T_0\|_{C^1},
 \label{estF_tau}
 \\
  \| |v|^2  \nabla_v f (x,v)\|  _{L^\infty(\O\times \R^3)}& \lesssim \| T_W-T_0\|_{C^1},\label{estF_v}
\end{align}
where $w_{\tilde{\theta}}(v)=e^{\tilde{\theta}|v|^2}$ with $\tilde{\theta}\ll 1$.

If we further assume $T_W(x)\in C^2(\partial \Omega)$, then for any $0\leq \beta<1$, the solution $F(x,v)$ belongs to $C^{1,\beta}(\bar{\Omega}\times \mathbb{R}^3\backslash \gamma_0)$:
\begin{equation}\label{estF_C1beta}
\Big\Vert w_{\tilde{\theta}}(v)|v|^2\min\{\frac{\alpha(x,v)}{|v|},\frac{\alpha(y,v)}{|v|}\}^{2+\beta} \frac{\nabla_x f(x,v)-\nabla_x f(y,v)}{|x-y|^\beta}\Big\Vert_\infty \lesssim \Vert T_W-T_0\Vert_{C^2},
\end{equation}
\begin{equation}\label{estF_C1betatang}
\Big\Vert w_{\tilde{\theta}/2}(v) |v|^2 \min\{\frac{\alpha(x,v)}{|v|},\frac{\alpha(y,v)}{|v|}\}^{1+\beta}\frac{|\nabla_\parallel f(x,v)-\nabla_\parallel f(y,v)|}{|x-y|^\beta}\Big\Vert_\infty\lesssim \Vert T_W-T_0\Vert_{C^2}.
\end{equation}

Moreover,
\begin{equation}\label{estF C1v betax}
\Big\Vert |v|\min\{\frac{\alpha(x,v)}{|v|},\frac{\alpha(y,v)}{|v|}\}^{1+\beta} \frac{\nabla_v f(x,v)-\nabla_v f(y,v)}{|x-y|^\beta}\Big\Vert_\infty \lesssim \Vert T_W-T_0\Vert_{C^2}.
\end{equation}

\end{mtheorem}

\begin{remark} For any $\tau \in\R^3$ with $\tau  \cdot n(x)=0$, we have $\tau ^t  (Id_{3 \times 3}-  n(x)\otimes  n(x) ) \nabla_x f= \tau^t  - \tau^t  n(x) n(x)^t \nabla_x f= \tau^t \nabla_x f = \tau \cdot \nabla_x f$. Therefore the second estimate (\ref{estF_tau}) implies that any tangential spatial derivatives of $f(x,v)$ is bounded in $L^\infty$ without any $\alpha$-weight. Also comparing the $C^{1,\beta}$ estimates~\eqref{estF_C1beta},\eqref{estF_C1betatang}, the weight in the tangential spatial derivative has a lower power in terms of $\alpha$.
\end{remark}

 An analog question for the dynamic problems has been recently answered by the second author and his collaborators in \cite{GKTT}. Namely, they construct a solution of the dynamic Boltzmann equation which is $C^1$ away from the grazing set (\ref{grazing}). Although one perhaps see a likelihood of analog result in the stationary problem, there are substantial difficulties which will be described behind the statement of the main theorem. Conceptually in the stationary setting one cannot rely on Gronwall-type estimate as the regularity estimate of dynamic problem in \cite{GKTT}. Moreover the result of \cite{GKTT} is essentially a local-in-time result. For the stationary regularity recently the authors obtained a weighted $C^1$ result for the linear Boltzmann equation in \cite{Ikun}.

Now we illustrate the main ideas of the proof of \textbf{Main Theorem}.

The main difficulty is to estimate the contribution of the collision operator in~\eqref{fx_x_4} and the boundary in~\eqref{fx_x_3}.

The collision operator is $h=K(f)+\Gamma(f,f)$, which is defined in~\eqref{linear operator}. To estimate $K(f)$, we use the definition of $K(f)$ in Lemma~\ref{Lemma: K,Gamma}. In~\eqref{fx_x_4} the contribution of $K(f)$ is approximately
\[\int^t_0 \int_{\mathbb{R}^3} \mathbf{k}(v,u)\underbrace{\nabla_x f(x-(t-s)v,u)}_{(*)}\dd u \dd s.\]
We utilize the $u$-integration and trace back along the trajectory again. To be more specific, we express (*) by~\eqref{fx_x_1}-\eqref{fx_x_5}. The key term to be estimated is the contribution of the double collision operator~\eqref{fx_x_4}, where the $K(f)$ is involved again. This term has a multiple fold integral as
\begin{equation}\label{double duhamel}
\int^t_0 \dd s\int_{\mathbb{R}^3}\mathbf{k}(v,u)\dd u\int^s_0 \dd s'\int_{\mathbb{R}^3}\mathbf{k}(u,u')\nabla_x f(x-(t-s)v-(s-s')u,u')\dd u'.
\end{equation}
A key observation is we can exchange the $x$-derivative into $u$-derivative:
\[\nabla_x f(x-(t-s)v-(s-s')u,u')=-\frac{\nabla_u [f(x-(t-s)v-(s-s')u,u')]}{s-s'}.\]
When $s-s'\geq \e$ has a lower bound, thanks to the $u$-integral from $K(f)$, we can perform an integration by parts for $\dd u$ and thus remove the derivative. Note that $f$ has an $L^\infty$ control from the \textbf{Existence Theorem} and thus we control \eqref{double duhamel} with $s-s'\geq \e$ by $O(\e^{-1})\Vert wf\Vert_\infty$. On the other hand, if $s-s'\leq \e$ does not have a lower bound, the $\dd s'$ corresponds to a small time integral and we can bound \eqref{double duhamel} by a small number $\e$ with the weighted $C^1$ norm $\Vert \alpha\nabla_x f\Vert_\infty$. To be more specific, we include the $\alpha$ weight to have
\[\eqref{double duhamel}\mathbf{1}_{s-s'\leq \e}\leq \Vert \alpha\nabla_x f\Vert_\infty\int^t_0 \dd s \int_{\mathbb{R}^3}\dd u\mathbf{k}(v,u)\underbrace{\int^s_{s-\e}\dd s' \int_{\mathbb{R}^3}\dd u' \frac{\mathbf{k}(u,u')}{\alpha(x-(t-s)v,u)} }_{(**)}.\]
Then we need to analyze the nonlocal to local estimate (**). This type of estimate has been studied in~\cite{GKTT} with a different power $\alpha^{p}$ with $3>p>1$. Here we are able generalize the result to power $p=1$ and extract a small number from the small time integral, for detail one can see Lemma \ref{Lemma: NLN}. Thus we bound~\eqref{double duhamel} with $s-s'\leq \e$ by $\frac{O(\e)\Vert \alpha\nabla_x f\Vert_\infty}{\alpha(x,v)}$.

Then we discuss the boundary condition~\eqref{fx_x_3}. Let $g$ be the diffuse BC in~\eqref{diffuse_f intro}, by a proper change of variable we can bring the derivative inside the integration:
\[\int_{n(\xb)\cdot v^1>0}\underbrace{\nabla_x f(\xb,v^1)}_{(***)}   \sqrt{\mu(v^1)}|n(\xb)\cdot v^1|\dd v^1 .  \]
Here we utilize the $v^1$ integral. We trace back along the trajectory by expressing (***) by~\eqref{fx_x_1}-\eqref{fx_x_5} . The contribution of the collision operator~\eqref{fx_x_4} can be handled in a similar way as~\eqref{double duhamel}. To be more specific, the contribution of the~\eqref{fx_x_4} reads
\[\int_{n(\xb)\cdot v^1>0} \int^{t^1}_0 \dd s\int_{\mathbb{R}^3} \mathbf{k}(v^1,u)\nabla_x f(\xb-(t^1-s)v^1,u)  \dd u \sqrt{\mu(v^1)}|n(\xb)\cdot v^1|\dd v^1 .\]
Since we have already had a $v^1$-integration from the boundary condition, we can exchange the $x$-derivative to $v^1$-derivative and perform the integration by parts for $\dd v^1$.

Then the main term to be estimated is the contribution of the boundary~\eqref{fx_x_4}, which means the trajectory hit the boundary for the second time. We use the notation $x^2$(defined in~\eqref{xi}) to denote such backward exit position. This term reads
\begin{equation}\label{double bdr}
 \int_{n(x^1)\cdot v^1>0} \nabla_{x^2} f(x^2,v^1)  \sqrt{\mu(v^1)}|n(x^1)\cdot v^1|\dd v^1 .
\end{equation}

Our idea is still integrating by part to remove the derivative $\nabla_{x^2}$ . To accomplish this goal, we observe that starting from a fixed position $x^1$, the map from $v^1$ to $(x^2,\tb^2)$ is one to one: $x^2=x^1-\tb^2 v^1$, thus we can perform the change of variable to change the $\dd v^1$ integral to $\dd x^2 \dd \tb^2$ integral. The change of variable and corresponding Jacobian is presented in Lemma \ref{Lemma: change of variable}. Then
\[\eqref{double bdr}=\int_{S_{x^2}}\int  \nabla_{x^2} f(x^2,v^1)  \sqrt{\mu(v^1)}|n(x^1)\cdot v^1|\times  \text{Jacobian } \dd \tb^2 \dd x^2 ,\]
and we integrate by part for $\dd x^2$ to remove the derivative $\nabla_{x^2}$. Thus with the control of $f$,~\eqref{double bdr} is bounded by $\Vert wf\Vert_\infty.$

For the tangential derivative~\eqref{estF_tau} and velocity derivative~\eqref{estF_v}, we have a better estimate for the derivative to $\xb$ as stated in~\eqref{simpleF_n}( for detail see Lemma \ref{Lemma: nabla tbxb} ). Thus we can use the Duhamel's formula with the integration by parts technique to obtain a better estimate, where the weight contains less singularity as in~\eqref{estF_tau} and~\eqref{estF_v}.

Now we illustrate the main idea of the weighted $C^{1,\beta}$ estimate. We use similar idea as~\eqref{double bdr}. More specifically, we express the difference $\frac{\nabla_x f(x,v)-\nabla_x f(y,v)}{|x-y|^\beta}$ using the difference of~\eqref{fx_x_1}-\eqref{fx_x_5}( see~\eqref{nablatbx-nablatby}-\eqref{partial e-e} for detail ). Among all the difference terms, the difference of the derivative of the backward exit position(or time) will generate higher singularity:
\begin{equation}\label{difference of nabla xb}
\begin{split}
    &\frac{\nabla_x\xb(x,v)-\nabla_x\xb(y,v)}{|x-y|^\beta} =\frac{\int_0^1 \nabla_\tau [\nabla_x\xb(x(\tau),v)] \dd \tau}{|x-y|^\beta} \\
     & =\frac{\int_0^1 \nabla_x \nabla_x\xb(x(\tau),v) \nabla_\tau x(\tau) \dd \tau}{|x-y|^\beta}=|x-y|^{1-\beta}\int_0^1 \nabla_x \nabla_x\xb(x(\tau),v) \nabla_\tau x(\tau) \dd \tau,
\end{split}
\end{equation}
where we denote $x(\tau)=\tau x+(1-\tau)y.$ As $\nabla_x \xb$ already contains singularity, \eqref{difference of nabla xb} is expected to have a higher singularity. This requires us to use a weight with higher power in terms of $\alpha$. Here the key factor to be estimated is the second derivative of $\xb(x(\tau),v)$ for $x(\tau)$ lying between $x$ and $y$. Since $\nabla_x \xb$ can be explicitly computed as in Lemma~\ref{Lemma: nabla tbxb}, we can directly take one more derivative to compute $\nabla_x \nabla_x \xb(x(\tau),v)$, approximately the computation reads
\begin{equation}\label{alpha cubic}
\nabla_x \frac{1}{n(\xb(x(\tau)),v)\cdot v}=O(1)\frac{\nabla_x \xb(x(\tau),v)}{|n(\xb(x(\tau)),v)\cdot v|^2}=\frac{O(1)}{\alpha^3(x(\tau),v)}.
\end{equation}
When $|x-y|$ does not have lower bound, we need to derive an estimate for $\alpha(x(\tau),v)$. Here we need a proper cutoff for $|x-y|$. We set the cutoff to be $\min\{\alpha(x,v),\alpha(y,v)\}$. When $|x-y|>\min\{\alpha(x,v),\alpha(y,v)\}$, combining with $\nabla_x \xb(x,v)$ the singularity in~\eqref{difference of nabla xb} is $\frac{1}{\min\{\alpha(x,v),\alpha(y,v)\}^{1+\beta}}$. Such singularity is already covered in the weight in~\eqref{estF_C1beta}. When $|x-y|<\min\{\alpha(x,v),\alpha(y,v)\}$, we can derive a lower bound for $\alpha(x(\tau),v)$ as:
\[\alpha(x(\tau),v)\gtrsim \min\{\alpha(x,v),\alpha(y,v)\}.\]
For detail we refer readers to Lemma \ref{Lemma: x-y}. Then counting the singularity using~\eqref{difference of nabla xb} and~\eqref{alpha cubic}, we need a weight with power $2+\beta$ as in~\eqref{estF_C1beta}.

Another difficulty occurs when we estimate the boundary. The term corresponds to the difference of the boundary condition is $\frac{\nabla_x f(\xb(x,v),v)-\nabla_x f(\xb(y,v),v)}{|\xb(x,v)-\xb(y,v)|^\beta}$. We express this term using the boundary condition~\eqref{diffuse_f intro}, approximately we have
\[\int_{n(\xb)\cdot v^1>0}\underbrace{\frac{\nabla_x f(\xb(x,v),v^1)-\nabla_x f(\xb(y,v),v^1)}{|\xb(x,v)-\xb(y,v)|^\beta}}_{(****)}|n(\xb)\cdot u|\sqrt{\mu(u)}\dd u.  \]
Then we express (****) along the trajectory again using the difference of~\eqref{fx_x_1}-\eqref{fx_x_5}. Similarly to~\eqref{double duhamel}, we get one term corresponds to the case that the trajectory hit the boundary twice. Roughly this term reads
\begin{equation}\label{bdr-bdr intro}
\int_{n(\xb(x,v))\cdot v^1>0}   \underbrace{\frac{\nabla_{\xb^2(x)} f(\xb^2(x),v^1)-\nabla_{\xb^2(y)} f(\xb^2(y),v^1)}{|\xb^2(x)-\xb^2(y)|^\beta}}_{\eqref{bdr-bdr intro}_*}|n(\xb)\cdot u|\sqrt{\mu(u)}\dd u.
\end{equation}
Here $\xb^2(x)$ and $\xb^2(y)$ are defined in~\eqref{second backward}, which denote the second backward exit position for the trajectory starting from $x$ and $y$ respectively. \eqref{bdr-bdr intro} looks similar to~\eqref{double bdr}, however, after removing the derivative using the same integration by parts technique, we have one more term that will generate singularity:
\[\frac{f(\xb^2(x),v^1)-f(\xb^2(y),v^1)}{|\xb^2(x)-\xb^2(y)|^\beta}.\]
We use $\Vert \alpha\nabla_x f\Vert_\infty$ to control this term and due to the weight $\alpha$ there is an extra singularity $\frac{1}{\alpha^\beta}$. Indeed due to this extra singularity,~\eqref{bdr-bdr intro} is no longer integrable after the integration by parts. To deal with this difficulty, we split this singularity into two cases. When $|n(\xb)\cdot v^1|>\e$ has a lower bound, clearly we can integrate by part to remove the derivative and do not need to concern the singularity.

When $|n(\xb)\cdot v^1|<\e$, we are not going to apply the integration by parts technique. Instead we aim to extract a small number from this integration. We note that $\xb^2(x),\xb^2(y)\in \partial\Omega$ are on the boundary, which corresponds to a surface with two dimension after transformation( see~\eqref{O_p}~\eqref{orthogonal} ), thus $\nabla_{\xb^2(x)} f(\xb^2(x),v^1)-\nabla_{\xb^2(y)} f(\xb^2(y),v^1)$ indeed corresponds to the tangential derivative. Due to~\eqref{simpleF_tau} and~\eqref{nabla_tbxb} the estimate for the tangential derivative is expected to have less singularity. Thus$~\eqref{bdr-bdr intro}_*$, which represents the tangential derivative, is expected to have a less singularity with power $1+\beta$ as in~\eqref{estF_C1betatang}. Combining with an extra term $|n(\xb)\cdot v^1|$ in the integrand, the singularity in~\eqref{bdr-bdr intro} becomes $\frac{1}{\min\{\alpha(\xb(x,v),v^1),\alpha(\xb(y,v),v^1)\}^{\beta}}$, which is integrable since $\beta<1$. Moreover, under the condition $|n(\xb)\cdot v^1|<\e$, this integral generates a small number $O(\e)$ with the norm~\eqref{estF_C1betatang}.

Hence in order to obtain the weighted $C^{1,\beta}$ estimate~\eqref{estF_C1beta}, we will estimate the weighted $C^{1,\beta}$ estimate for the tangential derivative~\eqref{estF_C1betatang} at the same time. The estimate for~\eqref{estF_C1betatang} will be similar as we will take the difference of $\nabla_\parallel f(x,v)$ and $\nabla_\parallel f(y,v)$ using the Duhamel's formula, and such term are expected to have a weighted with power $1+\beta$. For detail we refer readers to Proposition \ref{Prop: C1beta}.\unhide

Below we state the outline for our paper. In section 2 we prove several lemmas which serve as preliminary. Section 3 and Section 4 are devoted to establish the ideas in Section 1.3 as well as the nonlocal-to-local estimate and (\ref{estF_n}).  In Section 5 and Section 6 we establish the rest of weighted $C^1$ estimates. Finally, in Section 7 we prove the weighted $C^{1,\beta}$ estimate.

\section{Preliminary}

\subsection{Basic Notions}

\hide
\begin{equation}\label{lesssim}
\textcolor{red}{f \lesssim  g \Leftrightarrow \text{there exists $0<C<\infty$ such that } f\leq Cg.}
\end{equation}
\begin{equation}\label{sim}
\textcolor{red}{f\thicksim g   \Leftrightarrow \text{there exists $0<C<\infty$ such that }
\frac{f}{C}\leq g\leq Cf.}
\end{equation}
\begin{equation}\label{big O}
\textcolor{red}{f = O(g) \Leftrightarrow \text{there exists $0<C<\infty$ such that } f = Cg.}
\end{equation}
\begin{equation}\label{little o}
\textcolor{red}{f = o(g) \Leftrightarrow \text{there exists $c \ll 1$ such that } f = cg.}
\end{equation}
\unhide

We record the unique existence theorem of \cite{EGKM}:
\begin{existenceThrm}\label{theorem_EGKM} Assume the domain is open bounded and the boundary is smooth. For $\mathfrak{m}>0$ and $0<\varrho< {1}/{4}$, if $\sup_{x \in \p\O}|T_W(x)-T_0|\ll  1$, then there exists a unique mild solution
\Be\label{F}
F(x,v)=\mathfrak{m}M_{1,0,T_0}(v)+ \sqrt{M_{1,0,T_0}(v)} f(x,v) \geq0,
\Ee
 with $\iint_{\O\times \R^3} f \sqrt{M_{1,0,T_0}(v)}  =0$ to (\ref{BE_F}) and (\ref{diffuseBC_F}) such that  
 \Be\label{infty_bound}
 \Vert w f\Vert_\infty   \lesssim \| T_W-T_0\|_{L^\infty(\partial\Omega)} , \ \ \ w(v):=e^{\varrho|v|^2} \ \text{with}  \ 0<\varrho< {1}/{4}.
 \Ee
\end{existenceThrm}

Without loss of generality, we assume $\mathfrak{m}=1, T_0=1$ in~\eqref{F}. Then we define the reference global Maxwellian and its perturbation:
\[\mu:=M_{1,0,1}, \ \ F(x,v)=\mu(v)+\sqrt{\mu(v)}f(x,v).\]
Plugging~\eqref{F} into~\eqref{BE_F} and~\eqref{diffuseBC_F} we obtain the equation and boundary condition for $f$:
 \begin{equation}\label{f}
 v\cdot \nabla_x f+\nu(v)f=K(f)+\Gamma(f,f),
 \end{equation}
 \begin{equation}\label{diffuse_f intro}
 f(x,v)|_{n(x)\cdot v<0}=\frac{M_W(x,v)}{\sqrt{\mu(v)}}\int_{n(x)\cdot u>0}f(x,u) \sqrt{\mu(u)} \{n(x)\cdot u\}\dd u+r(x,v).
 \end{equation}
 Here $\nu(v),K(f),\Gamma(f,f)$ are the linear Boltzmann operator(see~\cite{R}) given by
 \begin{equation}\label{linear operator}
   \nu(v)f=-\frac{Q(\mu,\sqrt{\mu}f)}{\sqrt{\mu}},\quad K(f)=\frac{Q(\sqrt{\mu}f,\mu)}{\sqrt{\mu}},\quad \Gamma(f,f)=\frac{Q(\sqrt{\mu}f,\sqrt{\mu}f)}{\sqrt{\mu}}.
 \end{equation}
The $r(x,v)$ is the remainder term. By $\sqrt{2\pi}\int_{n(x)\cdot u>0}
\sqrt{\mu(u)} \{n(x)\cdot u\}\dd u= 1$, this term is given by
\begin{equation}\label{remainder term}
r(x,v):=\frac{M_W(x,v)/\sqrt{2\pi}-\mu(v)}{\sqrt{\mu(v)}}.
\end{equation}

Consider a linear transport equation with the inflow boundary condition
\begin{align}
v\cdot \nabla_x f + \nu(v) f &= h(x,v), \ \ (x,v) \in \O \times \R^3,
\label{inhomo_transport}
\\
f( x,v)&= g(x,v), \ \  (x,v )\in\gamma_-. \label{inflowBC}
\end{align}
As we can not rely on the Gronwall-type estimate, we will use the Duhamel's formula to express the equation along the trajectory:
\begin{equation}\label{trajectory}
\begin{split}
  f(x,v) & = \mathbf{1}_{t\geq \tb}e^{-\nu(v)\tb(x,v)} f(\xb(x,v),v) \\
   & +\mathbf{1}_{t<\tb} e^{-\nu(v)t}f(x-tv,v) \\
   & +\int^t_{\max\{0,t-\tb\}} e^{-\nu(v)(t-s)}h(x-(t-s)v,v) \dd s.
\end{split}
\end{equation}
Here we fix $t\gg 1$.

In order to obtain $C^1$ estimate we take the spatial derivative to~\eqref{trajectory} to get
\begin{align}
\p_{x_j}   f(x,v)&=  \mathbf{1}_{t\geq \tb } e^{-\nu(v)\tb(x,v)}\p_{x_j}   [f(\xb(x,v),v) ]
 \label{fx_x_1}
\\
    & -  \mathbf{1}_{t\geq \tb } \nu(v) \p_{x_j} \tb(x,v)  e^{-\nu(v)\tb(x,v)} f(\xb(x,v),v)
 \label{fx_x_2}
    \\
    &
    +\mathbf{1}_{t<\tb }e^{-\nu(v)t} \p_{x_j} [ f(x-tv,v)]
     \label{fx_x_3}
    \\
    & +\int_{\max\{0,t-\tb \}}^t e^{-\nu(v)(t-s)} \p_{x_j}  [h (   x-(t-s) v,v) ]\dd s
     \label{fx_x_4}\\
    &- \mathbf{1}_{t\geq \tb}\partial_{x_j} \tb e^{-\nu(v)\tb}h(x-\tb v,v) ,\label{fx_x_5}
\end{align}
where $\xb(x,v)$ and $\tb(x,v)$ represent the backward exit position and time which are defined in~\eqref{BET}. The derivative of $\tb(x,v)$ and $\xb(x,v)$ has singular behavior as stated in~\eqref{nabla_tbxb}, such singularity will be cancelled by our weight $\alpha$ defined in~\eqref{kinetic_distance}. With a compatibility condition it is standard to check the piecewise formula (\ref{fx_x_1})-(\ref{fx_x_5}) is actually a weak derivative of $f$ and continuous across $\{t=\tb(x,v)\}$ (see \cite{GKTTBV}) for the details.

\begin{definition}\label{definition: sto cycle}

 Recall the backward exit position $\xb$ and backward exit time $\tb$ in~\eqref{BET}, we define a stochastic cycles as $(x^0,v^0)= (x,v) \in \bar{\O} \times \R^3$ and inductively
\begin{align}
&x^1:= \xb(x,v), \   v^1 \in \{v^1\in \mathbb{R}^3:n(x^1)\cdot v^1>0\} , \label{vi}\\
&  v^{k}\in \{v^{k}\in \mathbb{R}^3:n(x^k)\cdot v^k>0\}, \ \ \text{for} \  k \geq 1, 
\label{xi}
\\
 &x^{k+1} := \xb(x^k, v^k) , \ \tb^{k}:= \tb(x^k,v^k) \ \ \text{for} \  n(x^k) \cdot v^k\geq 0  . \label{tbi}
\end{align}
Choose $t\geq0.$ We define $t^0=t$ and
\Be\label{ti}
t^{k} =  t-  \{ \tb + \tb^1 + \cdots + \tb^{k-1}\},  \ \ \text{for} \  k \geq 1.
\Ee

\end{definition}

\begin{remark}
 Here $x^{k+1}$ depends on $(x,v,x^1,v^1,\cdots, x^k,v^k)$, while $v^k$ is a free parameter whose domain~\eqref{xi} only depends on $x^k$. 
\end{remark}

 Recall \eqref{O_p}. Since the boundary is compact and $C^3$, for fixed $0<\delta_1 \ll 1$ we may choose a finite number of $p \in \tilde{\mathcal{P}} \subset\p\O$ and $0<\delta_2\ll 1$ such that $\mathcal{O}_p=\eta_p(
B_+(0; \delta_1)) \subset B(p;\delta_2) \cap \bar{\O}$ and $\{\mathcal{O}_p \}$ forms a finite covering of $\partial \Omega$. We further choose an interior covering $\mathcal{O}_0 \subset \O$ such that $\{ \mathcal{O}_p\}_{p \in \mathcal{P}}$ with $\mathcal{P} = \tilde{\mathcal{P}}\cup \{0\}$ forms an open covering of $\bar\O$. 
We define a partition of unity 
\Be\label{iota} 
\mathbf{1}_{\bar\O} (x)= 
\sum_{p \in \mathcal{P}} \iota_p(x) 
 \text{ such that }0 \leq \iota_p(x) \leq 1, \ \ 
  \iota_p(x) \equiv 0 \ \text{for} \ x  \notin \mathcal{O}_p.
 \Ee
 Without loss of generality (see \cite{KL}) we can always reparametrize $\eta_p$ such that $\partial_{\mathbf{x}_{p,i}} \eta_p \neq 0$ for $i=1,2,3$ at $\mathbf{x}_{p,3}=0$, and an \textit{orthogonality} holds as
\Be\label{orthogonal}
  \partial_{\mathbf{x}_{p,i}}\eta_p \cdot \partial_{\mathbf{x}_{p,j}}\eta_p =0 \ \ \text{at} \ \ \mathbf{x}_{p,3}=0 \text{  for  } i\neq j \text{ and } i,j\in \{1,2,3\}. 
\Ee
 For simplicity, we denote 
\begin{equation}\label{partial_i eta}
 \partial_i \eta_p(\mathbf{x}_p): = \partial_{\mathbf{x}_{p,i}} \eta_p. 
\end{equation}

\begin{definition} For $x \in \bar{\O}$, we choose $p \in\mathcal{P}$ as in (\ref{O_p}). We define
\begin{align}
    T_{\mathbf{x}_p}&
    =\left(
                               \begin{array}{ccc}
           \frac{\p_1 \eta_p(\mathbf{x}_p)}{\sqrt{g_{p,11}(\mathbf{x}_p) }}
           &      \frac{\p_2 \eta_p(\mathbf{x}_p)}{\sqrt{g_{p,22}(\mathbf{x}_p) }}
           &     \frac{\p_3 \eta_p(\mathbf{x}_p)}{\sqrt{g_{p,33}(\mathbf{x}_p) }}
            \\
                               \end{array}
                             \right)^t,\label{T}
\end{align}
with $g_{p,ij}(\mathbf{x}_p)    =\langle \partial_i \eta_p(\mathbf{x}_p),\partial_j \eta_p(\mathbf{x}_p)\rangle$ for $i,j\in \{1,2,3\}$. 
Here $A^t$ stands the transpose of a matrix $A$. Note that when $\mathbf{x}_{p,3}=0$, $T_{\mathbf{x}_p}       \frac{\p_i \eta_p(\mathbf{x}_p)}{\sqrt{g_{p,ii}(\mathbf{x}_p) }}
  = e_i$ for $i=1,2,3$ where $\{e_i\}$ is a standard basis of $\R^3$.

We define
\Be\label{bar_v}
\mathbf{v}_j(\mathbf{x}_p) = \frac{\p_j \eta_p(\mathbf{x}_p)}{\sqrt{g_{p,jj}(\mathbf{x}_p) }} \cdot  v.
\Ee
\end{definition}

We note that from (\ref{orthogonal}), the map $T_{\mathbf{x}_p}$ is an orthonormal matrix when $\mathbf{x}_{p,3}=0$. Therefore both maps $v \rightarrow \mathbf{v} (\mathbf{x}_p )$ and $\mathbf{v} (\mathbf{x}_p ) \rightarrow v$ have a unit Jacobian. This fact induces a new representation of boundary integration of diffuse boundary condition in (\ref{diffuse_f intro}): For $x \in \p\O$ and $p \in\mathcal{P}$ as in (\ref{O_p}),
\begin{equation}
\begin{split}
\label{eqn: diffuse for f}
 \int_{n(x)\cdot v>0}f(x,v)\sqrt{\mu(v)}\{n(x)\cdot v\}\dd v
 = \int_{\mathbf{v} _{p ,3}>0}f( \eta_{p } (\mathbf{x}_{p }    ), T^t_{\mathbf{x} _{p }} \mathbf{v} ( \mathbf{x}_p) )\sqrt{\mu(\mathbf{v} (\mathbf{x}_p))}\mathbf{v} _{ 3}(\mathbf{x}_p) \dd\mathbf{v}  (\mathbf{x}_p).
\end{split}
\end{equation}
We have used the fact of $\mu(v )=\mu(|v |)=\mu(|T^t_{\mathbf{x}_{p } }\mathbf{v} (\mathbf{x}_p) |)=\mu(|\mathbf{v} (\mathbf{x}_p)  |)=\mu(\mathbf{v} (\mathbf{x}_p) ) $ and $\mathbf{x}_{p,3}=0$. 

 Now we reparametrize the stochastic cycle using the local chart defined in Definition \ref{definition: sto cycle}. 

\begin{definition}\label{definition: chart}
Recall the stochastic cycles (\ref{xi}). For each cycle $x^k$ let us choose $p^k \in \mathcal{P}$ in (\ref{O_p}). Then we denote
\Be\begin{split}\label{xkvk}
\mathbf{x}^k_{p^k}&:= (\mathbf{x}^k_{p^k,1}, \mathbf{x}^k_{p^k,2},0)   \text{ such that }
\eta_{p^k} (\mathbf{x}^k_{p^k}) = x^k, \ \ \text{for} \  k \geq 1,
\\
\mathbf{v}^k_{p^k,i}&:=   \frac{\p_j \eta_{p^k}(\mathbf{x}_{p^k}^k)}{\sqrt{g_{p^k,jj}(\mathbf{x}_{p^k}^k) }} \cdot  v^k ,  \ \ \text{for} \  k \geq 1.
\end{split}
\Ee
Conventionally we denote
\Be\label{x0v0}
\mathbf{x}^0_{p^0}:= x^0=x, \  \ \mathbf{v}^0_{p^0}:= v^0=v.
\Ee

We define
 \Be
\p_{\mathbf{x}^{k}_{p^{k},i}}[
a( \eta_{p^{k}} ( \mathbf{x}_{p^{k} }^{k}  ),    {v}^{k} ) ]
 :=
\frac{\p \eta_{p^{k}}(\mathbf{x}^{k}_{p^{k},i})}{\p \mathbf{x}^{k}_{p^{k},i}}
\cdot \nabla_x a ( \eta_{p^{k}} ( \mathbf{x}_{p^{k} }^{k}  ), v^{k}) , \ \ i=1,2.
\label{fBD_x1}
\Ee
Conventionally we denote $\underline{\nabla}_{x^k} a(x^k, v^k)
=\big( \p_{\mathbf{x}^{k}_{p^{k},1}}[
a( \eta_{p^{k}} ( \mathbf{x}_{p^{k} }^{k}  ),    {v}^{k} ) ], \p_{\mathbf{x}^{k}_{p^{k},2}}[
a( \eta_{p^{k}} ( \mathbf{x}_{p^{k} }^{k}  ),    {v}^{k} ) ]
\big).
$
\end{definition}

\subsection{Properties of stochastic cycle}
 
In this subsection we list useful properties of the stochastic cycle defined in Definition \ref{definition: sto cycle} and Definition \ref{definition: chart}. 

\begin{lemma}\label{Lemma: nabla tbxb}
For the $\tb$ and $\xb$ defined in~\eqref{xi} and~\eqref{tbi}, the derivative reads
\begin{equation}\label{xi deri tb}
\frac{\p \tb^{k+1}}{\p x_j^{k+1}} =
	\frac{1}{ \V_{ 3 }(\mathbf{x}_{p^{k+2}})  }
	\frac{  \p_{3} \eta_{{p}^{k+2}}(x^{k+2})}{\sqrt{g_{{p}^{k+2},33}(x^{k+2})}} \cdot
	e_j, 
\end{equation}

\begin{equation}\label{vi deri tb}
\frac{\partial \tb^{k+1}}{\partial v_j^{k+1}}=-\frac{\tb^{k+1}e_j}{\mathbf{v}^{k+1}_{p^{k+2},3}}\cdot \frac{\partial_3 \eta_{p^{k+2}}}{\sqrt{g_{p^{k+2},33}}}\Big|_{x^{k+2}}.
\end{equation}
And thus
\Be
\begin{split}\label{nabla_tbxb}
\nabla_x \tb = \frac{n(\xb)}{n(\xb) \cdot v},\ \
\nabla_v \tb = - \frac{\tb n(\xb)}{n(\xb) \cdot v},\\
\nabla_x \xb = Id_{3\times 3} - \frac{n(\xb) \otimes v}{n(\xb) \cdot v},\ \
\nabla_v \xb = - \tb Id + \frac{ \tb n(\xb) \otimes v}{n(\xb) \cdot v}.
\end{split}\Ee

For $i=1,2$,
\begin{equation}\label{xi deri xbp}
\frac{\partial \mathbf{x}_{p^{k+2},i}^{k+2}}{\partial x^{k+1}_j}=\frac{1}{\sqrt{g_{p^{k+2}, ii} (\mathbf{x}^{k+2}_{p^{k+2} } )}}
\left[
\frac{\p_{i} \eta_{p^{k+2}} (\mathbf{x}^{k+2}_{p^{k+2} } ) }{\sqrt{g_{p^{k+2},ii}(\mathbf{x}^{k+2}_{p^{k+2} } ) }}
- \frac{\mathbf{v}_{p^{k+2}, i}}{\mathbf{v}_{p^{k+2}, 3}}
\frac{\p_{3} \eta_{p^{k+2}} (\mathbf{x}^{k+2}_{p^{k+2} } )  }{\sqrt{g_{p^{k+2},33}(\mathbf{x}^{k+2}_{p^{k+2} } )}}
\right] \cdot e_j,
\end{equation}

\begin{equation}\label{xip deri xbp}
\frac{\p \mathbf{x}^{k +2}_{p^{k+2},i}}{\p{\mathbf{x}^{k+1  }_{p^{k+1 },j}}} = \frac{1}{\sqrt{g_{p^{k+2}, ii} (\mathbf{x}^{k +2}_{p^{k+2} } )}}
\left[
\frac{\p_{i} \eta_{p^{k+2}} (\mathbf{x}^{k +2}_{p^{k+2} } ) }{\sqrt{g_{p^{k+2},ii}(\mathbf{x}^{k +2}_{p^{k+2} } ) }}
- \frac{\mathbf{v}^{k+2}_{p^{k+2}, i}}{\mathbf{v}^{k+2}_{p^{k+2}, 3}}
\frac{\p_{3} \eta_{p^{k+2}} (\mathbf{x}^{k +2}_{p^{k+2} } )  }{\sqrt{g_{p^{k+2},33}(\mathbf{x}^{k +2}_{p^{k+2} } )}}
\right] \cdot \p_j \eta_{p^{k+1}}(\mathbf{x}^{k +1}_{p^{k+1} } ) ,
\end{equation}

\begin{equation}\label{vi deri xbp}
\frac{\partial \mathbf{x}^{k+2}_{p^{k+2},i}}{\partial v_j^{k+1}}=-\tb^{k+1}e_j \cdot \frac{1}{\sqrt{g_{p^{k+2},ii}(\mathbf{x}_{p^{k+2}}^{k+2}))}} \Big[\frac{\partial_i \eta_{p^{k+2}}}{\sqrt{g_{p^{k+2},ii}}}\Big|_{x^{k+2}}- \frac{\mathbf{v}_{p^{k+1},i}^{k+1}}{\mathbf{v}_{p^{k+2},3}^{k+1}}\frac{\partial_3 \eta_{p^{k+2}}}{\sqrt{g_{p^{k+2},33}}}\Big|_{x^{k+2}} \Big].
\end{equation}

\end{lemma}

\begin{proof}
First of all we have
\begin{equation}\label{xk+2 xk+1}
\begin{split}
x^{k+2}=\eta_{p^{k+2}}(\mathbf{x}_{p^{k+2}}^{k+2})&=x^{k+1}-\tb^{k+1}v^{k+1} \\
&=\eta_{p^{k+1}}(\mathbf{x}_{p^{k+1}}^{k+1})-\tb^{k+1}v^{k+1}.
\end{split}
\end{equation}

\textit{Proof of~\eqref{xi deri tb}.} We take $\frac{\partial}{\partial x_j^{k+1}}$ to~\eqref{xk+2 xk+1} to get
\begin{equation}\label{x:partial x_k}
 \sum_{l=1,2}\frac{\partial \mathbf{x}^{k+2}_{p^{k+2},l}}{\partial x_j^{k+1}}\frac{\partial \eta_{p^{k+2}}}{\partial \mathbf{x}^{k+2}_{p^{k+2},l}}\Big|_{x^{k+2}}=-\tb^{k+1}\frac{\partial v^{k+1}}{\partial x_j^{k+1}}-\frac{\partial \tb^{k+1}}{\partial x_j^{k+1}}v^{k+1}+e_j=-\frac{\p \tb^{k+1}}{\p x_j^{k+1}}v^{k+1}+e_j.
\end{equation}

Then we take an inner product with $\frac{\p_{{3}} \eta_{{p}^{k+2}}}{\sqrt{g_{{p}^{k+2},33}}}\Big|_{x^{k+2}}$ to~\eqref{x:partial x_k} to have
	\begin{equation} \label{x:diff pos iden dot 3}
	\begin{split}
 & \sum_{l=1,2} \frac{\p \X^{k+2}_{{p}^{k+2},l}}{\p x_j^{k+1}} \frac{\p\eta_{{p}^{k+2}}}{\p \X^{k+2}_{{p}^{k+2},l}}\Big\vert_{x^{k+2}}  \cdot
		\frac{\p_{{3}} \eta_{{p}^{k+2}} }{\sqrt{g_{{p}^{k+2},33}}}   \Big\vert_{x^{k+2}}   =-   \frac{\p \tb^{k+1}}{\p x_j^{k+1}} v^{k+1} 	
		\cdot \frac{ \p_{3} \eta_{{p}^{k+2}}}{\sqrt{g_{{p}^{k+2},33}}}\Big|_{ x^{k+2}} +
		e_j  \cdot \frac{ \p_{3} \eta_{{p}^{k+2}}}{
			\sqrt{g_{p^{k+2},33}}
		}\Big|_{x^{k+2}} .
	\end{split}
	\end{equation}
Due to~\eqref{orthogonal} the LHS equals zero. Now we consider the RHS. From~\eqref{bar_v}
	\[ v^{k+1} \cdot \frac{ \p_{3} \eta_{{p}^{k+2}}}{\sqrt{g_{{p}^{k+2},33}}}\big|_{ x^{k+2}} = \V_{3}(\mathbf{x}_{p^{k+2}}). \]
	From (\ref{x:diff pos iden dot 3}), we conclude~\eqref{xi deri tb}.

\textit{Proof of~\eqref{vi deri tb}.} We apply $\partial v^{k+1}_j$ to~\eqref{xk+2 xk+1} and take $\cdot \frac{\partial_3 \eta_{p^{k+2}}}{\sqrt{g_{p^{k+2},33}}}\Big|_{x^{k+2}}$ to have
\begin{align*}
  \frac{\partial x^{k+2}}{\partial v^{k+1}_j}\cdot \frac{ \partial_3 \eta_{p^{k+2}}} {\sqrt{g_{p^{k+2},33}}}\Big|_{x^{k+2}} = & \sum_{l=1}^2   \frac{\partial \mathbf{x}^{k+2}_{p^{k+2},l}}{\partial v_j^{k+1}} \frac{\partial \eta_{p^{k+2}}(\mathbf{x}_{p^{k+2}}^{k+2})}{\partial \mathbf{x}^{k+2}_{p^{k+2},l}}\cdot \frac{ \partial_3 \eta_{p^{k+2}} }{\sqrt{g_{p^{k+2},33}}}\Big|_{x^{k+2}} \\
  = & -\Big\{\tb^{k+1}e_j+ v^{k+1}\frac{\partial \tb^{k+1}}{\partial v_j^{k+1}} \Big\}\cdot \frac{\partial_3 \eta_{p^{k+2}}}{\sqrt{g_{p^{k+2},33}}}\Big|_{x^{k+2}}.
\end{align*}
Thus we apply~\eqref{orthogonal} and~\eqref{xkvk} and use~\eqref{bar_v} to obtain~\eqref{vi deri tb}.

\textit{Proof of~\eqref{nabla_tbxb}.} The first line of~\eqref{nabla_tbxb} follows directly from~\eqref{xi deri tb} and~\eqref{vi deri tb}. For the second line we take $\partial x_j^{k+1}$ and $\partial v_j^{k+1}$ to~\eqref{xk+2 xk+1}. Again using~\eqref{xi deri tb} and~\eqref{vi deri tb} we conclude~\eqref{nabla_tbxb}.

\textit{Proof of~\eqref{xi deri xbp}.} We take inner product with $\frac{\partial_i \eta_{p^{k+2}}}{g_{p^{k+2},ii}}\Big|_{x^{k+2}}$ to~\eqref{x:partial x_k} to have
\begin{equation*}
  \begin{split}
     & \sum_{l=1,2}\frac{\partial \mathbf{x}^{k+2}_{p^{k+2},l}}{\p x_j ^{k+1}}\frac{\partial \eta_{p^{k+2}}}{\partial \mathbf{x}^{k+2}_{p^{k+2},l}}\Big|_{x^{k+2}}\cdot \frac{\partial_i \eta_{p^{k+2}}}{g_{p^{k+2},ii}}\Big|_{x^{k+2}}=\frac{\partial \mathbf{x}_{p^{k+2},i}^{k+2}}{\partial x_j ^{k+1}} = -\frac{\partial \tb^{k+1}}{\partial x_j^{k+1}}  v^{k+1} \cdot \frac{\partial_i \eta_{p^{k+2}}}{g_{p^{k+2},ii}}\Big|_{x^{k+2}}+e_j\cdot \frac{\partial_i \eta_{p^{k+2}}}{g_{p^{k+2},ii}}\Big|_{x^{k+2}}.
  \end{split}
\end{equation*}

By~\eqref{bar_v},
\[v^{k+1}\cdot \frac{\partial_i \eta_{p^{k+2}}}{g_{p^{k+2},ii}}\Big|_{x^{k+2}}=\frac{\mathbf{v}_i(\mathbf{x}_{p^{k+2}})}{\sqrt{g_{p^{k+2},ii}}}.\]
Then from ~\eqref{xi deri tb} we conclude~\eqref{xi deri xbp}.

\textit{Proof of~\eqref{xip deri xbp}.} Since
\[\frac{\partial \mathbf{x}^{k+2}_{p^{k+2},i}}{\partial \mathbf{x}_{p^{k+1},j}^{k+1}}=\nabla_{x^{k+1}} \mathbf{x}_{p^{k+2},i}^{k+2}\cdot \partial_{\mathbf{x}_{p^{k+1},j}^{k+1}}\eta_{p^{k+1}}(\mathbf{x}_{p^{k+1}}^{k+1}),\]
by~\eqref{xi deri xbp} we conclude~\eqref{xip deri xbp}.

\textit{Proof of~\eqref{vi deri xbp}.} For $i=1,2$, $j=1,2,3$, we apply $\partial v^{k+1}_j$ to~\eqref{xk+2 xk+1} and take $\cdot \frac{\partial_i \eta_{p^{k+2}}}{\sqrt{g_{p^{k+2},ii}}}\Big|_{x^{k+2}}$ to obtain
\begin{align*}
  \frac{\partial x^{k+2}}{\partial v^{k+1}_j}\cdot \frac{\partial_i \eta_{p^{k+2}}}{\sqrt{g_{p^{k+2},ii}}}\Big|_{x^{k+2}}= & \sum_{l=1}^2   \frac{\partial \mathbf{x}^{k+2}_{p^{k+2},l}}{\partial v_j^{k+1}} \frac{\partial \eta_{p^{k+2}}(\mathbf{x}_{p^{k+2}}^{k+2})}{\partial \mathbf{x}^{k+2}_{p^{k+2},l}}\cdot \frac{\partial_i \eta_{p^{k+2}}}{\sqrt{g_{p^{k+2},ii}}}\Big|_{x^{k+2}} \\
 =  & \frac{\partial \mathbf{x}^{k+2}_{p^{k+2},i}}{\partial v_j^{k+1}} \sqrt{g_{p^{k+2},ii}(\mathbf{x}_{p^{k+2}}^{k+2}))}\\
  = &-\Big\{\frac{\partial \tb^{k+1}}{\partial v_j^{k+1}}v^{k+1}+\tb^{k+1}\frac{\partial v^{k+1}}{\partial v^{k+1}_j} \Big\} \cdot \frac{\partial_i \eta_{p^{k+2}}}{\sqrt{g_{p^{k+2},ii}}}\Big|_{x^{k+2}}\\
  = &-\Big\{\tb^{k+1}e_j-\frac{\tb^{k+1}e_j}{\mathbf{v}_{p^{k+2},3}^{k+1}}\cdot \frac{\partial_3 \eta_{p^{k+2}}}{\sqrt{g_{p^{k+2},33}}}\Big|_{x^{k+2}} v^{k+1}\Big\}\cdot \frac{\partial_i \eta_{p^{k+2}}}{\sqrt{g_{p^{k+2},ii}}}\Big|_{x^{k+2}}.
\end{align*}
Then we apply~\eqref{vi deri tb} obtain~\eqref{vi deri xbp}. \end{proof}

The following two lemmas are immediate consequences of Lemma \ref{Lemma: nabla tbxb}.

\begin{lemma}\label{Lemma: nabla tbxb bounded}
\begin{equation}\label{tb bounded}
\tb(x,v)\lesssim \frac{|n(\xb(x,v))\cdot v|}{|v|^2},
\end{equation}
and thus
\begin{equation}\label{nablav tb xb}
|\nabla_v \tb| \lesssim \frac{1}{|v|^2}, \quad |\nabla_v \xb| \lesssim \frac{1}{|v|}.
\end{equation}
\begin{equation}\label{v deri of T}
|\nabla_v T^t_{\mathbf{x}_p^1}|\lesssim \frac{\Vert \eta\Vert_{C^2}}{|v|}.
\end{equation}

\end{lemma}

\begin{proof} Clearly, \eqref{tb bounded} follows from $\frac{n(\xb)\cdot v}{|v|}\gtrsim x-\xb=\tb |v|$.

By~\eqref{nabla_tbxb} and~\eqref{tb bounded} we have
\[|\nabla_v \tb|\lesssim \frac{|n(\xb)\cdot v|}{|n(\xb)\cdot v|}\frac{1}{|v|^2}\lesssim \frac{1}{|v|^2},\]
\[|\nabla_v \xb|\lesssim \frac{|n(\xb)\cdot v|}{|v|^2}+\frac{|n(\xb)\cdot v|  |v|}{|n(\xb)\cdot v||v|^2}\lesssim \frac{1}{|v|}.\]

For~\eqref{v deri of T} by the definition of $T_{\mathbf{x}_p}$ in~\eqref{T}, and using~\eqref{vi deri xbp}, we have
\begin{align*}
  |\nabla_v T^t_{\mathbf{x}_p^1}| & \lesssim \Vert \eta\Vert_{C^2}\times  |\nabla_v [\mathbf{x}^1_{p^1,1}+\mathbf{x}^1_{p^2,1}]|  \lesssim \Vert \eta\Vert_{C^2}  \frac{|v|\tb}{|n(\eta_{p^1}(\mathbf{x}_{p^1}^1))\cdot v|}\lesssim \frac{\Vert \eta\Vert_{C^2}}{|v|}  ,
\end{align*}
where we have used~\eqref{tb bounded} in the last inequality.

Then the lemma follows.

\end{proof}

\begin{lemma}\label{Lemma: change of variable} The following map is one-to-one
\Be\label{map_v_to_xbtb}
v^{k+1} \in   \{ n(x^{k+1}) \cdot v^{k+1} >0: \xb(x^{k+1},v^{k+1}) \in B(p^{k+2}, \delta_2)\} \mapsto
(\mathbf{x}^{k+2}_{p^{k+2},1}, \mathbf{x}^{k+2}_{p^{k+2},2}, \tb^{k+1}),
\Ee
with
\Be\label{jac_v_to_xbtb}
\det\left(\frac{\p (\mathbf{x}^{k+2}_{p^{k+2},1}, \mathbf{x}^{k+2}_{p^{k+2},2}, \tb^{k+1})}{\p v^{k+1}}\right)= \frac{1}{\sqrt{ g_{p^{k+2},11}(\mathbf{x}^{k+2}_{p^{k+2}})  g_{p^{k+2},22}(\mathbf{x}_{p^{k+2}}^{k+2}) }}
\frac{|\tb^{k+1}|^3}{  |n(x^{k+2}) \cdot v^{k+1}| }.
\Ee

\end{lemma}
\begin{proof}
Combining~\eqref{vi deri tb} and~\eqref{vi deri xbp} we conclude
\begin{equation*}
  \begin{split}
  &\det\left(\frac{\p (\mathbf{x}^{k+2}_{p^{k+2},1}, \mathbf{x}^{k+2}_{p^{k+2},2}, \tb^{k+1})}{\p v^{k+1}}\right)\\
      &= |\tb^{k+1}|^3\det\left(
                                     \begin{array}{c}
                                       -\frac{1}{\mathbf{v}^{k+1}_{p^{k+2},3}} \frac{\partial_3 \eta_{p^{k+2}}}{\sqrt{g_{p^{k+2},33}}}\Big|_{x^{k+2}} \\
                                       \frac{1}{\sqrt{g_{p^{k+2},11}(\mathbf{x}_{p^{k+2}}^{k+2})}} \Big[\frac{\partial_1 \eta_{p^{k+2}}}{\sqrt{g_{p^{k+2},11}}}\Big|_{x^{k+2}}- \frac{\mathbf{v}_{p^{k+1},1}^{k+1}}{\mathbf{v}_{p^{k+2},3}^{k+1}}\frac{\partial_3 \eta_{p^{k+2}}}{\sqrt{g_{p^{k+2},33}}}\Big|_{x^{k+2}} \Big] \\
                                                                              \frac{1}{\sqrt{g_{p^{k+2},22}(\mathbf{x}_{p^{k+2}}^{k+2})}} \Big[\frac{\partial_2 \eta_{p^{k+2}}}{\sqrt{g_{p^{k+2},22}}}\Big|_{x^{k+2}}- \frac{\mathbf{v}_{p^{k+1},2}^{k+1}}{\mathbf{v}_{p^{k+2},3}^{k+1}}\frac{\partial_3 \eta_{p^{k+2}}}{\sqrt{g_{p^{k+2},33}}}\Big|_{x^{k+2}} \Big] \\
                                     \end{array}
                                   \right) \\
       &=  -|\tb^{k+1}|^3\frac{1}{\mathbf{v}^{k+1}_{p^{k+2},3}}    \frac{1}{\sqrt{g_{p^{k+2},11}(\mathbf{x}_{p^{k+2}}^{k+2})g_{p^{k+2},22}(\mathbf{x}_{p^{k+2}}^{k+2})}} \frac{\partial_3 \eta_{p^{k+2}}}{\sqrt{g_{p^{k+2},33}}}\Big|_{x^{k+2}}   \\
        &\quad \cdot \bigg(\Big[\frac{\partial_1 \eta_{p^{k+2}}}{\sqrt{g_{p^{k+2},11}}}\Big|_{x^{k+2}}- \frac{\mathbf{v}_{p^{k+1},1}^{k+1}}{\mathbf{v}_{p^{k+2},3}^{k+1}}\frac{\partial_3 \eta_{p^{k+2}}}{\sqrt{g_{p^{k+2},33}}}\Big|_{x^{k+2}} \Big]\times \Big[\frac{\partial_2 \eta_{p^{k+2}}}{\sqrt{g_{p^{k+2},22}}}\Big|_{x^{k+2}}- \frac{\mathbf{v}_{p^{k+1},2}^{k+1}}{\mathbf{v}_{p^{k+2},3}^{k+1}}\frac{\partial_3 \eta_{p^{k+2}}}{\sqrt{g_{p^{k+2},33}}}\Big|_{x^{k+2}} \Big] \bigg)\\
      & =\frac{1}{\sqrt{g_{p^{k+2},11}(\mathbf{x}_{p^{k+2}}^{k+2})g_{p^{k+2},22}(\mathbf{x}_{p^{k+2}}^{k+2})}}\frac{|\tb^{k+1}|^3}{\mathbf{v}^{k+1}_{p^{k+2},3}}=~\eqref{jac_v_to_xbtb},
   \end{split}
\end{equation*}
where we have used ~\eqref{orthogonal}.

Now we prove the map~\eqref{map_v_to_xbtb} is one to one. Assume that there exists $v$ and $\tilde{v}$ satisfy $\xb(x^{k+1},v)=\xb(x^{k+1},\tilde{v})$ and $\tb(x^{k+1},v)=\tb(x^{k+1},\tilde{v})$. We choose $p\in \partial \Omega$ near $\xb(x^{k+1},v)$ and use the same parametrization. Then by an expansion, for some $\bar{v}\in \overline{\tilde{v}v},$
\[0=\left(
      \begin{array}{c}
        \nabla_v \mathbf{x}_{p,1}(x^{k+1},\tilde{v})    \\
         \nabla_v \mathbf{x}_{p,2}(x^{k+1},\tilde{v})\\
        \nabla_v \tb(x^{k+1},\tilde{v})  \\
      \end{array}
    \right)-\left(
              \begin{array}{c}
              \nabla_v \mathbf{x}_{p,1}(x^{k+1},v)    \\
               \nabla_v \mathbf{x}_{p,2}(x^{k+1},v)   \\
               \nabla_v \tb(x^{k+1},v)    \\
              \end{array}
            \right)=\left(
                      \begin{array}{c}
                        \nabla_v \mathbf{x}_{p,1}(x,\bar{v}) \\
                      \nabla_v \mathbf{x}_{p,2}(x,\bar{v})   \\
                  \nabla_v \tb(x,\bar{v})       \\
                      \end{array}
                    \right)(\tilde{v}-v).
\]
This equality can be true only if the determinant of the Jacobian matrix equals zero. Then~\eqref{jac_v_to_xbtb} implies that $\tb(x^{k+1},\bar{v})=0$. But this implies $x^{k+1}=\xb(x^{k+1},\bar{v})$ and hence $n(x^{k+1})\cdot \bar{v}=0$ which is out of our domain.

\end{proof}

The next lemma describe the properties of a convex domain.
\begin{lemma}\label{Lemma: nv<v2}
Given a $C^2$ convex domain defined in~\eqref{convex},
\Be\label{nv<v2}
\begin{split}
|n_{p^{k+j}} (\mathbf{x}_{p^{k+j}} ^{k+j}) \cdot  (x^{k+1} -
 \eta_{p^{k+2}} (\mathbf{x}_{p^{k+2}} ^{k+2})
 )| \sim  |x^{k+1} -
 \eta_{p^{k+2}} (\mathbf{x}_{p^{k+2}} ^{k+2})
 |^2,
 \ \  &j=1,2,
 \\
  {|\mathbf{v}^{k+1}_{p^{k+1}, 3}|  }  / {|\mathbf{v}^{k+1}_{p^{k+1}  }|}  \sim |x^{k+1} -
 \eta_{p^{k+2}} (\mathbf{x}_{p^{k+2}} ^{k+2})
 |.
 \end{split}
\Ee
For $j'=1,2$,
\Be\label{bound_vb_x}
  \bigg| \frac{\p[ n_{p^{k+j}} (\mathbf{x}_{p^{k+j}} ^{k+j}) \cdot  (x^{k+1} -
 \eta_{p^{k+2}} (\mathbf{x}_{p^{k+2}} ^{k+2})
 )]}{\p {\mathbf{x}_{p^{k+2},j^\prime}^{k+2}}}  \bigg| \lesssim \| \eta \|_{C^2}
 |x^{k+1} -
 \eta_{p^{k+2}} (\mathbf{x}_{p^{k+2}} ^{k+2})|
 ,
  \ \ j=1,2.
\Ee
\end{lemma}

\begin{proof}
First we prove~\eqref{nv<v2}. By Taylor's expansion, for $x,y\in \partial \Omega$ and some $0\leq t\leq 1$, 
\begin{align*}
  \xi(y)-\xi(x)  & =0-0=\nabla \xi(x)\cdot (y-x) + \frac{1}{2}(y-x)^T \nabla^2 \xi(x+t(y-x)) (y-x). 
\end{align*}
 Thus from~\eqref{normal} 
\[ |n(x)\cdot (x-y)| \sim  (y-x)^T \nabla^2 \xi(x+t(y-x)) (y-x).  \]

From the convexity~\eqref{convex}, we have
\[|n_{p^{k+j}} (\mathbf{x}_{p^{k+j}} ^{k+j}) \cdot  (x^{k+1} -
 \eta_{p^{k+2}} (\mathbf{x}_{p^{k+2}} ^{k+2})
 )|\geq C_\Omega |x^{k+1} -
 \eta_{p^{k+2}} (\mathbf{x}_{p^{k+2}} ^{k+2})
 |^2.\]
 Since $\xi$ is $C^2$ at least, 
\[|\{x^{k+1}-y\}\cdot n(x^{k+1})|\leq \Vert \xi\Vert_{C^2}|x^{k+1}-y|^2.\]

Also notice that
\[|n_{p^{k+1}} (\mathbf{x}_{p^{k+1}} ^{k+1}) \cdot  (x^{k+1} -
 \eta_{p^{k+2}} (\mathbf{x}_{p^{k+2}} ^{k+2})
 )|=|\mathbf{v}^{k+1}_{p^{k+1}, 3}| (t^{k+2}-t^{k+1}),\]
thus
\begin{equation*}
  \begin{split}
    \frac{|\mathbf{v}^{k+1}_{p^{k+1}, 3}|  }{|\mathbf{v}^{k+1}_{p^{k+1}  }|}\geq  &  \frac{1}{|\mathbf{v}^{k+1}_{p^{k+1}  }|}\frac{C_\Omega}{|t^{k+1}-t^{k+2}|}\Big|x^{k+1}-x^{k+2} \Big|^2\\
     = & C_\Omega|x^{k+1}-x^{k+2}|
     =  C_\Omega |x^{k+1}-\eta_{p^{k+2}}(\mathbf{x}_{p^{k+2}}^{k+2})|.
  \end{split}
\end{equation*}
By the same computation we can easily conclude
\[\frac{|\mathbf{v}^{k+1}_{p^{k+1}, 3}|  }{|\mathbf{v}^{k+j}_{p^{k+j}  }|} \leq C_\xi |x^{k+1}-\eta_{p^{k+2}}(\mathbf{x}_{p^{k+2}}^{k+2})| .\]

Then prove~\eqref{bound_vb_x}. For $j=1$, $j'=1,2$ we have
\begin{equation}\label{tang*nor}
\begin{split}
   & \bigg| \frac{\p[ n_{p^{k+1}} (\mathbf{x}_{p^{k+1}} ^{k+1}) \cdot  (x^{k+1} -
 \eta_{p^{k+2}} (\mathbf{x}_{p^{k+2}} ^{k+2})
 )]}{\p {\mathbf{x}_{p^{k+2},j^\prime}^{k+2}}}  \bigg|\leq |n_{p^{k+1}} (\mathbf{x}_{p^{k+1}} ^{k+1}) \cdot \partial_{j'}\eta_{p^{k+2}}(\mathbf{x}_{p^{k+2}} ^{k+2})| \\
   & =\Big|n_{p^{k+1}} (\mathbf{x}_{p^{k+1}} ^{k+1}) \cdot \partial_{j'}\eta_{p^{k+1}}(\mathbf{x}_{p^{k+1}} ^{k+1})+n_{p^{k+1}} (\mathbf{x}_{p^{k+1}} ^{k+1}) \cdot \big[\partial_{j'}\eta_{p^{k+2}}(\mathbf{x}_{p^{k+2}} ^{k+2})-\partial_{j'}\eta_{p^{k+1}}(\mathbf{x}_{p^{k+1}} ^{k+1})\big]\Big|\\
   & \leq 0+ \Vert \eta\Vert_{C^2} |x^{k+1} -
 \eta_{p^{k+2}} (\mathbf{x}_{p^{k+2}} ^{k+2})|,
\end{split}
\end{equation}
where we applied~\eqref{orthogonal}.

For $j=2$, we have
\begin{equation*}
  \begin{split}
     \bigg| \frac{\p[ n_{p^{k+2}} (\mathbf{x}_{p^{k+2}} ^{k+2}) \cdot  (x^{k+1} -
 \eta_{p^{k+2}} (\mathbf{x}_{p^{k+2}} ^{k+2})
 )]}{\p {\mathbf{x}_{p^{k+2},j^\prime}^{k+2}}}  \bigg|\lesssim &  |n_{p^{k+2}} (\mathbf{x}_{p^{k+2}} ^{k+2}) \cdot \partial_{j'}\eta_{p^{k+2}}|+\Vert \eta\Vert_{C^2}|x^{k+1} -
 \eta_{p^{k+2}} (\mathbf{x}_{p^{k+2}} ^{k+2})| \\
     = & \Vert \eta\Vert_{C^2}|x^{k+1} -
 \eta_{p^{k+2}} (\mathbf{x}_{p^{k+2}} ^{k+2})|,
  \end{split}
\end{equation*}
where we applied~\eqref{orthogonal}.

\end{proof}

\subsection{Properties of tangential derivative} 
Aiming the regularity estimate of ~\eqref{estF_tau} without the $\alpha-$weight, we establish several properties of the tangential derivative. We summarize them in Lemma \ref{Lemma: equivalent} - \ref{Lemma: ndot nabla xb}.

\begin{lemma}\label{Lemma: equivalent}
For $x=\eta_{p}(\mathbf{x}_p)\in \partial\Omega$, we have the following equivalence:
\begin{equation}\label{equivalent}
|G(x)\nabla_x f(x,v)|\thicksim  \sum_{j=1,2} \partial_{\mathbf{x}_{p,j}} f(\eta_p(\mathbf{x}_p),v).
\end{equation}
\end{lemma}

\begin{proof}
By~\eqref{T} we have
\[\partial_i \eta_p(\mathbf{x}_p)=\sqrt{g_{p,ii}(\mathbf{x}_p)}T^t_{\mathbf{x}_p}e_i.\]
Denote $\mathfrak{F}(x)=\nabla_x f(x,v) T^t_{\mathbf{x}_p}$, we have
\begin{align*}
 \sum_{j=1,2}\partial_{\mathbf{x}_{p,j}}f(\eta_p(\mathbf{x}_p),v)&= \sqrt{g_{p,11}(\mathbf{x}_p)}\nabla_x f(x,v) T^t_{\mathbf{x}_p}e_1+  \sqrt{g_{p,22}(\mathbf{x}_p)}\nabla_x f(x,v) T^t_{\mathbf{x}_p}e_2
    \\
   & =  \sqrt{g_{p,11}(\mathbf{x}_p)}\mathfrak{F}e_1+\sqrt{g_{p,22}(\mathbf{x}_p)}\mathfrak{F}e_2.
\end{align*}

We also have
\begin{align*}
  G(x)\nabla_x f(x,v) &=\nabla_x f \Big(T_{\mathbf{x}_p}^t T_{\mathbf{x}_p}- T^t_{\mathbf{x}_p}e_3 e_3^t T_{\mathbf{x}_p} \Big)  \\
   & =\mathfrak{F}\Big(T_{\mathbf{x}_p}-e_3e_3^t T_{\mathbf{x}_p} \Big)=\mathfrak{F}\Big(I-e_3\otimes e_3 \Big)T_{\mathbf{x}_p}\\
   &=\left(
       \begin{array}{ccc}
         \mathfrak{F}e_1 & \mathfrak{F}e_2 & 0 \\
       \end{array}
     \right)T_{\mathbf{x}_p} = \mathfrak{F}e_1 \frac{\partial_1 \eta_p(\mathbf{x}_p)}{\sqrt{g_{p,11}(\mathbf{x}_p)}}+\mathfrak{F}e_2 \frac{\partial_2 \eta_p(\mathbf{x}_p)}{\sqrt{g_{p,22}(\mathbf{x}_p)}}.
\end{align*}
Since $\partial_1 \eta_p(\mathbf{x}_p) \perp \partial_2 \eta_p(\mathbf{x}_p)$,
\[|G(x)\nabla_x f(x,v)|\thicksim \sqrt{g_{p,11}(\mathbf{x}_p)}\mathfrak{F}e_1+\sqrt{g_{p,22}(\mathbf{x}_p)}\mathfrak{F}e_2\thicksim \sum_{j=1,2}\partial_{\mathbf{x}_{p,j}}f(\eta_p(\mathbf{x}_p),v).\]


\end{proof}

\begin{lemma}\label{Lemma: nx-nxb}
For any $s\in [0,\tb]$, we have
\begin{equation}\label{n(x)-n(xb)}
|G(x)-G(x-sv)| \lesssim \frac{\tilde{\alpha}(x,v)}{|v|}.
\end{equation}

And thus
\begin{equation}\label{Gf bdd}
|G(x)\nabla_x f(x-sv)|\lesssim \frac{\Vert w_{\tilde{\theta}/2}|v|\nabla_\parallel f\Vert+\Vert w_{\tilde{\theta}}\alpha\nabla_x f\Vert_\infty}{|v|w_{\tilde{\theta}/2}(v)}.
\end{equation}

\end{lemma}

\begin{proof}
By the definition~\eqref{G}
\[|G(x)-G(x-sv)|\leq \big|n(x-sv)\otimes (n(x-sv)-n(x))\big|+\big|(n(x-sv)-n(x))\otimes n(x)\big|.\]
\begin{align*}
 |\nabla_x G(x)-\nabla_x G(x-sv)|  &  |\nabla_x n(x-sv)\otimes (n(x-sv)-n(x))|+ | n(x-sv)\otimes \nabla_x (n(x-sv)-n(x))| \\
   & +|\nabla_x (n(x-sv)-n(x))\otimes n(x)|+|(n(x-sv)-n(x))\otimes \nabla_x n(x)|.
\end{align*}
Then by~\eqref{normal} we have
\begin{equation*}
\begin{split}
\nabla_x n(x)
   & \lesssim \nabla_x \Big[\chi'_{\e/2}(dist(x,\partial \Omega)) \Big]+\chi'_{\e/2}(dist(x,\partial \Omega))\nabla_x \frac{\nabla_x \xi(x)}{|\nabla \xi(x)|} \\
    & \lesssim   \chi''\times \nabla_x dist(x,\partial \Omega)+ \chi'_{\e/2}(dist(x,\partial \Omega))\frac{|\nabla \xi(x)|\nabla^2 \xi(x)-\nabla \xi(x)\otimes \frac{\nabla^2 \xi(x) \nabla \xi(x)}{|\nabla \xi(x)|} }{|\nabla \xi(x)|^2}\\
    &\lesssim 1+\chi'_{\e/2}(dist(x,\partial \Omega))\frac{|\nabla^2\xi(x)|}{|\nabla \xi(x)|}.
\end{split}
\end{equation*}
From~\eqref{chi} we have $|\nabla \xi(x)|\gtrsim 1$ when $dist(x,\partial\Omega)\ll 1$. When $dist(x,\partial \Omega)\gtrsim 1$ we take $\e$ to be small enough such that $\chi_\e'(dist(x,\partial \Omega))=0$. Hence
\begin{equation*}
|\nabla_x n(x)|\lesssim \Vert \xi\Vert_{C^2}.
\end{equation*}

Then we use~\eqref{tb bounded} to have
\begin{equation*}
|n(x-sv)-n(x)|\lesssim \tb|v|\Vert \xi\Vert_{C^2}\lesssim \frac{\tilde{\alpha}(x,v)}{|v|}.
\end{equation*}
Thus we conclude~\eqref{n(x)-n(xb)}.

Last we prove~\eqref{Gf bdd}. We rewrite
\begin{align*}
   G(x)\nabla_x f(x-sv,v)&=G(x-sv)\nabla_x f(x-sv,v)+[G(x)-G(x-sv)]\nabla_x f(x-sv,v)  \\
   & \lesssim       \frac{\Vert w_{\tilde{\theta}/2}|v|\nabla_\parallel f\Vert_\infty}{|v|w_{\tilde{\theta}/2}(v)}+\frac{\tilde{\alpha}(x,v)}{|v|}\nabla_x f(x-sv,v)\\
   &\lesssim \frac{\Vert w_{\tilde{\theta}/2}|v|\nabla_\parallel f\Vert_\infty}{|v|w_{\tilde{\theta}/2}(v)}+\frac{\Vert w_{\tilde{\theta}/2}\tilde{\alpha}\nabla_x f\Vert_\infty}{|v|w_{\tilde{\theta}/2}(v)}\\
   &\lesssim \frac{\Vert w_{\tilde{\theta}/2}|v|\nabla_\parallel f\Vert_\infty+\Vert w_{\tilde{\theta}}\alpha\nabla_x f \Vert_\infty}{|v|w_{\tilde{\theta}/2}(v)},
\end{align*}
where we have used $w_{\tilde{\theta}/2}(v)\tilde{\alpha}(x,v)\lesssim w_{\tilde{\theta}}(v)\alpha(x,v).$ Then we conclude the lemma.\end{proof}

\begin{lemma}\label{Lemma: ndot nabla xb}
For $\xb(x,v)=\eta_{p^1}(\mathbf{x}_{p^1}^1)$ and $i=1,2$, we have
\begin{equation}\label{nx nabla xb}
\Big|G(x)\nabla_x \mathbf{x}_{p^1,i}^1\Big|\lesssim 1,
\end{equation}

\begin{equation}\label{nx nabla tb}
\Big|G(x) \nabla_x \tb(x,v)\Big|\lesssim \frac{1}{|v|}.
\end{equation}

And thus from~\eqref{equivalent},
\begin{equation}\label{tang nabla tb}
|\partial_{\mathbf{x}_{p^1,j}^1}\tb(\eta_{p^1}(\mathbf{x}_{p^1}^1),v^1)|\lesssim \frac{1}{|v^1|},
\end{equation}

\begin{equation}\label{tang nabla xb}
|G(x)\nabla_x \xb(x,v)|\lesssim 1,\quad |\partial_{\mathbf{x}_{p^1,j}^1}\xb(\eta_{p^1}(\mathbf{x}_{p^1}^1),v^1)|\lesssim 1.
\end{equation}

\end{lemma}

\begin{proof}
By~\eqref{n(x)-n(xb)} in Lemma \ref{Lemma: nx-nxb} we have

\begin{equation}\label{I-n otimes nablax x}
\Big| G(x) \nabla_x \mathbf{x}_{p^1,i}^1 \Big|  \lesssim \underbrace{G(\xb) \nabla_x \mathbf{x}_{p^1,i}^1}_{\eqref{I-n otimes nablax x}_1} +\underbrace{\frac{\tilde{\alpha}(x,v)}{|v|}|\nabla_x \mathbf{x}_{p^1,i}^1|}_{\eqref{I-n otimes nablax x}_2}.
\end{equation}

By~\eqref{xi deri xbp}, the definition of $n(\xb)$, $\mathbf{v}_{p^1,3}$ in~\eqref{bar_v} and~\eqref{tang times normal}, we have
\[\eqref{I-n otimes nablax x}_1\lesssim 1+|v|G(\xb)\frac{n(\xb)}{|n(\xb)\cdot v|}=1.\]
Again by~\eqref{xi deri xbp} we have
\[\eqref{I-n otimes nablax x}_2\lesssim \frac{\tilde{\alpha}(x,v)}{|v|}+    \frac{\tilde{\alpha}(x,v)}{|v|}  \frac{|v|}{|n(\xb)\cdot v|}\lesssim 1.\]
We conclude~\eqref{nx nabla xb}.

For~\eqref{nx nabla tb} by~\eqref{xi deri tb} we have
\[\eqref{I-n otimes nablax x}\lesssim 0+\frac{\tilde{\alpha}}{|v|}\frac{1}{|n(\xb)\cdot v|}\lesssim \frac{1}{|v|}.\]
We conclude~\eqref{nx nabla tb}.\end{proof}

\subsection{Properties of H\"older estimate}

To prove the $C^{1,\beta}$ estimate~\eqref{estF_C1beta} we need several $C^{1,\beta}$ estimate for $\xb$ and $\tb$. We summarize them in Lemma \ref{Lemma: min max} and Lemma \ref{Lemma: Gf-Gf}. Lemma \ref{Lemma: x-y} serves as a key ingredient to prove Lemma \ref{Lemma: min max}.

\begin{lemma}\label{Lemma: min max}
We have the following estimates:

\begin{equation}\label{min: xb}
\big|\frac{\xb(x,v)-\xb(y,v)}{|x-y|^\beta}\big|\lesssim \frac{1}{\min\{\frac{\alpha(x,v)}{|v|},\frac{\alpha(y,v)}{|v|}\}^{\beta}},
\end{equation}

\begin{equation}\label{min: tb}
\frac{|e^{-C\nu\tb(x,v)}-e^{-C\nu\tb(y,v)}|^\beta}{|x-y|^\beta}\lesssim \big|  \frac{e^{-C\nu\tb(x,v)}-e^{-C\nu\tb(y,v)}}{|x-y|^\beta}\big|\lesssim \frac{1}{|v|\min\{\frac{\alpha(x,v)}{|v|},\frac{\alpha(y,v)}{|v|}\}^{\beta}},
\end{equation}

\begin{equation}\label{min: nxb}
\big|   \frac{n(\xb(x,v))-n(\xb(y,v))}{|x-y|^\beta}\big|\lesssim \Vert \xi\Vert_{C^2}\frac{1}{\min\{\frac{\alpha(x,v)}{|v|},\frac{\alpha(y,v)}{|v|}\}^{\beta}},
\end{equation}

\begin{equation}\label{min: nabla xb}
\big|  \frac{\nabla_x\xb(x,v)-\nabla_x\xb(y,v)}{|x-y|^\beta}\big|\lesssim \frac{1}{\min\{\frac{\alpha(x,v)}{|v|},\frac{\alpha(y,v)}{|v|}\}^{2+\beta}},
\end{equation}

\begin{equation}\label{min: nabla tb}
\big|    \frac{\nabla_x\tb(x,v)-\nabla_x\tb(y,v)}{|x-y|^\beta}\big|\lesssim \frac{1}{|v|\min\{\frac{\alpha(x,v)}{|v|},\frac{\alpha(y,v)}{|v|}\}^{2+\beta}},
\end{equation}

\begin{equation}\label{min: I nabla xb}
\big|G(y)  \frac{\nabla_x\xb(x,v)-\nabla_x\xb(y,v)}{|x-y|^\beta}\big|\lesssim \frac{1}{\min\{\frac{\alpha(x,v)}{|v|},\frac{\alpha(y,v)}{|v|}\}^{1+\beta}},
\end{equation}

\begin{equation}\label{min: I nabla tb}
\big|  G(y)  \frac{\nabla_x\tb(x,v)-\nabla_x\tb(y,v)}{|x-y|^\beta}\big|\lesssim \frac{1}{|v|\min\{\frac{\alpha(x,v)}{|v|},\frac{\alpha(y,v)}{|v|}\}^{1+\beta}},
\end{equation}

\begin{equation}\label{min: f}
 \big|\frac{f(x,v)-f(y,v)}{|x-y|^\beta}\big|\lesssim  \frac{\Vert wf\Vert_\infty^{1-\beta}\Vert w_{\tilde{\theta}}\alpha\nabla_x f\Vert_\infty^\beta}{w_{2\tilde{\theta}}(v)\min\{\alpha(x,v),\alpha(y,v)\}^{\beta}}.
\end{equation}

 When $x,y\in \partial \Omega$,

\begin{equation}\label{min: xb tang}
\frac{|\xb(x,v)-\xb(y,v)|}{|x-y|}\lesssim 1,
\end{equation}

\begin{equation}\label{min: M_w}
\frac{|M_W(x,v)-M_W(y,v)|}{\sqrt{\mu(v)}|x-y|^\beta}\lesssim \Vert T_W-T_0\Vert_{C^1}.
\end{equation}

For $x=\eta_{p(x)}(\mathbf{x}_{p(x)}),y=\eta_{p(y)}(\mathbf{x}_{p(y)}) \in \partial \Omega$, and $\xb(x,v)=\eta_{p^1(x)}(\mathbf{x}_{p^1(x)}^1), \xb(y,v)=\eta_{p^1(y)}(\mathbf{x}_{p^1(y)}^1),$( see the definition~\eqref{xpk x} in section 7). For $i,j\in \{1,2\}$ we have 
\begin{equation}\label{min: partial xip1 xip2}
\frac{\big|\partial_{\mathbf{x}_{p(x),j}}\mathbf{x}_{p^1(x),i}^1 -\partial_{\mathbf{x}_{p(y),j}} \mathbf{x}_{p^1(y),i}^1  \big|}{|x-y|^\beta}\lesssim \frac{1}{\min\{\frac{\alpha(x,v)}{|v|},\frac{\alpha(y,v)}{|v|}\}^{3}}.
\end{equation}

%
%
%

\end{lemma}

We need the following lemma to prove it.
\begin{lemma}\label{Lemma: x-y}
Define
\begin{equation}\label{xtau}
x(\tau):=(1-\tau)x+\tau y,\quad |\dot{x}(\tau)|=|x-y|.
\end{equation}
If $|x-y|\leq \e\min\{\frac{\tilde{\alpha}(x,v)}{|v|},\frac{\tilde{\alpha}(y,v)}{|v|}\}\ll 1$, then
\begin{equation}\label{alpha geq min}
\tilde{\alpha}(x(\tau),v)\gtrsim \min\{\tilde{\alpha}(x,v),\tilde{\alpha}(y,v)\}.
\end{equation}

\end{lemma}

\begin{proof}

By the definition~\eqref{kinetic_distance} we have
\begin{equation}\label{alpha square}
\tilde{\alpha}^2(x(\tau),v) =|\nabla \xi(x(\tau))\cdot v|^2-2\xi(x(\tau))(v\cdot \nabla^2 \xi(x(\tau))\cdot v) .
\end{equation}
We expand $|\nabla \xi(x(\tau))\cdot v|^2$ and $-2\xi(x(\tau))(v\cdot \nabla^2 \xi(x(\tau))\cdot v)$ separately: we expand in $\tau$ as
\begin{equation}\label{alpha^2}
|\nabla \xi(x(\tau))\cdot v|^2=|\nabla \xi(x(0))\cdot v|^2+\underbrace{\int_0^\tau \dd \tau' 2(\nabla\xi(x(\tau'))\cdot v)\dot{x}(\tau)\cdot \nabla^2 \xi(x(\tau'))\cdot v}_{\eqref{alpha^2}_*},
\end{equation}

\begin{equation}\label{alpha^2 2}
-2\xi(x(\tau))(v\cdot \nabla^2 \xi(x(\tau))\cdot v)=-2\xi(x(\tau))\{v\cdot \nabla^2 \xi(x(0))\cdot v+O(|x-y|)\Vert \xi\Vert_{C^3}|v^2|\},
\end{equation}
where we have used~\eqref{xtau}.

For$~\eqref{alpha^2}_*$ we further expand in $\tau^\prime$ and obtain

\begin{align}
 \eqref{alpha square}=  & |\nabla \xi(x)\cdot v|^2+2(\nabla \xi(x)\cdot v)O(|x-y|)|v|\Vert \xi\Vert_{C^2} \notag\\
   & +\int_0^\tau \dd \tau' \int_0^{\tau'} \dd \tau'' \dot{x}(\tau'')\cdot \nabla^2 \xi(x(\tau''))\cdot v \dot{x}(\tau'')\cdot \nabla^2 \xi(x(\tau''))\cdot v \label{double ftc 1}\\
   & +\int_0^\tau \dd \tau' \int_0^{\tau'}\dd \tau''    2(\nabla \xi(x(\tau''))\cdot v) \dot{x}(\tau'')\dot{x}(\tau'')\nabla^3 \xi(x(\tau''))\cdot v  \label{double ftc 2}\\
   & -2\xi(x(\tau))O(|v|^2). \notag
\end{align}

From the convexity~\eqref{convex} we have
\begin{equation}\label{estimate for double ftc}
\eqref{double ftc 1}+\eqref{double ftc 2}=O(1) |\dot{x}|^2 \Vert \xi\Vert_{C^3}|v|^2= O(
\e^2)\min\{\frac{\alpha(x,v)}{|v|},\frac{\alpha(y,v)}{|v|}\}^2|v|^2.
\end{equation}

From~\eqref{estimate for double ftc} we have
\begin{equation}\label{alpha2+2xi}
\begin{split}
  \eqref{alpha square}+2\xi(x(\tau))O(|v|^2) & =|\nabla \xi(x)\cdot v|^2+O(\e)\alpha(x,v)\min\{\alpha(x,v),\alpha(y,v)\} \\
   & +O(\e^2)\min\{\alpha(x,v),\alpha(y,v)\}^2.
\end{split}
\end{equation}

Now we claim
\begin{equation}\label{xi geq min}
-\xi(x(\tau))\geq \min\{-\xi(x),-\xi(y)\}.
\end{equation}

From $\frac{\dd}{\dd \tau}(-\xi(x(\tau)))=-\dot{x}(\tau)\cdot \nabla_x \xi(x(\tau))$ and convexity~\eqref{convex},
\[\frac{\dd^2}{\dd \tau^2}(-\xi(x(\tau)))=-\dot{x}(\tau)\cdot \nabla_x^2 \xi(x(\tau))\cdot \dot{x(\tau)}\lesssim -|\dot{x}(\tau)|^2 \leq 0.\]
Thus $-\xi(x(\tau))$ is a concave function of $\tau$. From $0\leq \tau\leq 1$, we prove our claim~\eqref{xi geq min} as
\begin{align*}
-\xi(x(\tau))   &=-\xi(x((1-\tau)\cdot 0+\tau\cdot 1))\geq -(1-\tau)\xi(x(0))-\tau\xi(x(1))  \\
   & =-(1-\tau)\xi(x)-\tau \xi(y)\geq \min\{-\xi(x),-\xi(y)\}.
\end{align*}

Now combining~\eqref{alpha2+2xi} and~\eqref{xi geq min} we conclude that
\begin{equation}\label{estimate for alpha2}
\begin{split}
\eqref{alpha square}&\gtrsim |\nabla \xi(x)\cdot v|^2+\min\{-\xi(x),-\xi(y)\}|v|^2\\
&+O(\e)\alpha(x,v)\min\{\tilde{\alpha}(x,v),\tilde{\alpha}(y,v)\}+O(\e^2)\min\{\tilde{\alpha}(x,v),\tilde{\alpha}(y,v)\}^2.
\end{split}
\end{equation}
Similarly we can set $x(\tau)=(1-\tau)y+\tau x$. From $x(0)=y$, following the same argument we derive
\begin{equation}\label{estimate for y}
\begin{split}
\eqref{alpha square}&\geq |\nabla \xi(y)\cdot v|^2+\min\{-\xi(x),-\xi(y)\}|v|^2\\
&+O(\e)\tilde{\alpha}(y,v)\min\{\tilde{\alpha}(x,v),\tilde{\alpha}(y,v)\}+O(\e^2)\min\{\tilde{\alpha}(x,v),\tilde{\alpha}(y,v)\}^2.
\end{split}
\end{equation}

From the definition of~\eqref{kinetic_distance} using~\eqref{estimate for alpha2} and~\eqref{estimate for y} we have
\[\eqref{alpha square}\geq \min\{\tilde{\alpha}(x,v),\tilde{\alpha}(y,v)\}^2-O(\e)\min\{\tilde{\alpha}(x,v),\tilde{\alpha}(y,v)\}^2.\]
Hence from $\e\ll 1$ we conclude~\eqref{alpha geq min}. \end{proof}

Then we start the proof of Lemma \ref{Lemma: min max}.

\begin{proof}[\textbf{Proof of Lemma \ref{Lemma: min max}}]
For all estimates we assume $|x-y|\leq \e \min\{\frac{\tilde{\alpha}(x,v)}{|v|},\frac{\tilde{\alpha}(y,v)}{|v|}\}$, otherwise the Lemma follows immediately by~\eqref{nabla_tbxb}. Thus we can apply~\eqref{alpha geq min} during the whole proof. We will use the $x(\tau)$ defined in~\eqref{xtau}.

\textit{Proof of~\eqref{min: xb}}.
We have
\begin{align*}
 \frac{|\xb(x,v)-\xb(y,v)|}{|x-y|^\beta}  & =\frac{1}{|x-y|^\beta}\int_0^1 \dd \tau \frac{\dd}{\dd \tau}\nabla_x \xb(x(\tau),v) \\
   & =\frac{1}{|x-y|^\beta}\int_0^1  |\dot{x}(\tau)||\nabla_x \xb(x(\tau),v)|\dd \tau\\
   & \lesssim \frac{1}{|x-y|^{\beta-1}}\int_0^1 \frac{|v|}{\tilde{\alpha}(x(\tau),v)}\lesssim \frac{|v|^\beta}{\min\{\alpha(x,v),\alpha(y,v)\}^\beta},
\end{align*}
where we have used Lemma \ref{Lemma: x-y} and~\eqref{n geq alpha} in the last line.

\textit{Proof of~\eqref{min: tb}}.
The first inequality is clear since $|e^{-C\nu \tb(x,v)}-e^{-C\nu\tb(y,v)}|\lesssim 1$.

To prove the second inequality we have
\begin{align*}
  \frac{|e^{-C\nu\tb(x,v)}-e^{-C\nu \tb(y,v)}|}{|x-y|^\beta} & =\frac{1}{|x-y|^\beta}\int_0^1 \dd \tau \frac{\dd}{\dd \tau} e^{-C\nu\tb(x(\tau),v)} \\
   & \lesssim\frac{1}{|x-y|^\beta}\int_0^1 \dd \tau (\nu \tb(x(\tau),v)) e^{-C\nu\tb(x(\tau),v)} |x-y| \frac{1}{n(\xb(x(\tau),v))\cdot v}\\
   &\lesssim  |x-y|^{1-\beta}\frac{1}{\min\{\tilde{\alpha}(x,v),\tilde{\alpha}(y,v)\}}\lesssim \frac{1}{|v|\min\{\frac{\alpha(x,v)}{|v|},\frac{\alpha(y,v)}{|v|}\}^{\beta}},
\end{align*}
where we have used~\eqref{nabla_tbxb} in the second line, Lemma \ref{Lemma: x-y} and~\eqref{n geq alpha} in the last line.

\textit{Proof of~\eqref{min: nxb}}.
Since
\[\frac{|n(\xb(x,v))-n(\xb(y,v))|}{|x-y|^\beta}=\frac{|n(\xb(x,v))-n(\xb(y,v))|}{|\xb(x,v)-\xb(y,v)|}\frac{|\xb(x,v)-\xb(y,v)|}{|x-y|^\beta}\lesssim \Vert \xi\Vert_{C^2}\frac{|\xb(x,v)-\xb(y,v)|}{|x-y|^\beta}.\]
By~\eqref{min: xb} we derive~\eqref{min: nxb}.

\textit{Proof of~\eqref{min: nabla xb}}. We have
\begin{align}
\frac{|\nabla_x \xb(x,v)-\nabla_x \xb(y,v)|}{|x-y|^\beta} & =\frac{1}{|x-y|^\beta}\int_0^1 \dd \tau \frac{\dd}{\dd \tau} \nabla_x \xb(x(\tau),v)  \notag\\
   & =\frac{1}{|x-y|^\beta}\int_0^1 |\dot{x}(\tau)||\nabla_x \nabla_x \xb(x(\tau),v)|\dd \tau  \label{nablanabla x}\\
   & \lesssim |x-y|^{1-\beta}\int_0^1 \frac{|v|^3}{|\tilde{\alpha}(x(\tau),v)|^3}.\notag
\end{align}
Here we have used~\eqref{nabla_tbxb} to have
\begin{align}
 |\nabla_x (\nabla_x \xb(x(\tau),v))|   &\lesssim \Big[\frac{\Vert \eta\Vert_{C^1}|v||n(\xb(x(\tau),v))\cdot v| }{|n(\xb(x(\tau),v))\cdot v|^2}+    \Vert \eta\Vert_{C^2}\frac{n(\xb(x(\tau),v))\otimes v}{|n(\xb(x(\tau),v))\cdot v|^2}      \Big]\times |\nabla_x \xb(\tau,v)| \label{second xb}\\
    & \lesssim \frac{|v|^2}{|n(\xb(x(\tau),v))\cdot v|^2}+\frac{|v|^3}{|n(\xb(x(\tau),v))\cdot v|^3}\lesssim \frac{|v|^3}{|n(\xb(x(\tau),v))\cdot v|^3} \label{second derivative xb},
\end{align}
where we have used $|n(\xb(x(\tau),v))\cdot v|\leq |v|$ in the last inequality. Then by Lemma \ref{Lemma: min max} and~\eqref{n geq alpha} we obtain~\eqref{min: nabla xb}.

%
%
%
%
%
%

\textit{Proof of~\eqref{min: nabla tb}}. We have
\begin{align*}
  \frac{|\nabla_x \tb(x,v)-\nabla_x \tb(y,v)|}{|x-y|^\beta} &=\frac{1}{|x-y|^\beta}\int_0^1 \dd \tau \frac{\dd}{\dd \tau}\nabla_x \tb(x(\tau),v)  \\
   & =\frac{1}{|x-y|^\beta}\int_0^1 |\dot{x}(\tau)||\nabla_x \nabla_x \tb(x(\tau),v)|\dd \tau\\
   &   \lesssim |x-y|^{1-\beta}\int_0^1     \frac{|v|^2}{|\tilde{\alpha}(x(\tau),v)|^3},
\end{align*}
where we have used~\eqref{nabla_tbxb} to conclude
\begin{equation}\label{proof of min nabla tb}
\begin{split}
|\nabla_x (\nabla_x \tb(x(\tau),v))|\lesssim    & \frac{\Vert \eta\Vert_{C^2}(|n(\xb(x(\tau),v))\cdot v|)}{|n(\xb(x(\tau),v))|^2}\nabla_x \xb(x(\tau),v) \\
   & + \frac{n(\xb(x(\tau)),v)\Vert \eta\Vert_{C^2}|v|}{|n(\xb(x(\tau),v))|^2}\nabla_x \xb(x(\tau),v)\\
 &  \lesssim \frac{|v|^2}{|n(\xb(x(\tau),v))\cdot v|^3}.
\end{split}
\end{equation}

Thus by Lemma \ref{Lemma: min max} and~\eqref{n geq alpha} we obtain~\eqref{min: nabla tb}.

\textit{Proof of~\eqref{min: I nabla xb}}.
From~\eqref{nablanabla x} and~\eqref{second xb}, we bound
\begin{align*}
 |G(y) \nabla_x (\nabla_x \xb(x(\tau),v))|  & \lesssim      \Big[\frac{\Vert \eta\Vert_{C^2}|v|}{|n( \xb(x(\tau),v))\cdot v|}+       \frac{|v|G(y)n( \xb(x(\tau),v))\otimes v}{|n( \xb(x(\tau),v))\cdot v|^2} \Big]  |\nabla_x \xb(\tau,v)| \\
   & \lesssim \frac{|v|^2}{|n( \xb(x(\tau),v))\cdot v|^2},
\end{align*}
where we have used
\begin{align}
|G(y)n( \xb(x(\tau),v)) |  &\lesssim   |G(y)n(y)|    +|n(\xb(x(\tau,v)))-n(y)| \notag\\
   & \lesssim |\xb(x(\tau),v)-y|\lesssim |\xb(x(\tau),v)-x(\tau)|+|x(\tau)-y| \notag\\
   & \lesssim\min\{\frac{\tilde{\alpha}(x,v)}{|v|},\frac{\tilde{\alpha}(y,v)}{|v|}\}.\label{G n}
\end{align}
Thus
\begin{align*}
 \big| G(y)\frac{\nabla_x \xb(x,v)-\nabla_x \xb(y,v)}{|x-y|^\beta}\big| & \lesssim |x-y|^{1-\beta}\int_0^1 \frac{|v|^2}{|\tilde{\alpha}(x(\tau),v)|^2} ,
\end{align*}
and we conclude~\eqref{min: I nabla xb} from~\eqref{n geq alpha}.

\textit{Proof of~\eqref{min: I nabla tb}}. From~\eqref{proof of min nabla tb} we have
\begin{align*}
G(y) \frac{\nabla_x \tb(x,v)-\nabla_x \tb(y,v)}{|x-y|^\beta}   & \lesssim \frac{1}{|x-y|^\beta}\times  \int_0^1 \dd \tau \Big| \frac{\Vert \eta\Vert_{C^2}}{|n(\xb(x(\tau),v))\cdot v|}\nabla_x \xb(x(\tau),v) \\
   & +  \Vert \eta\Vert_{C^2}|v|  \frac{G(y) n(\xb(x(\tau),v))}{|n(\xb(x(\tau),v))\cdot v|^2}  \nabla_x \xb(x(\tau),v) \Big|\\
   & \lesssim \frac{1}{|x-y|^\beta}\int_0^1 \dd \tau\frac{|v|}{|n(\xb(x(\tau),v))\cdot v|^2}\lesssim \frac{1}{|v|\min\{\frac{\alpha(x,v)}{|v|},\frac{\alpha(y,v)}{|v|}\}^{1+\beta}},
\end{align*}
where we have used~\eqref{G n} to $G(y)n(\xb(x(\tau),v))$ and~\eqref{n geq alpha}.

\textit{Proof of~\eqref{min: f}}. By~\eqref{alpha geq min} in Lemma \ref{Lemma: x-y} we have
\begin{align*}
  \frac{|f(x,v)-f(y,v)|}{|x-y|^\beta}
   & \lesssim w^{-\beta}(v)\Vert wf\Vert_\infty^{1-\beta}  \frac{|f(x,v)-f(y,v)|^\beta}{|x-y|^\beta}\\
    &\lesssim w^{-\beta}(v)\Vert wf\Vert_\infty^{1-\beta}\frac{1}{|x-y|^\beta}\Big|\int_0^1 \dd\tau \dot{x}(\tau)\cdot \nabla_x f(x(\tau),v) \Big|^\beta\\
    & \lesssim w^{-\beta}(v)w^{\beta-1}_{\tilde{\theta}}(v)\Vert wf\Vert_\infty^{1-\beta}\Vert w_{\tilde{\theta}}\alpha\nabla_x f\Vert_\infty^\beta     \Big|\int_0^1 \dd \tau \frac{1}{\tilde{\alpha}(x(\tau),v)}\Big|^\beta\\
    &\lesssim w^{-1}_{2\tilde{\theta}}(v)\frac{\Vert wf\Vert_\infty^{1-\beta}\Vert w_{\tilde{\theta}}\alpha\nabla_x f\Vert_\infty^\beta}{\min\{\alpha(x,v),\alpha(y,v)\}^\beta},
\end{align*}
where we have used $\tilde{\theta}\ll \varrho$ to have
\[w^{\beta-1}_{\tilde{\theta}}(v)w^{-\beta}(v)=e^{|v|^2[(\beta-1)\varrho-\beta \tilde{\theta}]}\leq e^{-2\tilde{\theta}|v|^2}.\]

\textit{Proof of~\eqref{min: xb tang}}.
Since $|x-y|\ll 1$, we can assume that $x,y \in B(p;\delta_2)$, where $B(p;\delta_2)$ is defined in~\eqref{iota}. Then both $x,y$ correspond to the same $p$. Only for proof of this estimate we denote
\[x=\eta_p(\mathbf{x}_p(x)),\quad y=\eta_p(\mathbf{x}_p(y)),\]
\[\eta_p(\mathbf{x}_p(\tau)):=\tau \eta_p(\mathbf{x}_{p}(x))+(1-\tau)\eta_p(\mathbf{x}_p(y)).\]

By mean value theorem
\begin{align*}
  x-y&= \eta_p(\mathbf{x}_p(x))-\eta_p(\mathbf{x}_p(y))=\nabla \eta_p\big(c\mathbf{x}_p(x)+(1-c)\mathbf{x}_p(y)\big) \big(\mathbf{x}_p(x)-\mathbf{x}_p(y)\big).
\end{align*}
Thus
\begin{align*}
&\frac{|\xb(\eta_p(\mathbf{x}_p(x)),v)-\xb(\eta_p(\mathbf{x}_p(y)),v)|}{|x-y|}
     \\
   & =   \frac{1}{|x-y|} \big| \int_0^1 \dd \tau \frac{\dd}{\dd \tau} \xb(\eta_{p}(\mathbf{x}_p(\tau)),v) \big|=\frac{1}{|x-y|} \big| \int_0^1 \dd \tau \nabla_x \xb(\eta_{p}(\mathbf{x}_p(\tau)),v) \frac{\dd }{\dd \tau}\eta_p(\mathbf{x}_p(\tau)) \big|\\
   &= \frac{1}{|x-y|}\int_0^1 \dd \tau \nabla_x \xb(\eta_{p}(\mathbf{x}_{p}(\tau)),v) \nabla \eta_p\big(c\mathbf{x}_p(x)+(1-c)\mathbf{x}_p(y)\big)  \big(\mathbf{x}_p(x)-\mathbf{x}_p(y)\big)\\
   &\lesssim \frac{|\mathbf{x}_p(x)-\mathbf{x}_p(y)|}{|x-y|} \Big|\int_0^1 \dd \tau \nabla_x \eta_p\big(c\mathbf{x}_p(x)+(1-c)\mathbf{x}_p(y)\big)-\frac{n(\xb(\eta_p(\mathbf{x}_p(\tau)),v))\otimes v}{|n(\xb(\eta_p(\mathbf{x}_p(\tau)),v))\cdot v|} \nabla \eta_p\big(c\mathbf{x}_p(x)+(1-c)\mathbf{x}_p(y)\big)\Big|\\
  & \lesssim  \Vert \eta\Vert_{C^1}+\frac{n(\eta_p\big(c\mathbf{x}_p(x)+(1-c)\mathbf{x}_p(y)\big)) \nabla \eta_p\big(c\mathbf{x}_p(x)+(1-c)\mathbf{x}_p(y)\big) |v|}{|n(\xb(\eta_p(\mathbf{x}_p(\tau)),v))\cdot v|}\\
&+\frac{|v|[n(\eta_p\big(c\mathbf{x}_p(x)+(1-c)\mathbf{x}_p(y)\big)) -n(\xb(\eta_p(\mathbf{x}_p(\tau)),v))]}{|n(\xb(\eta_p(\mathbf{x}_p(\tau)),v))\cdot v|}   \\
&\lesssim 1+\frac{\Vert \xi\Vert_{C^2}|\xb(\eta_p(\mathbf{x}_p(\tau)),v)-\eta_p(\mathbf{x}_p(\tau))|+|\eta_p(\mathbf{x}_p(\tau))-\eta_p\big(c\mathbf{x}_p(x)+(1-c)\mathbf{x}_p(y)\big)| }{|n(\xb(\eta_p(\mathbf{x}_p(\tau)),v))\cdot v|} \\
  &\lesssim 1+ \frac{\alpha(n(\xb(\eta_p(\mathbf{x}_p(\tau))),v))+|v||x-y|}{|n(\xb(\eta_p(\mathbf{x}_p(\tau)),v))\cdot v|}\lesssim 1.
\end{align*}
 In the fourth line we have used~\eqref{nabla_tbxb}. In the last two lines we have used~\eqref{n(x)-n(xb)} and $|x-y|\leq O(\e)\min\{\frac{\tilde{\alpha}(x,v)}{|v|},\frac{\tilde{\alpha}(y,v)}{|v|}\}$.

\textit{Proof of~\eqref{min: M_w}.} Since $\Vert T_W-T_0\Vert_\infty \ll 1$ from Existence Theorem, by the definition of $M_W$ in~\eqref{Wall Maxwellian} we apply the mean value theorem to have
\begin{align*}
  \frac{|M_W(x,v)-M_W(y,v)|}{\sqrt{\mu(v)}|x-y|^\beta} & \lesssim \frac{|M_W(x,v)-M_W(y,v)|}{\sqrt{\mu(v)}|x-y|} \leq \Big\Vert \frac{\nabla_x M_W(x,v)}{\sqrt{\mu(v)}}\Big\Vert_{\infty} \\
   & \lesssim_{T_0} \Big\Vert \nabla_x T_W   |v|^2\frac{M_W(x)}{\sqrt{\mu(v)}}   \Big\Vert_\infty\lesssim \Vert T_W-T_0\Vert_{C^1}.
\end{align*}

\textit{Proof of~\eqref{min: partial xip1 xip2}}. From~\eqref{O_p} it is equivalent to compute
\begin{align*}
  & \frac{|\partial_j \eta_{p(x)}(\mathbf{x}_{p(x)})\nabla_{x}\mathbf{x}^1_{p^1(x),i}  -\partial_j \eta_{p(y)}(\mathbf{x}_{p(y)})\nabla_{x}\mathbf{x}^1_{p^1(y),i}|}{|x-y|^\beta}  \\
  & \lesssim \frac{|\partial_j \eta_{p(x)}(\mathbf{x}_{p(x)})-\partial_j \eta_{p(y)}(\mathbf{x}_{p(y)})| |\nabla_x \mathbf{x}_{p^1(x),i}^1|}{|x-y|^\beta}+ \frac{\big| \partial_j \eta_{p(x)}(\mathbf{x}_{p(x)})[\nabla_{x}\mathbf{x}^1_{p^1(x),i}-\nabla_{x}\mathbf{x}^1_{p^1(y),i}]\big|}{|x-y|^\beta} \\
  &\lesssim \frac{\Vert \eta\Vert_{C^2}|v|}{\alpha(x,v)}+ \Vert \eta\Vert_{C^1}\frac{|\nabla_{x}\mathbf{x}^1_{p^1(x),i}-\nabla_{x}\mathbf{x}^1_{p^1(y),i}|}{|x-y|^\beta} ,
\end{align*}
where we have used~\eqref{xip deri xbp}. Denote $\xb(x(\tau),v) = \eta_{p^1(x(\tau))}(\mathbf{x}_{p^1(x(\tau))}^1)$. Then
\[\frac{\dd}{\dd \tau} \big(\nabla_x\mathbf{x}^1_{p^1(x(\tau)),i}\big)= \frac{\dd}{\dd \tau} x(\tau) \nabla_{x(\tau)} \nabla_x \mathbf{x}_{p^1(x(\tau)),i}^1=\frac{\dd}{\dd \tau} x(\tau) \nabla_{x(\tau)}\mathbf{x}_{p^1(x(\tau))}^1     \nabla_{\mathbf{x}_{p^1(x(\tau))}^1}        \nabla_x \mathbf{x}_{p^1(x(\tau)),i}^1 .\]
Applying~\eqref{xi deri xbp} we further bound
\begin{align*}
  &\frac{|\nabla_{x}\mathbf{x}^1_{p^1(x),i}-\nabla_{x}\mathbf{x}^1_{p^1(y),i}|}{|x-y|^\beta} \\
  &=\frac{1}{|x-y|^\beta} \Big|\int_0^1     \dd \tau  \frac{\dd}{\dd \tau} \big(\nabla_x\mathbf{x}^1_{p^1(x(\tau)),i}\big)\Big|=\frac{1}{|x-y|^\beta} \int_0^1     \dd \tau  |\dot{x}(\tau)| | \nabla_{x(\tau)}\mathbf{x}_{p^1(x(\tau))}^1     \nabla_{\mathbf{x}_{p^1(x(\tau))}^1}        \nabla_x \mathbf{x}_{p^1(x(\tau)),i}^1| \\
  &\lesssim \frac{\nabla_x \mathbf{x}_{p^1(x(\tau))}^1}{g_{p^1(x(\tau)),ii}(\mathbf{x}_{p^1(x(\tau))}^1)}\Vert \eta\Vert_{C^2}\big[\frac{1}{g_{p^1(x(\tau)),ii}(\mathbf{x}^1_{p^1(x(\tau))})}+\frac{|v|^2}{|n(\xb(x(\tau),v))\cdot v|^2 g_{p^1(x(\tau)),33}(\mathbf{x}^1_{p^1(x(\tau))})}\big]\\
  &\lesssim_{\eta} \frac{|v|^3}{\alpha^3(x(\tau),v)} \lesssim \frac{1}{\min\{\frac{\alpha(x,v)}{|v|},\frac{\alpha(y,v)}{|v|}\}^3}.
\end{align*}
In the third line we have taken derivative to~\eqref{xi deri xbp}. In the fourth line we have used~\eqref{n geq alpha} and $\nabla_x \mathbf{x}_{p^1}^1\lesssim \frac{|v|}{n(\xb(x(\tau),v))\cdot v}$ from~\eqref{xi deri xbp}.

\end{proof}

\begin{lemma}\label{Lemma: Gf-Gf}
For any $s\in [t-\tb,t]$, we have
\begin{align}
   & \frac{G(x)\nabla_x f(x-sv,v)-G(y)\nabla_x f(y-sv,v)}{|x-y|^\beta} \notag\\
   &\lesssim \frac{\nabla_\parallel f(x-sv,u)-\nabla_\parallel f(y-sv,v)}{|x-y|^\beta}+\frac{\tilde{\alpha}(x,v)}{|v|}\frac{\nabla_x f(x-sv,v)-\nabla_x f(y-sv,v)}{|x-y|^\beta}+ \frac{\Vert w_{\tilde{\theta}}\alpha\nabla_x f\Vert_\infty}{w_{\tilde{\theta}}(v)\alpha(y-sv,v)}. \label{Gf-Gf 1}
\end{align}

\end{lemma}

\begin{proof}
First we rewrite
\begin{align}
   &G(x)\nabla_x f(x-sv,v)-G(y)\nabla_x f(y-sv,v) \notag \\
   & =G(x-sv)\nabla_x f(x-sv,v)-G(y-sv)\nabla_x f(y-sv,v) \label{rewrite Gfx-Gfy 1} \\
   &+\big(G(x)-G(x-sv)\big)\nabla_x f(x-sv,v)\label{rewrite Gfx-Gfy 2}\\
   &+\big( G(y-sv)-G(y)\big) \nabla_x f(y-sv,v)\label{rewrite Gfx-Gfy 3}.
\end{align}
Note that from~\eqref{tang deri} a contribution of~\eqref{rewrite Gfx-Gfy 1} appears in~\eqref{Gf-Gf 1} .

For~\eqref{rewrite Gfx-Gfy 2} and~\eqref{rewrite Gfx-Gfy 3} we apply~\eqref{n(x)-n(xb)} and rearrange terms to derive that
\begin{align}
   &  [G(x)-G(x-sv)] \nabla_x f(x-sv,v)-[G(y)-G(y-sv)] \nabla_x f(y-sv,v)\notag\\
   &= [G(x)-G(x-sv)][ \nabla_x f(x-sv,v)-\nabla_x f(y-sv,v)]\notag\\
   &+[G(x)-G(y)+G(y-sv)-G(x-sv)]\nabla_x f(y-sv,v)\notag\\
   &\lesssim \frac{\tilde{\alpha}(x,v)}{|v|}[ \nabla_x f(x-sv,v)-\nabla_x f(y-sv,v)]+ [n(x)-n(y)+n(y-sv)-n(x-sv)]\frac{\Vert \alpha \nabla_x f\Vert_\infty}{\alpha(y-sv,u)}         .\notag
\end{align}
By~\eqref{min: nxb} we conclude the lemma.

\end{proof}

\subsection{Properties of boundary condition and collision operators} 
In this subsection we list some properties of the boundary condition and collision operators. We summarize the property of diffuse boundary condition in Lemma \ref{Lemma: bc estimate}. The property of the collision operator is summarized in Lemma \ref{Lemma: K,Gamma} and Lemma \ref{Lemma: k tilde}. 
\begin{lemma}\label{Lemma: bc estimate}
For the diffuse boundary condition of $f$ in~\eqref{diffuse_f}, let $\xb(x,v)=\eta_{p^1(x)}(\mathbf{x}_{p^1(x)})\in \partial\Omega$ ( see~\eqref{xpk x} ), we have
\begin{equation}\label{Estimate for r}
\Vert r\Vert_\infty< \infty,\quad |\partial_{\mathbf{x}_{p^1,i}}r(\eta_{p^1}(\mathbf{x}_{p^1}),v) |\lesssim \Vert T_W-T_0\Vert_{C^1},\quad ||v|^2\nabla_v r(\xb(x,v),v)|\lesssim 1,
\end{equation}

\begin{equation}\label{Estimate for r beta}
w_{\tilde{\theta}}(v)|v|^2\frac{\partial_{\mathbf{x}_{p^1(x),i}}r(\eta_{p^1(x)}(\mathbf{x}_{p^1(x)}),v)-\partial_{\mathbf{x}_{p^1(y),i}}r(\eta_{p^1(y)}(\mathbf{x}_{p^1(y)}),v)}{|\xb(x,v)-\xb(y,v)|^\beta}\lesssim \Vert T_W-T_0\Vert_{C^2},
\end{equation}

\begin{equation}\label{Estimate for M_W beta}
w_{\tilde{\theta}}(v)|v|^2\frac{\partial_{\mathbf{x}_{p^1(x),i}}M_W(\eta_{p^1(x)}(\mathbf{x}_{p^1(x)}),v)-\partial_{\mathbf{x}_{p^1(y),i}}M_W(\eta_{p^1(y)}(\mathbf{x}_{p^1(y)}),v)}{\sqrt{\mu(v)}|\xb(x,v)-\xb(y,v)|^\beta}\lesssim \Vert T_W-T_0\Vert_{C^2},
\end{equation}

\begin{equation}\label{int f-int f}
\begin{split}
    & \frac{1}{|\xb(x,v)-\xb(y,v)|^\beta}\times \Big[\int_{n(\xb(x,v))\cdot v^1>0}f(\xb(x,v),v^1)\sqrt{\mu(v^1)}\{n(\xb(x,v))\cdot v^1\}\dd v^1 \\
     & -\int_{n(\xb(y,v))\cdot v^1>0}f(\xb(y,v),v^1)\sqrt{\mu(v^1)}\{n(\xb(y,v))\cdot v^1\}\dd v^1\Big]\lesssim \Vert \alpha\nabla_x f\Vert_\infty.
\end{split}
\end{equation}

\end{lemma}

\begin{proof}
From~\eqref{remainder term}, it is easy to derive the estimate for $\Vert r\Vert_\infty$. We take derivative to $r$ to obtain
\begin{align}
&\partial_{\mathbf{x}_{p^1,i}}r(\eta_{p^1}(\mathbf{x}_{p^1}),v)\notag
\\&= \frac{\partial_i \eta_{p^1}(\mathbf{x}_{p^1})}{\sqrt{2\pi\mu(v)}} \nabla_x \frac{1}{2\pi [T_W(\xb(x,v))]^2}e^{-\frac{|v|^2}{2T_W(\xb(x,v))}} \notag \\
& =\frac{\partial_i \eta_{p^1}(\mathbf{x}_{p^1})\nabla_x T_W(\xb(x,v))}{\sqrt{2\pi \mu}}    \Big(\frac{-1}{\pi[T_W(\xb(x,v))]^3}+ \frac{|v|^2}{4\pi[T_W(\xb(x,v))]^4} \Big)e^{-\frac{|v|^2}{2T_W(\xb(x,v))}}\label{r C1beta}\\
&\lesssim_{T_0} \Vert T_W-T_0\Vert_{C^1},   \notag
\end{align}
where we have used $\Vert T_W-T_0\Vert_\infty \ll 1 $ from Existence Theorem.

Then we take $v$ derivative to have
\begin{align*}
  &|\nabla_v r(\xb(x,v),v)| \notag\\
  &= |\nabla_v \frac{M_W}{\sqrt{2\pi \mu}}-\nabla_v \sqrt{\mu}| \lesssim   \Big| \nabla_v  \frac{e^{-\frac{|v|^2}{2T_W(\xb(x,v))}}}{\sqrt{\mu(v)}[T_W(\xb(x,v))]^2 }    \Big| \\
  & \lesssim_{T_W} \frac{1}{\mu(v)}\times \big[\nabla_v e^{-\frac{|v|^2}{2T_W(\xb(x,v))}} \sqrt{\mu(v)} +\nabla_v [\sqrt{\mu(v)} T_W^2(\xb(x,v))] e^{-\frac{|v|^2}{2T_W(\xb(x,v))}}   \big]\\
  &\lesssim_{T_W} \frac{1}{\mu(v)}\times e^{-\frac{|v|^2}{2T_W(\xb(x,v))}}\sqrt{\mu(v)} |v|^2|\nabla_v \xb(x,v)|  \lesssim \frac{e^{-\frac{|v|^2}{2T_W(\xb(x,v))}}|v|}{\sqrt{\mu(v)}},
\end{align*}
where we have used~\eqref{nablav tb xb} in the last line. Since the coefficient for $|v|^2$ is negative, from $\Vert T_W-T_0\Vert_\infty \ll 1$ we conclude~\eqref{Estimate for r}.

For~\eqref{Estimate for r beta} from~\eqref{r C1beta} we apply the mean value theorem to bound
\begin{align*}
\frac{|\partial_i \eta_{p^1(x)}(\mathbf{x}_{p^1(x)})-\partial_i \eta_{p^1(y)}(\mathbf{x}_{p^1(y)})|}{|\xb(x,v)-\xb(y,v)|^\beta}&\lesssim  \Vert \eta\Vert_{C^2},\\
\frac{|\nabla_x T_W(\xb(x,v))-\nabla_x T_W(\xb(y,v))|}{|\xb(x,v)-\xb(y,v)|^\beta}  & \lesssim \Vert \nabla_x^2 T_W\Vert_{\infty},\\
w_{\tilde{\theta}}(v)|v|^2\frac{\big|e^{-\frac{|v|^2}{2T_W(\xb(x,v))}}-e^{-\frac{|v|^2}{2T_W(\xb(y,v))}}\big|}{|\xb(x,v)-\xb(y,v)|^\beta}  &\lesssim \Vert \nabla_x T_W\Vert_{\infty},
\end{align*}
\begin{align*}
  &w_{\tilde{\theta}}(v)|v|^2\frac{\frac{-1}{\pi[T_W(\xb(x,v))]^3}+\frac{1}{\pi[T_W(\xb(y,v))]^3}+\frac{|v|^2}{4\pi [T_W(\xb(x,v))]^4}-\frac{|v|^2}{4\pi[T_W(\xb(x,v))]^4}}{|\xb(x,v)-\xb(y,v)|^\beta}e^{-\frac{|v|^2}{2T_W(\xb(x,v))}}  \\
  &   \lesssim \Vert T_W\Vert_{C^1}|v|^4e^{-\frac{|v|^2}{2T_W(\xb(x,v))}}\lesssim \Vert \nabla_x^2 T_W\Vert_{\infty},
\end{align*}
and thus~\eqref{Estimate for r beta} follows from $\tilde{\theta}\ll \frac{1}{T_W(x)}$.

Since $\partial_{\mathbf{x}_{p^1,i}}r(\eta_{p^1}(\mathbf{x}_{p^1}),v)=\frac{1}{\sqrt{2\pi}}\partial_{\mathbf{x}_{p^1,i}}M_W(\eta_{p^1}(\mathbf{x}_{p^1}),v)$, \eqref{Estimate for M_W beta} also follows.

Last we prove~\eqref{int f-int f}. We rewrite the LHS of~\eqref{int f-int f} as
\begin{align}
&\frac{1}{|\xb(x,v)-\xb(y,v)|^\beta}\times\Big[ \int_{n(\xb(x,v))\cdot v^1>0,n(\xb(y,v))\cdot v^1>0}  f(\xb(x,v),v^1)\sqrt{\mu}(v^1)|n(\xb(x,v))\cdot v^1| \notag\\
&-f(\xb(y,v),v^1)\sqrt{\mu(v^1)}|n(\xb(y,v))\cdot v^1|\dd v^1 \Big]  \label{int f-int f 1}\\
& +\frac{  \int_{|n(\xb(x,v))-n(\xb(y,v))| \geq \frac{n(\xb(x,v))\cdot v^1}{|v^1|}>0} f(\xb(x,v),v^1)\sqrt{\mu(v^1)}|n(\xb(x,v))\cdot v^1|}{|\xb(x,v)-\xb(y,v)|^\beta} \notag\\
 &+ \frac{\int_{|n(\xb(x,v))-n(\xb(y,v))| \geq \frac{n(\xb(x,v))\cdot v^1}{|v^1|}>0}f(\xb(y,v),v^1)\sqrt{\mu(v^1)}|n(\xb(y,v))\cdot v^1|}{|\xb(x,v)-\xb(y,v)|^\beta}  \label{int f-int f 2}.
\end{align}
Clearly from~\eqref{min: f} and~\eqref{min: nxb}, we have
\begin{align*}
 \eqref{int f-int f 1} &\lesssim \int_{n(\xb(x,v))\cdot v^1>0,n(\xb(y,v))\cdot v^1>0}  \frac{\Vert \alpha\nabla_x f\Vert_\infty \sqrt{\mu(v^1)}}{\min\{\alpha(\xb(x,v),v^1),\alpha(\xb(y,v),v^1)\}^\beta}+ \Vert \eta\Vert_{C^2}\Vert wf\Vert_\infty \sqrt{\mu(v^1)}   \\
  &\lesssim \Vert \alpha\nabla_x f\Vert_\infty+\Vert wf\Vert_\infty.
\end{align*}

For~\eqref{int f-int f 2}, from~\eqref{min: nxb} we bound
\begin{align*}
  & \frac{  \int_{|n(\xb(x,v))-n(\xb(y,v))| \geq \frac{n(\xb(x,v))\cdot v^1}{|v^1|}>0} f(\xb(x,v),v^1)\sqrt{\mu(v^1)}|n(x)\cdot v^1|}{|\xb(x,v)-\xb(y,v)|^\beta} \\
  &\lesssim \frac{|n(\xb(x,v))-n(\xb(y,v))|}{|\xb(x,v)-\xb(y,v)|^\beta}\int f(\xb(x,v),v^1)\sqrt{\mu(v^1)}|v^1|\lesssim \Vert wf\Vert_\infty.
\end{align*}

Then we conclude the lemma.

\end{proof}

Besides the boundary condition, we also need to estimate the collision operator. The next two lemmas describe the properties of the collision operator $K$ and $\Gamma$.

\begin{lemma}\label{Lemma: K,Gamma}
The linear Boltzmann operator $K(f)$ in~\eqref{linear operator} is given by
\[Kf(x,v)=\int_{\mathbb{R}^3}\mathbf{k}(v,u)f(x,u)\dd u.\]

The kernel $\mathbf{k}(v,u)$ satisfies:
\Be\label{k_varrho}
 |\mathbf{k}  (v,u)| \lesssim \mathbf{k}_\varrho (v,u), \  | \nabla_ u\mathbf{k}  (v,u)| \lesssim  \langle u\rangle\mathbf{k}_\varrho (v,u)/|v-u|, \ \mathbf{k}_\varrho (v,u) := e^{- \varrho |v-u|^2}/ |v-u|.
  \Ee
And for $3>c\geq 0$,
\begin{equation}\label{integrate k u}
\int_{\mathbb{R}^3}\mathbf{k}_\varrho(v,u)\frac{1}{|u|^c}\dd u\lesssim \frac{1}{|v|^c}.
\end{equation}

Moreover, for the operator $\nu$ and $\Gamma$ in~\eqref{linear operator}, we have
\begin{equation}\label{h bounded}
|K(f)+\Gamma(f,f)|= O(1)\Vert f\Vert_\infty=O(1)\Vert wf\Vert_\infty,
\end{equation}

\begin{equation}\label{nablav nu}
\nu\gtrsim 1,\quad |\nabla_v \nu| \lesssim 1,
\end{equation}

\begin{equation}\label{nablav K Gamma}
|\nabla_v \Gamma(f,f)|\lesssim \frac{\Vert wf\Vert_\infty^2}{|v|^2}+\frac{\Vert wf\Vert_\infty \Vert |v|^2\nabla_v f\Vert_\infty}{|v|^2},
\end{equation}

 \begin{align}
| \nabla_x \Gamma(f,f)(v) |= O(\| wf \|_\infty)\Big\{
  | \nabla_x f(v) |+
  \int_{\R^3} \mathbf{k}_\varrho(v,u)
 |\nabla_x f(u)|
  \dd u
  \Big\},\label{Gamma_est}
 \end{align}

\begin{equation}\label{G Gamma}
|G(x)\nabla_x \Gamma(f,f)(x,v)|=O(\Vert wf\Vert_\infty)\{|G(x)\nabla_x f(x,v)|+\int_{\mathbb{R}^3}\mathbf{k}_\varrho(v,u)|G(x)\nabla_x f(x,u)|\}.
\end{equation}

\end{lemma}

\begin{proof}

\textit{Proof of~\eqref{k_varrho}}.
Due to the Grad estimate in~\cite{R},
\begin{equation}\label{eqn: Grad estimate for gain}
 \Gamma_{\text{gain}}(\sqrt{\mu},f)+\Gamma_{\text{gain}}(f,\sqrt{\mu})=\int_{\mathbb{R}^3}\mathbf{k}_2(v,u)f(u)du,
\end{equation}
\[\nu(\sqrt{\mu}f)=\int_{\mathbb{R}^3}\mathbf{k}_1(v,u)f(u)du,\]
where
\begin{eqnarray}
\mathbf{k}_1(v,u)   &=&  C_{\mathbf{k}_1}|u-v|e^{-\frac{|v|^2+|u|^2}{2}}, \notag\\
 \mathbf{k}_2(v,u)  &=&  C_{\mathbf{k}_2}\frac{1}{|u-v|}e^{-\frac{1}{4}|u-v|^{2}-\frac{1}{4}
\frac{(|u|^{2}-|v|^{2})^{2}}{|u-v|^{2}}}. \notag
\end{eqnarray}

We compute the derivative:
\[|\nabla_u \mathbf{k}_1(v,u)|\lesssim e^{-\frac{|u|^2+|v|^2}{2}}+ |u||u-v|e^{-\frac{|u|^2+|v|^2}{2}}\lesssim e^{-\frac{|v-u|^2}{4}}\lesssim \frac{e^{-\varrho|v-u|^2}}{|v-u|^2}.\]
And
\begin{equation*}
  \begin{split}
    |\nabla_u \mathbf{k}_2(v,u)| &\lesssim \frac{1}{|v-u|^2}e^{-\frac{1}{4}|v-u|^2}+ \frac{1}{|v-u|}e^{-\frac{1}{4}|v-u|^2-\frac{1}{4}\frac{(|u|^2-|v|^2)^2}{|v-u|^2}} \\
   &\quad  \times \Big[|v-u|+\frac{|u|\Big||u|^2-|v|^2\Big||v-u|^2-(|u|^2-|v|^2)^2|v-u|}{|v-u|^4} \Big]  \\
      & \lesssim \frac{e^{-\varrho|v-u|^2}\langle u\rangle}{|v-u|^2},
  \end{split}
\end{equation*}
where we have used
\begin{eqnarray*}
e^{-\frac{1}{4}\frac{(|u|^2-|v|^2)^2}{|v-u|^2}} \frac{|u|\Big||u|^2-|v|^2\Big||v-u|^2}{|v-u|^4}  &\lesssim& \frac{|u|}{|v-u|},  \\
e^{-\frac{1}{4}\frac{(|u|^2-|v|^2)^2}{|v-u|^2}} \frac{(|u|^2-|v|^2)^2|v-u|}{|v-u|^4}  &\lesssim &  \frac{1}{|v-u|}.
\end{eqnarray*}

\textit{Proof of~\eqref{integrate k u}.} We consider two cases. When $|u|>\frac{|v|}{2}$, we have
\begin{align*}
 \int_{|u|>\frac{|v|}{2}} \mathbf{k}_\varrho(v,u)\frac{1}{|u|^c}\dd u &\lesssim \int_{|u|>\frac{|v|}{2}} \frac{e^{-\varrho|v-u|^2}}{|v-u|}\frac{1}{|u|^c}\dd u  \\
   &  \lesssim \frac{1}{|v|^c}   \int_{\mathbb{R}^3}\frac{e^{-\varrho|v-u|^2}}{|v-u|}\dd u\lesssim \frac{1}{|v|^c}.
\end{align*}

When $|u|\leq \frac{|v|}{2}$ we bound $|v-u|\geq \frac{|v|}{2}$, and thus
\begin{align*}
  \int_{|u|\leq\frac{|v|}{2}} \mathbf{k}(v,u)\frac{1}{|u|^c}\dd u  &\lesssim \frac{e^{-\varrho|v|^2/2}}{|v|} \int_{|u|\leq \frac{|v|}{2}} \frac{1}{|u|^c} \dd u \\
   & \lesssim \frac{e^{-\varrho|v|^2/2}}{|v|}  \int_{0\leq r\leq \frac{|v|}{2}} \int_{\partial B(0,r)}\dd S \frac{1}{|r|^c} \dd r\\
   &\lesssim \frac{e^{-\varrho|v|^2/2}}{|v|} |v|^{3-c}\lesssim \frac{e^{-\varrho|v|^2/2}|v|^{2}}{|v|^c} \lesssim \frac{1}{|v|^c}.
\end{align*}
In the second line we used the polar coordinate with $|u|=|r|$. In the third line we used $c<3$ to the $r$ integral.

Then we conclude~\eqref{integrate k u}.

\textit{Proof of~\eqref{h bounded}}.
For $K(f)$ we bound
\begin{align*}
K(f)\lesssim \Vert f\Vert_\infty \int_{\mathbb{R}^3} |\mathbf{k}(v,u)|\dd u\lesssim \Vert f\Vert_\infty\lesssim\Vert wf\Vert_\infty,   \\
\end{align*}
where we have used $|\mathbf{k}(v,u)|\lesssim \mathbf{k}_\varrho(v,u)\in L^1_u$.

For $\Gamma$, clearly

\begin{equation}\label{gamma bounded}
|\Gamma_{\mathrm{gain}} (f, f)| \ \lesssim \ |\Gamma_{\mathrm{gain}} (e^{-\varrho |v|^{2}} , |f|)| \times || wf ||_{\infty}.
\end{equation}
By~\eqref{eqn: Grad estimate for gain} we bound $|\Gamma_{\mathrm{gain}} (e^{-\theta |v|^{2}} , |f|)|$ using different exponent of $\mathbf{k}_2(v,u)$, we conclude that
\[\Gamma_{\text{gain}}(f,f)\lesssim \Vert wf\Vert_\infty^2\lesssim \Vert wf\Vert_\infty.\]
For the other term we bound
\begin{align*}
\begin{split}
|\nu(\sqrt{\mu}f)f(v)|&\lesssim \Vert wf\Vert_\infty \int_{\mathbb{R}^3}|v-u|e^{-\varrho |v|^2}\sqrt{\mu(u)} |f(u)| \\
&\lesssim \Vert wf\Vert_\infty^2 \int_{\mathbb{R}^3} |v-u|e^{-C|v-u|^2}\lesssim \Vert wf\Vert_\infty \Vert f\Vert_\infty\lesssim \Vert f\Vert_\infty\lesssim \Vert wf\Vert_\infty,
\end{split}
\end{align*}
where we have used
\[e^{-\varrho|v|^2}e^{-\varrho|u|^2}\lesssim e^{-C(|v|^2+|u|^2)}\lesssim   e^{-\frac{C}{2}|v-u|^2}.\]

The proof for~\eqref{nablav nu} is standard.

\textit{Proof of~\eqref{nablav K Gamma}}.
The velocity derivative for the nonlinear Boltzmann operator reads
\begin{align}
  \nabla_v \Gamma(f,f) &=\nabla_v \left( \Gamma_{\text{gain}}(f,f)-\Gamma_{\text{loss}}(f,f) \right) \notag \\
   & \Gamma_{\textrm{gain}} (\nabla_v f,f) + \Gamma_{\textrm{gain}} ( f,\nabla_v f) - \Gamma_{\textrm{loss}} (\nabla_v f,f) -\Gamma_{\textrm{loss}} ( f,\nabla_v f)\label{gamma nablav f f} \\
   &  + \Gamma_{v,\text{gain}} (f,f)-\Gamma_{v,\text{loss}}(f,f).\label{gamma_v f}
\end{align}
Here we have defined
\begin{equation}\begin{split}\label{Gamma_v}
          & \Gamma_{v,\textrm{gain}}(f,f) - \Gamma_{v,\textrm{loss}}(f,f) \\
			&:=\int_{\mathbb{R}^{3}}   \int_{\S^2}   | u \cdot \omega|
			f(v+ u_\perp) f(v + u_\parallel)
			\nabla_v\sqrt{\mu(v+u)} \dd \omega  \mathrm{d}u  \\
			&\quad - \int_{\mathbb{R}^{3}}   \int_{\S^2}   | u \cdot \omega|
			f(v+u) f(v)
			\nabla_v \sqrt{\mu(v+u)} \dd \omega  \mathrm{d}u.
		\end{split}\end{equation}
Replacing the $\nabla_x$ by $\nabla_v$ in~\eqref{different exponent} and~\eqref{gamma loss}, we use~\eqref{integrate k u} with $c=2$ to conclude
\begin{align*}
  \eqref{gamma nablav f f} & \lesssim \Vert wf\Vert_\infty \int_{\mathbb{R}^3} \frac{e^{-\varrho |v-u|^2}}{|v-u|} |\nabla_v f| \\
   & \lesssim \Vert wf\Vert_\infty \Vert |v|^2\nabla_v f\Vert_\infty \int_{\mathbb{R}^3}\frac{e^{-\varrho |v-u|^2}}{|v-u|}\frac{|1}{|u|^2} \\
   & \lesssim \frac{\Vert |v|^2\nabla_v f\Vert_\infty}{|v|^2}.
\end{align*}

Then we further compute
\begin{align*}
  \eqref{gamma_v    f} &\lesssim     \Vert wf\Vert_\infty^2 \int_{\mathbb{R}^3} |u|[e^{-\varrho|v+u_\perp|^2}e^{-\varrho |v+u_\parallel|^2}+e^{-\varrho|v+u|^2}e^{-\varrho|v|^2}] e^{-\frac{|v+u|^2}{2}}  \\
   & \lesssim   \Vert wf\Vert_\infty^2      \int_{\mathbb{R}^3} |u|e^{-c|v|^2}e^{-c|u|^2}\dd u\\
   & \lesssim \Vert wf\Vert_\infty^2   \frac{|v|^2e^{-c|v|^2}}{|v|^2}\lesssim \frac{\Vert wf\Vert_\infty^2}{|v|^2},
\end{align*}
where we have used
\begin{align*}
   & e^{-\varrho |v+u_\perp|^2}e^{-\varrho|v+u_\parallel|^2}= e^{-(\varrho|v|^2+2v\cdot (u_\perp+u_\parallel)+\varrho|u|^2 )} e^{-\varrho|v|^2} \\
   & =e^{-\varrho|v+u|^2}e^{-\varrho |v|^2},
\end{align*}
and
\begin{align*}
e^{-\varrho|v+u|^2}e^{-\varrho |v|^2}   &=e^{-\varrho|v|^2/2} e^{-\varrho(3|v|^2/2+2v\cdot u+|u|^2)}  \\
   & =e^{-\varrho|v|^2/2}e^{-\varrho (\sqrt{3/2}v+\sqrt{2/3}u)^2  }e^{-u^2/3}.
\end{align*}

\textit{Proof of~\eqref{Gamma_est}}.
From~\eqref{gamma bounded} we have
\begin{equation}\label{different exponent}
|\Gamma_{\text{gain}}(f,\partial_x f)+\Gamma_{\text{gain}}(\partial_x f,f)|\lesssim    \Vert wf\Vert_\infty \int_{\mathbb{R}^3}\frac{e^{-\varrho |v-u|^2}}{|v-u|}|\partial_x f|du       .
\end{equation}
For $|\nu(\sqrt{\mu}\partial_x f)f(v)|$ we have

\begin{align}
 |\nu(\sqrt{\mu}\partial_x f)f(v)|  & \lesssim \Vert wf\Vert_\infty e^{-\varrho|v|^2}\nu(\sqrt{\mu} \partial_x f)(v)\notag\\
   & \lesssim \int_{\mathbb{R}^3}|v-u|e^{-\varrho|v|^2}\sqrt{\mu(u)}|\partial_x f(u)|\lesssim \int_{\mathbb{R}^3}\frac{e^{-\varrho|v-u|^2}}{|v-u|}|\partial_x f(u)|du,\label{gamma loss}
\end{align}
where we have used $e^{-\varrho|v|^2}|v-u|\sqrt{\mu(u)}\lesssim \frac{e^{-\varrho|v-u|^2}}{|v-u|}$.

\textit{Proof of~\eqref{G Gamma}.} Since
\begin{align}
   & G(x)\nabla_x \Gamma(f,f)(x,v)=G(x)\Gamma(\nabla_x f,f)+G(x)\Gamma(f,\nabla_x f)\notag \\
   &  =\Gamma(G(x)\nabla_xf,f )+\Gamma(f,G(x)\nabla_x f),  \label{G nabla Gamma}
\end{align}
from~\eqref{Gamma_est} we conclude~\eqref{G Gamma}.

\end{proof}

\begin{lemma}\label{Lemma: k tilde}
If $0<\frac{\tilde{\theta}}{4}<\varrho$, if $0<\tilde{\varrho}<  \varrho- \frac{\tilde{\theta}}{4}$,
\begin{equation}\label{k_theta}
\mathbf{k}_{\varrho}(v,u) \frac{e^{\tilde{\theta} |v|^2}}{e^{\tilde{\theta} |u|^2}} \lesssim  \mathbf{k}_{\tilde{\varrho}}(v,u) ,
\end{equation}
where $\mathbf{k}_{\varrho}$ is defined in~\eqref{k_varrho}.

\end{lemma}

\begin{proof}

Note
		\Be\notag
		\begin{split}
			\mathbf{k}_{  \varrho}(v,u) \frac{e^{\tilde{\theta} |v|^2}}{e^{\tilde{\theta} |u|^2}}
			=  \frac{1}{|v-u| } \exp\left\{- {\varrho} |v-u|^{2}
			-  {\varrho} \frac{ ||v|^2-|u|^2 |^2}{|v-u|^2} + \tilde{\theta} |v|^2 - \tilde{\theta} |u|^2
	\right\}.
		\end{split}\Ee

		Let $v-u=\eta $ and $u=v-\eta $. Then the exponent equals
		\begin{eqnarray*}
			&&- \varrho|\eta |^{2}-\varrho\frac{||\eta |^{2}-2v\cdot \eta |^{2}}{%
				|\eta |^{2}}-\tilde{\theta} \{|v-\eta |^{2}-|v|^{2}\} \\
			&=&-2 \varrho |\eta |^{2}+ 4 \varrho v\cdot \eta - 4 \varrho\frac{|v\cdot
				\eta |^{2}}{|\eta |^{2}}-\tilde{\theta} \{|\eta |^{2}-2v\cdot \eta \} \\
			&=&(-2 \varrho-\tilde{\theta}  )|\eta |^{2}+(4 \varrho+2\tilde{\theta} )v\cdot \eta -%
			4 \varrho\frac{\{v\cdot \eta \}^{2}}{|\eta |^{2}}.
		\end{eqnarray*}%
		If $0<\tilde{\theta} <4 \varrho$ then the discriminant of the above quadratic form of
		$|\eta |$ and $\frac{v\cdot \eta }{|\eta |}$ is
		\begin{equation*}
		(4 \varrho+2\tilde{\theta} )^{2}-4
		(-2 \varrho-\tilde{\theta}  )(-%
		4 \varrho)
		=4\tilde{\theta} ^{2}- 16 \varrho \tilde{\theta}<0.
		\end{equation*}%
		Hence, the quadratic form is negative definite. We thus have, for $%
		0<\tilde{\varrho}< \varrho - \frac{\tilde{\theta}}{4}  $, the following perturbed quadratic form is still negative definite: $
		-(\varrho - \tilde{\varrho})|\eta |^{2}-(\varrho - \tilde{\varrho})\frac{||\eta
			|^{2}-2v\cdot \eta |^{2}}{|\eta |^{2}}-\tilde{\theta} \{|\eta |^{2}-2v\cdot \eta \}  \leq 0.$\end{proof}

%
%
%
%
%
%
%
%
%
%
%
%
%

\section{Differentiation along the Stochastic Cycles: Mixing via diffuse reflection and transport}
The main purpose of this section is to provide crucial differentiation form of the transport equation with the diffuse reflection boundary condition, which will be stated in Proposition \ref{prop_fx}. Several geometric integration by parts will be employed as being described in Section 1.3.

Consider a sequence of linear transport equation for $\ell\geq 1$ with the inflow boundary condition
\begin{align}
v\cdot \nabla_x f^{\ell} + \nu(v) f^{\ell} &= h^{\ell}(x,v), \ \ (x,v) \in \O \times \R^3,
\label{inhomo_transport}
\\
f^{\ell}( x,v)&= g^{\ell}(x,v), \ \  (x,v )\in\gamma_-. \label{inflowBC}
\end{align}
Here we set $f^{0}=0.$

Later we will substitute the $h^{\ell}$ by the sequence of collision operator~\eqref{h_K} and $g^{\ell}$ by the sequence of boundary condition:
\begin{equation}\label{diffuse_f}
\begin{split}
f^{\ell}(x,v)|_{\gamma_-}&=\frac{M_W(x,v)}{\sqrt{\mu(v)}}\int_{n(x)\cdot v^1>0}f^{\ell-1}(x,v^1)\sqrt{\mu(v^1)} \{n(x)\cdot v^1\}\dd v^1
+r(x,v),
\end{split}
\end{equation}
where $r(x,v)$ is defined in~\eqref{remainder term}.

 Note that from the collision operator~\eqref{h_K} and boundary condition~\eqref{diffuse_f} and $f^0 = 0$, 
\[ h^1(x,v)=0, \ \ \ g^1(x,v) = r(x,v).\]

We have the following expansion:
\begin{proposition}\label{prop_fx} Suppose $f$ solves inhomogeneous steady transport equation (\ref{inhomo_transport}) with the diffuse BC (\ref{diffuse_f}). Then
\begin{align}
&w_{\tilde{\theta}}(v)\p_{x_i} f^\ell(x,v)\notag\\
& =O(1) w_{\tilde{\theta}}(v)\frac{ n_i (x^1) }{\alpha(x,v)}
\bigg\{
 \frac{\nu(v)}{w(v)}\| w f^\ell\|_\infty+
|v|   \frac{M_W(x^1, v)}{\sqrt{\mu(v)}}\| wf ^{\ell-1}\|_\infty
+ |v| |  \underline{\nabla}_{x^1}  r(x^1)| (1+ \| wf^{\ell-1} \|_\infty)
\bigg\}
\label{px_f_1}\\
& +  \frac{O(1)}{\alpha(x,v)} e^{-\nu(v) t}  (w_{\tilde{\theta}}\alpha \p_{x_i} f^\ell)(x-tv,v)
\label{px_f_2}
\\
&  +\int_{\max\{0,t-\tb \}}^t e^{-\nu(v)(t-s)}w_{\tilde{\theta}}(v) \p_{x_i}  \underbrace{ h^\ell (   x-(t-s) v,v) }_{(\ref{px_f_3})_*} \dd s
\label{px_f_3}
\\
& + O(1)e^{-\nu (v) \tb} \frac{n_i(x^1) |v|}{\alpha(x,v)} \frac{w_{\tilde{\theta}}(v) M_W(x^1, v)}{\sqrt{\mu(v)}}
\int_{n(x^1) \cdot v^1>0}  \dd v^1\{n^1 \cdot v^1\}\sqrt{\mu(v^1)}\label{px_f_int}\\
& \ \ \ \   \ \ \ \    \times  \bigg\{
  e^{-\nu(v^1) t^1} \frac{1}{\alpha(x^1 ,v^1)} (\alpha  \underline{\nabla}_{x^1}  f^{\ell-1})(x^1-t^1 v^1,v^1)
  \label{px_f_4} \\
& \ \ \ \   \ \ \ \   \ \ \ \  +  \int^{t^1}_{\max\{0,  t^1- \tb^1\}} e^{- \nu (v^1) (t^1-s^1)}
 \underline{\nabla}_{x^1}   \underbrace{  h^{\ell-1}(x^1- (t^1-s^1) v^1,v^1)}_{(\ref{px_f_5})_*}   \dd s^1\bigg\}.\label{px_f_5}
\end{align}
Here $ \underline{\nabla}_{x^1}  a(x^1+ \cdot )$ stands the tangential derivative $\nabla_{\mathbf{x}^1_{p^1}}[a(\eta_{p^1} (\mathbf{x}^1_{p^1})+ \cdot  )]$ in a local coordinate of (\ref{O_p}) as in (\ref{fBD_x1}).

\end{proposition}

 \begin{proof}[\textbf{Proof of Proposition \ref{prop_fx}}] Consider $f^\ell$ solves (\ref{inhomo_transport}) and (\ref{diffuse_f}). Choose $t\gg 1$. Recall (\ref{fBD_x1}). Same as (\ref{fx_x_1})-(\ref{fx_x_4}), for $k\geq 1$, $n(\mathbf{x}^k_{p^k}) \cdot v^k>0$, and $i=1,2$,  or $k=0$ with $i=1,2,3,$ 
\begin{align}
 w_{\tilde{\theta}}(v^k)\p_{\mathbf{x}^{k}_{p^{k},i}}[
f^\ell( \eta_{p^{k}} ( \mathbf{x}_{p^{k} }^{k}  ),    {v}^{k} ) ]
=   &  \mathbf{1}_{t^k\geq \tb^k}e^{- \nu^k\tb^k}  w_{\tilde{\theta}}(v^k) \p_{\mathbf{x}^{k}_{p^{k},i}}[  f^\ell(\xb(\eta_{p^{k}} ( \mathbf{x}_{p^{k} }^{k}  )  ,    {v}^{k} ), v^k)]  \label{fx1_x1_1}\\
- & \mathbf{1}_{t^k\geq \tb^k}  \nu^k \p_{\mathbf{x}^k_{p^k,i}} \tb^k e^{- \nu^k\tb^k } w_{\tilde{\theta}}(v^k) f^\ell(\xb(\eta_{p^{k}} ( \mathbf{x}_{p^{k} }^{k}  )  ,    {v}^{k} ), v^k)
\label{fx1_x1_2}\\
 +&
  \mathbf{1}_{t^k < \tb^k}
 e^{- \nu^k t^k }w_{\tilde{\theta}}(v^k) \p_{\mathbf{x}^k_{p^k,i}}  [ f ^\ell
 (\eta_{p^{k}} ( \mathbf{x}_{p^{k} }^{k}  )- t^k v^k, v^k
 )]\label{fx1_x1_3}\\
 + & \int^{t^k}_{\max\{
 0,t^k-\tb ^k
 \}}
 e^{-\nu ^k (t ^k- s^k)}w_{\tilde{\theta}}(v^k) \p_{\mathbf{x}^k_{p^k,i}}
 [h ^\ell(\eta_{p^{k}} ( \mathbf{x}_{p^{k} }^{k}  )- (t^k-s^k) v^k, v^k
 ) ]\dd s^k\label{fx1_x1_4}\\
+&  \partial_{\mathbf{x}_{p^k,i}^k}\tb^k e^{-\nu^k \tb^k} w_{\tilde{\theta}}(v^k) h^\ell(\xb(\eta_{p^k}(\mathbf{x}_{p^k}^k),v^k),v^k), \label{fx1_x1_5}
 \end{align}
where we  denoted $\nu^k = \nu(v^k)$.

\smallskip

\textit{Estimate of (\ref{fx1_x1_1}).}
From (\ref{diffuse_f}) and (\ref{eqn: diffuse for f}) with replacing $f$ by $f^\ell$, for $k\geq 1$ with $i=1,2$, or $k=0$ with $i=1,2,3$, if $\xb(\eta_{p^{k}} ( \mathbf{x}_{p^{k} }^{k}  ) ,v^{k } ) \in \mathcal{O}_{p^{k+1}}$ then 
\begin{align}
&w_{\tilde{\theta}}(v^k)\p_{\mathbf{x}^{k } _{p^{k },i}}[ f^\ell( \xb(\eta_{p^{k}} ( \mathbf{x}_{p^{k} }^{k}  ) ,v^{k } ), v^{k } )]
\notag \\
 =& \sum_{j=1,2} \frac{\p \mathbf{x}^{k+1 }_{p^{k+1},j}}{\p{\mathbf{x}^{k  }_{p^{k },i}}}  w_{\tilde{\theta}}(v^k) \p_{\mathbf{x}_{p^{k+1},j}^{k+1}} [  f^\ell( \eta_{p^{k+1}} ( \mathbf{x}^{k+1}_{p^{k+1} } ),
 v^{k }
 )]\label{fBD_x0}
\\
=& \sum_{j=1,2} \frac{\p \mathbf{x}^{k+1}_{p^{k+1}, j}}{\p{\mathbf{x} ^{k }_{p^{k },i}}}\Big[
\frac{w_{\tilde{\theta}}(v^k)M_W
( \eta_{p^{k+1}} ( \mathbf{x}^{k+1}_{p^{k+1} } ),
v^{k }
 )
}{\sqrt{\mu
(v^{k })
}}
\notag
\\
&\times
\int_{\mathbf{v}^{k+1}_{p^{k+1},3}>0}  \underbrace{\p_{\mathbf{x}^{k+1}_{p^{k+1},j}}
 [f^{\ell-1}( \eta_{p^{k+1}} ( \mathbf{x}_{p^{k+1} }^{k+1}  ), T^t_{\mathbf{x}^{k+1}_{p^{k+1}}} \mathbf{v}^{k+1}_{p^{k+1}}) ]}_{(\ref{fBD_x})_*}\sqrt{\mu(\mathbf{v}^{k+1}_{p^{k+1}})}
 \mathbf{v}^{k+1}_{p^{k+1},3}
\dd  \mathbf{v}^{k+1}_{p^{k+1}}\label{fBD_x}
\\
+& \
O(1) \sum_{j=1,2}  w_{\tilde{\theta}}(v^k)\p_{\mathbf{x}^{k+1}_{p^{k+1},j}} r( \eta_{p^{k+1}} ( \mathbf{x}^{k+1}_{p^{k+1}} ), v^{k }) \{1+ \| w f ^{\ell-1}\|_\infty \}\Big].  \label{fBD_xr}
\end{align}
Note that the above equalities for $k=0$ gives an identity of $\p_{x_i}[ f^\ell( \xb(x  ,v  ), v   )]$.

It is relatively simple to derive that, from (\ref{xip deri xbp}),
\Be\label{est:fBD_xr}
(\ref{fBD_xr}) = O(1) |w_{\tilde{\theta}} r (x^{k+1})|
\{1+ \| w f^{\ell-1} \|_\infty\}
.
\Ee

Now we consider (\ref{fBD_x}). We compute $(\ref{fBD_x})_*= (\ref{fBD_x1})_{(k,i,a)\rightarrow (k+1, j, f^{\ell-1})}+ (\ref{v_under_v})$. Here (\ref{v_under_v}) is given by
\Be\begin{split}
& \left(\p_{\mathbf{x}^{k+1}_{p^{k+1},j}} T^t_{\mathbf{x}^{k+1}_{p^{k+1}}}   \mathbf{v}^{k+1}_{p^{k+1}}\right) \cdot
\nabla_v
 f^{\ell-1} ( \eta_{p^{k+1}} ( \mathbf{x}_{p^{k+1} }^{k+1}  ), T^t_{\mathbf{x}^{k+1}_{p^{k+1}}} \mathbf{v}^{k+1}_{p^{k+1}})\\
= &
  \sum_{l,m} \frac{\p }{ \p{\mathbf{x}^{k+1}_{p^{k+1}, j}}  }\left(
  \frac{\p_m \eta_{p^{k+1},  l} (\mathbf{x}^{k+1}_{p^{k+1}})}{\sqrt{g_{p^{k+1},mm}(\mathbf{x}^{k+1}_{p^{k+1}})} }\right)
   \mathbf{v}_{p^{k+1},m}^{k+1}\p_{v_ l}f^{\ell-1} (  \eta_{p^{k+1}} ( \mathbf{x}_{p^{k+1} }^{k+1}  ), T^t_{\mathbf{x}^{k+1}_{p^{k+1}}} \mathbf{v}^{k+1}_{p^{k+1}})\\
=&  \sum_{  m,n}
(\ref{v_under_v_mn})_{mn}
\mathbf{v}_{p^{k+1},m}^{k+1} \p_{\mathbf{v}^{k+1}_{p^{k+1},n}} [f^{\ell-1}(  \eta_{p^{k+1}} ( \mathbf{x}_{p^{k+1} }^{k+1}  ) , T^t_{\mathbf{x}^{k+1}_{p^{k+1}}}\mathbf{v}^{k+1}_{p^{k+1}})],
\label{v_under_v}
\end{split}\Ee
where
\Be\label{v_under_v_mn}
(\ref{v_under_v_mn})_{mn} :=\sum_l \frac{\p }{ \p{\mathbf{x}^{k+1}_{p^{k+1}, j}}  }\left(
  \frac{\p_m \eta_{p^{k+1},l} (\mathbf{x}^{k+1}_{p^{k+1}})}{\sqrt{g_{p^{k+1},mm}(\mathbf{x}^{k+1}_{p^{k+1}})} }\right)
  \frac{\p_n \eta_{p^{k+1},l} (\mathbf{x}^{k+1}_{p^{k+1}})}{\sqrt{g_{p^{k+1},nn}(\mathbf{x}^{k+1}_{p^{k+1}})} } .
\Ee
Here we have used (\ref{T}) and (\ref{bar_v}). 

\hide
From (\ref{T})
\Be
\begin{split}\label{v_under_v}
  ( \p_{x^1_{j}}T^t_{x^1} \underline{v}^1)_\ell  \p_{v_\ell}f (\cdot, T^t_{x^1} \underline{v}^1)
= &   \sum_m  \p_{x^1_{j}} (\tau_{p,m})_\ell  \underline{v}_m^1\p_{v_\ell}f (\cdot, T^t_{x^1} \underline{v}^1)\\
=&  \sum_{  m,k}  \underline{v}_m^1 \p_{x^1_{j}} (\tau_{p,m})_\ell
 (\tau_{p,k})_\ell \p_{\underline{v}^1_k} [f(\cdot, T^t_{x^1} \underline{v}^1)].
 \end{split}\Ee
 \unhide
First we consider a contribution of (\ref{v_under_v}) in (\ref{fBD_x}). We substitute (\ref{v_under_v})-(\ref{v_under_v_mn}) for $(\ref{fBD_x})_*$ and then apply the integration by parts with respect to $\p_{\mathbf{v}_{p^{k+1}}^{k+1}}$ to derive that
\Be\begin{split}\label{IBP_v}
& \int_{\mathbf{v}^{k+1}_{p^{k+1},3}>0} f^{\ell-1} ( \eta_{p^{k+1}} ( \mathbf{x}_{p^{k+1} }^{k+1}  ), T^t_{\mathbf{x}^{k+1}_{p^{k+1}}} \mathbf{v}^{k+1}_{p^{k+1}})\\
& \ \ \ \  \ \ \ \   \times
\sum_{m,n}   (\ref{v_under_v_mn})_{mn}
\p_{\mathbf{v}^{k+1}_{p^{k+1},n}}
\big[
  \mathbf{v}^{k+1}_{p^{k+1}, m}  \mathbf{v}^{k+1}_{p^{k+1},3} \sqrt{\mu(\mathbf{v}^{k+1}_{p^{k+1}})}
 \big]
\dd
 \mathbf{v}^{k+1}_{p^{k+1}}
  =   O(1)\| \eta \|_{C^2}
  \| w f^{\ell-1} \|_\infty
  \hide\sum_{j=1,2} \frac{\p \mathbf{x}^{k}_{p^{k}, j}}{\p{\mathbf{x} ^{k-1}_{p^{k-1},i}}}
\frac{M_W
( \eta_{p^{k}} ( \mathbf{x}^{k}_{p^{k} } ),
| \mathbf{v}^{k-1}_{p^{k-1}}|
 )
}{\sqrt{\mu
( \mathbf{v}^{k-1}_{p^{k-1}} )
}}\unhide.
\end{split}\Ee
Here we have used $f^{\ell-1} ( \eta_{p^{k+1}} ( \mathbf{x}_{p^{k+1} }^{k+1}  ), T^t_{\mathbf{x}^{k+1}_{p^{k+1}}} \mathbf{v}^{k+1}_{p^{k+1}})
\sum_{m,n}   (\ref{v_under_v_mn})_{mn}
  \mathbf{v}^{k+1}_{p^{k+1}, m}  \mathbf{v}^{k+1}_{p^{k+1},3} \sqrt{\mu(\mathbf{v}^{k+1}_{p^{k+1}})}
\equiv 0
$ when $\mathbf{v}^{k+1}_{p^{k+1},3}=0$ for $\|wf^{\ell-1}\|_\infty<\infty$.


\hide

Then
\[~\eqref{eqn: diffuse for f}=\int_{n(x_1^1,x_2^1,0)\cdot v^1>0} f(\eta_{p^1}(x_1^1,x_2^1),v^1)\sqrt{\mu(v^1)}\Big(n(x_1^1,x_2^1,0)\cdot v^1 \Big) dv^1\]
\[=\int_{\underline{v}_3^1>0}   f(\eta_{p^1}(x_1^1,x_2^1,0),T^{-1}(x_1^1,x_2^1,0)\underline{v}^1)  \sqrt{\mu(\underline{v}^1)}\underline{v}_3^1 d\underline{v}^1.         \]

 Note that for $\ell=0$
 \begin{align}
  (\ref{fBD_x})_{\ell=0} 
&= \sum_{j=1,2}\frac{\p x^1_j }{\p x_i} \frac{M_W}{\sqrt{\mu}} \int_{n(x^1) \cdot v^1>0 }
\underline{\p_{x^1_j}[
f( \eta_{p^1} (x_\parallel^1 (x,v ),0),  v^1)]}\sqrt{\mu( v^1)} \{n(x^1) \cdot v^1\} \dd {v}^1
\label{boundary_x_parallel}
\\
&+  \sum_{j=1,2}\frac{\p x^1_j }{\p x_i} \frac{M_W}{\sqrt{\mu}}   \int_{\underline{v}^1_3>0}
\sum_\ell \underline{( \p_{x_j^1}T^t_{x^1} \underline{v}^1)_\ell
\p_{v_\ell }
f( \eta_{p^1} (x_\parallel^1 (x,v ),0), T^t_{x^1} \underline{v}^1)}\sqrt{\mu(\underline{v}^1)}\underline{v}^1_3 \dd\underline{v}^1 .
\label{boundary_v}
 \end{align}
\hide
\Be\begin{split}
   \p_{\underline{v}_\ell^1} [ g(T^t_{x^1} \underline{v}^1)]
& = \sum_k \frac{\p (T^t_{x^1} \underline{v}^1)_k}{\p \underline{v}_\ell^1}
\p_{v_k} g  (T^t_{x^1} \underline{v}^1) \\
&= \sum_k \p_{\underline{v}_\ell^1}
\big[
(\tau _{p,1})_k \underline{v}_1^1 + (\tau _{p,2})_k \underline{v}_2^1
+ (n_p)_k \underline{v}_3^1
\big] \p_{v_k} g  (T^t_{x^1} \underline{v}^1)\\
&= \sum_k (\tau _{p,\ell})_k  \p_{v_k} g  (T^t_{x^1} \underline{v}^1)
 \end{split}
\Ee
On the other hand \unhide
\unhide

\smallskip

\textit{Estimate of a contribution of $(\ref{fBD_x1})_{(k,i,a)\rightarrow (k+1, j, f^{\ell-1})}$ in (\ref{fBD_x}).} Since the velocity variables of

$(\ref{fBD_x1})_{(k,i,a)\rightarrow (k+1, j, f^{\ell-1})}$ is written in Cartesian coordinate as $v^{k+1}$ (not $\mathbf{v}^{k+1}_{p^{k+1}}$) we rewrite the $\mathbf{v}_{p^{k+1}}^{k
+1}$-integration of (\ref{fBD_x}) in $v^{k+1}$-integration. Then, along the trajectory, $(\ref{fBD_x1})_{(k,i,a)\rightarrow (k+1, j, f^{\ell-1})}$ can be represented by (\ref{fx1_x1_1})-(\ref{fx1_x1_4}) with $(k,i,\ell)\rightarrow (k+1, j,\ell-1)$. Here we further replace $(\ref{fx1_x1_1})_{(k,i,\ell)\rightarrow (k+1, j,\ell-1)}$ by $\mathbf{1}_{t^{k+1} \geq \tb^{k+1}} \times $ $
e^{-\nu^{k+1} \tb^{k+1}} (\ref{fBD_x0})_{(k,i,j,\ell)\rightarrow (k+1, j,j^\prime,\ell-1)}$. We note that we do not use a further expansion of (\ref{fBD_x})-(\ref{fBD_xr}). Throughout the process, we derive an identity
\begin{align}
& (\ref{fBD_x}) \text{ with } [(\ref{fBD_x})_*=(\ref{fBD_x1})_{(k,i,a)\rightarrow (k+1, j,f^{\ell-1})}] \notag
\\
\hide&=
\int_{
n^{k+1} \cdot v^{k+1}>0
}
\sum_{p \in \mathcal{P}}
\iota_{p} (\xb(x^{k+1}, v^{k+1}))
\mathbf{1}_{t^{k+1} \geq \tb^{k+1}} e^{-\nu^{k+1} \tb^{k+1}}
\sum_{j^\prime=1,2}
\frac{\p \mathbf{x}^{k +2}_{p^{k+2},j^\prime}}{\p{\mathbf{x}^{k+1  }_{p^{k+1 },j}}}  \p_{\mathbf{x}_{p^{k+2},j^\prime}^{k+2}} [  f^{\ell-1}( \eta_{p^{k+2}} ( \mathbf{x}^{k+2}_{p^{k+2} } ),
 {v}^{k+1}
 )]\label{fBD_x_e1} \\
 & \ \ \ \ \ \ \ \ \ \ \ \ \ \ \  \ \ \ \ \  \times
  \sqrt{\mu (v^{k+1 })}
\{n ^{k+1}  \cdot v^{k+1}\} 
\dd  v^{k+1}  
\notag
\\
\unhide
&=\int_{
n^{k+1} \cdot v^{k+1}>0
}
\mathbf{1}_{t^{k+1} \geq \tb^{k+1}} e^{-\nu^{k+1} \tb^{k+1}}
\sum_{p^{k+2} \in \mathcal{P}}
\iota_{p^{k+2}} (\xb(x^{k+1}, v^{k+1}))
\notag \\
 & \ \ \ \ \ \ \ \    \times \sum_{j^\prime=1,2}
\frac{\p \mathbf{x}^{k +2}_{p^{k+2},j^\prime}}{\p{\mathbf{x}^{k+1  }_{p^{k+1 },j}}}  \p_{\mathbf{x}_{p^{k+2},j^\prime}^{k+2}} [  f^{\ell-1}( \eta_{p^{k+2}} ( \mathbf{x}^{k+2}_{p^{k+2} } ),
 {v}^{k+1}
 )]
  \sqrt{\mu (v^{k+1 })}
\{n ^{k+1}  \cdot v^{k+1}\} 
\dd  v^{k+1}  
\label{fBD_x_e1}\\
 &  +
\int_{
n ^{k+1}  \cdot v^{k+1}>0
}
\sum_{p^{k+2} \in \mathcal{P}}
\iota_{p^{k+2}} (\xb(x^{k+1}, v^{k+1}))
[
(\ref{fx1_x1_2}) + (\ref{fx1_x1_3}) + (\ref{fx1_x1_4})]_{(k,i, \ell ) \rightarrow (k+1, j,\ell-1)}
\notag\\
 & \ \ \ \ \ \ \ \    \times \sqrt{\mu (v^{k+1 })}
\{n ^{k+1}  \cdot v^{k+1}\} 
\dd  v^{k+1} . \label{fBD_x_e2}
\end{align}
Here we have denoted $n^{k+1}= n(x^{k+1})$. It is relatively easy to derive
\Be\label{est:fBD_x_e2}
(\ref{fBD_x_e2})= O(1) \| w f ^{\ell-1} \|_\infty + O(1)
[
(\ref{px_f_4}) + (\ref{px_f_5})
]_{(t^1,x^1,v^1,f) \rightarrow (t^{k+1},x^{k+1},v^{k+1},f^{\ell-1})} ,
 \Ee
where we have used $|\p_{\mathbf{x}^{k+1}_{p^{k+1}, j}} \tb ^{k+1}| \leq \frac{1}{n(x^{k+2}) \cdot v^{k+1}}$ from (\ref{nabla_tbxb}).

In order to take off  $\p_{\mathbf{x}_{p^{k+2},j^\prime}^{k+2}}$ from $f^{\ell-1}$ in (\ref{fBD_x_e1}) we use the change of variables of (\ref{map_v_to_xbtb}). Note that
\Be\label{v_k_xtb}
v^{k+1} = (x^{k+1} - \eta_{p^{k+2}} (\mathbf{x}_{p^{k+2}} ^{k+2})) /\tb^{k+1}.
\Ee
Now we apply the change of variables of (\ref{map_v_to_xbtb}) and derive that
\Be \label{p_xf_total1_under}
\begin{split}
 (\ref{fBD_x_e1})
=&
\sum_{p^{k+2} \in \mathcal{P}}\iint _{|\mathbf{x}_{p^{k+2}}^{k+2}|< \delta_1}
\int^{t^{k+1}}_0
e^{- \nu(v^{k+1}) \tb^{k+1}}
\iota_{p^{k+2}} (
\eta_{p^{k+2}} (\mathbf{x}_{p^{k+2}}^{k+2} )
)\\
& \times
\sum_{j^\prime=1,2}
\frac{\p \mathbf{x}^{k +2}_{p^{k+2},j^\prime}}{\p{\mathbf{x}^{k+1  }_{p^{k+1 },j}}}  \p_{\mathbf{x}_{p^{k+2},j^\prime}^{k+2}} [  f^{\ell-1}( \eta_{p^{k+2}} ( \mathbf{x}^{k+2}_{p^{k+2} } ),
 {v}^{k+1}
 )]\\
&
\times\frac{n_{p^{k+2}} (\mathbf{x}_{p^{k+1}} ^{k+1}) \cdot  (x^{k+1} -
 \eta_{p^{k+2}} (\mathbf{x}_{p^{k+2}} ^{k+2})
 )
}{\tb^{k+1}}
 \frac{ n_{p^{k+2}}(\mathbf{x}^{k+2}_{p^{k+2}}) \cdot (x^{k+1} -
 \eta_{p^{k+2}} (\mathbf{x}_{p^{k+2}} ^{k+2})
 ) }{|\tb^{k+1}|^4}\\
 &  \times  \sqrt{\mu ( v^{k+1})}
 \dd \tb^{k+1}
\sqrt{g_{p^{k+2},11}g_{p^{k+2},22}  }
 \dd \mathbf{x}^{k+2}_{p^{k+2},1}\dd \mathbf{x}^{k+2}_{p^{k+2},2}.
\end{split}
\Ee
Here we read $g_{p^{k+2}, ii}$ at $\mathbf{x}_{p^{k+2}} ^{k+2}$.

We apply the integration by parts with respect to $\p_{\mathbf{x}^{k+2}_{p^{k+2}, j^\prime}}$ for $j^{\prime}=1,2$. For $\iota_{p^{k+2} } (
\eta_{p^{k+2}} (\mathbf{x}_{p^{k+2}}^{k+2} )
)=0$ when $|\mathbf{x}_{p^{k+2}}^{k+2}|= \delta_1$ from (\ref{iota}), such contribution of $|\mathbf{x}_{p^{k+2}}^{k+2}|= \delta_1$ vanishes. Then we derive
\begin{align}
 (\ref{fBD_x_e1})
=&
\sum_{p^{k+2} \in \mathcal{P}}\iint 
\int^{t^{k+1}}_0
 \p_{\mathbf{x}_{p^{k+2},j^\prime}^{k+2}} \Big[
e^{- \nu(v^{k+1}) \tb^{k+1}}
\sqrt{\mu(v^{k+1})}
\iota_{p^{k+2}} (
\eta_{p^{k+2}} (\mathbf{x}_{p^{k+2}}^{k+2} )
)\Big] \cdots \label{p_xf_total1_under1}
\\
+ &
\sum_{p^{k+2} \in \mathcal{P}}\iint 
\int^{t^{k+1}}_0
 \p_{\mathbf{x}_{p^{k+2},j^\prime}^{k+2}} \Big[
 \sum_{j^\prime=1,2}
\frac{\p \mathbf{x}^{k +2}_{p^{k+2},j^\prime}}{\p{\mathbf{x}^{k+1  }_{p^{k+1 },j}}}
\sqrt{g_{p^{k+2},11}g_{p^{k+2},22}  } \Big] \cdots\label{p_xf_total1_under2}
\\
+ &
\sum_{p^{k+2} \in \mathcal{P}}\iint 
\int^{t^{k+1}}_0
 \p_{\mathbf{x}_{p^{k+2},j^\prime}^{k+2}} \bigg[
\frac{n_{p^{k+1}} (\mathbf{x}_{p^{k+1}} ^{k+1}) \cdot  (x^{k+1} -
 \eta_{p^{k+2}} (\mathbf{x}_{p^{k+2}} ^{k+2})
 )
}{\tb^{k+1}} \notag\\
& \ \ \ \ \ \ \ \  \  \ \ \ \ \ \ \ \   \ \ \ \ \ \ \ \   \ \ \ \ \ \ \ \  \cdot
 \frac{ n_{p^{k+2}}(\mathbf{x}^{k+2}_{p^{k+2}}) \cdot (x^{k+1} -
 \eta_{p^{k+2}} (\mathbf{x}_{p^{k+2}} ^{k+2})
 ) }{|\tb^{k+1}|^4}\bigg]\cdots.
 \label{p_xf_total1_under3}
\end{align}

From~\eqref{xip deri xbp} and (\ref{nv<v2}), (\ref{bound_vb_x}), we derive that
\Be\begin{split}\label{px_pxx}
&\bigg| \frac{\p }{\p {\mathbf{x}_{p^{k+2},j^\prime}^{k+2}}} \bigg(
\sum_{j^\prime=1,2}
\frac{\p \mathbf{x}^{k +2}_{p^{k+2},j^\prime}}{\p{\mathbf{x}^{k+1  }_{p^{k+1 },j}}}
\sqrt{g_{p^{k+2},11}g_{p^{k+2},22}  }
\bigg) \bigg|\\
 \lesssim& \  \| \eta \|_{C^2} \Big\{1+
\frac{|\mathbf{v}^{k+2}_{p^{k+2} ,\parallel}|}{|\mathbf{v}^{k+2}_{p^{k+2}, 3}|^2  } | \p_3 \eta_{p^{k+2}}
(
\mathbf{x}^{k+2  }_{p^{k+2 }}
)
\cdot \p_j \eta_{p^{k+1}} (
\mathbf{x}^{k+1  }_{p^{k+1 }}
)|
\Big\}\\
  \leq & \ O(\| \eta \|_{C^2})
   \Big\{1+  \frac{|\mathbf{v}^{k+2}_{p^{k+2}  }|}{|\mathbf{v}^{k+2}_{p^{k+2}, 3}|^2  }
   |x^{k+1} - \eta_{p^{k+2}}
(
\mathbf{x}^{k+2  }_{p^{k+2 }}
) |
   \Big\}\\
   \leq & \  O(\| \eta\|_{C^2}) \frac{1}{|\mathbf{v}^{k+2}_{p^{k+2}, 3}|}=O(\Vert \eta\Vert_{C^2})\frac{|\tb^{k+1}|}{|n_{p^{k+2}}(\mathbf{x}_{p^{k+1}}^{k+1})  \cdot  (x^{k+1} -
 \eta_{p^{k+2}} (\mathbf{x}_{p^{k+2}} ^{k+2})|} .
\end{split}
\Ee

Now using (\ref{v_k_xtb}) for (\ref{p_xf_total1_under1}), (\ref{px_pxx}) for (\ref{p_xf_total1_under2}), and (\ref{bound_vb_x}) for (\ref{p_xf_total1_under3}), we derive that %
%
\Be\begin{split}\label{est:fBD_x_e1}
&|(\ref{fBD_x_e1})|\\
&\lesssim \| \eta \|_{C^2} \| w f^\ell \|_\infty
\iint \int_0^{t^{k+1}} e^{-\nu_0 \tb^{k+1}}\Big[
\frac{|x^{k+1} - \eta_{p^{k+2} } (\mathbf{x}^{k+2}_{p^{k+2}})|^3}{|\tb^{k+1}|^5}+ \frac{|x^{k+1} - \eta_{p^{k+2} } (\mathbf{x}^{k+2}_{p^{k+2}})|^2}{|\tb^{k+1}|^4}\Big]
e^{-\frac{|x^{k+1} - \eta_{p^{k+2} } (\mathbf{x}^{k+2}_{p^{k+2}})|^2}{4|\tb^{k+1}|^2}}\\
&\lesssim  \| \eta \|_{C^2} \| w f^\ell \|_\infty
 \int_0^{\infty}
\frac{e^{-\nu_0 \tb^{k+1}}}{|\tb^{k+1}|^{1/2}}\iint \frac{1}{|x^{k+1} - \eta_{p^{k+2} } (\mathbf{x}^{k+2}_{p^{k+2}})|^{3/2}}\\
&\lesssim  \| \eta \|_{C^2} \| w f^\ell \|_\infty,
\end{split}\Ee
where we have used
\Be
\begin{split}\notag
&\Big[
\frac{|x^{k+1} - \eta_{p^{k+2} } (\mathbf{x}^{k+2}_{p^{k+2}})|^3}{|\tb^{k+1}|^5}+ \frac{|x^{k+1} - \eta_{p^{k+2} } (\mathbf{x}^{k+2}_{p^{k+2}})|^2}{|\tb^{k+1}|^4}\Big]
e^{-\frac{|x^{k+1} - \eta_{p^{k+2} } (\mathbf{x}^{k+2}_{p^{k+2}})|^2}{4|\tb^{k+1}|^2}}\\
&\leq \frac{1}{|\tb^{k+1}|^{1/2}} \frac{1}{|x^{k+1} - \eta_{p^{k+2} } (\mathbf{x}^{k+2}_{p^{k+2}})|^{3/2}}
\Big[\frac{|x^{k+1} - \eta_{p^{k+2} } (\mathbf{x}^{k+2}_{p^{k+2}})|^{9/2}}{|\tb^{k+1}|^{9/2}}+\frac{|x^{k+1} - \eta_{p^{k+2} } (\mathbf{x}^{k+2}_{p^{k+2}})|^{7/2}}{|\tb^{k+1}|^{7/2}}\Big]
\\
&\times e^{-\frac{|x^{k+1} - \eta_{p^{k+2} } (\mathbf{x}^{k+2}_{p^{k+2}})|^2}{4|\tb^{k+1}|^2}}\\
&\lesssim \frac{1}{|\tb^{k+1}|^{1/2}} \frac{1}{|x^{k+1} - \eta_{p^{k+2} } (\mathbf{x}^{k+2}_{p^{k+2}})|^{3/2}}.
\end{split}
\Ee

Finally collecting terms (\ref{fx1_x1_1})-(\ref{fx1_x1_4}), and (\ref{est:fBD_xr}), (\ref{est:fBD_x_e1}) and setting $k=0$, we prove the Proposition \ref{prop_fx}.\end{proof}

\section{Mixing via the binary collision and transport}
In this section we mainly establish the integration by parts technique mentioned in Section 1.3 using the mixing of the binary collision and the transport operator. In particular, we will prove Proposition \ref{est:K_Gamma}. As direct consequence of Proposition \ref{prop_fx} and Proposition \ref{est:K_Gamma} we will give a proof of the (\ref{estF_n}) in \textbf{Main Theorem}.

We consider a solution of the Boltzmann equation (\ref{inhomo_transport}) with
 \Be\label{h_K}
 h^\ell(x,v):= K(f^{\ell-1}) + \Gamma(f^{\ell-1},f^{\ell-1}),
 \Ee
 and the diffuse BC (\ref{diffuse_f}). The main result is an estimate of (\ref{h_K})-contribution in (\ref{px_f_3}) and (\ref{px_f_5}):
 \begin{proposition}\label{est:K_Gamma}
 We bound $(\ref{px_f_3})_{(\ref{px_f_3})_*=Kf^{\ell-1} }$, $(\ref{px_f_3})_{(\ref{px_f_3})_*=  \Gamma(f^{\ell-1},f^{\ell-1})}$, and $(\ref{px_f_int}) \cdot (\ref{px_f_5})_{(\ref{px_f_5})_*=  Kf^{\ell-2} + \Gamma(f^{\ell-2},f^{\ell-2})}$ respectively as
 \Be\begin{split}
 \label{est:K_Gamma1}
& \int_{\max\{0,t-\tb \}}^t e^{-\nu(v)(t-s)} w_{\tilde{\theta}}(v)\p_{x_i}   Kf^{\ell-1}(   x-(t-s) v,v)  \dd s\\
&\leq  \frac{O(1)}{\alpha(x,v)} \Big\{
\Big(
\e  +  \sup_{i\geq 0} \| w f^{\ell-1-i}\|_\infty
\Big)
 \sup_{i\geq 0}  \| \alpha \nabla_x f^{\ell-1-i}\|_\infty +  \e^{-1}   \sup_{i\geq 0} \| w f^{\ell-1-i}\|_\infty\Big\},
 \end{split}\Ee
 \Be
 \begin{split} \label{est:K_Gamma2}
& e^{-\nu (v) \tb} \frac{n_i(x^1) |v|}{\alpha(x,v)} \frac{w_{\tilde{\theta}}(v)M_W(x^1, v)}{\sqrt{\mu(v)}}
\int_{n(x^1) \cdot v^1>0}  \dd v^1\{n^1 \cdot v^1\}\sqrt{\mu(v^1)}\\
& \ \ \ \ \times\int^{t^1}_{\max\{0,  t^1- \tb^1\}} e^{- \nu (v^1) (t^1-s^1)}
 \underline{\nabla}_{x^1}    K f^{\ell-1}(x^1- (t^1-s^1) v^1,v^1)    \dd s^1\\
& \leq \frac{O(1)}{\alpha(x,v)} \times \Big\{  \e  \sup_{i\geq 0}  \| w_{\tilde{\theta}}\alpha \nabla_x f^{\ell-1-i}\|_\infty +  \e^{-1}  \sup_{i\geq 0} \| w f^{\ell-1-i}\|_\infty\Big\},
\end{split} \Ee
 \Be\label{est:K_Gamma3}
 \begin{split}
  &\int_{\max\{0,t-\tb \}}^t e^{-\nu(v)(t-s)}w_{\tilde{\theta}} \p_{x_i}   \Gamma(f^{\ell-1},f^{\ell-1})(   x-(t-s) v,v)  \dd s\\
   &   + e^{-\nu (v) \tb} \frac{n_i(x^1) |v|}{\alpha(x,v)} \frac{w_{\tilde{\theta}}(v)M_W(x^1, v)}{\sqrt{\mu(v)}}
\int_{n(x^1) \cdot v^1>0}  \dd v^1\{n^1 \cdot v^1\}\sqrt{\mu(v^1)}\\
& \ \ \ \ \times\int^{t^1}_{\max\{0,  t^1- \tb^1\}} e^{- \nu (v^1) (t^1-s^1)}
 \underline{\nabla}_{x^1}   \Gamma( f^{\ell-1}, f^{\ell-1})(x^1- (t^1-s^1) v^1,v^1)    \dd s^1\\
& \leq  \frac{O(1)}{\alpha(x,v)}
\Big(
\e  +  \sup_{i\geq 0} \| w f^{\ell-1-i}\|_\infty
\Big)
 \sup_{i\geq 0}  \|w_{\tilde{\theta}} \alpha \nabla_x f^{\ell-1-i}\|_\infty  .
 \end{split}
 \Ee
\end{proposition}

\begin{proof}[\textbf{Proof of ~\eqref{estF_n}} in \textbf{Main Theorem}]
Combining Proposition \ref{prop_fx} and Proposition \ref{est:K_Gamma} we obtains that for $t\gg 1$ and $\e \ll 1$,
\[\Vert w_{\tilde{\theta}}\alpha\nabla_x f^\ell\Vert_\infty \leq o(1)\sup_{i\leq\ell-1}\Vert w_{\tilde{\theta}}\alpha\nabla_x f^{i}\Vert_\infty+C(\e)\Vert T_W-T_0\Vert_{C^1}\sup_{i\geq 0}\Vert wf^i\Vert_\infty,\]
where the $\Vert T_W-T_0\Vert_{C^1}$ comes from $| \underline{\nabla}_{x^1} r(x^1)|$ in~\eqref{px_f_1}.

By a standard argument we pass the limit and conclude that the unique solution in \textbf{Existence Theorem} satisfies the weighted $C^{1}$ estimate~\eqref{estF_n}.
\end{proof}

From now we give a proof of the proposition.

\subsection{Convert $\nabla_x$ for $\nabla_v$ along the trajectory using binary collision}

\

\textit{Expansion of $(\ref{px_f_3})_{(\ref{px_f_3})_*=Kf^{\ell-1}(x -(t -s )v,v )}.$ } First we consider (\ref{px_f_3}) with $(\ref{px_f_3})_*=Kf^{\ell-1}(x -(t -s )v, v )
= \int_{\R^3} \mathbf{k}(v,u)  f^{\ell-1}(x -(t -s )v,u) \dd u
$. Temporarily denote $y= x-(t-s)v$. Proposition \ref{prop_fx} gives a formula of $w_{\tilde{\theta}}(u)\p_{x_i} f^{\ell-1}(y,u)$ by (\ref{px_f_1})-(\ref{px_f_5}) with $h^\ell=(\ref{h_K})$ and $(x,v,\ell) \rightarrow (y, u,\ell-1)$. We split a contribution of (\ref{px_f_3}) with $(\ref{px_f_3})_*=Kf^{\ell-2}(y-(s-s^0) u, u)$, which is (\ref{px_f_3_Ku}), and the rest. The rest is given as
\begin{align}
&
\int^t_{\max\{0,t-\tb\}}
\dd s \,
 e^{-\nu(v)(t-s)}
\int_{\R^3} \dd u \, \mathbf{k}(v,u)\frac{w_{\tilde{\theta}}(v)}{w_{\tilde{\theta}}(u)}\notag
\\
\times &\bigg\{O(1) \frac{ n_i (x^1)w_{\tilde{\theta}}(u) }{\alpha(y,u)}
\bigg(
\Big( \frac{\nu(u)}{w(u)} +
|u|   \frac{M_W(x^1,u)}{\sqrt{\mu(u)}}
\Big)
\| wf^{\ell-1} \|_\infty
+ |u| | \underline{\nabla}_{x^1}  r(x^1)| (1+ \| wf^{\ell-1} \|_\infty)
\bigg)
\label{px_f_1_K}\\
&  \ +  \frac{O(1)}{\alpha(y,u)} e^{-\nu(u) s}  (w_{\tilde{\theta}}\alpha \p_{x_i} f^{\ell-1})(y-su,u)
\label{px_f_2_K}
\\
 & \  +\int_{\max\{0,s-\tb \}}^s e^{-\nu(u)(t-s)} w_{\tilde{\theta}}(u)\p_{x_i} 
 \Gamma(f^{\ell-2},f^{\ell-2}) (   y-(s-s') v,v)
 \dd s'
 \label{px_f_3_K}
 \\
& \  + O(1)e^{-\nu (u) \tb} \frac{n_i(x^1) |u|}{\alpha(y,u)} \frac{w_{\tilde{\theta}}(u)M_W(x^1, u)}{\sqrt{\mu(u)}}
\int_{n(x^1) \cdot v^1>0}  \dd v^1\{n^1 \cdot v^1\}\sqrt{\mu(v^1)}\label{px_f_int_K}\\
& \ \ \ \   \ \ \ \    \times  \bigg(
  e^{-\nu(v^1) t^1} \frac{1}{\alpha(x^1 ,v^1)} (\alpha  \underline{\nabla}_{x^1}  f^{\ell-2})(x^1-t^1 v^1,v^1)
  \label{px_f_4_K} \\
& \ \ \ \   \ \ \ \   \ \ \ \  +  \int^{t^1}_{\max\{0,  t^1- \tb^1\}} e^{- \nu (v^1) (t^1-s^1)}
 \underline{\nabla}_{x^1}   \underbrace{  h^{\ell-2}(x^1- (t^1-s^1) v^1,v^1)}_{ h^{\ell-2} \ \text{in} \ (\ref{h_K})}   \dd s^1
\bigg)
\bigg\},\label{px_f_5_K}
\end{align}
where we intentionally have abused the notations as $x^1=x^1(y, u), t^1=s-\tb(y,u)$ for the sake of simplicity (see (\ref{xi}) and (\ref{tbi})). We will estimate (\ref{px_f_1_K})-(\ref{px_f_5_K}) later together with the other expansions.

Now we focus on the contribution of (\ref{px_f_3}) of $(\ref{px_f_3})_*=Kf^{\ell-2}
(y-(s-s^0) u, u)
$. We split the time integration in $s^0$ as
 \Be\label{px_f_3_Ku}
 \begin{split}
 & \int^t_{\max\{0, t-\tb\}}
 \dd s   \,   
 e^{-\nu(v) (t-s)}
  \int_{\R^3} \dd u   \,   \mathbf{k} (v,u)w_{\tilde{\theta}}(v)
  \int^s_{\max\{0, s-\tb(y, u)\}} \dd s^0   \,
  e^{-\nu(u) (s-s^0)}\\
 & \times
 \{
 \mathbf{1}_{s^0\leq s-\e} + \mathbf{1}_{s^0> s-\e}
 \}
  \int_{\R^3} \dd u^\prime    \,
 \mathbf{k} (u,u^\prime)   \p_{x_j}f^{\ell-2} (x-(t-s)v- (s-s^0)u, u^\prime) .
 \end{split}
 \Ee

 Note that
 $\p_{x_i}f^{\ell-2}(x-(t-s)v- (s-s^0)u, u^\prime)
 =\frac{-1}{s-s^0}\p_{u_i}[f^{\ell-2}(x-(t-s)v- (s-s^0)u, u^\prime)]$. Applying an integration by parts with respect to $\p_{u_i}$, we derive an identity of a contribution of $\{s^0
 \leq s- \e
 \}$ in (\ref{px_f_3_Ku}) as

 \Be\begin{split}\label{est:px_f_3_K_1}
 & \int^t_{\max\{0, t-\tb\}}
 \dd s   \,   
 e^{-\nu(v) (t-s)}w_{\tilde{\theta}}(v)
  \int_{\R^3} \dd u
  \int^s_{\max\{0, s-\tb(y, u)\}} \dd s^0   \,
  e^{-\nu(u) (s-s^0)} \frac{\mathbf{1}_{s^0\leq s-\e}}{s-s^0}  \\
 & \times
  \int_{\R^3} \dd u^\prime    \,
\p_{u_i}[ \mathbf{k} (v,u) \mathbf{k} (u,u^\prime) ]f^{\ell-2}(x-(t-s)v- (s-s^0)u, u^\prime) \\
-&\int^t_{\max\{0, t-\tb\}}
 \dd s   \,   
 e^{-\nu(v) (t-s)}w_{\tilde{\theta}}(v)
  \int_{\R^3} \dd u
  \int^s_{\max\{0, s-\tb(y, u)\}} \dd s^0   \,
 \p_{u_i} \nu(u)  e^{-\nu(u) (s-s^0)}  {\mathbf{1}_{s^0\leq s-\e}}  \\
 & \times
  \int_{\R^3} \dd u^\prime    \,
\mathbf{k} (v,u) \mathbf{k} (u,u^\prime) f^{\ell-2}(x-(t-s)v- (s-s^0)u, u^\prime)\\
+& \int^t_{\max\{0, t-\tb\}}
 \dd s   \,   
 e^{-\nu(v) (t-s)}w_{\tilde{\theta}}(v)
  \int_{\R^3} \dd u
  \mathbf{1}_{s \geq \tb(y,u)}
  e^{-\nu(u) \tb(y,u)} \frac{\mathbf{1}_{\tb(y,u) \geq \e}}{\tb(y,u)}\frac{\partial \tb(y,u)}{\partial u_i}  \\
 & \times
  \int_{\R^3} \dd u^\prime    \,
 \mathbf{k} (v,u) \mathbf{k} (u,u^\prime) f^{\ell-2}(x-(t-s)v- (s-s^0)u, u^\prime).
 \end{split}\Ee
From~\eqref{k_varrho} and Lemma \ref{Lemma: k tilde}, for the first term in~\eqref{est:px_f_3_K_1} we have
\begin{equation}\label{cal for first term}
\begin{split}
   & w_{\tilde{\theta}}(v) \p_{u_i}[ \mathbf{k} (v,u) \mathbf{k} (u,u^\prime) ]f^{\ell-2}(x-(t-s)v- (s-s^0)u, u^\prime)\\
    & \lesssim \frac{w_{\tilde{\theta}}(v)k_\varrho(v,u)}{w_{\tilde{\theta}}(u)} \frac{w_{\tilde{\theta}}(u)k_\varrho(u,u^\prime)\langle u\rangle}{|w_{\tilde{\theta}}(u')|\langle u'\rangle} \Big(\frac{1}{|v-u|}+\frac{1}{|u-u^\prime|} \Big) \frac{w_{\tilde{\theta}}(u')\langle u'\rangle}{w(u')} \Vert wf^{\ell-2}\Vert_\infty \\
    & \lesssim \frac{\mathbf{k}_{\tilde{\varrho}}(v,u)}{|v-u|}\frac{\mathbf{k}_{\tilde{\varrho}}(u,u')}{|u-u'|}\Vert wf^{\ell-2}\Vert_\infty.
\end{split}
\end{equation}
Since $\frac{\mathbf{k}_\varrho(v,u)}{|v-u|},\frac{\mathbf{k}_\varrho(u,u')}{|u-u'|}\in L^1_u$, the first term of~\eqref{est:px_f_3_K_1} is bounded by
\begin{equation}\label{first term}
O(\e^{-1})\Vert wf^{\ell-2}\Vert_\infty.
\end{equation}

For the second term in~\eqref{est:px_f_3_K_1}, similarly to~\eqref{cal for first term} we have
\begin{equation*}
\begin{split}
   & w_{\tilde{\theta}}(v)  \mathbf{k} (v,u) \mathbf{k} (u,u^\prime) f^{\ell-2}(x-(t-s)v- (s-s^0)u, u^\prime)\\
    & \lesssim \mathbf{k}_{\tilde{\varrho}}(v,u)\mathbf{k}_{\tilde{\varrho}}(u,u')\Vert wf^{\ell-2}\Vert_\infty.
\end{split}
\end{equation*}
Thus the second term is bounded by
\begin{equation}\label{second term}
O(\e^{-1})\Vert wf^{\ell-2}\Vert_\infty.
\end{equation}

From~\eqref{k_varrho} and~\eqref{nabla_tbxb}, we conclude the third term in~\eqref{est:px_f_3_K_1} is bounded by
\begin{equation}\label{third term}
  \begin{split}
  &    \Vert wf^{\ell-2}\Vert_\infty  \int^t_{\max\{0, t-\tb\}}
 \dd s   \,   
 e^{-\nu(v) (t-s)}
  \int_{\R^3} \dd u  \mathbf{k}_\varrho(v,u) \frac{\mathbf{1}_{\tb(y,u)\geq \e}}{\tb(y,u)}\frac{\partial \tb(y,u)}{\partial u_i} \int_{\mathbb{R}^3} du^\prime \mathbf{k}_\varrho(u,u^\prime)       \\
     & \lesssim \Vert wf^{\ell-2}\Vert_\infty\int^t_{\max\{0, t-\tb\}}
 \dd s   \,   
 e^{-\nu(v) (t-s)}
  \int_{\R^3} \dd u
 \frac{\mathbf{k}_\varrho(v,u)}{\alpha(y,u)}.
  \end{split}
\end{equation}
This term will be estimated later using Lemma \ref{Lemma: NLN}.

 On the other hand by Lemma \ref{Lemma: k tilde} a contribution of $\{s^0> s-\e\}$ in (\ref{px_f_3_Ku}) is controlled by
  \Be\label{est:px_f_3_K_2}
 \begin{split}
 &  \| w_{\tilde{\theta}}\alpha\nabla_{x}f^{\ell-2}  \|_\infty\int^t_{\max\{0, t-\tb\}}
 \dd s   \,   
 e^{-\nu(v) (t-s)}
  \int_{\R^3} \dd u   \,   \mathbf{k} (v,u)\frac{w_{\tilde{\theta}}(v)}{w_{\tilde{\theta}}(u)}
 \\
 & \times
 \int^s_{\max\{0, s-\tb(y, u)\}}      \int_{\R^3}
 \mathbf{1}_{s^0> s-\e}  e^{-\nu_0(u) (s-s^0)} \mathbf{k} (u,u^\prime)
 \frac{w_{\tilde{\theta}}(u)}{w_{\tilde{\theta}}(u')\alpha (y- (s-s^0)u, u^\prime)}\dd u^\prime \dd s^0
\\
&\lesssim \| w_{\tilde{\theta}}\alpha\nabla_{x}f^{\ell-2}  \|_\infty\int^t_{\max\{0, t-\tb\}}
 \dd s   \,   
 e^{-\nu(v) (t-s)}
  \int_{\R^3} \dd u   \,   \mathbf{k}_{\tilde{\varrho}} (v,u)\\
  &\times  \int^s_{\max\{0, s-\tb(y, u)\}}      \int_{\R^3}
 \mathbf{1}_{s^0> s-\e}  e^{-\nu_0\langle u\rangle (s-s^0)} \mathbf{k}_{\tilde{\varrho}} (u,u^\prime)
 \frac{1}{\alpha (y- (s-s^0)u, u^\prime)}\dd u^\prime \dd s^0.
\end{split} \Ee
This term will be estimate later using Lemma \ref{Lemma: NLN}.

\smallskip

\textit{Expansion of $(\ref{px_f_5})_{(\ref{px_f_5})_*=Kf^{\ell-2}(x^1-(t^1-s^1)v^1,v^1)}$ and $(\ref{px_f_5_K})_{h^{\ell-2}=Kf^{\ell-3}(x^1-(t^1-s^1),v^1)}$. }  We split
  \begin{align}
(\ref{px_f_5})_{(\ref{px_f_5})_*=Kf^{\ell-2}(x^1-(t^1-s^1)v^1,v^1)}=   \underbrace{
\int^{t^1}_{\max\{0, t^1- \tb^1 \}} \mathbf{1}_{s^1 \leq t^1-\e } \cdots }_{(\ref{px_f_5_split})_1}
+  \underbrace{
\int^{t^1}_{\max\{0, t^1- \tb^1 \}}
\mathbf{1}_{s^1 \geq t^1-\e } \cdots
 }_{(\ref{px_f_5_split})_2},\label{px_f_5_split} \\
(\ref{px_f_5_K})_{h^{\ell-2}=Kf^{\ell-3}(x^1-(t^1-s^1),v^1)}=   \underbrace{
\int^{t^1}_{\max\{0, t^1- \tb^1 \}} \mathbf{1}_{s^1 \leq t^1-\e } \cdots }_{(\ref{px_f_5_K_split})_1}
+  \underbrace{
\int^{t^1}_{\max\{0, t^1- \tb^1 \}}
\mathbf{1}_{s^1 \geq t^1-\e } \cdots
 }_{(\ref{px_f_5_K_split})_2}.\label{px_f_5_K_split}
 \end{align}

We simply derive an intermediate estimate (see (\ref{fBD_x1}))
\Be
\begin{split}\label{est1:px_f_5_split2}
&
|(\ref{px_f_5_split})_2|
+ |(\ref{px_f_5_K_split})_2|
\\
\leq& \sum_{i=0,1} \int^{t^1}_{ \max\{t^2, t^1-\e\}} e^{-\nu_0 \langle v^1 \rangle (t^1-s^1)} \int_{\R^3} \mathbf{k}_\varrho (v^1, u^\prime)
|\nabla_{\mathbf{x}_{p^1}^1} \eta_{p^1}|  \frac{ \alpha |\nabla_x f^{\ell-2-i}(x^1
-(t^1-s^1) v^1, u^\prime
)|}{\alpha (x^1
-(t^1-s^1) v^1, u^\prime
)}
\dd u^\prime \dd s^1\\
 \leq& \  \| \eta \|_{C^1}\sum_{i=0,1}  \| \alpha \nabla_x f^{\ell-2-i} \|_\infty \sup_{t^1,x^1, v^1} \int^{t^1}_{ \max\{t^2, t^1-\e\}} e^{-\nu_0 \langle v^1 \rangle (t^1-s^1)}
\int_{\R^3} \mathbf{k}_\varrho (v^1, u^\prime)
  \frac{ 1}{\alpha (x^1
-(t^1-s^1) v^1, u^\prime
)}
\dd u^\prime \dd s^1.
\end{split}
\Ee
 This term, together with (\ref{est:px_f_3_K_2}), will be estimate later using Lemma \ref{Lemma: NLN}.

 Now we consider $(\ref{px_f_5_split})_1$ and $(\ref{px_f_5_K_split})_1$. Recall (\ref{fBD_x1}). The key observation is the following interchange of spatial derivatives and velocity derivatives: For $t^1\neq s^1$ and $i=0,1,$
  \begin{align}
 & \p_{\mathbf{x}_{p^1,j}^1} \Big[ \int_{\R^3} \mathbf{k} (v^1, u^\prime )  f^{\ell-2-i}(
\eta_{p^1} (\mathbf{x}_{p^1}^1)
- (t^1-s^1) v^1, u^\prime)
 \dd u^\prime\Big] \notag\\
 =& \sum_{\ell=1}^3\frac{\p \eta_{p^1, \ell}(\mathbf{x}_{p^1}^1)}{\p\mathbf{x}_{p^1,j}^1}
 \int_{\R^3} \mathbf{k} (v^1, u^\prime )
 \p_{x_\ell}  f^{\ell-2-i}(
\eta_{p^1} (\mathbf{x}_{p^1}^1)
- (t^1-s^1) v^1, u^\prime)
 \dd u^\prime\notag\\
 =& -\frac{1}{t^1-s^1}\sum_{\ell=1}^3\frac{\p \eta_{p^1, \ell}(\mathbf{x}_{p^1}^1)}{\p\mathbf{x}_{p^1,j}^1}
 \int_{\R^3} \mathbf{k} (v^1, u^\prime )
 \p_{v^1_\ell} [ f^{\ell-2-i}(
\eta_{p^1} (\mathbf{x}_{p^1}^1)
- (t^1-s^1) v^1, u^\prime)]
 \dd u^\prime\notag\\
 =&
  -\frac{1}{t^1-s^1}\sum_{\ell=1}^3\frac{\p \eta_{p^1, \ell}(\mathbf{x}_{p^1}^1)}{\p\mathbf{x}_{p^1,j}^1}
  \p_{v^1_\ell} \Big[
 \int_{\R^3} \mathbf{k} (v^1, u^\prime )
f^{\ell-2-i}(
\eta_{p^1} (\mathbf{x}_{p^1}^1)
- (t^1-s^1) v^1, u^\prime)
 \dd u^\prime\Big]
 \label{xK_to_vK1}
 \\
 & + \frac{1}{t^1-s^1}\sum_{\ell=1}^3\frac{\p \eta_{p^1, \ell}(\mathbf{x}_{p^1}^1)}{\p\mathbf{x}_{p^1,j}^1}
 \int_{\R^3}   \p_{v^1_\ell}  \mathbf{k} (v^1, u^\prime )
f^{\ell-2-i}(
\eta_{p^1} (\mathbf{x}_{p^1}^1)
- (t^1-s^1) v^1, u^\prime)
 \dd u^\prime .\label{xK_to_vK2}
 \end{align}
 Now we consider a contribution of (\ref{xK_to_vK1}) in $(\ref{px_f_5_split})_1$ and $(\ref{px_f_5_K_split})_1$ inside $v^1$-integration in (\ref{px_f_int}) and (\ref{px_f_int_K}). From the integration by parts with respect to $\p_{v^1_\ell}$, for $i=0,1$,
 \Be \begin{split}\label{est:xK_to_vK1}
 &\int_{n ^1  \cdot v^1 >0}
 [
\text{a contribution of } (\ref{xK_to_vK1}) \text{ in }
 (\ref{px_f_5_split})_1 \text{ and }  (\ref{px_f_5_K_split})_1]
 \sqrt{\mu(v^1)}\{n^1 \cdot v^1\}  \dd v^1
 \\
 =& \int_{n ^1  \cdot v^1 >0}
 \int^{t^1}_{\max\{0, t^1- \tb^1\}} \mathbf{1}_{s^1 \leq t^1- \e}
 e^{- \nu(v^1) (t^1-s^1)}
\frac{1}{t^1-s^1}\sum_{\ell=1}^3\frac{\p \eta_{p^1, \ell}(\mathbf{x}_{p^1}^1)}{\p\mathbf{x}_{p^1,j}^1}
\\
& \ \ \ \ \ \   \times
 \p_{v^1_\ell}  \Big[ \int_{\R^3}  \mathbf{k} (v^1, u^\prime )
f^{\ell-2-i}(
\eta_{p^1} (\mathbf{x}_{p^1}^1)
- (t^1-s^1) v^1, u^\prime)
 \dd u^\prime
\Big] \dd s^1
 \sqrt{\mu(v^1)}\{n^1 \cdot v^1\}  \dd v^1 \\
 =&
 - \int_{n ^1 \cdot v^1 >0}
 \int^{t^1}_{\max\{0, t^1- \tb^1\}}
\frac{\mathbf{1}_{\e \leq t^1- s^1  }}{t^1-s^1}
 \sum_{\ell=1}^3\frac{\p \eta_{p^1, \ell}(\mathbf{x}_{p^1}^1)}{\p\mathbf{x}_{p^1,j}^1}   \p_{v^1_\ell}  \big[ \sqrt{\mu(v^1)}\{n^1 \cdot v^1\}  e^{- \nu(v^1) (t^1-s^1)}\big]
  \\
 & \ \ \ \ \ \   \times
 \int_{\R^3}  \mathbf{k} (v^1, u ^\prime)
f^{\ell-2-i}(
\eta_{p^1} (\mathbf{x}_{p^1}^1)
- (t^1-s^1) v^1, u^\prime)
 \dd u^\prime
 \dd s^1
 \dd v^1 \\
 &+
  \int_{n ^1  \cdot v^1 =0}
 \int^{t^1}_{\max\{0, t^1- \tb^1\}}
 \cdots
 %
 \int_{\R^3}  \mathbf{k} (v^1, u^\prime )
f^{\ell-2-i}(
\cdot , u^\prime)
 \dd u^\prime
 \dd s^1
 \sqrt{\mu(v^1)}\{n^1 \cdot v^1\}  \dd v^1
  \\
&+ \int_{n ^1  \cdot v^1 >0}
\mathbf{1}_{t^1 \geq \tb^1  \geq  \e}
\frac{e^{- \nu(v^1)  \tb^1 }}{\tb^1}
\partial_{\mathbf{x}^1_{p^1,j}} \eta \cdot \nabla_{v^1} \tb^1
\\
& \ \ \ \ \ \  \times  \int_{\R^3}  \mathbf{k} (v^1, u ^\prime)
f^{\ell-2-i}(
x^2, u^\prime)
 \dd u ^\prime
 \sqrt{\mu(v^1)}\{n^1 \cdot v^1\}  \dd v^1
 \\
= & \ O(\e^{-1})
 \| \eta \|_{C^1} \| wf^{\ell-2-i} \|_\infty
 \times \Big\{1 +
 \int_{n ^1  \cdot v^1>0}
\frac{|n ^1  \cdot v^1|}{|n ^2  \cdot v^1|} \sqrt{\mu(v^1)} \dd v^1
 \Big\}\\
 =& \ O(\e^{-1})
 \| \eta \|_{C^1} \| wf ^{\ell-2-i}\|_\infty.
  \end{split} \Ee

Here we have used (\ref{nabla_tbxb}) and (\ref{Velocity_lemma}). Also we have used $\| w f^{\ell-2-i}\|_\infty<\infty$ and derived $\int_{n(x^1) \cdot v^1 =0}\cdots =0$.

  From (\ref{k_varrho}), (\ref{O_p}) and $t^1-s^1\geq \e$, we bound a contribution of (\ref{xK_to_vK2}) in $(\ref{px_f_5_split})_1$ and $(\ref{px_f_5_K_split})_1$ by
\Be
\begin{split}
&\int_{n(x^1) \cdot v^1 >0}
 [
\text{a contribution of } (\ref{xK_to_vK2}) \text{ of }
 (\ref{px_f_5_split})_1 \text{ and }  (\ref{px_f_5_K_split})_1]
 \sqrt{\mu(v^1)}\{n^1 \cdot v^1\}  \dd v^1 \\
 & =O
(\e^{-1})\| \eta \|_{C^1} \sup_{i=0,1} \| w f^{\ell-2-i} \|_\infty.\label{est:xK_to_vK2}
\end{split}\Ee

Now we consider contributions of $\Gamma$ in (\ref{px_f_3}), (\ref{px_f_5}), (\ref{px_f_3_K}), and (\ref{px_f_5_K}). From (\ref{Gamma_est})
 \begin{align}
& |(\ref{px_f_3})_{(\ref{px_f_3})_*=\Gamma(f^{\ell-1},f^{\ell-1})}|
+ |(\ref{px_f_5})_{(\ref{px_f_5})_*=\Gamma(f^{\ell-2},f^{\ell-2})}|
+ |(\ref{px_f_3_K})| + |(\ref{px_f_5_K})_{h= \Gamma(f^{\ell-3},f^{\ell-3})}|
\notag\\
 \leq & \
\sup_{i} O(\| wf^{\ell-1-i} \|_\infty)
\notag
\\
& \times  \sup_i
 \sum_{j=0,1} \Big\{
 \int^{t^j}_{\max\{0,{t^j}-\tb^j\}} e^{-\nu_0\langle v^j\rangle   (t^j-s^j)}
w_{\tilde{\theta}}(v^j) |\nabla_x f^{\ell-1-i}(x^j-(t^j-s^j)v^j,v^j)  |
 \dd u
 \dd s^j\notag\\
& \ \ \ \ \ \  +  \int^{t^j}_{\max\{0,{t^j}-\tb^j\}} e^{-\nu_0\langle v^j\rangle   (t^j-s^j)}
 \int_{\R^3}
 \mathbf{k}_\varrho (v^j,u)\frac{w_{\tilde{\theta}}(v)}{w_{\tilde{\theta}}(u)}
 |\nabla_x f^{\ell-1-i}(x^j-(t^j-s^j)v^j,u)  |
 \dd u
 \dd s^j
 \Big\}
 \notag\\
 \leq & \
\sup_i O(\| wf^{\ell-1-i} \|_\infty)\sup_i\|w_{\tilde{\theta}} \alpha \nabla_x f^{\ell-1-i}  \|_\infty\notag\\
 & \times \Big\{  \sum_{j=0,1}
  \int^{t^j}_{\max\{0,t^j-\tb^j\}}  e^{-\nu_0\langle v^j\rangle   (t^j-s^j)}
 \frac{1}{\alpha(x^j-(t^j-s^j)v^j,v^j)}
 \dd s
\label{int1/alpha0} \\
&  \ \ \ \
+   \sum_{j=0,1} \int^{t^j}_{\max\{0,t^j-\tb^j\}}  e^{-\nu_0\langle v^j\rangle   (t^j-s^j)}
\int_{\R^3}    \mathbf{k}_{\tilde{\varrho}}(v^j,u)
 \frac{1 }{\alpha(x^j-(t^j-s^j)v^j,u)}
\dd u \dd s^j \Big\},\label{int1/alpha}
 \end{align}
where we have used~\eqref{Lemma: k tilde} in the last line.

 From (\ref{Velocity_lemma})
 \Be\label{est:int1/alpha0}
(\ref{int1/alpha0})\leq  e^{C_\O}
\sum_{j=0,1}
  \int^t_0 e^{-\nu_0 \langle v^j\rangle (t^j-s^j)} \dd s^j \times
\alpha(x^j,v^j)^{-1}
  \leq   \frac{C_\O}{\nu_0 \langle v^j\rangle } \alpha(x^j,v^j)^{-1}.
 \Ee
 We will estimate (\ref{int1/alpha}), together with (\ref{est1:px_f_5_split2}) and (\ref{est:px_f_3_K_2}), later using Lemma \ref{Lemma: NLN}.

 \hide
 Similarly we derive
 \Be\label{int1/alpha_K}
  |(\ref{px_f_3})_{(\ref{px_f_3})_*=Kf}|
+ |(\ref{px_f_5})_{(\ref{px_f_5})_*=Kf}| \leq O(\|\alpha \nabla_x f  \|_\infty) \times (\ref{int1/alpha}).
 \Ee
 \Be
 \begin{split}
 |(\ref{px_f_3})_{(\ref{px_f_3})_*=Kf}| & \leq
 \int^t_{\max\{0,t-\tb\}} e^{-\nu(v) (t-s)}
 \int_{\R^3}
 \mathbf{k}_\varrho (v,u)
 |\nabla_x f(x-(t-s)v,u)  |
 \dd u
 \dd s\\
 &
 \end{split}
 \Ee
 \unhide

 \subsection{Nonlocal-to-Local estimate and Small time contributions }
In this subsection we estimate (\ref{est:px_f_3_K_2}), (\ref{est1:px_f_5_split2}), (\ref{int1/alpha0}), and (\ref{int1/alpha}). The key lemma is the following \textit{Nonlocal-to-Local estimate}:


\begin{lemma}\label{Lemma: NLN}
Denote $x'=x-(t-s)v,\quad y'=y-(t-s)v$. Assume $(t,x,v)\in \lbrack 0,\infty )\times \bar{\Omega}%
\times \mathbb{R}^{3}$ and $t-\tb(x,v)\leq t-t_1 \leq t-t_2 \leq t$. Then for $0<\beta<1$ and some $C_1>0$,
\begin{align}
&\int^{t-t_1}_{t-t_2} \int_{\mathbb{R}^{3}}\frac{e^{-C\langle v\rangle (t-s)} e^{-\varrho |v-u|^{2}}}{|v-u| \alpha (  x', {u}) }\mathrm{d}u\mathrm{d}s \notag\\
  &\lesssim  |e^{-C_1\nu t_1 }-e^{-C_1\nu t_2}|^{\beta}\big[ \langle v\rangle^{-1/2}\big(
1+  |\ln |v|| +  |\ln  {\alpha}(x,v)|\big) +  \frac{1}{|v|^{1-\beta}}\big] \notag \\
   & \lesssim \frac{|e^{-C_1\nu t_1 }-e^{-C_1\nu t_2}|^{\beta}}{\alpha(x,v)}.\label{NLN general}
\end{align}

Thus
\begin{equation}\label{est:nonlocal_wo_e}
\begin{split}
  \int^{t}_{t- \tb(x,v)} \int_{\mathbb{R}^{3}}\frac{e^{-C\langle v\rangle (t-s)} e^{-\varrho |v-u|^{2}}}{|v-u| \alpha (  x', {u}) }\mathrm{d}u\mathrm{d}s
  &
 \lesssim \frac{1}{\alpha(x,v)},
\end{split}
\end{equation}
and for $\e \ll 1$,
\begin{equation}\label{est:NLN}
\begin{split}
  \int^{t}_{t- \tb(x,v) } \int_{\mathbb{R}^{3}}\mathbf{1}_{s\geq t-\e} \frac{e^{-C\langle v\rangle (t-s)} e^{-\varrho |v-u|^{2}}}{|v-u| \alpha (  x', {u}) }\mathrm{d}u\mathrm{d}s
   \leq
\frac{O(\e)}{\alpha(x,v)}.
\end{split}
\end{equation}

For $1< p< 3$, we have
\begin{equation}\label{integrate alpha beta}
\int^{t}_{t-\tb} e^{-\nu(t-s)}\dd s \int_{\mathbb{R}^3}\dd u\frac{\mathbf{k}(v,u)}{|u|^2\min\{\frac{\alpha(x',u)}{|u|},\frac{\alpha(y',u)}{|u|}\}^p}\lesssim \frac{\min\{1,O(t)\}}{|v|^2\min\{\frac{\alpha(x,v)}{|v|},\frac{\alpha(y,v)}{|v|}\}^p},
\end{equation}
and
\begin{equation}\label{integrate alpha beta small}
\int^{t}_{t-\tb} e^{-\nu(t-s)}\dd s \int_{\mathbb{R}^3}\dd u
 \mathbf{1}_{s\geq t-\e}  \frac{\mathbf{k}(v,u)}{|u|^2\min\{\frac{\alpha(x',u)}{|u|},\frac{\alpha(y',u)}{|u|}\}^p}\lesssim \frac{\e}{|v|^2\min\{\frac{\alpha(x,v)}{|v|},\frac{\alpha(y,v)}{|v|}\}^p}.
\end{equation}

For $\beta<1$,
\begin{equation}\label{integrate xi beta/2}
\int^t_{t-\tb(x,v)} \frac{e^{-C\langle v\rangle(t-s)}}{|v|\min\{\xi(x'),\xi(y')\}^{\beta/2}}\lesssim \frac{1}{|v|^2\min\{\frac{\alpha(x,v)}{|v|},\frac{\alpha(y,v)}{|v|}\}^{\beta}}.
\end{equation}

\begin{equation}\label{integrate beta<1}
\int_{\mathbb{R}^3} \mathbf{k}(v,u)\frac{1}{[\alpha(x,u)]^\beta}\dd u \lesssim 1.
\end{equation}

\begin{equation}\label{integrate k/v-u}
  \int^{t}_{t- \tb(x,v)} \int_{\mathbb{R}^{3}}\frac{e^{-C\langle v\rangle (t-s)} e^{-\varrho |v-u|^{2}}}{|v-u|^2 }\frac{1}{\min\{\alpha(x',u),\alpha(y',u)\}^\beta}\mathrm{d}u\mathrm{d}s\lesssim \frac{1}{\min\{\alpha(x,v),\alpha(y,v)\}^\beta}.
\end{equation}

\end{lemma}
\begin{remark}
 We note that~\eqref{est:NLN} can be considered as a boarderline case of Lemma 10 in~\cite{GKTT} in which the integral ${1} / {{ {\alpha}}^{\beta}}$ is considered for $\beta \gneqq1$.

Actually to prove Proposition \ref{est:K_Gamma} we only need~\eqref{est:nonlocal_wo_e} and~\eqref{est:NLN}. But in the later section we will prove the weighted $C^{1,\beta}$ estimate~\eqref{estF_C1beta}. This type of estimate will be involved with different power in terms of $\alpha$. We summarize all these estimates in this single lemma.
\end{remark}

\begin{proof}
During the whole proof we assume $\alpha=\tilde{\alpha}$. For the other case, when $\alpha\gtrsim 1$, the lemma follows from $\mathbf{k}(v,u)\in L^1_u$.

\textit{Proof of~\eqref{NLN general}~\eqref{est:nonlocal_wo_e} and~\eqref{est:NLN}}. We only prove~\eqref{NLN general}. \eqref{est:nonlocal_wo_e} and~\eqref{est:NLN} follows directly from~\eqref{NLN general}.
%

\textit{Step 1. }We claim that for $y\in \bar{\Omega}$ and $\varrho>0$,
\begin{equation}\label{eqn: int alpha du}
\int_{\mathbb{R}^3}\frac{e^{-\varrho|v-u|^2}}{|v-u|  }
\frac{1}{ \alpha(y,u)}
\dd u \lesssim 1+ |\ln
  |\xi(y)|   |  +  |\ln
  |v|   | .
\end{equation}
Recall (\ref{bar_v}) and set $\mathbf{v}= \mathbf{v} (y)$ and $\mathbf{u}= \mathbf{u} (y)$. For $|u|\geq O(1){|v|} $,
\Be\label{com:alpha}
\big[|\mathbf{u}_3 (y)|^2+|\xi(y)||u|^2 \big]^{1/2}
\gtrsim \big[|\mathbf{u}_3 (y)|^2+|\xi(y)||v|^2 \big]^{1/2}.
\Ee
Thus
\Be
\begin{split} \label{int_1/a_1}
 {\int_{\frac{|v|}{2} \leq |u|\leq 2|v| }}
\lesssim & \
\iint
 \frac{e^{-\varrho|\mathbf{v}_\parallel-\mathbf{u}_\parallel|^2}}{|\mathbf{v}_\parallel-\mathbf{u}_\parallel|}\dd \mathbf{u}_\parallel
 \int_0^{ 2|v|}\frac{  \dd \mathbf{u}_3
 }{\big[ |\mathbf{u}_3|^2+|\xi (y)||v|^2\big]^{1/2}}
 \\
 \leq& \  \int_0^{2|v|}\frac{
\dd \mathbf{u}_3
 }{\big[ |\mathbf{u}_3|^2+|\xi (y)||v|^2\big]^{1/2}}
 \\
= & \ \ln\Big( \sqrt{|\mathbf{u}_3|^2+|\xi(y)||v|^2}+|\mathbf{u}_3|\Big)\Big|_{0}^{2|v|}\\
=& \
\ln (\sqrt{4|v|^2 + |\xi(y)||v|^2}
+ 4 |v|^2
)- \ln (\sqrt{  |\xi(y)||v|^2}
)\\
\leq &  \
\ln |v| + \ln |\xi(y)|.
 \end{split}
 \Ee


If $|u|\geq 2|v|$ then $|u-v|^2 \geq |v|^2+|u|^2$ and hence $e^{- \varrho |v-u|^2}
\leq e^{- \frac{\varrho}{2} |v|^2}e^{- \frac{\varrho}{2} |u|^2} e^{- \frac{\varrho}{2} |v-u|^2}
$. This, together with (\ref{com:alpha}), implies
\Be
\begin{split} \label{int_1/a_2}
{\int_{|u|\geq 2|v|}}
\lesssim & \  e^{- \frac{\varrho}{2}|v|^2}
\iint
 \frac{e^{-\frac{\varrho}{2}|\mathbf{v}_\parallel-\mathbf{u}_\parallel|^2}}{|\mathbf{v}_\parallel-\mathbf{u}_\parallel|}\dd \mathbf{u}_\parallel
 \int_0^{\infty}\frac{
 e^{-\frac{\varrho}{2}|\mathbf{u}_3|^2}
 }{\big[ |\mathbf{u}_3|^2+|\xi (y)||v|^2\big]^{1/2}}   \dd \mathbf{u}_3
 \\
 \lesssim & \  e^{- \frac{\varrho}{2}|v|^2}
  \int_0^{\infty}\frac{
 e^{-\frac{\varrho}{2}|\mathbf{u}_3|^2}
 }{\big[ |\mathbf{u}_3|^2+|\xi (y)||v|^2\big]^{1/2}}   \dd \mathbf{u}_3\\
 \lesssim & \ e^{- \frac{\varrho}{2}|v|^2} +  e^{- \frac{\varrho}{2}|v|^2}
  \int_0^{1}\frac{
 \dd \mathbf{u}_3
 }{\big[ |\mathbf{u}_3|^2+|\xi (y)||v|^2\big]^{1/2}}
 \\
= & \ e^{- \frac{\varrho}{2}|v|^2}+ e^{- \frac{\varrho}{2}|v|^2}\ln\Big( \sqrt{|\mathbf{u}_3|^2+|\xi(y)||v|^2}+|\mathbf{u}_3|\Big)\Big|_{0}^{1}\\
=& \
e^{- \frac{\varrho}{2}|v|^2} \Big\{1+
\ln (\sqrt{1 + |\xi(y)||v|^2}
+ 1
)- \ln (\sqrt{  |\xi(y)||v|^2}
)
\Big\}
\\
\leq &  \
e^{- \frac{\varrho}{2}|v|^2}
\{\ln |v| + \ln |\xi(y)|\} .
 \end{split}
 \Ee

For $|u|\leq \frac{|v|}{2}$, we have $|v-u|\geq \big||v|-|u|\big|\geq |v|-\frac{|v|}{2}\geq \frac{|v|}{2}$. We have
\Be\notag
\begin{split}
{\int_{|u|\leq \frac{|v|}{2}}}
\lesssim    \frac{
e^{- \frac{\varrho}{2}|v|^2}
}{|v|}\int_{|\mathbf{u}_3|+|\mathbf{u}_\parallel|\leq \frac{|v|}{2}} \frac{ \dd\mathbf{u}_3 \dd \mathbf{u}_\parallel }{\Big[|\mathbf{u}_3|^2+|\xi(y)||\mathbf{u}_\parallel|^2 \Big]^{1/2}}
  \lesssim |v|
e^{- \frac{\varrho}{2}|v|^2}\int_{|
\mathbf{\tilde{u}}_3
|\leq \frac{1}{2}}
 \int_{|\mathbf{\tilde{u}}_\parallel|\leq \frac{1}{2}}
 \frac{
 \dd \mathbf{\tilde{u}}_\parallel \dd \mathbf{\tilde{u}}_3
}{\Big[|\mathbf{\tilde{u}}_3|^2+|\xi(y)||\mathbf{\tilde{u}}_\parallel|^2 \Big]^{1/2}},
\end{split}\Ee
where we have used $|v|\tilde{u} = u$. Using the polar coordinate $\mathbf{\tilde{u}}_{1}=|\mathbf{\tilde{u}}_\parallel|\cos\rho, \mathbf{\tilde{u}}_{2}=|\mathbf{\tilde{u}}_\parallel|\sin \rho$, 
we have
\Be
\begin{split} \label{int_1/a_3}
{\int_{|u|\leq \frac{|v|}{2}}} \lesssim& \  |v|
e^{- \frac{\varrho}{2}|v|^2}\int_0^{ \frac{1}{2}} \dd \mathbf{\tilde{u}}_3
\int_{0}^{2\pi}\int_{0}^{\sqrt{1/2}}
 \frac{   |\mathbf{\tilde{u}}_\parallel|\dd |\mathbf{\tilde{u}}_\parallel|  \dd \rho
}{\Big[|\mathbf{\tilde{u}}_3|^2+|\xi(y)||\mathbf{\tilde{u}}_\parallel|^2 \Big]^{1/2}}\\
 \lesssim & \
  |v|
e^{- \frac{\varrho}{2}|v|^2}
\int_0^{ \frac{1}{2}}
\dd \mathbf{\tilde{u}}_3
\int_{0}^{ {1/2}}
 \frac{   \dd |\mathbf{\tilde{u}}_\parallel|^2
}{\Big[|\mathbf{\tilde{u}}_3|^2+|\xi(y)||\mathbf{\tilde{u}}_\parallel|^2 \Big]^{1/2}}\\
=& \   |v|
e^{- \frac{\varrho}{2}|v|^2}
\int_0^{ \frac{1}{2}} \dd \mathbf{\tilde{u}}_3
\frac{1}{|\xi(y)|}
\Big( \sqrt{
|\mathbf{\tilde{u}}_3|^2+\frac{|\xi(y)|}{2} } - |\mathbf{\tilde{u}}_3|
\Big)\\
=& \ \frac{|v|
e^{- \frac{\varrho}{2}|v|^2}}{|\xi(y)|}
  \bigg\{\frac{1}{2}|\mathbf{\tilde{u}}_3|\sqrt{|\mathbf{\tilde{u}}_3|^2+\frac{|\xi|}{2}}+\frac{|\xi|}{4}\log\Big(\sqrt{|\mathbf{\tilde{u}}_3|^2+\frac{|\xi|}{2}}+|\mathbf{\tilde{u}}_3| \Big)-\frac{1}{2}|\mathbf{\tilde{u}}_3|^{1/2} \bigg\}\bigg|_{|\mathbf{\tilde{u}}_3|=0}^{|\mathbf{\tilde{u}}_3|=1/2} \\
  = & \  \frac{|v|e^{-C|v|^2}}{|\xi|}\bigg\{\frac{1}{4}\sqrt{\frac{1}{4}+\frac{|\xi|}{2}}+\frac{\xi}{4}\log\Big( \sqrt{\frac{1}{4}+\frac{|\xi|}{2}}+\frac{1}{2}\Big)-\frac{|\xi|}{4}\log\Big(\sqrt{\frac{|\xi|}{2}} \Big)-\frac{1}{8} \bigg\} \\
 \lesssim& \   \frac{|v|e^{-C|v|^2}}{|\xi|}\Big[|\xi|\log(|\xi|)+|\xi|\log\Big(1+\sqrt{1+|\xi|}\Big) \Big]\\\lesssim  & \ 1+\log(|\xi(y)|).
\end{split}
\Ee
\hide
\[
~\eqref{eqn: three cases}_2\lesssim |v|e^{-C|v|^2} \int_{|u_n|\leq \frac{1}{2}}du_n \int_{0}^{2\pi}\int_{0}^{\sqrt{1/2}}\frac{rdr\theta}{\Big[|u_n|^2+|\xi|r^2 \Big]^{1/2}}\]
\[\lesssim  |v|e^{-C|v|^2} \int_{|u_n|\leq \frac{1}{2}}du_n \int_0^{\sqrt{1/2}} \frac{dr^2}{\Big[|u_n|^2+|\xi||u_\tau|^2 \Big]^{1/2}}\]
\[\lesssim  |v|e^{-C|v|^2} \int_{|u_n|\leq \frac{1}{2}}du_n \frac{1}{|\xi|}\Big[\sqrt{|u_n|^2+\frac{1}{2}|\xi|}-|u_n| \Big]\]
\[\lesssim\frac{|v|e^{-C|v|^2}}{|\xi|} \bigg\{\frac{1}{2}|u_n|\sqrt{|u_n|^2+\frac{|\xi|}{2}}+\frac{|\xi|}{4}\log\Big(\sqrt{|u_n|^2+\frac{|\xi|}{2}}+|u_n| \Big)-\frac{1}{2}|u_n|^{1/2} \bigg\}\bigg|_{|u_n|=0}^{|u_n|=1/2}       \]
\[=\frac{|v|e^{-C|v|^2}}{|\xi|}\bigg\{\frac{1}{4}\sqrt{\frac{1}{4}+\frac{|\xi|}{2}}+\frac{\xi}{4}\log\Big( \sqrt{\frac{1}{4}+\frac{|\xi|}{2}}+\frac{1}{2}\Big)-\frac{|\xi|}{4}\log\Big(\sqrt{\frac{|\xi|}{2}} \Big)-\frac{1}{8} \bigg\}\]
\[\lesssim \frac{|v|e^{-C|v|^2}}{|\xi|}\Big[|\xi|\log(|\xi|)+|\xi|\log\Big(1+\sqrt{1+|\xi|}\Big) \Big]\lesssim 1+\log(|\xi|).\]

For $|\xi(y)|\gtrsim 1$. Then we choose any orthonormal basis $(\tau_1,\tau_2,n)=\{e_1,e_2,e_3\}$ and thus decompose $u\in \mathbb{R}^3$ as $u=u_1e_1+u_2e_2+u_3e_3=u_{\tau,1}e_1+u_{\tau,2}e_2+u_{n}e_3$. Then
\[\tilde{\alpha}(y,u)=|u\cdot \nabla \xi(y)|^2-2\xi(y)\{u\cdot \nabla^2\xi(y)\cdot u\}\geq 2|\xi(y)|\{u\cdot \nabla^2\xi(y)\cdot u\}\]
\[\gtrsim |u|^2+|\xi(y)|\{u\cdot \nabla^2 \xi(y)\cdot u\}\gtrsim |u_n|^2+|\xi(y)||u|^2.\]
Then we follow the same proof as the case $|\xi(y)|\ll 1$ to conclude~\eqref{eqn: int alpha du}.
\unhide

Collecting terms from (\ref{int_1/a_1}), (\ref{int_1/a_2}), and (\ref{int_1/a_3}), we prove (\ref{eqn: int alpha du}).

  \textit{Step 2. } We prove the following statement: We can choose $0 < \tilde{\delta}\ll_\O 1$ such that
\Be
\begin{split}
 \tilde{\delta}^{1/2}   |v\cdot \nabla \xi ( x-(t-s) v) |\gtrsim_\O&  \ |v|\sqrt{-\xi (
x-(t-s) v)},
 \\
& \ \ for  \ s\in
  \Big[t-\tb(x,v),
t-\tb(x,v)+ \tilde{\delta}\frac{\alpha(x,v)}{|v|^2} \Big]\cup
 \Big[t- \tilde{\delta}\frac{\alpha(x,v)}{|v|^2},
t  \Big]
,
\label{vpxxi_alpha_1}
\end{split} \Ee
\Be
 \tilde{\delta}^{1/2} \times
\alpha(x,v)  \lesssim_\O    |v|\sqrt{-\xi (
x-(t-s) v)}, \ \ for  \  s\in \Big[
t-\tb(x,v)+ \tilde{\delta}\frac{\alpha(x,v)}{|v|^2}, t- \tilde{\delta} \frac{\alpha(x,v)}{|v|^2}
\Big]. \label{vpxxi_alpha_2}
\Ee

If $v=0$ or $v\cdot  \nabla \xi (x)=0$ then (\ref{vpxxi_alpha_1}) and (\ref{vpxxi_alpha_2}) hold clearly. We may assume $v\neq 0$ and $v\cdot \nabla \xi(x)>0$. Due to (\ref{Velocity_lemma}),  $%
v\cdot \nabla\xi (x_{\mathbf{b}}(x,v) ) <0$. By the mean value theorem there exists at least one $t^{\ast } \in (t-\tb(x,v)  , t)$ such that $v\cdot \nabla \xi
(x- t^*v)=0$. Moreover due to the convexity in (\ref{convex}) we have $\frac{d}{ds}\big(v\cdot \nabla \xi (x-(t-s)v)\big)=v\cdot
\nabla ^{2}\xi (x-(t-s)v )\cdot v\geq C_{\xi }|v|^{2},$ and therefore $t^{\ast }\in (t-\tb(x,v), t)$ is unique. 

Let $s \in \big[t-
\tilde{\delta} \frac{\alpha(x,v)}{|v|^2},t\big]$ for $0<\tilde{\delta}\ll 1$. Then from the fact that $v\cdot \nabla_x \xi(x-(t-\tau )v)$ is non-decreasing function in $\tau  \in [t-t^*,t]$,
\Be \label{exp_xi}
|v|^2  (-1) \xi(x-(t-s)v)
=
\int^t_s |v|^2   v\cdot \nabla_x \xi(x-(t-\tau )v)    \dd \tau
\leq \tilde{\delta}  \alpha(x,v)
|v\cdot \nabla_x \xi(x)| .
\Ee
Since $ |v\cdot \nabla_x \xi (x)|\leq \alpha(x,v) \leq C_{\O} \alpha(x-(t-s)v,v)\leq C_\O \{
|v \cdot \nabla \xi (x-(t-s)v)| + \| \nabla_x^2 \xi \|_\infty |v| \sqrt{- \xi(x-(t-s)v)}
\},$ we choose $\tilde{\delta}\ll (C_\O \| \nabla_x^2 \xi \|_\infty)^{-2}$ and absorb
\[
\tilde{\delta} \alpha(x,v)  \times C_\O  \{
|v \cdot \nabla \xi (x-(t-s)v)| + \| \nabla_x^2 \xi \|_\infty |v| \sqrt{- \xi(x-(t-s)v)}
\}
 \leq
\tilde{\delta} \times \{C_\O \| \nabla_x^2 \xi \|_\infty |v| \sqrt{- \xi (x-(t-s)v)}\}^2\]
 by the left hand side of (\ref{exp_xi}). This gives (\ref{vpxxi_alpha_1}) for $s \in \big[t-
\tilde{\delta} \frac{\alpha(x,v)}{|v|^2},t\big]$. The proof for $s \in \big[t-\tb(x,v),
t-\tb(x,v)+ \tilde{\delta}\frac{\alpha(x,v)}{|v|^2} \big]$ is same.

For $\xi(x-(t-s)v)$ is non-increasing in $s \in [t-t^*,t]$, 
we have $|v|^2 (-1) \xi(x-(t-s)v)\geq |v|^2 (-1) \xi\Big(x-\tilde{\delta} \frac{\alpha(x,v)}{|v|^2}v\Big)$ for $s\in \big[t- t^*, t- \tilde{\delta} \frac{\alpha(x,v)}{|v|^2}
\big]$. By an expansion, for $s^*:=t-\tilde{\delta} \frac{\alpha(x,v)}{|v|^2}$,
\hide

\Be\begin{split}
|v|^2 (-1) \xi(x-(t-s)v)\geq |v|^2 (-1) \xi\Big(x-\tilde{\delta} \frac{\alpha(x,v)}{|v|^2}v\Big)  \geq  \int^t_s |v|^2 v\cdot \nabla_x \xi(x-(t-s)v) \dd \tau
\end{split}\Ee
\unhide
\Be\label{exp_xi_td}
|v|^2  (-1) \xi\Big(x-\tilde{\delta} \frac{\alpha(x,v)}{|v|^2}v\Big)
=
 |v|^2 ( v\cdot \nabla_x \xi(x ))\tilde{\delta} \frac{\alpha(x,v)}{|v|^2}
+
 \int^t_{s^*}
\int^\tau_{s^*}
|v|^2 v\cdot \nabla_x^2 \xi(x-(t-\tau^\prime) v) \cdot v
 \dd \tau^\prime
 \dd \tau.
\Ee
 The last term of (\ref{exp_xi_td}) is bounded by $ \| \nabla_x^2 \xi \|_\infty\tilde{\delta}^2 \big(\frac{\alpha(x,v)}{|v|^2}\big)^2 |v|^4 \leq  \| \nabla_x^2 \xi \|_\infty\tilde{\delta}^2   \alpha(x,v) ^2$. Since $v \cdot \nabla_x \xi(x) \leq \alpha(x,v)$, for $\tilde{\delta} \ll \| \nabla_x^2 \xi \|_\infty^{-1/2}$, the right hand side of (\ref{exp_xi_td}) is bounded below by $\frac{\tilde{\delta}}{2} \alpha(x,v)^2$. This completes the proof of (\ref{vpxxi_alpha_2}) when $s\in \big[t- t^*, t- \tilde{\delta} \frac{\alpha(x,v)}{|v|^2}
\big]$. The proof for the case of $s\in \big[ t - \tb(x,v)+ \delta \frac{\alpha(x,v)}{|v|^2}, t- t^*
\big]$ is same.
\hide

From the choice of $s$ and (\ref{Velocity_lemma}) we derive that
\Be
\begin{split}\notag
(\ref{exp_xi})_{1} & \leq \tilde{\delta} |v\cdot \nabla_x \xi(x-(t-s)v)| \times  \alpha(x,v)
\leq \tilde{\delta} |v\cdot \nabla_x \xi(x-(t-s)v)| \times C_\O  \alpha(x-(t-s)v,v)\\
& \leq   4C_\O \tilde{\delta} |v\cdot \nabla_x \xi(x-(t-s)v)|^2
+ C_\O \tilde{\delta} |v|^2  (-1) \xi(x-(t-s)v),\\
(\ref{exp_xi})_{2} & \leq C_\O^2 \tilde{\delta}^2  \alpha(x-(t-s)v,v)^2
\leq C_\O^2 \tilde{\delta}^2  |v\cdot \nabla_x \xi(x-(t-s)v)|^2
+  C_\O^2\tilde{\delta}^2 |v|^2  (-1) \xi(x-(t-s)v).
\end{split}\Ee
Then from (\ref{exp_xi}) we have
\Be\notag
\big\{
1 -C_\O \tilde{\delta}- C_\O^2 \tilde{\delta}^2
\big\}
|v|^2  (-1) \xi(x-(t-s)v) \leq \big\{
4 C_\O \tilde{\delta} + C_{\O}^2 \tilde{\delta}^2
\big\} |v\cdot \nabla_x \xi(x-(t-s)v)|,
\Ee
 and this completes the proof of (\ref{vpxxi_alpha_1}) for $s \in \big[t-
\tilde{\delta} \frac{\alpha(x,v)}{|v|^2},t\big]$. The proof for $s \in \big[t-\tb(x,v),
t-\tb(x,v)+ \tilde{\delta}\frac{\alpha(x,v)}{|v|^2} \big]$ is same.

 such that $t-s \in [0,t^*  ]$ and $|v\cdot \nabla \xi ( x-(t-\tau) v) |\geq  |v|\sqrt{-\xi (
x-(t-\tau) v)}$ holds $s \leq \tau \leq t$.

Applying the choice of $s$ to the first term and the convexity to the second term, this equals $ \tilde{\delta} \alpha(x,v ) ( v\cdot \nabla_x \xi(x-(t-s)v))
 +O(1) \tilde{\delta}^2\alpha(x,v )^2$.

\Be\notag
\begin{split}
  & \tilde{\delta} \alpha(x,v ) ( v\cdot \nabla_x \xi(x-(t-s)v))
 +O(1) \tilde{\delta}^2\alpha(x,v )^2\\
=& O(\tilde{\delta}) \alpha(x,v )^2  \\
\gtrsim& |v\cdot \nabla \xi ( x-(t-\tau) v) |^2.
\end{split}
\Ee

Define $\Phi (s)= |v\cdot \nabla \xi (x-(t-s)v)|^{2}+|v|^{2}\xi (x-(t-s)v) .$ From 
(\ref{convex}), $\frac{d}{ds}\Phi (s)=\big(v\cdot \nabla \xi (x-(t-s)v)\big)%
\big\{2\big(v\cdot \nabla ^{2}\xi (x-(t-s)v)\cdot v\big)%
+|v|^{2}\big\}<0$ for $t-s\in \lbrack 0,t^{\ast }]$ and $\frac{d}{ds}\Phi (s)>0$ for $t-s\in \lbrack t^{\ast }, t-\tb(x,v)]$. Note that $\Phi
(0)>0$ and $\Phi(\tb(x,v))>0$ from $v\cdot \nabla \xi(x) >0$ and $v\cdot  \nabla \xi(x_{\mathbf{b}}(x,v))  <0$. Note that $\Phi$ is continuous
function on the interval $[0,  t_{\mathbf{b}}(x,v) ]$ so that it has a minimum.
If $\min_{[0, t_{\mathbf{b}} ]}\Phi (s)\leq 0,$ there exist $\sigma
_{1},\sigma _{2}>0$ satisfying
\begin{equation}
\begin{split}
\Phi (\min\{\tau,t_{\mathbf{b}}\}+\sigma _{1})&=\Phi (t_{\mathbf{b}%
}(x,v))+\int_{0}^{\sigma _{1}}\frac{d}{ds}\Phi (s)\mathrm{d}s=0, \\
\Phi (\min\{\tau,t_{\mathbf{b}}\}-\sigma _{2})&=\Phi (\min\{\tau,t_{\mathbf{b}}\}-\int_{\min\{\tau,t_{\mathbf{b}}\} -\sigma _{2}}^{\min\{\tau,t_{\mathbf{b}}\}}\frac{d}{ds}\Phi (s)%
\mathrm{d}s=0,
\end{split}
\notag
\end{equation}%
then $\sigma _{1}\leq t^{\ast }$ and $\min\{\tau,t_{\mathbf{b}}\}-\sigma _{2}\geq
t^{\ast }$ and there is no other $s\in [0,\min\{\tau,t_{\mathbf{b}}\}]$ satisfying $%
\Phi (s)=0$. Moreover we have $\Phi (s)\leq 0$ for $s\in \lbrack \sigma
_{1},\min\{\tau,t_{\mathbf{b}}\}-\sigma _{2}]$. If $\min_{[0,\min\{\tau,t_{\mathbf{b}}\}]}\Phi (s)>0,$ there do not exist such $\sigma _{1}$ and $\sigma
_{2} $ then we let $\sigma _{1}=t^{\ast }$ and $\sigma _{2}=\min\{\tau,t_{\mathbf{b}}\}-t^{\ast }$. This proves (\ref{sigma}).

Secondly we prove (\ref{COV_xi_s}). By the proof of ({\ref{sigma}}) and the fact
\begin{equation*}
\frac{d |\xi|}{ds} = -\frac{d}{ds} \xi ( x-(\min\{\tau,t_{\mathbf{b}}\} -s)v) = -
v\cdot \nabla_{x} \xi(x-(\min\{\tau,t_{\mathbf{b}}\} -s)v ),
\end{equation*}
and the inverse function theorem we prove (\ref{COV_xi_s}).

First we establish (\ref{sigma}) and (\ref{COV_xi_s}).
For small $0 < \tilde{\delta}\ll 1,$ we define
\begin{equation}  \label{sigma_delta}
\tilde{\sigma}_{1} := \min \Big\{\sigma_{1}, \tilde{\delta} \frac{
\tilde{\alpha}(x,v)}{|v|^{2}}\Big\} , \ \ \ \tilde{\sigma}_{2} := \min \Big\{%
\sigma_{2}, \tilde{\delta} \frac{\tilde{\alpha}(x,v)}{|v|^{2}}\Big\}.
\end{equation}

When $\alpha(x-(t-s)v,u)\geq 1$, by the definition in~\eqref{eqn: kinetic weight} we have
\[\int^{t}_{t-\min\{\tau,t_{\mathbf{b}}\}}  e^{-C\langle v\rangle (t-s)} \int_{\mathbb{R}^{3}}\frac{e^{-\theta |v-u|^{2}}}{|v-u|[\alpha (  x-(t-s)v, {u})]}\mathrm{d}u\mathrm{d}s \]
\[\lesssim \int^{t}_{t-\min\{\tau,t_{\mathbf{b}}\}}  e^{-C\langle v\rangle (t-s)} \int_{\mathbb{R}^{3}}\mathbf{k}_\varrho(v,u)\mathrm{d}u\mathrm{d}s\lesssim \int_{t-\min\{\tau,t_b\}}^t e^{-C\langle v\rangle (t-s)}ds\lesssim~\eqref{eqn: integrate alpha along trajectory 2}. \]
When $\alpha(x-(t-s)v,u)\leq 1$ since we have the relation~\eqref{eqn: relation of alpha and tildealpha}, we only need to prove this lemma with replacing all $\alpha$ by $\tilde{\alpha}$.

We first assume $v\cdot \nabla\xi(x)\geq 0$ and $x\in \partial \Omega.$ There exist $\sigma _{1},\sigma _{2}>0$ such that
\begin{equation}  \label{sigma}
\begin{split}
& |v\cdot\nabla \xi ( x-(\min\{\tau,t_{\mathbf{b}}\}-s) v ) |\gtrsim \tilde{\alpha} (
x-(\min\{\tau,t_{\mathbf{b}}\}-s) v,v)\  \\
& \ \ \ \ \ \ \ \ \ \ \ \ \ \ \ \ \ \ \ \ \ \ \ \ \ \ \ \ \ \ \ \ \ \ \ \ \
\ \ \ \ \ \ \ \ \ \ \ \ \ \ \ \ \ \ \ \ \ \text{for all }s\in \lbrack
0,\sigma _{1}]\cup \lbrack \min\{\tau,t_{\mathbf{b}}\}-\sigma _{2},t_{\mathbf{b}%
}(x,v)], \\
\end{split}%
\end{equation}
and $|v|\sqrt{-\xi ( x-(\min\{\tau,t_{\mathbf{b}}\}-s) v )}\gtrsim \tilde{\alpha} (
x-(\min\{\tau,t_{\mathbf{b}}\}-s) v ,v) \ \   \text{for all }s\in \lbrack \sigma
_{1},\min\{\tau,t_{\mathbf{b}}\}-\sigma _{2}].$ The mapping $s \mapsto \xi(x-(t_{%
\mathbf{b}}(x,v) -s )v)$ is one-to-one and onto on $s \in [0,\sigma_{1}]$ or
on $s \in [\min\{\tau,t_{\mathbf{b}}\} -\sigma_{2}, \min\{\tau,t_{\mathbf{b}}\}]$. Moreover
this mapping $s \mapsto \xi(x-(\min\{\tau,t_{\mathbf{b}}\} -s )v)$ is a diffeomorphism
and we have a change of variables on $s \in [0,\sigma_{1}] \ \text{ or } \
s\in [\min\{\tau,t_{\mathbf{b}}\} - \sigma_{2}, \min\{\tau,t_{\mathbf{b}}\}]$,
\begin{equation}  \label{COV_xi_s}
\mathrm{d}s = \frac{\mathrm{d} |\xi|}{ | \nabla \xi (x-(\min\{\tau,t_{\mathbf{b}}\}
-s )v ) \cdot v |} \lesssim \frac{\mathrm{d}|\xi|}{\tilde{\alpha} (x-(\min\{\tau,t_{\mathbf{b}}\} -s )v )}.
\end{equation}

Firstly we prove (\ref{sigma}). Recall the definition of $\tilde{\alpha}$ in~\eqref{eqn: kinetic weight tilde}. It suffices to show

\vspace{8pt}

\noindent\textit{Step 2.} For small $0 < \tilde{\delta}\ll 1,$ we define
\begin{equation}  \label{sigma_delta}
\tilde{\sigma}_{1} := \min \Big\{\sigma_{1}, \tilde{\delta} \frac{
\tilde{\alpha}(x,v)}{|v|^{2}}\Big\} , \ \ \ \tilde{\sigma}_{2} := \min \Big\{%
\sigma_{2}, \tilde{\delta} \frac{\tilde{\alpha}(x,v)}{|v|^{2}}\Big\}.
\end{equation}
Actually, thanks to the fact $\sigma_{i} \leq \tb \lesssim \frac{\tilde{\alpha}(x,v)}{|v|^{2}}$, by Lemma 3 in~\cite{GKTT}, we have $\tilde{\sigma}_{i} :=   \tilde{\delta}  {\tilde{\alpha}(x,v)}/{|v|^{2}}$.

Then both of (\ref{sigma}) and (\ref{COV_xi_s}) hold for $s\in [0,\tilde{
\sigma}_{1}]\cup [\min\{\tau,t_{\mathbf{b}}\}-\tilde{ \sigma}_{2}, t_{\mathbf{b}%
}(x,v)]$ without changing the constant. Moreover, if $s\in [0,\tilde{ \sigma}%
_{1}]\cup [\min\{\tau,t_{\mathbf{b}}\}-\tilde{ \sigma}_{2}, \min\{\tau,t_{\mathbf{b}}\}]$
then by the Velocity lemma
\begin{equation}  \label{max_xi}
\max \{|\xi|\} := \max_{s\in [0, \tilde{\sigma}_{1}] \cup [\min\{\tau,t_{\mathbf{b}}\} - \tilde{\sigma}_{2}, \min\{\tau,t_{\mathbf{b}}\} ]} |\xi(x-(t_b-s)v)| \ \lesssim \ \tilde{\delta} \frac{\tilde{\alpha}^2(x,v)}{|v|^{2}}.
\end{equation}
For $s\in [ \tilde{ \sigma}_{1}, \min\{\tau,t_{\mathbf{b}}\} -\tilde{\sigma}_{2} ]$
we have the following estimate with $\tilde{\delta}-$dependent constant:
\begin{equation}  \label{lower_delta}
\begin{split}
&|v| \sqrt{-\xi(x-(\min\{\tau,t_{\mathbf{b}}\}-s)v)} \  \ \gtrsim_{\xi, \tilde{\delta}} \ \  \tilde{\alpha}(x-(\min\{\tau,t_{\mathbf{b}}\}-s)v,v).
\end{split}%
\end{equation}

The proof of (\ref{max_xi}) is due to, for $s\in \lbrack 0,\tilde{\sigma}%
_{1}],$
\begin{equation}\label{tz}
\begin{split}
|\xi (x-(\min\{\tau,t_{\mathbf{b}}\}-s)v)| &\leq  \int_{0}^{s}|v\cdot \nabla \xi
(x-(\min\{\tau,t_{\mathbf{b}}\}-\tau )v)|\mathrm{d}\tau \\
&\lesssim \tilde{ \tilde{\alpha}} (x,v)|s|  \lesssim  \min \left\{ \tilde{\alpha} \min\{\tau,t_{\mathbf{b}}\},\frac{\tilde{\delta}\tilde{\alpha}^2 }{%
|v|^{2}}\right\} =: B,
\end{split}
\end{equation}
where we have used $\tilde{\alpha} (x-(\min\{\tau,t_{\mathbf{b}}\}-s)v,v)\lesssim _{\xi }\tilde{\alpha} (x,v)$ from the Velocity lemma (Lemma 2 in~\cite{GKTT}
). The proof for $s\in \lbrack \min\{\tau,t_{\mathbf{b}}\}-\tilde{%
\sigma}_{2},\min\{\tau,t_{\mathbf{b}}\}]$ is exactly same.

Now we prove (\ref{lower_delta}). Recall that for $t^{*} \in[0, t_{\mathbf{b}%
}(x,v)]$ in the previous step we proved: $v\cdot \nabla \xi(x-(t_{\mathbf{b}%
}(x,v)-t^{*})v)=0$. Clearly $|\xi(x-(t-s)v,v)|$ is an increasing
function on $[0, t^{*}]$ and a decreasing function on $[t^{*}, t_{%
\mathbf{b}}(x,v)]$. This is due to the convexity of $\xi$:
\begin{equation*}
\frac{d^{2}}{ds^{2}} [-\xi(s-(\min\{\tau,t_{\mathbf{b}}\}-s)v)] = v\cdot
\nabla^{2}\xi(x-(\min\{\tau,t_{\mathbf{b}}\}-s)v) \cdot v \gtrsim |v|^{2},
\end{equation*}
and $v\cdot \nabla \xi(x)>0$ and $v\cdot \nabla \xi(x_{\mathbf{b}}(x,v))<0.$

Therefore, from $\xi(x)= 0=\xi(\xb)$,
\begin{equation}
\begin{split}
-\xi (x-(\min\{\tau,t_{\mathbf{b}}\}-s)v)& =-\xi (x)-\int_{t_{\mathbf{b}%
}(x,v)}^{s}v\cdot \nabla \xi (x-(\min\{\tau,t_{\mathbf{b}}\}-\tau )v)\mathrm{d}\tau
\\
& =\int_{s}^{\min\{\tau,t_{\mathbf{b}}\}}v\cdot \nabla \xi (x-(t_{\mathbf{b}%
}(x,v)-\tau )v)\mathrm{d}\tau \\
& \geq (\min\{\tau,t_{\mathbf{b}}\}-s)(v\cdot \nabla \xi (x-(t_{\mathbf{b}%
}(x,v)-s)v)) \\
& \geq \tilde{\sigma}_{2}|v\cdot \nabla \xi (x-\tilde{\sigma}_{2}v)|\ \ \
\text{for}\ \ s\in \lbrack t^{\ast },\min\{\tau,t_{\mathbf{b}}\}-\tilde{\sigma}_{2}],
\end{split}
\notag
\end{equation}%
and
\begin{equation}
\begin{split}
-\xi (x-(\min\{\tau,t_{\mathbf{b}}\}-s)v)& =-\xi (x_{\mathbf{b}}(x,v))-\int_{0}^{s}v%
\cdot \nabla \xi (x-(\min\{\tau,t_{\mathbf{b}}\}-\tau )v)\mathrm{d}\tau \\
& \geq s|v\cdot \nabla \xi (x-(t_{\mathbf{b}}-s)v)| \\
& \geq \tilde{\sigma}_{1}|v\cdot \nabla \xi (x_{\mathbf{b}}(x,v)+\tilde{%
\sigma}_{1}v)|\ \ \ \text{for}\ \ s\in \lbrack 0, \tilde{\sigma}_{1}].
\end{split}
\notag
\end{equation}%

Hence, for $s \in [\tilde{\sigma}_{1}, \min\{\tau,t_{\mathbf{b}}\}-\tilde{\sigma}%
_{2}],$
\begin{equation}
\begin{split}
|\xi(x -(t_{\mathbf{b}} -s )v )| &\geq \min \Big\{ |\xi(x - \tilde{\sigma}%
_{2} v )|, | \xi( x_{\mathbf{b}}(x,v) + \tilde{\sigma}_{1} v ) |\Big\} \\
&\geq \min \Big\{ \tilde{\sigma}_{2} |v \cdot \nabla \xi(x -\tilde{\sigma}%
_{2} v)|, \tilde{\sigma}_{1} | v \cdot \nabla \xi(x_{\mathbf{b}} (x,v)+
\tilde{\sigma}_{1} v)|\Big\}.
\end{split}
\notag
\end{equation}
From the definition of $\tilde{\sigma}_{1}$ and $\tilde{\sigma}_{2}$ in (\ref%
{sigma_delta}) we have
\begin{equation}  \label{xi_lower1_tilde}\notag
|v |^{2} |\xi(x -(\min\{\tau,t_{\mathbf{b}}\} -s) v)| \geq \tilde{\delta} \tilde{\alpha}(x , v ) \min \Big\{ |v \cdot \nabla \xi(x -\tilde{\sigma}_{2} v)|, |
v \cdot \nabla \xi(x_{\mathbf{b}} (x,v)+ \tilde{\sigma}_{1} v)|\Big\} .
\end{equation}
Without loss of generality we may assume $|v \cdot \nabla \xi(x -\tilde{%
\sigma}_{2} v)|=\min \big\{ 
|v \cdot \nabla \xi(x -\tilde{\sigma}_{2} v)|, 
| v \cdot \nabla \xi(x_{\mathbf{b}} (x,v)+ \tilde{\sigma}_{1} v)|\big\}.$
Then by the Velocity lemma we have $\tilde{\alpha}(x,v)\gtrsim_{\xi}
|v||\xi(x-\tilde{\sigma}_{2} v)|^{1/2}$. Then we choose $s= t_{\mathbf{b}%
}(x,v) - \tilde{\sigma}_{2}$ to have $|v|^{2} |\xi(x-\tilde{\sigma}_{2}
v)|\geq \tilde{\delta} |v| |\xi(x-\tilde{\sigma}_{2} v)|^{1/2} \times
|v\cdot \nabla \xi(x-\tilde{\sigma}_{2}v)| $ and
\begin{equation}
\begin{split}
|v| |\xi(x-\tilde{\sigma}_{2} v)|^{1/2}\gtrsim \tilde{\delta} \times |v\cdot
\nabla \xi(x-\tilde{\sigma}_{2}v)|.
\end{split}
\notag
\end{equation}
The left hand side is the lower bound of $|v|^{2}|\xi(x-(t_{\mathbf{b}%
}(x,v)-s)v)|$ for $s\in [\tilde{\sigma}_{1}, \min\{\tau,t_{\mathbf{b}}\}-\tilde{%
\sigma}_{2}]$ and the right hand side is bounded below by the Velocity
lemma: $e^{-\mathcal{C}|v|\min\{\tau,t_{\mathbf{b}}\}} \tilde{\alpha}^2(x,v)\gtrsim_{\xi}
\tilde{\alpha}^2(x,v)$. Therefore we conclude (\ref{lower_delta}).

\unhide

 \textit{Step 3. } From (\ref{eqn: int alpha du}), for the proof of Lemma \ref{Lemma: NLN}, it suffices to estimate
 \Be\label{int_lnxi+lnv}
 \begin{split}
     &  \int^t_{t-\tb(x,v)} \mathbf{1}_{t-t_2\geq s \geq t-t_1 } e^{-C \langle v\rangle (t-s)} \big| \ln |\xi(x-(t-s)v)|\big| \dd s \\
      &  + \int^t_{t-\tb(x,v)} \mathbf{1}_{t-t_2\geq s \geq t-t_1 } e^{-C \langle v\rangle (t-s)}\big(1+ \big| \ln | v|\big| \big)\dd s .
 \end{split}
 \Ee
We simply bound the second term of (\ref{int_lnxi+lnv}) as
\Be\label{small_t}
(1+|\ln|v||)\int^{t-t_2}_{t-t_1}e^{-C\langle v\rangle (t-s)}   \lesssim (1+|\ln|v||)\langle v\rangle^{-1} |e^{-C\langle v\rangle t_2}-e^{-C\langle v\rangle t_1}|.
\Ee

For utilizing (\ref{vpxxi_alpha_1}) and (\ref{vpxxi_alpha_2}), we split the first term of (\ref{int_lnxi+lnv}) as
 \Be\label{int_ln}
 \underbrace{ \int^t_{t- \tilde{\delta} \frac{ \alpha(x,v)}{|v|^2}}
 + \int^{t -\tb(x,v)+ \tilde{\delta} \frac{\alpha(x,v)}{|v|^2}}_{t -\tb(x,v) }}_{(\ref{int_ln})_1}
 + \underbrace{ \int_{t -\tb(x,v)+ \tilde{\delta} \frac{\alpha(x,v)}{|v|^2}} ^{t- \tilde{\delta} \frac{\alpha(x,v)}{|v|^2}}}_{(\ref{int_ln})_2}.
 \Ee
Without loss of generality, we assume $t-t_2\in [t-\tilde{\delta}\frac{\alpha(x,v)}{|v|^2}]$, $t-t_1\in [t-\tb(x,v)+\tilde{\delta}\frac{\alpha(x,v)}{|v|^2}]$. For the first term $(\ref{int_ln})_1$ we use a change of variables $s \mapsto   -\xi(x-(t-s)v)$ in $s \in [t-\tb(x,v), t-t^*]$ and $s \in [t-t^*,t ]$ separately with $\dd s = |v\cdot \nabla_x \xi(x-(t-s)v)|^{-1} \dd |\xi|$. From (\ref{exp_xi}) we have $|\xi(x-(t-s)v) | \leq \tilde{\delta}\frac{\alpha^2(x,v)}{|v|^2}$. Then applying H\"{o}lder inequality with $\beta+(1-\beta)=1$ and using (\ref{vpxxi_alpha_1}), we get
 \Be\label{est:int_ln_1}
 \begin{split}
 &(\ref{int_ln})_1 \mathbf{1}_{\{t-t_2\in [t-\tilde{\delta}\frac{\alpha(x,v)}{|v|^2},t]\},t-t_1\in [t-\tb(x,v)+\tilde{\delta}\frac{\alpha(x,v)}{|v|^2}]}  \\
   & \lesssim \Big(\big[\int^{t-t_2}_{t-\tilde{\delta}\frac{\alpha(x,v)}{|v|^2}}          e^{-C\langle v\rangle (t-s)/\beta} \dd s \big]^{\beta}+\big[\int^{t-\tb(x,v)+\tilde{\delta}\frac{\alpha(x,v)}{|v|^2}}_{t-t_1}          e^{-C\langle v\rangle (t-s)/\beta} \dd s \big]^{\beta}\Big)\\
   &\times  \Big[ \int_0^{\tilde{\delta}\frac{\alpha^2(x,v)}{|v|^2}}
|\ln|\xi| |^{1/(1-\beta)}
 \frac{\dd |\xi|}{ \tilde{\delta}^{-1/2}|v| \sqrt{ | \xi |}}  \Big]^{1-\beta} \\
      &\lesssim |e^{-C\langle v\rangle t_2/\beta }-e^{-C\langle v\rangle t_1/\beta }|^\beta\frac{1}{|v|^{1-\beta}},
 \end{split}
\Ee
where we have used $t-\tilde{\delta}\frac{\alpha(x,v)}{|v|^2}>t-t_1, \quad t-t_2>t-\tb(x,v)+\tilde{\delta}\frac{\alpha(x,v)}{|v|^2}$ and $\frac{|\ln |\xi||^{1/(1-\beta)}}{\sqrt{\xi}} \in L^1_{loc}(0,\infty)$ for $\beta<1$ in the last line.

On the other hand, from (\ref{vpxxi_alpha_2}),
\Be\begin{split}\label{est:int_ln_2}
 (\ref{int_ln})_2& \leq \int^{t-t_2}_{t-t_1} e^{-C \langle v\rangle (t-s)} \Big| \ln \Big(\tilde{\delta} \frac{ \alpha(x,v)^2}{|v|^2}\Big) \Big| \dd s
 \leq
2 \int^{t-t_2}_{t-t_1} e^{-C \langle v\rangle (t-s)}  \{
| \ln \tilde{\delta} |  + |\ln \alpha(x,v)| + |\ln |v||\}
 \dd s
\\
& \leq 2
|e^{-C\langle v\rangle t_2}-e^{-C\langle v\rangle t_1}| \langle v\rangle^{-1/2}\{
 | \ln \tilde{\delta} | +  |\ln \alpha(x,v)|   +  |\ln |v||\},
\end{split}\Ee
 where we have used a similar estimate of (\ref{small_t}). Then we conclude the first inequality in (\ref{est:NLN}) using $\beta<1$.

For the second inequality, from (\ref{kinetic_distance}) and (\ref{chi}) we bound a term of the upper bound of (\ref{est:NLN}) as
 \Be
 \begin{split}\notag
&  \{ \langle v\rangle^{-1/2}\big(
1+  |\ln |v|| +  |\ln  {\alpha}(x,v)|\big) +  \frac{1}{|v|^{1-\beta}} \} \times \frac{\alpha(x,v)}{\alpha(x,v)} \\
\leq &\{1 +   \langle v\rangle^{-1/2} \min\{1, |v|\}| \ln |v||
+ \alpha| \ln |\alpha||  + \frac{\alpha(x,v)}{|v|^{1-\beta}}
 \} \times \frac{1}{\alpha(x,v)}
 \lesssim   \frac{1}{\alpha(x,v)},
\end{split} \Ee
where we have used $\alpha(x,v) \leq \min\{1,|v|\}$ and $1-\beta<1$.

   \hide
 We prove~\eqref{est:NLN}. From~\eqref{eqn: int alpha du} with $y= x-(\min\{\tau,t_{\mathbf{b}}\}
-s)v$
\begin{equation}
\begin{split}
& \int_{0}^{\min\{\tau,t_{\mathbf{b}}\}}e^{-C\langle v\rangle s}\int_{\mathbb{R}^{3}}\frac{e^{-\theta |v-u|^{2}}}{|v-u|\tilde{\alpha} (x-(t_{\mathbf{b}%
}-s)v,u)}\mathrm{d}u\mathrm{d}s \\
& \lesssim {\int_{0}^{\min\{\tau,t_{\mathbf{b}}\}}e^{-C\langle v\rangle s}\Big[1+\big|\log|\xi|\big|+\big| \log|v|\big| \Big]\mathrm{d}s}.
\end{split}
\notag
\end{equation}%
According to (\ref{sigma_delta}) we split the time integration as
\begin{equation}\label{eqn: two cases}
{\int_{0}^{\min\{\tau,t_{\mathbf{b}}\}}}e^{-C\langle v\rangle s}\Big[1+\big|\log|\xi|\big|+\big| \log|v|\big| \Big]\mathrm{d}s=\underbrace{
\int_{0}^{\tilde{\sigma}_{1}}+\int_{\min\{\tau,t_{\mathbf{b}}\}-\tilde{\sigma}%
_{2}}^{\min\{\tau,t_{\mathbf{b}}\}}}_{~\eqref{eqn: two cases}_1}+\underbrace{\int_{\tilde{\sigma}%
_{1}}^{\min\{\tau,t_{\mathbf{b}}\}-\tilde{\sigma}_{2}}}_{~\eqref{eqn: two cases}_2}.
\end{equation}%
For the first two terms$~\eqref{eqn: two cases}_1$, we use the mapping of (\ref%
{COV_xi_s})
\begin{equation*}
s\in \lbrack 0,\tilde{\sigma}_{1}]\cup \lbrack \min\{\tau,t_{\mathbf{b}}\}-\tilde{%
\sigma}_{2},\min\{\tau,t_{\mathbf{b}}\}]\mapsto |\xi (x-(\min\{\tau,t_{\mathbf{b}}\}-s)v)|\in \big[0,B\big),
\end{equation*}%
where the range of $|\xi |$ has been bounded in (\ref{max_xi}), and $B$ is
given by (\ref{tz}). By the change of variables of (\ref{COV_xi_s})
\begin{equation}
\begin{split}
~\eqref{eqn: two cases}_1& \lesssim \ \sup_{0\leq s\leq \min\{\tau,t_{\mathbf{b}}\}}\bigg[\int_{0}^{B}\big| \log|\xi|\big|\frac{%
\mathrm{d}|\xi |}{\tilde{\tilde{\alpha} }(x,v)}+\big| \log|v|\big|\int_0^B e^{-C\langle v\rangle |\xi|} \frac{d|\xi|}{\tilde{\alpha}(x,v)} \bigg]\\
& \lesssim \sup_{0\leq s\leq \min\{\tau,t_{\mathbf{b}}\}}\frac{1}{\tilde{\alpha}(x,v)}\bigg[\Big[|\xi|\big| \log |\xi|\big|-|\xi| \Big]_{|\xi|=0}^{|\xi|=B}+\big| \frac{\big| \log|v|\big|(1-e^{-C\langle v\rangle B})}{\langle v\rangle}   \bigg]
\end{split}
\notag
\end{equation}%
~\eqref{est:NLN} follows with $B$ given by
(\ref{tz}).

For $~\eqref{eqn: two cases}_2$ we use $\tilde{\alpha}(x-(\min\{\tau,t_{\mathbf{b}}\}-s)v)\lesssim_{\xi,%
\tilde{\delta}} |v|\sqrt{-\xi(x-(\min\{\tau,t_{\mathbf{b}}\}-s)v)}$ for $s\in [ \tilde{\sigma%
}_{1}, \min\{\tau,t_{\mathbf{b}}\}-\tilde{\sigma}_{2}]$, from (\ref{sigma}), to have
\begin{equation*}
\big|\log|\xi|\big|+\big| \log|v|\big|=\big| \log(|v||\sqrt{-\xi}|)\big|\lesssim \big| \log|\tilde{\alpha}(x,v)|\big|
\end{equation*}
Finally
\begin{equation*}
~\eqref{eqn: two cases}_2 \lesssim  \big| \log|\tilde{\alpha}(x,v)|\big|\int_{0}^{\min\{\tau,t_{\mathbf{b}}\}} e^{-C\langle v\rangle s}   \mathrm{d}s \lesssim
\big|\log|\tilde{\alpha}(x,v)|\big|\frac{1-e^{-C\langle v\rangle\min\{\tau,t_{\mathbf{b}}\}}}{\langle v\rangle}\}  .
\end{equation*}

Now we assume $x\notin \partial \Omega .$ We find $\bar{x}\in \partial
\Omega $ and $\bar{t}_{\mathbf{b}}$ so that
\begin{equation*}
x-(\min\{\tau,t_{\mathbf{b}}\}-s)v=\bar{x}-(\overline{\min\{\tau,t_{\mathbf{b}}\}} -s)v.
\end{equation*}%
Therefore, by the \textit{Step 1} and the fact $\bar{x}\in \partial \Omega $%
, we just need to estimate
\begin{equation*}
\int_{0}^{\overline{\min\{\tau,t_{\mathbf{b}}\}}} \int_{\mathbb{R}^{3}}\frac{e^{-\theta |v-u|^{2}}}{|v-u| \big[\tilde{\alpha}(%
\bar{x} -(\overline{\min\{\tau,t_{\mathbf{b}}\}} -s )v,u )\big]} \mathrm{d}u
\mathrm{d}s
\end{equation*}%
We then deduce~\eqref{est:NLN} since $\tilde{\alpha} (\bar{x},v)\simeq \tilde{\alpha} (x,v)$ via
the Velocity Lemma and the fact $\overline{\min\{\tau,t_{\mathbf{b}}\}}|v|\lesssim_{\Omega}
1. $

The~\eqref{eqn: integrate alpha along trajectory 2} follows immediately by the inequality $1-e^{-x}\leq x$ and thus we deduce our lemma.
\unhide

\textit{Proof of~\eqref{integrate alpha beta} and~\eqref{integrate alpha beta small}.} Again we only prove~\eqref{integrate alpha beta}. Clearly
\begin{align*}
   & \int^{t}_0 e^{-\nu(t-s)}\dd s \int_{\mathbb{R}^3}\dd u\frac{\mathbf{k}(v,u)}{|u|^2\min\{\frac{\alpha(x-(t-s)v,u)}{|u|},\frac{\alpha(y-(t-s)v,u)}{|u|}\}^p}
 \\
   & \leq \int^{t}_0 e^{-\nu(t-s)}\dd s \int_{\mathbb{R}^3}\dd u\frac{\mathbf{k}(v,u)}{|u|^2\big(\frac{\alpha(x-(t-s)v,u)}{|u|}\big)^p}+\int^{t}_0 e^{-\nu(t-s)}\dd s \int_{\mathbb{R}^3}\dd u\frac{\mathbf{k}(v,u)}{|u|^2\big(\frac{\alpha(y-(t-s)v,u)}{|u|}\big)^p}.
\end{align*}

By Lemma 1 in \cite{GKTT} we have
\begin{align*}
   &\int^{t}_0 e^{-\nu(t-s)}\dd s \int_{\mathbb{R}^3}\dd u\frac{\mathbf{k}(v,u)}{\alpha^p(x-(t-s)v,u)}\frac{|u|^{p-2}}{|v|^{p-2}}|v|^{p-2}
  \\
   & \lesssim |v|^{p-2}\Big[ \min\{1,t\}\times \min\{\frac{1}{|v|^2 \alpha^{p-2}(x,v)},\frac{\alpha^{1/2-p/2}(x,v)}{|v|^{p-1}}\}+\frac{1}{\alpha^{p-1}(x,v)}\int_0^t e^{-\frac{C}{2}\langle v\rangle (t-s)}\dd s \Big].
\end{align*}
We bound
\[\frac{1}{||v|^2\alpha^{p-2}(x,v)}\leq \frac{\alpha^2(x,v)}{|v|^2}\frac{1}{\alpha^p(x,v)}\leq \frac{1}{\alpha^p(x,v)},\]
\[\frac{\alpha^{1/2-p/2}(x,v)}{|v|^{p-1}}\leq \frac{\alpha^{p-1}(x,v)}{|v|^{p-1}}\frac{1}{\alpha^{p/2-1/2+p-1}(x,v)}\leq \frac{1}{\alpha^{3p/2-3/2}(x,v)}\lesssim \frac{1}{\alpha^p(x,v)},\]
\[\frac{1}{\alpha^{p-1}(x,v)}\int_0^t e^{-\frac{C}{2}\langle v\rangle (t-s)}\dd s\lesssim \min\{1,O(t)\}\frac{1}{\alpha^p(x,v)},\]
where we have used $\alpha\leq 1$.

Then we conclude~\eqref{integrate alpha beta}.

\textit{Proof of~\eqref{integrate xi beta/2}}
Since $\frac{1}{\min\{\xi(x'),\xi(y')\}^{\beta/2}}\leq \frac{1}{\xi^{\beta/2}(x')}+\frac{1}{\xi^{\beta/2}(y')}$, we only need to prove
\[\int^t_{t-\tb(x,v)} \frac{e^{-C\langle v\rangle(t-s)}}{|v|\xi(x')^{\beta/2}}\lesssim \frac{1}{|v|^2\min\{\frac{\alpha(x,v)}{|v|},\frac{\alpha(y,v)}{|v|}\}^{\beta}}.\]

We split the integral as
\begin{equation*}
  \int^t_{t- \tilde{\delta} \frac{ \tilde{\alpha}(x,v)}{|v|^2}}
 + \int^{t -\tb(x,v)+ \tilde{\delta} \frac{\tilde{\alpha}(x,v)}{|v|^2}}_{t -\tb(x,v) }
 +  \int_{t -\tb(x,v)+ \tilde{\delta} \frac{\tilde{\alpha}(x,v)}{|v|^2}}^{t- \tilde{\delta} \frac{\tilde{\alpha}(x,v)}{|v|^2}}.
\end{equation*}
Similarly to~\eqref{est:int_ln_1} the first two terms are bounded by
\[2\int_0^{\tilde{\delta}\frac{\alpha^2(x,v)}{|v|}} \frac{1}{|v| |\xi|^{\beta/2}} \frac{d|\xi|}{\tilde{\delta}^{-1/2}|v|\sqrt{|\xi|}}\lesssim \frac{1}{|v|^2}.\]

Similarly to~\eqref{est:int_ln_2} the third term is bounded by
\[\int^t_{0}\frac{e^{-C\langle v\rangle(t-s)}}{|v|} \frac{1}{|\tilde{\delta}|^{\beta/2} \frac{\alpha^\beta(x,v)}{|v|^\beta}} \dd s\lesssim  \frac{1}{|v|\langle v\rangle\min\{\frac{\alpha(x,v)}{|v|},\frac{\alpha(y,v)}{|v|}\}^\beta}.\]

\textit{Proof of~\eqref{integrate beta<1}.}
Similarly to the proof of~\eqref{est:NLN} we consider three cases: $\frac{|v|}{2}\leq |u|\leq 2|v|$, $|u|\geq 2|v|$, $|u|\leq \frac{|v|}{2}$.

By a similar computation as~\eqref{int_1/a_1},\eqref{int_1/a_2} and \eqref{int_1/a_3}, we obtain
\begin{equation*}
\begin{split}
 {\int_{\frac{|v|}{2} \leq |u|\leq 2|v| }}
\lesssim & \
\iint
 \frac{e^{-\varrho|\mathbf{v}_\parallel-\mathbf{u}_\parallel|^2}}{|\mathbf{v}_\parallel-\mathbf{u}_\parallel|}\dd \mathbf{u}_\parallel
 \int_0^{ 2|v|}\frac{  \dd \mathbf{u}_3
 }{\big[ |\mathbf{u}_3|^2+|\xi (y)||v|^2\big]^{\beta/2}}
 \lesssim 1,
 \end{split}
 \end{equation*}
\begin{equation*}
\begin{split}
{\int_{|u|\geq 2|v|}}
\lesssim & \  e^{- \frac{\varrho}{2}|v|^2}
\iint
 \frac{e^{-\frac{\varrho}{2}|\mathbf{v}_\parallel-\mathbf{u}_\parallel|^2}}{|\mathbf{v}_\parallel-\mathbf{u}_\parallel|}\dd \mathbf{u}_\parallel
 \int_0^{\infty}\frac{
 e^{-\frac{\varrho}{2}|\mathbf{u}_3|^2}
 }{\big[ |\mathbf{u}_3|^2+|\xi (y)||v|^2\big]^{\beta/2}}   \dd \mathbf{u}_3
 \lesssim 1,
 \end{split}
 \end{equation*}
\Be\notag
\begin{split}
{\int_{|u|\leq \frac{|v|}{2}}}
\lesssim |v|
e^{- \frac{\varrho}{2}|v|^2}\int_{|
\mathbf{\tilde{u}}_3
|\leq \frac{1}{2}}
 \int_{|\mathbf{\tilde{u}}_\parallel|\leq \frac{1}{2}}
 \frac{
 \dd \mathbf{\tilde{u}}_\parallel \dd \mathbf{\tilde{u}}_3
}{|\mathbf{\tilde{u}}_3|^\beta} \lesssim 1.
\end{split}\Ee
Here we use $\frac{1}{|u|^\beta}\in L^1_u$. And thus we conclude the proof.

\textit{Proof of~\eqref{integrate k/v-u}}.
By H$\ddot{o}$lder inequality with $\frac{1}{3}+\frac{1}{3/2}=1$ and split $|v-u|^2=|v-u|^{4/3+\e}|v-u|^{2/3-\e}$ we have
\begin{align}
   &  \int^{t}_{t- \tb(x,v)} \int_{\mathbb{R}^{3}}\frac{e^{-C\langle v\rangle (t-s)} e^{-\varrho |v-u|^{2}}}{|v-u|^2 }\frac{1}{\min\{\alpha(x-(t-s)v,u),\alpha(y-(t-s)v,u)\}^\beta}\mathrm{d}u\mathrm{d}s \notag \\
   & \lesssim         \int^t_{t-\tb(x,v)}e^{-C \langle v\rangle (t-s)} \dd s     \Big(\int_{\mathbb{R}^3} \frac{e^{\varrho|v-u|^2}}{|v-u|^{2+\e}} \dd u\Big)^{2/3} \notag\\
   & \times \Big(\int_{\mathbb{R}^3}  \frac{ e^{-\varrho|v-u|^2}}{|v-u|^{2-\e}}\frac{1}{\min\{\alpha(x-(t-s)v,u),\alpha(y-(t-s)v,u)\}^{3\beta}} \dd u   \Big)^{1/3} \notag\\
   &  \lesssim \Big(\int^t_{t-\tb(x,v)}    e^{-C\langle v\rangle(t-s)} \int_{\mathbb{R}^3}  \frac{ e^{-\varrho|v-u|^2}}{|v-u|^{2-\e}}\frac{1}{\min\{\alpha(x-(t-s)v,u),\alpha(y-(t-s)v,u)\}^{3\beta}} \dd u \Big)^{1/3}\notag\\
    & \times \Big(\int^t_{t-\tb(x,v)} e^{-C\langle v\rangle(t-s)} \Big)^{2/3}         \label{Holder}.
\end{align}
Since $3\beta<3$ by~\eqref{integrate alpha beta} we have
\begin{align*}
  \eqref{Holder} & \lesssim    \Big( \frac{1}{\min\{\alpha(x,v),\alpha(y,v)\}^{3\beta}}\Big)^{1/3}   ,
\end{align*}
then we finish the proof.

\end{proof}

\begin{proof}[\textbf{Proof of Proposition \ref{est:K_Gamma}}]
Clearly
\[\Vert \alpha\nabla_x f^\ell\Vert_\infty\leq \Vert w_{\tilde{\theta}}\alpha\nabla_x f^\ell \Vert_\infty.\]
Applying Lemma \ref{Lemma: NLN} we bound
 \Be
 \begin{split}\label{apply_NLN}
 |  (\ref{est:px_f_3_K_2}) |+
 |(\ref{est1:px_f_5_split2}) |   \leq \frac{O(\e)}{\alpha(x,v)} \sup_{i\geq 0} \|w_{\tilde{\theta}} \alpha \nabla_x f^{\ell-1-i} \|_\infty,\
  \
  |  (\ref{int1/alpha}) |   \leq \frac{O(\sup_{i\geq 0} \| wf^{\ell-1-i} \|_\infty)}{\alpha(x,v)} \sup_{i\geq 0}\|w_{\tilde{\theta}} \alpha \nabla_x f^{\ell-1-i} \|_\infty.
\end{split} \Ee

First we prove (\ref{est:K_Gamma1}). From (\ref{est:NLN})
\Bes
\int^t_{\max\{0,t-\tb\}}
\dd s \,
 e^{-\nu(v)(t-s)}
\int_{\R^3} \dd u \, \mathbf{k}(v,u)(\ref{px_f_1_K})\lesssim \frac{1+ \| \theta\|_{C^1}}{\alpha(x,v)} (1+ \| w f^{\ell-1} \|_\infty),\\
\int^t_{\max\{0,t-\tb\}}
\dd s \,
 e^{-\nu(v)(t-s)}
\int_{\R^3} \dd u \, \mathbf{k}(v,u)
\{(\ref{px_f_2_K}) + (\ref{px_f_int_K})(\ref{px_f_4_K})\}
\lesssim
\frac{  e^{-\nu_0 t }}{\alpha(x,v)} \| w_{\tilde{\theta}}\alpha \nabla_x f^{\ell-1} \|_\infty.
\Ees
From (\ref{int1/alpha0})-(\ref{est:int1/alpha0}) and (\ref{apply_NLN})
\Be\begin{split}\notag
&\int^t_{\max\{0,t-\tb\}}
\dd s \,
 e^{-\nu(v)(t-s)}
\int_{\R^3} \dd u \, \mathbf{k}(v,u)
\{(\ref{px_f_3_K}) + (\ref{px_f_int_K})(\ref{px_f_5_K})
_{h^{\ell-2}=\Gamma(f^{\ell-3}, f^{\ell-3})}
\}\\
&\leq \ \frac{O(\sup_{i\geq 0} \| wf^{\ell-1-i} \|_\infty)}{\alpha(x,v)} \sup_{i\geq 0}\|w_{\tilde{\theta}} \alpha \nabla_x f^{\ell-1-i} \|_\infty.
\end{split}\Ee

From $(\ref{px_f_5_K_split})\leq (\ref{px_f_5_K_split})_1 + (\ref{est1:px_f_5_split2})$ and (\ref{est:xK_to_vK1}), (\ref{est:xK_to_vK2}), (\ref{apply_NLN})
\Be\begin{split}\notag
&\int^t_{\max\{0,t-\tb\}}
\dd s \,
 e^{-\nu(v)(t-s)}
\int_{\R^3} \dd u \, \mathbf{k}(v,u)(\ref{px_f_5_K})_{h^{\ell-2}=Kf^{\ell-3}(x^1-(t^1-s^1),v^1)}\\
& \leq\frac{1}{\alpha(x,v)} \times \Big\{ O(\e) \sup_{i\geq 0}  \| \alpha \nabla_x f^{\ell-1-i}\|_\infty + O(\e^{-1})  \sup_{i\geq 0} \| w f^{\ell-1-i}\|_\infty\Big\}.
\end{split}\Ee

From $|(\ref{px_f_3_Ku})|\leq |(\ref{est:px_f_3_K_1})|+ |(\ref{est:px_f_3_K_2})|$ and (\ref{apply_NLN}),
 \Be\notag
|(\ref{px_f_3_Ku})|\leq \frac{1}{\alpha(x,v)} \times \Big\{ O(\e^{-1} ) \| w f^{\ell-2}\|_\infty + O(\e) \sup_{i\geq 0} \| w_{\tilde{\theta}}\alpha \nabla_x f^{\ell-1-i} \|_\infty
\Big\}.
\Ee
Collecting terms we complete the proof of (\ref{est:K_Gamma1}).

The proof of (\ref{est:K_Gamma2}) comes from (\ref{px_f_5_split}), (\ref{est1:px_f_5_split2}), (\ref{est:xK_to_vK1}), and (\ref{est:xK_to_vK2}). We prove (\ref{est:K_Gamma3}) from (\ref{int1/alpha0})-  (\ref{est:int1/alpha0}) and (\ref{est:K_Gamma2}).\end{proof}

\section{$C^1$ estimate of tangential derivative and continuity of $C^1$ solution}
\subsection{$C^1$ estimate of tangential derivative}
In this subsection we prove~\eqref{estF_tau} in the Main Theorem. Section 3 and 4 already conclude the estimate~\eqref{estF_n}, from now on we will drop the super index in $f^{\ell}$ and only analyze the property of $\nabla_x f$.

\begin{proof}[\textbf{Proof of~\eqref{estF_tau}}]
For $x\in \Omega$ we use~\eqref{fx_x_1}-\eqref{fx_x_5} to have
\begin{align}
G(x)\nabla_x f(x,v)   &  =  \mathbf{1}_{t\geq \tb}e^{-\nu(v)\tb} G(x) \sum_{i=1,2}\nabla_x \mathbf{x}_{p^1,i}^1  \partial_{\mathbf{x}_{p^1,i}^1} f(\eta_{p^1}(\mathbf{x}_{p^1}^1),v)  \label{tang: nabla fxb} \\
   &  -\mathbf{1}_{t\geq \tb}\nu(v)\nabla_x \tb e^{-\nu(v)\tb}  G(x) f(\xb(x,v),v) \label{tang: nabla tb}\\
   &  +\mathbf{1}_{t<\tb}e^{-\nu(v)t}    G(x) \nabla_x f(x-tv,v)\label{tang: nabla f}\\
   &+\mathbf{1}_{t\geq \tb}G(x)\nabla_x \tb  e^{-\nu(v)\tb}h(x-\tb v,v) \label{tang: h}     \\
   &+\int_{\max\{0,t-\tb\}}^{t} G(x) e^{-\nu(v)(t-s)}\nabla_x h(x-(t-s)v,v)\dd s ,\label{tang: nabla h}
\end{align}
where $h=K(f)+\Gamma(f,f)$.

We focus on the estimate of~\eqref{tang: nabla h}.~\eqref{tang: nabla fxb}-\eqref{tang: h} will be estimated with~\eqref{tang: nabla h} together.

\textbf{Estimate of~\eqref{tang: nabla h} with $h=K(f)$}. Let $y=x-(t-s)v$. Rewriting $G(x)=G(x)-G(y)+G(y)$ and applying~\eqref{n(x)-n(xb)} in Lemma \ref{Lemma: nx-nxb} to $G(x)-G(y)$ we have
\begin{align}
\eqref{tang: nabla h}\mathbf{1}_{h=K(f)}   &\lesssim \int_{\max\{0,t-\tb\}}^{t}e^{-\nu(v)(t-s)}  G(y) \int_{\mathbb{R}^3}\mathbf{k}(v,u)\nabla_x f(y,u) \dd u \dd s \label{tang: k term}\\
   & + \frac{\tilde{\alpha}(x,v)}{|v|}\int_{\max\{0,t-\tb\}}^{t}   \int_{\mathbb{R}^3}\mathbf{k}(v,u)\nabla_x f(y,u) \dd u \dd s. \label{tang: alpha in k}
\end{align}
Then applying~\eqref{est:nonlocal_wo_e} in Lemma \ref{Lemma: NLN} with $y=x-(t-s)v$ and~\eqref{k_theta} we obtain
\begin{align*}
\eqref{tang: alpha in k}  & \lesssim \frac{\tilde{\alpha}(x,v)}{|v|w_{\tilde{\theta}}(v)}\int_{\max\{0,t-\tb\}}^{t}   \int_{\mathbb{R}^3}\frac{\mathbf{k}(v,u)}{\alpha(y,u)} \frac{w_{\tilde{\theta}}(v)}{w_{\tilde{\theta}}(u)} w_{\tilde{\theta}}(u)\alpha(y,u)\nabla_x f(y,u) \dd u \dd s\\
  & \lesssim \Vert w_{\tilde{\theta}}\alpha\nabla_x f\Vert_\infty \frac{\alpha(x,v)}{w_{\tilde{\theta}/2}(v)|v|}\int_{\max\{0,t-\tb\}}^{t}   \int_{\mathbb{R}^3}\frac{\mathbf{k}_{\tilde{\varrho}}(v,u)}{\alpha(y,u)} \dd u \dd s\lesssim \frac{\Vert w_{\tilde{\theta}}\alpha \nabla_x f\Vert_\infty}{w_{\tilde{\theta}/2}(v)|v|}.
  \end{align*}

We focus on~\eqref{tang: k term}. We further expand $G(y) \nabla_x f(y,u)$ along $u$:
\begin{align}
  \eqref{tang: k term} &\lesssim \int_{\max\{0,t-\tb\}}^{t}e^{-\nu(v)(t-s)} \int_{\mathbb{R}^3}\mathbf{k}(v,u) G(y) \sum_{i=1,2}   \nabla_x \mathbf{x}_{p^1(u),i}^1  \partial_{\mathbf{x}_{p^1(u),i}^1} f(\eta_{p^1(u)}(\mathbf{x}_{p^1(u)}^1),u) \dd u \dd s  \label{tang: k nabla f}\\
   &    +    \int_{\max\{0,t-\tb\}}^{t}e^{-\nu(v)(t-s)} \int_{\mathbb{R}^3}\mathbf{k}(v,u) G(y) \nabla_x \tb(y,u) f(y-\tb(y,u)u,u) \dd u \dd s        \label{tang: k nabla tb}\\
   &    +\int_{\max\{0,t-\tb\}}^{t}e^{-\nu(v)(t-s)}e^{-\nu(u)s} \int_{\mathbb{R}^3}\mathbf{k}(v,u) G(y) \nabla_x f(y-su,u) \dd u \dd s \label{tang: k s}   \\
   &+  \int_{\max\{0,t-\tb\}}^{t}e^{-\nu(v)(t-s)} \int_{\mathbb{R}^3}\mathbf{k}(v,u) G(y) \nabla_x \tb(y,u)e^{-\nu(u)\tb(y,u)}h(y-\tb(y,u)u,u) \dd u \dd s   \label{tang: k h}\\
   &+   \int_{\max\{0,t-\tb\}}^{t} \dd s\int_{\mathbb{R}^3}\mathbf{k}(v,u)\dd u G(y)\int_{\max\{0,s-\tb(y,u)\}}^s \dd s' \int_{\mathbb{R}^3}\mathbf{k}(u,u')\nabla_x f(y-(s-s')u,u')    \label{tang: kk}\\
   &+ \int_{\max\{0,t-\tb\}}^{t} \dd s\int_{\mathbb{R}^3}\mathbf{k}(v,u)\dd u G(y)\int_{\max\{0,s-\tb(y,u)\}}^s \dd s' \nabla_x \Gamma(f,f)(y-(s-s')u,u'). \label{tang: k gamma}
\end{align}
In~\eqref{tang: k nabla f} we denoted $\xb(x-(t-s)v,u)=\eta_{p^1(u)}(\mathbf{x}_{p^1(u)}^1).$

Then we estimate~\eqref{tang: nabla fxb}-\eqref{tang: h} together with~\eqref{tang: k nabla   f}-\eqref{tang: k h}.

First we estimate~\eqref{tang: nabla fxb} and \eqref{tang: k nabla f}. We start from~\eqref{tang: nabla fxb}. From section 3 we have
\[ |\partial_{\mathbf{x}_{p^1,i}^1} f(\eta_{p^1}(\mathbf{x}_{p^1}^1),v)|=\eqref{fBD_x}+\eqref{fBD_xr}\lesssim \eqref{est:fBD_xr}+\eqref{IBP_v}+\eqref{est:fBD_x_e2}+\eqref{est:fBD_x_e1}\lesssim \frac{M_W(\xb,v)}{\sqrt{\mu(v)}}\Vert T_W-T_0\Vert_{C^1}\Vert wf\Vert_\infty. \]
Thus by~\eqref{nx nabla xb}
\begin{equation}\label{tang: bound for nabla fxb}
\eqref{tang: nabla fxb}\lesssim \frac{\Vert T_W-T_0\Vert_{C^1}\Vert wf\Vert_\infty}{w_{\tilde{\theta}/2}(v)|v|}.
\end{equation}

For~\eqref{tang: k nabla   f} similarly by~\eqref{nx nabla xb} we apply~\eqref{k_theta} and~\eqref{integrate k u} with $c=1$ to have
\begin{align}
  \eqref{tang: k nabla f} & \lesssim \frac{\Vert T_W-T_0\Vert_{C^1}\Vert wf\Vert_\infty}{w_{\tilde{\theta}/2}(v)} \int^t_{\max\{0,t-\tb\}}e^{-\nu(v)(t-s)} \int_{\mathbb{R}^3}\mathbf{k}(v,u) \frac{w_{\tilde{\theta}/2}(v)}{w_{\tilde{\theta}/2}(u)}\frac{1}{|u|}\dd u\dd s  \lesssim \frac{\Vert wf\Vert_\infty}{w_{\tilde{\theta}/2}(v)|v|} .\label{tang: k bound for nabla f}
\end{align}

Then we estimate~\eqref{tang: nabla tb} and~\eqref{tang: k nabla tb}. We start from~\eqref{tang: nabla tb}. By~\eqref{nx nabla tb} we conclude
\begin{equation}\label{tang: bound for nabla tb}
\eqref{tang: nabla tb}\lesssim \frac{\Vert w_{\tilde{\theta}/2}f\Vert_\infty}{w_{\tilde{\theta}/2}(v)|v|}\lesssim \frac{\Vert wf\Vert_\infty}{w_{\tilde{\theta}/2}(v)|v|}.
\end{equation}

For~\eqref{tang: k nabla tb} similarly by~\eqref{nx nabla tb} we apply~\eqref{integrate k u} with $c=1$ to have
\begin{align}
  \eqref{tang: k nabla tb} &\lesssim \frac{\Vert wf\Vert_\infty}{w_{\tilde{\theta}/2}(v)}\int^t_{\max\{0,t-\tb\}} \int_{\mathbb{R}^3}\mathbf{k}(v,u)\frac{w_{\tilde{\theta}/2}(v)}{w_{\tilde{\theta}/2}(u)}\frac{1}{|u|} \dd u\dd s    \lesssim \frac{\Vert wf\Vert_\infty}{w_{\tilde{\theta}/2}(v)|v|}. \label{tang: k bound for nabla tb}
\end{align}

Then we estimate~\eqref{tang: nabla f} and~\eqref{tang: k s}. For~\eqref{tang: nabla f} we apply~\eqref{Gf bdd} to have
\begin{align}
  \eqref{tang: nabla f}  & \lesssim e^{-t}\big[\frac{\Vert w_{\tilde{\theta}/2}|v|\nabla_\parallel f(x,v)\Vert_\infty}{w_{\tilde{\theta}/2}(v)|v|}+\frac{\Vert w_{\tilde{\theta}}\alpha\nabla_x f\Vert_\infty}{w_{\tilde{\theta}/2}(v)|v|}\big]\label{tang: bound for nabla f}.
\end{align}

Similarly for~\eqref{tang: k s} applying~\eqref{Gf bdd} and~\eqref{integrate k u} with $c=1$ we have
\begin{align}
  \eqref{tang: k s} &\lesssim e^{-t} \frac{\Vert w_{\tilde{\theta}/2}|v|\nabla_\parallel f\Vert_\infty+\Vert w_{\tilde{\theta}}\alpha\nabla_x f\Vert_\infty }{w_{\tilde{\theta}/2}(v)}\int^t_{\max\{0,t-\tb\}} \int_{\mathbb{R}^3}\mathbf{k}(v,u)\frac{w_{\tilde{\theta}/2}(v)}{w_{\tilde{\theta}/2}(u)}\frac{1}{|u|}\dd u\dd s  \notag\\
   &\lesssim e^{-t}\big[\frac{\Vert w_{\tilde{\theta}/2}|v|\nabla_\parallel f(x,v)\Vert_\infty}{|v|}+\frac{\Vert w_{\tilde{\theta}}  \alpha\nabla_x f\Vert_\infty}{w_{\tilde{\theta}/2}(v)|v|}\big]. \label{tang: k bound for ks}
\end{align}

Then we estimate~\eqref{tang: h} and~\eqref{tang: k h}. For~\eqref{tang: h}, by~\eqref{nx nabla tb} and~\eqref{h_K} we have
\begin{align}
  \eqref{tang: h} &\lesssim \frac{ K(f)+\Gamma(f,f)}{|v|}  \lesssim \frac{(\Vert wf\Vert_\infty+1)}{|v|}\int_{\mathbb{R}^3}|\mathbf{k}(v,u)f(x-\tb v,u)|\dd u\notag\\
  &\lesssim \frac{\Vert w_{\tilde{\theta}}f\Vert_\infty}{w_{\tilde{\theta}/2}(v)|v|}\int_{\mathbb{R}^3}\mathbf{k}_\varrho(v,u)\frac{w_{\tilde{\theta}/2}(v)}{w_{\tilde{\theta}/2}(u)}\dd u\lesssim \frac{\Vert wf\Vert_\infty}{w_{\tilde{\theta}/2}(v)|v|}.\label{tang: bound for h}
\end{align}

For~\eqref{tang: k h} similarly by~\eqref{nx nabla tb} and~\eqref{integrate k u} with $c=1$ we have
\begin{align}
  \eqref{tang: k h} &\lesssim \frac{\Vert wf\Vert_\infty}{w_{\tilde{\theta}/2}(v)} \int^t_{\max\{0,t-\tb\}} e^{-\nu(v)(t-s)}\int_{\mathbb{R}^3}\mathbf{k}(v,u)\frac{w_{\tilde{\theta}/2}(v)}{w_{\tilde{\theta}/2}(u)}\frac{1}{|u|} \lesssim \frac{\Vert wf\Vert_\infty}{w_{\tilde{\theta}/2}(v)|v|} \label{tang: k bound for kh}.
\end{align}

Last we estimate~\eqref{tang: kk}. This estimate is the most delicate one. We apply the decomposition~\eqref{px_f_3_Ku} to $\dd s'$.

When $s'>s-\e$, we apply~\eqref{Gf bdd} in Lemma \ref{Lemma: nx-nxb} to have
\begin{align}
  &\eqref{tang: kk}\mathbf{1}_{s'>s-\e}\notag\\
   & \lesssim \frac{\Vert |v|\nabla_\parallel f\Vert_\infty + \Vert w_{\tilde{\theta}} \alpha\nabla_x f\Vert_\infty }{w_{\tilde{\theta}/2}(v)} \\
   &\times \int^t_{\max\{0,t-\tb\}} \int_{\mathbb{R}^3}\mathbf{k}(v,u)\frac{w_{\tilde{\theta}/2}(v)}{w_{\tilde{\theta}/2}(u)}\dd u \dd s\int^s_{s-\e} \dd s' \int_{\mathbb{R}^3}\mathbf{k}(u,u')\frac{w_{\tilde{\theta}/2}(u)}{w_{\tilde{\theta}/2}(u')} \frac{|u|}{|u'|}\frac{|v|}{|u|} \notag\\
   &\lesssim o(1)\frac{\Vert w_{\tilde{\theta}} \alpha\nabla_x f\Vert_\infty+\Vert w_{\tilde{\theta}/2} |v|\nabla_\parallel f\Vert_\infty}{w_{\tilde{\theta}/2}(v)|v|} ,\label{tang: k bound for kk1}
\end{align}
where we have applied~\eqref{integrate k u} twice with $c=1$.

On the other hand when $s'<s-\e$, we exchange $\nabla_x$ for $\nabla_u$:
\[\nabla_x f(x-(t-s)v- (s-s')u, u^\prime)
=\frac{-1}{s-s'}\nabla_u[f(x-(t-s)v- (s-s')u, u^\prime)].\]
Then we perform an integration by parts with respect to $\dd u$ and obtain
\Be\begin{split}\label{tang: IBP u}
\eqref{tang: kk} &= \int^t_{\max\{0, t-\tb\}}
 \dd s   \,   
 e^{-\nu(v) (t-s)}
  \int_{\R^3} \dd u
  \int^s_{\max\{0, s-\tb(y, u)\}} \dd s'   \,
  e^{-\nu(u) (s-s')} \frac{\mathbf{1}_{s'\leq s-\e}}{s-s'}  \\
 & \times
  \int_{\R^3} \dd u^\prime    G(y)\,
\nabla_u[ \mathbf{k} (v,u) \mathbf{k} (u,u^\prime) ]f(y- (s-s')u, u^\prime) \\
-&\int^t_{\max\{0, t-\tb\}}
 \dd s   \,   
 e^{-\nu(v) (t-s)}
  \int_{\R^3} \dd u
  \int^s_{\max\{0, s-\tb(y, u)\}} \dd s'   \,
 \nabla_u \nu(u) e^{-\nu(u) (s-s')}  {\mathbf{1}_{s'\leq s-\e}}  \\
 & \times
  \int_{\R^3} \dd u^\prime    G(y)\,
\mathbf{k} (v,u) \mathbf{k} (u,u^\prime) f(y- (s-s')u, u^\prime)\\
+& \int^t_{\max\{0, t-\tb\}}
 \dd s   \,   
 e^{-\nu(v) (t-s)}
  \int_{\R^3} \dd u
  \mathbf{1}_{s \geq \tb(y,u)}
  e^{-\nu(u) \tb(y,u)} \frac{\mathbf{1}_{\tb(y,u) \geq \e}}{\tb(y,u)}\nabla_u \tb(y,u)  \\
 & \times
  \int_{\R^3} \dd u^\prime   G(y) \,
 \mathbf{k} (v,u) \mathbf{k} (u,u^\prime) f(y- \tb(y,u)u, u^\prime).
 \end{split}\Ee
We bound $|G(y)|\leq 1$. Then applying~\eqref{k_varrho} and~\eqref{nablav nu} the first and second terms of~\eqref{tang: IBP u} are bounded by
\begin{align*}
   & \frac{O(\e^{-1})\Vert w_{\tilde{\theta}}f\Vert_\infty}{w_{\tilde{\theta}/2}(v)}\int^t_{\max\{0,t-\tb\}} \dd s e^{-\nu(v)(t-s)}\int_{\mathbb{R}^3}\int^s_{\max\{0,s-\tb(y,u)\}}\dd s' e^{-\nu(u)(s-s')} \\
   &  \times \int_{\mathbb{R}^3}\frac{ \mathbf{k}(v,u)\langle u\rangle^2}{|v-u|} \frac{w_{\tilde{\theta}/2}(v)}{w_{\tilde{\theta}}(u)}  \frac{|u|}{|u|}\frac{\mathbf{k}(u,u')}{|u-u'|}\frac{w_{\tilde{\theta}}(u)}{w_{\tilde{\theta}}(u')}    \\
   & \lesssim \frac{O(\e^{-1})\Vert wf\Vert_\infty}{w_{\tilde{\theta}/2}(v)}\int^t_{\max\{0,t-\tb\}} \dd s e^{-\nu(v)(t-s)}\int_{\mathbb{R}^3}\int^s_{\max\{0,s-\tb(y,u)\}}\dd s' e^{-\nu(u)(s-s')} \\
   &  \times \int_{\mathbb{R}^3}\frac{ \mathbf{k}(v,u)}{|v-u|} \frac{w_{\tilde{\theta}/2}(v)}{w_{\tilde{\theta}/2}(u)}  \frac{1}{|u|}\frac{\mathbf{k_{\tilde{\varrho
}}}(u,u')}{|u-u'|}    \lesssim \frac{O(\e^{-1})\Vert wf\Vert_\infty}{w_{\tilde{\theta}/2}(v)|v|},
\end{align*}
where we have used~\eqref{k_theta}, $\langle u\rangle^2 |u|w^{-1}_{\tilde{\theta}}(u)\lesssim w^{-1}_{\tilde{\theta}/2}(u)$ and~\eqref{integrate k u} with $c=1$.

For the third term we apply~\eqref{nabla_tbxb} and~\eqref{n(x)-n(xb)} to have
\begin{align}
   & G(y)\frac{\nabla_u \tb(y,u)}{\tb(y,u)}=\frac{G(y)n(\xb(y,u)) }{n(\xb(y,u))\cdot u} \notag\\
   & =\frac{[G(y)-G(\xb(y,u))]n(\xb(y,u))}{n(\xb(y,u))\cdot u}+\frac{G(\xb(y,u))n(\xb(y,u))}{n(\xb(y,u))\cdot u}\notag\\
   &\lesssim       \frac{\tilde{\alpha}(y,u)}{|n(\xb(y,u))\cdot u||u|}\lesssim \frac{1}{|u|},\label{G nabla_u tb}
\end{align}
Thus applying~\eqref{integrate k u} with $c=1$ the third term is bounded by
\begin{align*}
& \frac{\Vert w_{\tilde{\theta}/2}f\Vert_\infty}{w_{\tilde{\theta}/2}(v)}   \int^t_{\max\{0,t-\tb\}} \dd s e^{-\nu(v)(t-s)}\\
&\times \int_{\mathbb{R}^3}\dd u \mathbf{k}(v,u)\frac{w_{\tilde{\theta}/2}(v)}{w_{\tilde{\theta}/2}(u)}\frac{1}{|u|} e^{-\nu \tb(y,u)}\int_{\mathbb{R}^3} \mathbf{k}(u,u')  \lesssim \frac{\Vert wf\Vert_\infty}{w_{\tilde{\theta}/2}(v)|v|}.
\end{align*}

Therefore, we conclude
\begin{equation}\label{tang: k bound for kk2}
\eqref{tang: kk}\mathbf{1}_{s'<s-\e}\lesssim O(\e^{-1})\frac{\Vert wf\Vert_\infty}{w_{\tilde{\theta}/2}(v)|v|}.
\end{equation}

We estimate~\eqref{tang: k gamma} together with $\eqref{tang: nabla h}\mathbf{1}_{h=\Gamma}$. We apply~\eqref{G Gamma},~\eqref{Gf bdd} and~\eqref{integrate k/v-u} with $c=1$ to have
\begin{align}
  &\eqref{tang: nabla h}\mathbf{1}_{h=\Gamma} \notag\\
   & \lesssim    \int^t_{\max\{t-\tb\}}e^{-\nu(v)(t-s)}\Vert wf\Vert_\infty[G(x)\nabla_x f(x-(t-s)v)+\int_{\mathbb{R}^3}\mathbf{k}_\varrho (v,u)|G(x)\nabla_x f(x-(t-s)v,u)|] \notag\\
        &\lesssim\frac{\Vert wf\Vert_\infty[\Vert w_{\tilde{\theta}/2}|v|\nabla_\parallel f\Vert_\infty+\Vert w_{\tilde{\theta}}\alpha\nabla_x f\Vert_\infty]}{w_{\tilde{\theta}/2}(v)}\notag\\
   & \times \Big[\frac{1}{|v|}+ \int^t_{\max\{0,t-\tb\}} e^{-\nu(v)(t-s)} \int_{\mathbb{R}^3}\mathbf{k}_\varrho(v,u)\frac{w_{\tilde{\theta}/2}(v)}{w_{\tilde{\theta}/2}(u)}\frac{1}{|u|}\dd u \dd s \Big] \notag  \\
   &\lesssim \Vert wf\Vert_\infty \frac{\Vert w_{\tilde{\theta}/2}|v|\nabla_\parallel f\Vert_\infty}{|v|}+\frac{\Vert wf\Vert_\infty \Vert w_{\tilde{\theta}}\alpha\nabla_x f\Vert_\infty}{|v|}\notag\\
   &\lesssim \frac{\Vert wf\Vert_\infty[\Vert w_{\tilde{\theta}/2}|v|\nabla_\parallel f\Vert_\infty+\Vert w_{\tilde{\theta}}\alpha\nabla_x f\Vert_\infty]}{w_{\tilde{\theta}/2}(v)|v|} \label{tang: bound for nabla h gamma}.
\end{align}

Then similarly we have
\begin{align}
  \eqref{tang: k gamma} &\lesssim \int^t_{\max\{t-\tb\}}e^{-\nu(v)(t-s)} \int_{\mathbb{R}^3}\mathbf{k}_\varrho(v,u) \frac{\Vert wf\Vert_\infty[\Vert w_{\tilde{\theta}/2}|v|\nabla_\parallel f\Vert_\infty+\Vert w_{\tilde{\theta}}\alpha\nabla_x f\Vert_\infty]}{w_{\tilde{\theta}/2}(u)|u|}  \notag\\
   & \lesssim \frac{\Vert wf\Vert_\infty[\Vert w_{\tilde{\theta}/2}|v|\nabla_\parallel f\Vert_\infty+\Vert w_{\tilde{\theta}}\alpha\nabla_x f\Vert_\infty]}{w_{\tilde{\theta}/2}(v)} \int^t_{\max\{t-\tb\}}e^{-\nu(v)(t-s)} \int_{\mathbb{R}^3}\mathbf{k}_\varrho(v,u)\frac{w_{\tilde{\theta}/2}(v)}{w_{\tilde{\theta}/2}(u)}\frac{1}{|u|} \notag\\
   & \lesssim \frac{\Vert wf\Vert_\infty[\Vert w_{\tilde{\theta}/2}|v|\nabla_\parallel f\Vert_\infty+\Vert w_{\tilde{\theta}}\alpha\nabla_x f\Vert_\infty]}{w_{\tilde{\theta}/2}(v)|v|}.\label{tang: bound for k gamma}
\end{align}

Then combining~\eqref{tang: k bound for   nabla f},\eqref{tang: k bound for nabla tb},\eqref{tang: k bound for ks},\eqref{tang: k bound for kh},\eqref{tang: k bound for kk1},\eqref{tang: k bound for kk2},\eqref{tang: bound for    nabla h gamma} and~\eqref{tang: bound for k    gamma} we conclude
\begin{equation}\label{tang: bound for nabla h}
\eqref{tang: nabla h}\lesssim \big[o(1)+e^{-t}\big]\Big[\frac{\Vert wf\Vert_\infty}{w_{\tilde{\theta}/2}(v)|v|}+\frac{\Vert w_{\tilde{\theta}/2}|v|\nabla_\parallel f\Vert_\infty}{w_{\tilde{\theta}/2}(v)|v|}\Big]+\frac{\Vert w_{\tilde{\theta}}\alpha\nabla_x f\Vert_\infty}{w_{\tilde{\theta}/2}(v)|v|}.
\end{equation}

Combining~\eqref{tang: bound for nabla f},\eqref{tang: bound for nabla fxb},\eqref{tang: bound for nabla h} and~\eqref{tang: bound for nabla tb} we conclude
\begin{equation}\label{nablaparallel f bound}
|\nabla_\parallel f|\lesssim O(\e^{-1})\Big[\frac{\Vert wf\Vert_\infty}{|v|}+\frac{\Vert w_{\tilde{\theta}}\alpha\nabla_x f\Vert_\infty}{w_{\tilde{\theta}/2}|v|}\Big]+ \big[\Vert wf\Vert_\infty+e^{-t}\big]\frac{\Vert w_{\tilde{\theta}/2}|v|\nabla_\parallel f\Vert_\infty}{w_{\tilde{\theta}/2}|v|}.
\end{equation}
Then from $t\gg 1$ and $\Vert wf\Vert_\infty\ll 1$ in the \textbf{Existence Theorem} we conclude~\eqref{estF_tau}.

\subsection{Continuity of $C^1$ solution}
It remains to prove the continuity of $\nabla_x f$. The continuity of $G(x)\nabla_x f$ will follow directly from the continuity of $\nabla_x f$. We only need to prove $\nabla_x f$ is continuous at $t=\tb(x,v).$ When $t=\tb(x,v)$,~\eqref{tang: nabla f} reads
\[e^{-\nu \tb}\nabla_x f(x-\tb v,v). \]
At the boundary $x-\tb v=\xb(x,v)=\eta_{p^1}(\mathbf{x}_{p^1}^1)$, we use the notation~\eqref{x0v0},\eqref{fBD_x1} and decompose the spatial derivative as
\[v\cdot \nabla_x f= \sum_{i=1}^2\mathbf{v}^1_{p^1,i}\frac{\partial_{\mathbf{x}^1_{p^1,i}}f}{\sqrt{g_{p^1,ii}(\mathbf{x}_{p^1})}}+ \mathbf{v}^1_{p^1,3}\frac{\partial_{\mathbf{x}^1_{p^1,3}}f}{\sqrt{g_{p^1,33}(\mathbf{x}_{p^1})}}.\]
Then from the equation~\eqref{f}, we derive
\begin{equation}\label{partial n f}
\frac{\partial_{\mathbf{x}^1_{p^1,3}}f}{\sqrt{g_{p^1,33}(\mathbf{x}_{p^1})}}=\underbrace{-\sum_{i=1}^2\frac{\mathbf{v}^1_{p^1,i}\partial_{\mathbf{x}^1_{p^1,i}}f}{\mathbf{v}^1_{p^1,3}\sqrt{g_{p^1,ii}(\mathbf{x}_{p^1})}}}_{\eqref{partial n f}_1}-\underbrace{\frac{\nu(v)f}{\mathbf{v}^1_{p^1,3}}}_{\eqref{partial n f}_2}+\underbrace{\frac{K(f)+\Gamma(f,f)}{\mathbf{v}^1_{p^1,3}}}_{\eqref{partial n f}_3}.
\end{equation}
Plugging these terms back the spatial derivative $\nabla_x f$, we derive that
\begin{align}
  \nabla_x f & =\nabla_x f T T^t \notag\\
   &=\Big(\frac{\partial_{\mathbf{x}^1_{p^1,1}}f}{\sqrt{g_{p^1,11}(\mathbf{x}_{p^1})}},\frac{\partial_{\mathbf{x}^1_{p^1,2}}f}{\sqrt{g_{p^1,22}(\mathbf{x}_{p^1})}},\eqref{partial n f} \Big) \left(
                  \begin{array}{c}
                    \frac{\partial_1 \eta_{p^1}(\mathbf{x}_{p^1}^1)}{\sqrt{g_{p^1,11}(\mathbf{x}_{p^1}^1)}} \\
                     \frac{\partial_2 \eta_{p^1}(\mathbf{x}_{p^1}^1)}{\sqrt{g_{p^1,22}(\mathbf{x}_{p^1}^1)}}  \\
                     \frac{\partial_3 \eta_{p^1}(\mathbf{x}_{p^1}^1)}{\sqrt{g_{p^1,33}(\mathbf{x}_{p^1}^1)}}  \\
                  \end{array}
                \right)\label{nabla f expand}.
\end{align}
Then the contribution of$~\eqref{partial n f}_1$ in~\eqref{nabla f expand} is
\[\Big(\frac{\partial_{\mathbf{x}^1_{p^1,1}}f}{\sqrt{g_{p^1,11}(\mathbf{x}_{p^1})}},\frac{\partial_{\mathbf{x}^1_{p^1,2}}f}{\sqrt{g_{p^1,22}(\mathbf{x}_{p^1})}},\eqref{partial n f}_1  \Big)\left(
                  \begin{array}{c}
                    \frac{\partial_1 \eta_{p^1}(\mathbf{x}_{p^1}^1)}{\sqrt{g_{p^1,11}(\mathbf{x}_{p^1}^1)}} \\
                     \frac{\partial_2 \eta_{p^1}(\mathbf{x}_{p^1}^1)}{\sqrt{g_{p^1,22}(\mathbf{x}_{p^1}^1)}}  \\
                     \frac{\partial_3 \eta_{p^1}(\mathbf{x}_{p^1}^1)}{\sqrt{g_{p^1,33}(\mathbf{x}_{p^1}^1)}}  \\
                  \end{array}
                \right) \]
\[=\left(
                         \begin{array}{c}
                           \sum_{i=1}^2  \Big[\frac{\partial_{\mathbf{x}^1_{p^1,i}}f}{\sqrt{g_{p^1,ii}(\mathbf{x}_{p^1})}} \frac{\partial_i \eta_{p^1}(\mathbf{x}_{p^1}^1)}{\sqrt{g_{p^1,ii}(\mathbf{x}_{p^1}^1)}}-\frac{\mathbf{v}_{p^1,i}^1 \partial_{\mathbf{x}_{p^1,i}^1}f}{\mathbf{v}_{p^1,3}^1\sqrt{g_{p^1,ii}(\mathbf{x}_{p^1})}}\frac{\partial_3 \eta_{p^1}(\mathbf{x}_{p^1}^1)}{\sqrt{g_{p^1,33}(\mathbf{x}_{p^1}^1)}} \Big]\cdot e_1  \\
                           \sum_{i=1}^2  \Big[\frac{\partial_{\mathbf{x}^1_{p^1,i}}f}{\sqrt{g_{p^1,ii}(\mathbf{x}_{p^1})}} \frac{\partial_i \eta_{p^1}(\mathbf{x}_{p^1}^1)}{\sqrt{g_{p^1,ii}(\mathbf{x}_{p^1}^1)}}-\frac{\mathbf{v}_{p^1,i}^1 \partial_{\mathbf{x}_{p^1,i}^1}f}{\mathbf{v}_{p^1,3}^1\sqrt{g_{p^1,ii}(\mathbf{x}_{p^1})}}\frac{\partial_3 \eta_{p^1}(\mathbf{x}_{p^1}^1)}{\sqrt{g_{p^1,33}(\mathbf{x}_{p^1}^1)}} \Big]\cdot e_2  \\
                            \sum_{i=1}^2  \Big[\frac{\partial_{\mathbf{x}^1_{p^1,i}}f}{\sqrt{g_{p^1,ii}(\mathbf{x}_{p^1})}} \frac{\partial_i \eta_{p^1}(\mathbf{x}_{p^1}^1)}{\sqrt{g_{p^1,ii}(\mathbf{x}_{p^1}^1)}}-\frac{\mathbf{v}_{p^1,i}^1 \partial_{\mathbf{x}_{p^1,i}^1}f}{\mathbf{v}_{p^1,3}^1\sqrt{g_{p^1,ii}(\mathbf{x}_{p^1})}}\frac{\partial_3 \eta_{p^1}(\mathbf{x}_{p^1}^1)}{\sqrt{g_{p^1,33}(\mathbf{x}_{p^1}^1)}} \Big]\cdot e_3 \\
                         \end{array}
                       \right)^t,
                     \]
which is exactly the same as $\sum_{i=1,2}\nabla_x \mathbf{x}_{p^1,i}^1  \partial_{\mathbf{x}_{p^1,i}^1} f(\eta_{p^1}(\mathbf{x}_{p^1}^1),v)$ in~\eqref{tang: nabla fxb} by applying~\eqref{xi deri xbp}.

The contribution of$~\eqref{partial n f}_2$ and $\eqref{partial n f}_3$ in~\eqref{nabla f expand} are
\begin{align}
( \eqref{partial n f}_2+\eqref{partial n f}_3)\cdot \frac{\partial_3 \eta_{p^1}(\mathbf{x}_{p^1}^1)}{\sqrt{g_{p^1,33}(\mathbf{x}_{p^1}^1)}}  &=-\left(
                                                                             \begin{array}{c}
                                                                              \frac{\nu(v)f}{\mathbf{v}_{p^1,3}^1}\frac{\partial_3 \eta_{p^1}(\mathbf{x}_{p^1}^1)}{\sqrt{g_{p^1,33}(\mathbf{x}_{p^1}^1)}}\cdot e_1  \\
                                                                               \frac{\nu(v)f}{\mathbf{v}_{p^1,3}^1}\frac{\partial_3 \eta_{p^1}(\mathbf{x}_{p^1}^1)}{\sqrt{g_{p^1,33}(\mathbf{x}_{p^1}^1)}}\cdot e_2   \\
                                                                                \frac{\nu(v)f}{\mathbf{v}_{p^1,3}^1}\frac{\partial_3 \eta_{p^1}(\mathbf{x}_{p^1}^1)}{\sqrt{g_{p^1,33}(\mathbf{x}_{p^1}^1)}}\cdot e_3  \\
                                                                             \end{array}
                                                                           \right)^t +\left(
                                                                             \begin{array}{c}
                                                                              \frac{K(f)+\Gamma(f,f)}{\mathbf{v}_{p^1,3}^1}\frac{\partial_3 \eta_{p^1}(\mathbf{x}_{p^1}^1)}{\sqrt{g_{p^1,33}(\mathbf{x}_{p^1}^1)}}\cdot e_1  \\
                                                                               \frac{K(f)+\Gamma(f,f)}{\mathbf{v}_{p^1,3}^1}\frac{\partial_3 \eta_{p^1}(\mathbf{x}_{p^1}^1)}{\sqrt{g_{p^1,33}(\mathbf{x}_{p^1}^1)}}\cdot e_2   \\
                                                                                \frac{K(f)+\Gamma(f,f)}{\mathbf{v}_{p^1,3}^1}\frac{\partial_3 \eta_{p^1}(\mathbf{x}_{p^1}^1)}{\sqrt{g_{p^1,33}(\mathbf{x}_{p^1}^1)}}\cdot e_3  \\
                                                                             \end{array}
                                                                           \right)^t, \notag
\end{align}
which is exactly the same as $-\nu(v)\nabla_x \tb f$ in~\eqref{tang: nabla f} and $\nabla_x \tb h$ in~\eqref{tang: h} by applying~\eqref{nabla_tbxb}.

Thus $\nabla_x f$ is continuous at $t=\tb$.

\end{proof}

\section{$C^1_v$ estimate}
In this section we prove the $C^1_v$ estimate, which is~\eqref{estF_v} in Main Theorem.

\begin{proof}[\textbf{Proof of~\eqref{estF_v}}]
We take the $v$ derivative to~\eqref{trajectory} and have
\begin{align}
  \nabla_v f(x,v) &= \mathbf{1}_{t\geq \tb} e^{-\nu \tb} \nabla_v [f(\xb,v)] \label{nablav: 1}\\
   & -\mathbf{1}_{t\geq \tb}  \nabla_v \tb(x,v)  e^{-\nu \tb} f(\xb,v) \label{nablav: 2}\\
   & -\mathbf{1}_{t\geq \tb}  \nabla_v \nu(v) e^{-\nu \tb} f(\xb,v) \label{nablav: 3}\\
   & + \mathbf{1}_{t\leq \tb}e^{-\nu t}   \nabla_v [f(x-tv,v)]\label{nablav: 4}\\
   & -\mathbf{1}_{t\leq \tb}  \nabla_v \nu(v) e^{-\nu t}f(x-tv,v)\label{nablav: 5}\\
   &-\int^t_{\max\{0,t-\tb\}}  \nabla_v \nu(v) e^{-\nu (t-s)} h(x-(t-s)v,v) \dd s\label{nablav: 6}\\
   &+\int^t_{\max\{0,t-\tb\}} e^{-\nu (t-s)} \nabla_v [ h(x-(t-s)v,v) ]\dd s\label{nablav: 7}\\
   &- \mathbf{1}_{t\geq \tb} \nabla_v \tb(x,v) e^{-\nu \tb}h(x-\tb v,v). \label{nablav: 8}
\end{align}

First we estimate~\eqref{nablav: 1} and~\eqref{nablav: 7}, which are the most delicate.

For~\eqref{nablav: 1}, we apply the boundary condition~\eqref{eqn: diffuse for f} to obtain
\begin{align}
   \eqref{nablav:   1} & \lesssim    \nabla_v \Big[\frac{M_W(\xb,v)}{\sqrt{\mu(v)}}  \int_{\mathbf{v}^1_{3}>0}   f(\xb,T^t_{\mathbf{x}_p^1}\mathbf{v}^1)\sqrt{\mu(\mathbf{v}^1)}  \mathbf{v}_3^1 \dd \mathbf{v}^1 + r(\xb,v)   \Big]   \notag\\
     & \lesssim   \frac{\Vert wf\Vert_\infty \Vert T_W-T_0\Vert_{C^1}+\Vert |v|^2 \nabla_v r\Vert_\infty}{|v|^2}  \notag\\
     & +\frac{|v|^2 M_W(\xb,v) }{\sqrt{\mu(v)}|v|^2}  \int_{\mathbf{v}^1_{3}>0}   \underbrace{\nabla_v [f(\xb,T^t_{\mathbf{x}_p^1}\mathbf{v}^1)]}_{\eqref{nablav: estimate of 1}_*}\sqrt{\mu(\mathbf{v}^1)}  \mathbf{v}_3^1 \dd \mathbf{v}^1 .
     \label{nablav: estimate of 1}
\end{align}

Applying~\eqref{nablav tb xb} and using the chain rule we have
\begin{align*}
  \eqref{nablav: estimate of 1}_* &\lesssim     \nabla_v \xb  \nabla_x f(\xb,T^t_{\mathbf{x}_p^1}\mathbf{v}^1)+  \nabla_v T^t_{\mathbf{x}_p^1}\mathbf{v}^1 \nabla_v f(\xb,T^t_{\mathbf{x}_p^1}\mathbf{v}^1)  \\
   &  \lesssim   \underbrace{\frac{\nabla_x f(\xb,T^t_{\mathbf{x}_p^1}\mathbf{v}^1)}{|v|}}_{\eqref{nablav: estimate of 1}_1}+  \underbrace{\nabla_v T^t_{\mathbf{x}_p^1}\mathbf{v}^1 \nabla_v f(\xb,T^t_{\mathbf{x}_p^1}\mathbf{v}^1)}_{\eqref{nablav: estimate of 1}_2}.
\end{align*}
Then we estimate the contribution of both terms in~\eqref{nablav: estimate of 1}. The contribution of$~\eqref{nablav: estimate of 1}_1$ is bounded by
\begin{align}
 & \frac{1}{|v|^2}\int_{n(\xb)\cdot v^1>0}   \frac{\alpha(\xb,v^1)}{|v^1|} \nabla_x f(\xb,v^1)   \sqrt{\mu(v^1)}   \dd v^1         \notag \\
   & \lesssim  \frac{1}{|v|^2}\int_{n(\xb)\cdot v^1>0}   \frac{\Vert \alpha\nabla_x f\Vert_\infty}{|v^1|}    \sqrt{\mu(v^1)} \dd v^1\lesssim
   \frac{\Vert \alpha\nabla_x f\Vert_\infty}{|v|^2}.\label{nablav: estimate of 1_1}
\end{align}

For the contribution of$~\eqref{nablav: estimate of 1}_2$, we exchange the $v$ derivative into $\mathbf{v}^1$ derivative:
\[\nabla_v f(\xb,T^t_{\mathbf{x}_p^1}\mathbf{v}^1)=\nabla_{\mathbf{v}^1}[f(\xb,T^t_{\mathbf{x}_p^1}\mathbf{v}^1) ] T_{\mathbf{x}_p^1}.  \]
Then the contribution of$~\eqref{nablav: estimate of 1}_2$ in~\eqref{nablav: estimate of 1_1} can be written as
\begin{align}
   &\frac{O(1)}{|v|}\int_{\mathbf{v}^1_3>0}   \nabla_v T^t_{\mathbf{x}_p^1} \mathbf{v}^1 \nabla_{\mathbf{v}^1}[f(\xb,T^t_{\mathbf{x}_p^1}\mathbf{v}^1) ] T_{\mathbf{x}_p^1}            \mathbf{v}^1_3  \sqrt{\mu(\mathbf{v}^1)}   \dd \mathbf{v}^1         \notag  \\
   & =\frac{O(1)}{|v|} \int_{\mathbf{v}^1_3>0}   \nabla_v T^t_{\mathbf{x}_p^1} f(\xb,T^t_{\mathbf{x}_p^1}\mathbf{v}^1)  T_{\mathbf{x}_p^1}            \nabla_{\mathbf{v}^1}  \big[\mathbf{v}^1 \mathbf{v}^1_3  \sqrt{\mu(\mathbf{v}^1)} \big]  \dd \mathbf{v}^1    \notag\\
   & \lesssim \frac{O(1)}{|v|}\int_{\mathbf{v}^1_3>0}   \frac{1}{|v|} \Vert \eta\Vert_{C^2}\Vert wf\Vert_\infty      \mu^{1/4}(\mathbf{v}^1)  \dd \mathbf{v}^1 \lesssim \frac{\Vert wf\Vert_\infty}{|v|^2}, \label{nablav: estimate of 1_2}
\end{align}
where we applied an integration by parts to $\dd v^1$ in the second line, and used~\eqref{v deri of T} in the third line.

Combining~\eqref{nablav: estimate of 1_1} and~\eqref{nablav: estimate of 1_2} we conclude
\begin{equation}\label{nablav: bound for estimate of 1}
\eqref{nablav:   1}\lesssim \frac{\Vert wf\Vert_\infty+\Vert \alpha\nabla_x f\Vert_\infty}{|v|^2}.
\end{equation}

Then we estimate~\eqref{nablav:  7}. For $h=K(f)$, we compute
  \begin{align}
   &     \int^t_{\max\{0,t-\tb\}}  e^{-\nu (t-s)} \int_{\mathbb{R}^3}\nabla_v \big[ \mathbf{k}(v,u) f(x-(t-s)v,u) \big] \dd s \notag\\
     & = \int^t_{\max\{0,t-\tb\}}  e^{-\nu (t-s)}   \int_{\mathbb{R}^3} \big[\nabla_v \mathbf{k}(v,u)   f(x-(t-s)v,u)  + \mathbf{k}(v,u) \nabla_v [f(x-(t-s)v,u)]   \big]      \notag\\
     &  \lesssim \int^t_{\max\{0,t-\tb\}}  e^{-\nu (t-s)}   \int_{\mathbb{R}^3} \big[\Vert wf\Vert_\infty    \frac{w^{-1}(u)\langle v\rangle \mathbf{k}_\varrho(v,u)}{|v-u|}+ \mathbf{k}(v,u) (t-s) \nabla_x f(x-(t-s)v,u) \big]\notag\\
     &\lesssim   \int^t_{\max\{0,t-\tb\}}  e^{-\nu (t-s)}   \int_{\mathbb{R}^3}\big[ \Vert wf\Vert_\infty   \frac{e^{-\varrho |v-u|^2/2}}{|v-u|^2} \frac{1}{|v|^2}+ \tb \Vert w_{\tilde{\theta}}\alpha\nabla_x f\Vert_\infty  \frac{w_{\tilde{\theta}}(v)\mathbf{k}(v,u)}{w_{\tilde{\theta}}(u)\alpha(x-(t-s)v,u)}\frac{1}{w_{\tilde{\theta}}(v)} \big]\notag\\
     &\lesssim   \frac{\Vert wf\Vert_\infty}{|v|^2}+ \frac{\tilde{\alpha}(x,v)}{w_{\tilde{\theta}}(v)|v|^2}\frac{\Vert \alpha\nabla_x f\Vert_\infty}{\alpha(x-(t-s)v,v)}\lesssim \frac{\Vert wf\Vert_\infty+\Vert w_{\tilde{\theta}}\alpha\nabla_x f\Vert_\infty}{|v|^2}.
     \label{nablav: estimate of 7K}
  \end{align}
In the fourth line we have used
\begin{align*}
  w^{-1}(u)e^{-\varrho |v-u|^2}\langle v\rangle &=   e^{-\varrho|v-u|^2/2}   e^{-\varrho|v-u|^2/2}e^{-\varrho|u|^2} \langle v\rangle     \\
   & \lesssim e^{-\varrho|v-u|^2/2} \frac{e^{-C|v|^2}\langle v\rangle |v|^2}{|v|^2}\lesssim \frac{e^{-\varrho|v-u|^2/2}}{|v|^2}
\end{align*}
In the last line we have applied $\frac{e^{-\varrho|v-u|^2/2}}{|v-u|^2} \in L^1_u$, \eqref{est:nonlocal_wo_e} in Lemma \ref{Lemma: NLN} and \eqref{tb bounded}.

For $h=\Gamma(f,f)$ we apply~\eqref{nablav K Gamma} to have
\begin{align}
\eqref{nablav:    7}\mathbf{1}_{h=\Gamma(f,f)}   &\lesssim \frac{1}{|v|^2}\int^t_{\max\{0,t-\tb\}} e^{-\nu(t-s)} \Vert wf\Vert_\infty^2+\Vert wf\Vert_\infty \Vert |v|^2\nabla_v f\Vert_\infty  \notag\\
   & \lesssim \frac{\Vert wf\Vert_\infty^2+\Vert wf\Vert_\infty \Vert|v|^2 \nabla_v f\Vert_\infty}{|v|^2}. \label{nablav: estimate of 7G}
\end{align}

Then we estimate all the other terms, which follow from more direct computation. For~\eqref{nablav: 2} we apply~\eqref{nablav tb xb}; for~\eqref{nablav: 3} we apply~\eqref{nablav nu}; for~\eqref{nablav: 4} we apply~\eqref{nablav tb xb} and $t\leq \tb\leq \frac{\tilde{\alpha}(x,v)}{|v|^2}$; for~\eqref{nablav: 5} we apply \eqref{nablav nu}; for~\eqref{nablav: 6} we apply \eqref{nablav nu} and~\eqref{h bounded}; for~\eqref{nablav: 8} we apply \eqref{nablav tb xb} and~\eqref{h bounded}, then we obtain the following bound:
\begin{equation}\label{nablav: estimate of 2}
\begin{split}
   \eqref{nablav: 2} & \lesssim \Vert wf\Vert_\infty \frac{\tilde{\alpha}(x,v)}{\alpha(x,v)|v|^2w^{-1}(v)}\lesssim \frac{\Vert wf\Vert_\infty}{|v|^2},
\end{split}
\end{equation}

\begin{equation}\label{nablav: estimate of 3}
\eqref{nablav: 3}\lesssim \frac{\Vert |v|^2f\Vert_\infty}{|v|^2}\lesssim \frac{\Vert wf\Vert_\infty}{|v|^2},
\end{equation}

\begin{equation}\label{nablav: estimate of 4}
\begin{split}
   \eqref{nablav: 4} & \lesssim e^{-t} t \nabla_x f(x-tv,v)+e^{-t}\nabla_v f(x-tv,v)  \\
     & \lesssim e^{-t}\frac{\Vert \alpha \nabla_x f\Vert_\infty}{|v^2|}+e^{-t}\frac{\Vert |v|^2 \nabla_v f\Vert_\infty}{|v|^2},
\end{split}
\end{equation}

\begin{equation}\label{nablav: estimate of 5}
\eqref{nablav: 5}\lesssim e^{-t}\frac{\Vert |v|^2f \Vert_\infty}{|v|^2}\lesssim \frac{e^{-t}\Vert wf\Vert_\infty}{|v|^2},
\end{equation}

\begin{equation}\label{nablav: estimate of 6}
\eqref{nablav: 6}\lesssim  \Vert wf\Vert_\infty\int^t_{\max\{0,t-\tb\}}  e^{-(t-s)}\dd s \lesssim \frac{\Vert wf\Vert_\infty}{|v|^2},
\end{equation}

\begin{equation}\label{nablav: estimate of 8}
\eqref{nablav: 8} \lesssim \frac{\Vert wf\Vert_\infty}{|v|^2}.
\end{equation}

Combining~\eqref{nablav: bound for estimate of 1},\eqref{nablav: estimate of 2},\eqref{nablav: estimate of 3},\eqref{nablav: estimate of 4},\eqref{nablav: estimate of 5},\eqref{nablav: estimate of 6},\eqref{nablav: estimate of 7G},\eqref{nablav: estimate of 7K} and~\eqref{nablav: estimate of 8} we conclude
\[|\nabla_v f|\lesssim  \Vert T_W-T_0\Vert_{C^1}\frac{(e^{-t}+\Vert wf\Vert_\infty)\Vert |v|^2\nabla_v f\Vert_\infty +\Vert wf\Vert_\infty+\Vert \alpha\nabla_x f\Vert_\infty}{|v|^2}.\]
Since $e^{-t}\ll 1$ from $t\gg 1$ and $\Vert wf\Vert_\infty \ll 1$ from Existence Theorem, we conclude the proof.

\end{proof}

\section{$C^{1,\beta}$ Solutions in Convex Domains}
 Given the continuity of the $C^1$ solution in section 5.2, in this section we prove the H\"{o}lder regularity, which are~\eqref{estF_C1beta} and~\eqref{estF_C1betatang} in the \textbf{Main Theorem}. Since the proof of (\ref{estF C1v betax}) is simpler than these two, we omit it.

For simplicity we denote
\begin{equation}\label{C1beta norm}
[\nabla_x f(\cdot,v)]_{C^{0,\beta}_{x;2+\beta}}:=\sup_{x,y\in \Omega}\big\Vert w_{\tilde{\theta}}(v)|v|^2\min\{\frac{\alpha(x,v)}{|v|},\frac{\alpha(y,v)}{|v|}\}^{2+\beta} \frac{\nabla_x f(x,v)-\nabla_x f(y,v)}{|x-y|^\beta}\big\Vert_{L_v^\infty},
\end{equation}
\begin{equation}\label{C1beta tang norm}
[\nabla_{x_\parallel}f(\cdot,v)]_{C^{0,\beta}_{x;1+\beta}}:=\sup_{x,y\in \Omega}\big\Vert w_{\tilde{\theta}/2}(v) |v|^2 \min\{\frac{\alpha(x,v)}{|v|},\frac{\alpha(y,v)}{|v|}\}^{1+\beta}\frac{|\nabla_\parallel f(x,v)-\nabla_\parallel f(y,v)|}{|x-y|^\beta}\big\Vert_{L_v^\infty}.
\end{equation}
Here $\nabla_\parallel=G(x)\nabla_x$ and $G$ is defined in~\eqref{G}.
We note that the weight in $\eqref{C1beta norm}$ and $\eqref{C1beta tang norm}$ are different in terms of the power.

To prove the weighted $C^{1,\beta}$ we will estimate the trajectory starting from two different positions $x$ and $y$. In result we define the backward exit time and positions corresponding to them.

The first backward exit position and time are denoted using the previous notation as
\[\xb(x,v),\quad \xb(y,v),\quad \tb(x,v),\quad \tb(y,v).\]
For simplicity we denote the second backward exit position and time as
\begin{equation}\label{second backward}
\xb^2(x)=\xb(\xb(x,v),v^1),\quad \xb^2(y)=\xb(\xb(y,v),v^1),\quad \tb^2(x)=\tb(\xb(x,v),v^1),\quad \tb^2(y)=\tb(\xb(y,v),v^1).
\end{equation}

Similarly to Definition \ref{definition: chart}, we choose $p^1(x),p^2(x),p^1(y),p^2(y)\in \mathcal{P}$ such that
\begin{equation}\label{xpk x}
\mathbf{x}_{p^i(x)}^i:=(\mathbf{x}_{p^i(x),1}^i,\mathbf{x}_{p^i(x),2}^i,0) \text{ such that }\eta_{p^i(x)}(\mathbf{x}_{p^i(x)}^i)=\left\{
                                     \begin{array}{ll}
                                       \xb(x,v), & \hbox{$i=1$;} \\
                                       \xb^2(x), & \hbox{$i=2$.}
                                     \end{array}
                                   \right.
\end{equation}
\begin{equation}\label{xpk y}
\mathbf{x}_{p^i(y)}^i:=(\mathbf{x}_{p^i(y),1}^i,\mathbf{x}_{p^i(y),2}^i,0) \text{ such that }\eta_{p^i(y)}(\mathbf{x}_{p^i(y)}^i)=\left\{
                                     \begin{array}{ll}
                                       \xb(y,v), & \hbox{$i=1$;} \\
                                       \xb^2(y), & \hbox{$i=2$.}
                                     \end{array}
                                   \right.
\end{equation}

Without of loss of generality and also for the purpose of simplicity, throughout this section we assume
\begin{equation}\label{simplicity}
\max\{\Vert wf\Vert_\infty,\Vert |v|\nabla_\parallel f \Vert_\infty,\Vert |v|^2\nabla_v f\Vert_\infty\}\leq \Vert \alpha\nabla_x f\Vert_\infty\leq \Vert \alpha\nabla_x f\Vert_\infty^2 \leq \Vert w_{\tilde{\theta}}\alpha\nabla_x f\Vert_\infty^2.
\end{equation}

Then the bound of~\eqref{C1beta norm} and~\eqref{C1beta tang norm} are given by the following proposition.
\begin{proposition}\label{Prop: C1beta}
Suppose $F=\mu+\sqrt{\mu}f$ solves the steady Boltzmann equation~\eqref{BE_F} with boundary condition~\eqref{diffuseBC_F}, then
\begin{equation}\label{Bound of C1beta}
[\nabla_{x}f(\cdot,v)]_{C^{0,\beta}_{x;2+\beta}}\lesssim o(1)[\nabla_{x_\parallel}f(\cdot,v)]_{C^{0,\beta}_{x;1+\beta}}+C_{\e}\Vert T_W-T_0\Vert_{C^2}\Vert w_{\tilde{\theta}}\alpha\nabla_x f \Vert_\infty^2,
\end{equation}
and
\begin{equation}\label{Bound of C1beta tangential}
[\nabla_{x_\parallel}f(\cdot,v)]_{C^{0,\beta}_{x;1+\beta}}\lesssim o(1)[\nabla_x f(\cdot,v)]_{C^{0,\beta}_{x;2+\beta}}+C_{\e}\Vert T_W-T_0\Vert_{C^2}\Vert w_{\tilde{\theta}}\alpha\nabla_x f \Vert_\infty^2,
\end{equation}
where $C_\e\gg 1$.
\end{proposition}

These two estimates together conclude~\eqref{estF_C1beta} and~\eqref{estF_C1betatang}.

Below we present two lemmas about the collision operators. We will use them in the proof when we estimate the difference of the collision operators.

\begin{lemma}\label{Lemma: h-h}
For $h(x,v)=Kf(x,v)+\Gamma(f,f)(x,v)$, we have
\begin{align}\label{h-h bounded}
\frac{|h(x,v)-h(y,v)|}{|x-y|^\beta}   &  \lesssim \frac{\Vert w_{\tilde{\theta}}\alpha\nabla_x f\Vert_\infty}{w_{\tilde{\theta}}(v)\min\{\alpha(x,v),\alpha(y,v)\}^\beta},
\end{align}

\begin{equation}\label{nabla h bounded}
\nabla_x h(x,v)\lesssim  \frac{\Vert w_{\tilde{\theta}}\alpha\nabla_x f\Vert_\infty}{w_{\tilde{\theta}}(v)} \int_{\mathbb{R}^3}\frac{\mathbf{k}_{\tilde{\varrho}}(v,u)}{\alpha(x,u)}\dd u+ \frac{\Vert w_{\tilde{\theta}}\alpha\nabla_x f\Vert_\infty}{w_{\tilde{\theta}}(v)\alpha(x,v)}.
\end{equation}

\end{lemma}

\begin{proof}
Since
\begin{equation}\label{Gammax-Gammay}
\begin{split}
|\Gamma(f,f)(x,v)-\Gamma(f,f)(y,v)|&=      \Gamma(f(x)-f(y),f(x))(v)+\Gamma(f(y),f(x)-f(y))(v)         \\
&\lesssim  \Vert wf\Vert_\infty \big( \int_{\mathbb{R}^3}\mathbf{k}(v,u)|f(x,u)-f(y,u)|\dd u+|f(x,v)-f(y,v)|\big) ,
\end{split}
\end{equation}
we have
\begin{align*}
   &\frac{h(x,v)-h(y,v)}{|x-y|^\beta} \\
   &=\frac{K[f(x,v)-f(y,v)]+\Gamma(f,f)(x,v)-\Gamma(f,f)(y,v)}{|x-y|^\beta}  \\
   & \lesssim \int_{\mathbb{R}^3} \mathbf{k}(v,u)\frac{f(x,u)-f(y,u)}{|x-y|^\beta} \dd u\\
   &+     \Vert wf\Vert_\infty \int_{\mathbb{R}^3} \mathbf{k}(v,u)\frac{f(x,u)-f(y,u)}{|x-y|^\beta} \dd u+ \frac{|f(x,v)-f(y,v)|}{|x-y|^\beta}\\
  & \lesssim   \frac{\Vert wf\Vert_\infty \Vert w_{\tilde{\theta}}\alpha\nabla_x f\Vert_\infty}{w_{\tilde{\theta}}(v)} \Big[ \int_{\mathbb{R}^3}\frac{w_{\tilde{\theta}}(u)\mathbf{k}_\varrho(v,u)}{w_{\tilde{\theta}}(u)\min\{\alpha(x,u),\alpha(y,u)\}^\beta}+\frac{1}{w_{\tilde{\theta}}(v)\min\{\alpha(x,v),\alpha(y,v)\}^\beta}\Big]\\
  & \lesssim \frac{\Vert w_{\tilde{\theta}}\alpha\nabla_x f\Vert_\infty}{w_{\tilde{\theta}}(v)\min\{\alpha(x,v),\alpha(y,v)\}^\beta},
\end{align*}
where we have applied~\eqref{min: f} and~\eqref{integrate beta<1} in Lemma \ref{Lemma: NLN}.

Then we prove~\eqref{nabla h bounded}. Clearly from Lemma \ref{Lemma: k tilde},
\begin{align*}
\nabla_x Kf(x,v)&=\int_{\mathbb{R}^3} \mathbf{k}(v,u)\nabla_x f(x,u)\lesssim \frac{\Vert w_{\tilde{\theta}}\alpha\nabla_x f\Vert_\infty}{w_{\tilde{\theta}}(v)}  \int_{\mathbb{R}^3} \frac{\mathbf{k}(v,u)}{\alpha(x,u)}\frac{w_{\tilde{\theta}}(v)}{w_{\tilde{\theta}}(u)}\lesssim \frac{\Vert w_{\tilde{\theta}}\alpha\nabla_x f\Vert_\infty}{w_{\tilde{\theta}}(v)}  \int_{\mathbb{R}^3} \frac{\mathbf{k}_{\tilde{\varrho}}(v,u)}{\alpha(x,u)}.
\end{align*}

For $\Gamma$, we bound
\begin{align*}
 \nabla_x \Gamma(f,f)(x,v)  & =\Gamma(\nabla_x f,f)+\Gamma(f,\nabla_x f)\lesssim \Vert wf\Vert_\infty \Big[|\nabla_x f(x,v)|+ \int_{\mathbb{R}^3} |\mathbf{k}(v,u)\nabla_x f(x,u)| \Big] \\
   & \lesssim \frac{\Vert w_{\tilde{\theta}}\alpha \nabla_x f \Vert_\infty}{w_{\tilde{\theta}}(v)\alpha(x,v)}+ \frac{\Vert w_{\tilde{\theta}}\alpha\nabla_x f\Vert_\infty}{w_{\tilde{\theta}}(v)} \int_{\mathbb{R}^3} \frac{\mathbf{k}(v,u)}{\alpha(x,u)}\frac{w_{\tilde{\theta}}(v)}{w_{\tilde{\theta}}(u)} \dd u,
\end{align*}
again by Lemma \ref{Lemma: k tilde} we conclude the lemma.

\end{proof}

\begin{lemma}\label{Lemma: gamma-gamma}
Denote $x'=x-(t-s)v$, $y'=y-(t-s)v$. For the difference of $\nabla_x\Gamma$, we have
\begin{align}
  & \int_0^te^{-\nu(v)(t-s)}\Big|\frac{\nabla_x \Gamma(f,f)(x',v)-\nabla_x \Gamma(f,f)(y',v)}{|x-y|^\beta}\Big|  \notag \\
   & \lesssim     \frac{o(1) [\nabla_x f(\cdot ,v)]_{C_{x,2+\beta}^{0,\beta}}  +\Vert w_{\tilde{\theta}} \alpha\nabla_x f\Vert^2_\infty}{w_{\tilde{\theta}}(v)|v|^2\min\{\frac{\alpha(x,v)}{|v|},\frac{\alpha(y,v)}{|v|}\}^{2+\beta}},
   \label{int gamma-gamma}
\end{align}
and

\begin{align}
  & \int_0^te^{-\nu(v)(t-s)}\Big|\frac{G(x')\nabla_x \Gamma(f,f)(x',v)-G(y')\nabla_x \Gamma(f,f)(y',v)}{|x-y|^\beta}\Big|  \notag \\
   & \lesssim     \frac{o(1)[\nabla_\parallel f(\cdot ,v)]_{C_{x,1+\beta}^{0,\beta}} +\Vert w_{\tilde{\theta}} \alpha\nabla_x f\Vert^2_\infty}{w_{\tilde{\theta}/2}(v)|v|^2\min\{\frac{\alpha(x,v)}{|v|},\frac{\alpha(y,v)}{|v|}\}^{1+\beta}}.
   \label{G gamma- G gamma}
\end{align}

\end{lemma}

\begin{proof}
We rewrite the difference of $\nabla_x\Gamma$ as
\begin{align*}
  \nabla_x \Gamma(f,f)(x,v)-\nabla_x \Gamma(f,f)(y,v) & =\Gamma(\nabla_x f(x)-\nabla_x f(y),f(x))(v)+\Gamma(f(x)-f(y),\nabla_x f(x))(v) \\
   & +\Gamma(\nabla_x f(y),f(x)-f(y))(v)+\Gamma(f(y),\nabla_x f(x)-\nabla_x f(y))(v).
\end{align*}
Then by~\eqref{gamma bounded} we bound
\[\Gamma(\nabla_x f(x)-\nabla_x f(y),f(x))(v)\lesssim  \Vert wf\Vert_\infty \int_{\mathbb{R}^3}\mathbf{k}_\varrho(v,u)|\nabla_x f(x,u)-\nabla_x f(y,u)|\dd u,  \]

\[\Gamma(f(y),\nabla_x f(x)-\nabla_x f(y))(v)\lesssim    \Vert wf\Vert_\infty \big(\nabla_x f(x,v)-\nabla_x f(y,v)\big)+\Vert wf\Vert_\infty \int_{\mathbb{R}^3}\mathbf{k}_\varrho(v,u)|\nabla_x f(x,u)-\nabla_x f(y,u)|.\]

For $\Gamma(f(x)-f(y),\nabla_x f(x))(v) $, we bound the loss term as
\[\Gamma_{\text{loss}}(f(x)-f(y),\nabla_x f(x))(x)\lesssim |\nabla_x f(x,v)|\nu(\sqrt{\mu}(f(x)-f(y)))\lesssim |\nabla_x f(x,v)|\int_{\mathbb{R}^3}\mathbf{k}_\varrho(v,u)|f(x,u)-f(y,u)|.\]
From the Grad's estimate~\cite{R}, the gain term can be written as
\begin{align}
   &\frac{\Gamma_{\text{gain}}(f(x)-f(y),\nabla_x f(x))(v)}{|x-y|^\beta}\notag\\
   &\lesssim \int_{\mathbb{R}^3} \dd u \frac{|\nabla_x f(x,u)|}{|v-u|} \int_{(u-v)\cdot \omega=0}  \frac{f(x,v+\omega)-f(y,v+\omega)}{|x-y|^\beta}  \label{gamma f-f nablaf}        \\
   &\lesssim \int_{\mathbb{R}^3} \dd u \mathbf{k}_\varrho(v,u)  |\nabla_x f(x,u)| \int_{\mathbb{S}^2} \frac{w^{-1}_{2\tilde{\theta}}(v+\omega)\Vert w_{\tilde{\theta}}\alpha\nabla_x f\Vert_\infty}{\min\{\alpha(x,v+\omega),\alpha(y,v+\omega)\}^\beta}\notag\\
   &  \lesssim \frac{\Vert w_{\tilde{\theta}}\alpha\nabla_x f\Vert_\infty^2}{w_{\tilde{\theta}}(v)}\int_{\mathbb{R}^3} \dd u \frac{w_{\tilde{\theta}}(v)\mathbf{k}_\varrho(v,u)}{w_{\tilde{\theta}}(u)\alpha(x,u)\min\{\xi(x),\xi(y)\}^{\beta/2}}   \int_{\mathbb{S}^2}\dd \omega e^{-\tilde{\theta}|v|^2/2} \frac{w_{\tilde{\theta}}^{-1}(v+\omega)}{|v+\omega|^\beta}\notag\\
&\lesssim  \frac{\Vert w_{\tilde{\theta}}\alpha\nabla_x f\Vert_\infty^2}{w^{3/2}_{\tilde{\theta}}(v)}    \int_{\mathbb{R}^3}\dd u \frac{\mathbf{k}_{\tilde{\varrho}}(v,u)}{\alpha(x,u) \min\{\xi(x),\xi(y)\}^{\beta/2}}\notag\\
   &      \lesssim \frac{w_{\tilde{\theta}}^{-3/2}(v)}{\min\{\xi(x),\xi(y)\}^{\beta/2}} [1+\log |\xi(y)|+|\log |v||]. \notag
\end{align}
Here we applied~\eqref{min: f} in the third line. In the fourth line we have used that for $|w|\leq 1$,
\begin{align*}
  w^{-1}_{2\tilde{\theta}}(v+\omega) &=w^{-1}_{\tilde{\theta}}(v+\omega)w^{-1}_{\tilde{\theta}}(v+\omega)\lesssim w^{-1}_{\tilde{\theta}}(v+\omega) e^{-\tilde{\theta}|v|^2}e^{2\tilde{\theta}v\cdot \omega} \\
   & \lesssim w^{-1}_{\tilde{\theta}}(v+\omega) e^{-\tilde{\theta}|v|^2}e^{\frac{\tilde{\theta}|v|^2}{2}}e^{2\tilde{\theta}|w|^2}\lesssim w^{-1}_{\tilde{\theta}}(v+\omega) e^{-\tilde{\theta}|v|^2/2}.
\end{align*}
and $\alpha(x,v)\geq \sqrt{-\xi(x)}|v|^2$. In the fifth line we use Lemma \ref{Lemma: k tilde} and in the last line we use~\eqref{eqn: int alpha du}.

Thus by~\eqref{min: f} with~\eqref{simplicity} we obtain
\begin{align}
& \frac{|\nabla_x \Gamma(f,f)(x',v)-\nabla_x \Gamma(f,f)(y',v)|}{|x-y|^\beta} \notag\\
&\leq  O(\Vert wf\Vert_\infty)\frac{|\nabla_x f(x',v)-\nabla_x f(y',v)|}{|x-y|^\beta}+O(\Vert wf\Vert_\infty)\int_{\mathbb{R}^3}\mathbf{k}_\varrho(v,u)\frac{|\nabla_x f(x',u)-\nabla_x f(y',u)|}{|x-y|^\beta} \label{Gamma-Gamma1} \\
   & +|\nabla_x f(x',v)|\int_{\mathbb{R}^3} \mathbf{k}_\varrho(v,u) \frac{| f(x',u)- f(y',u)|}{|x-y|^\beta}\dd u+\frac{|f(x',v)-f(y',v)|}{|x-y|^\beta}\int_{\mathbb{R}^3}\mathbf{k}_\varrho(v,u)|\nabla_x f(x',u)|\dd u \label{Gamma-Gamma2}\\
   &+\frac{\Vert w_{\tilde{\theta}}\alpha\nabla_x f\Vert_\infty^2}{w^{3/2}_{\tilde{\theta}}(v)} \Big[\frac{|\log|\xi(x')||}{\min\{\xi(x'),\xi(y')\}^{\beta/2}}+\frac{|\log |v||}{\min\{\xi(x'),\xi(y')\}^{\beta/2}}\Big]\label{Gamma-Gamma3}.
\end{align}

We bound $w_{\tilde{\theta}}^{-3/2}(v)|v|\lesssim w_{\tilde{\theta}}^{-1}(v).$ By \eqref{integrate alpha beta} and~\eqref{integrate beta<1} in Lemma \ref{Lemma: NLN} we have

\begin{align*}
  \int_0^t e^{-\nu(v)(t-s)} \eqref{Gamma-Gamma1} &\lesssim   \frac{\Vert wf\Vert_\infty [\nabla_x f(\cdot,v)]_{C_{x,2+\beta}^{0,\beta}} }{w_{\tilde{\theta}}(v)|v|^2\min\{\frac{\alpha(x,v)}{|v|},\frac{\alpha(y,v)}{|v|}\}^{2+\beta}}           \\
   & +\frac{\Vert wf\Vert_\infty [\nabla_x f(\cdot,v)]_{C_{x,2+\beta}^{0,\beta}}}{w_{\tilde{\theta}}(v)}  \int_0^t e^{-\nu(v)(t-s)}  \int_{\mathbb{R}^3} \frac{w_{\tilde{\theta}}(v)\mathbf{k}_\varrho(v,u)}{w_{\tilde{\theta}}(u)|u|^2\min\{\frac{\alpha(x',u)}{|u|},\frac{\alpha(y',u)}{|u|}\}^{2+\beta}}\\
   &\lesssim \frac{\Vert wf\Vert_\infty [\nabla_x f(\cdot,v)]_{C_{x,2+\beta}^{0,\beta}}}{w_{\tilde{\theta}}(v)|v|^2\min\{\frac{\alpha(x,v)}{|v|},\frac{\alpha(y,v)}{|v|}\}^{2+\beta}}.
\end{align*}

By \eqref{min: f},\eqref{integrate alpha beta} and~\eqref{integrate beta<1} in Lemma \ref{Lemma: NLN} we have
\begin{align}
  \int_0^t e^{-\nu(v)(t-s)} \eqref{Gamma-Gamma2} &\lesssim   \frac{w^{-2}_{\tilde{\theta}}(v)|v|\Vert w_{\tilde{\theta}}\alpha\nabla_x f\Vert_\infty^2}{|v|\alpha(x,v)} \int_0^t e^{-\nu(v)(t-s)}  \int_{\mathbb{R}^3} \frac{w_{\tilde{\theta}}(v)\mathbf{k}_\varrho(v,u)}{w_{\tilde{\theta}}(u)\min\{\alpha(x',u),\alpha(y',u)\}^\beta}    \notag        \\
   & + \frac{w^{-2}_{\tilde{\theta}}(v)|v|\Vert w_{\tilde{\theta}}\alpha\nabla_x f\Vert_\infty^2}{|v|\min\{\alpha(x,v),\alpha(y,v)\}^\beta}   \int_0^t e^{-\nu(v)(t-s)}  \int_{\mathbb{R}^3} \frac{w_{\tilde{\theta}}(v)\mathbf{k}_\varrho(v,u)}{w_{\tilde{\theta}}(u)\alpha(x',u)}\notag\\
   &\lesssim \frac{\Vert w_{\tilde{\theta}} \alpha\nabla_x f\Vert_\infty^2}{w_{\tilde{\theta}}(v)|v|^2\min\{\frac{\alpha(x,v)}{|v|},\frac{\alpha(y,v)}{|v|}\}^{1+\beta}}
   \label{inte Gamma-Gamma 2}\\
   &\lesssim \frac{\Vert w_{\tilde{\theta}} \alpha\nabla_x f\Vert_\infty^2}{w_{\tilde{\theta}}(v)|v|^2\min\{\frac{\alpha(x,v)}{|v|},\frac{\alpha(y,v)}{|v|}\}^{2+\beta}}. \notag
\end{align}

For~\eqref{Gamma-Gamma3}, we bound $|\log\xi(x')|\lesssim \frac{1}{\min\{\xi(x'),\xi(y')\}^\e}$ and $w_{\tilde{\theta}}^{-3/2}(v)|v|\lesssim w_{\tilde{\theta}}^{-1}(v)$. Then we have
\begin{align}
 &\int_0^t e^{-\nu(v)(t-s)} \eqref{Gamma-Gamma3}  \notag\\
  & \lesssim \Vert w_{\tilde{\theta}}\alpha\nabla_x f\Vert_\infty^2\Big[ \int_0^t \frac{w_{\tilde{\theta}}^{-3/2}(v)e^{-\nu(v)(t-s)}}{\min\{\xi(x'),\xi(y')\}^{\beta/2+\e}}+\int_0^t \frac{w_{\tilde{\theta}}^{-3/2}(v)e^{-\nu(v)(t-s)}|\log|v||}{\min\{\xi(x'),\xi(y')\}^{\beta/2}}\Big] \notag \\
   & \lesssim   \Vert w_{\tilde{\theta}}\alpha\nabla_x f\Vert_\infty^2\Big[  \int_0^t \frac{w_{\tilde{\theta}}^{-3/2}(v)|v|e^{-\nu(v)(t-s)}}{|v|\min\{\xi(x'),\xi(y')\}^{\beta/2+\e}}+\int_0^t \frac{w_{\tilde{\theta}}^{-3/2}(v)|v||\log|v||e^{-\nu(v)(t-s)}}{|v|\min\{\xi(x'),\xi(y')\}^{\beta/2}}\notag\\
   &\lesssim \int_0^t \frac{w^{-1}_{\tilde{\theta}}(v)e^{-\nu(v)(t-s)}}{|v|\min\{\xi(x'),\xi(y')\}^{\beta/2+\e}}+\int_0^t \frac{w^{-1}_{\tilde{\theta}}(v)e^{-\nu(v)(t-s)}}{|v|\min\{\xi(x'),\xi(y')\}^{\beta/2}}\Big]\notag\\
   &\lesssim \frac{\Vert w_{\tilde{\theta}}\alpha\nabla_x f\Vert_\infty^2}{w_{\tilde{\theta}}(v)|v|^2\min\{\frac{\alpha(x,v)}{|v|},\frac{\alpha(y,v)}{|v|}\}^{\beta}}            ,\label{bound for gamma-gamma3}
\end{align}
where we have applied Lemma \ref{Lemma: NLN} in the last line with small enough $\e$ such that$\beta/2+\e\ll 1$.

Using $\Vert wf\Vert_\infty \ll 1$ from Main Theorem we conclude~\eqref{int gamma-gamma}.

Then we prove~\eqref{G gamma- G gamma}. From~\eqref{Gamma-Gamma1}-\eqref{Gamma-Gamma3} we can rewrite
\begin{align}
  &\frac{G(x')\nabla_x \Gamma(f,f)(x',v)-G(y')\nabla_x \Gamma(f,f)(y',v)}{|x-y|^\beta}   \notag\\
   &   =\frac{G(y')-G(x')}{|x-y|^\beta}\nabla_x \Gamma(f,f)(y',v) +G(x')\frac{\nabla_x \Gamma(f,f)(x',v)-\nabla_x\Gamma(f,f)(y',v)}{|x-y|^\beta}\notag\\
   &\lesssim     \frac{G(y')-G(x')}{|x'-y'|^\beta} \Vert wf\Vert_\infty\Vert w_{\tilde{\theta}}\alpha\nabla_x f\Vert_\infty\big[\frac{w^{-1}_{\tilde{\theta}}(v)|v|}{|v||\alpha(y,v)|}+w^{-1}_{\tilde{\theta}/2}(v)\int_{\mathbb{R}^3}\frac{w_{\tilde{\theta}/2}(v)|u|\mathbf{k}_\varrho(v,u)}{w_{\tilde{\theta}}(u)|u|\alpha(y-(t-s)v,u)} \big] \label{G gamma-gamma1}\\
   &+ G(x')\times \big[\eqref{Gamma-Gamma1}+\eqref{Gamma-Gamma2}+\eqref{Gamma-Gamma3}\big]
   \label{G gamma-gamma3}.
 \end{align}
We bound $w_{\tilde{\theta}}^{-1}(v)|v|\lesssim w_{\tilde{\theta}/2}^{-1}(v)$. Then we have
\begin{align*}
  \int_0^t e^{-\nu(v)(t-s)}\eqref{G gamma-gamma1} &\lesssim \frac{\Vert w_{\tilde{\theta}}\alpha\nabla_x f\Vert_\infty^2}{w_{\tilde{\theta}/2}(v)|v|^2\min\{\frac{\alpha(x,v)}{|v|},\frac{\alpha(y,v)}{|v|}\}}\\
  &+ \frac{\Vert w_{\tilde{\theta}}\alpha\nabla_x f\Vert_\infty^2}{w_{\tilde{\theta}/2}(v)}\int_0^t e^{-\nu(v)(t-s)} \int_{\mathbb{R}^3} \frac{\mathbf{k}_{\tilde{\varrho}}(v,u)}{|u|^2\min\{\frac{\alpha(x',u)}{|u|},\frac{\alpha(y',u)}{|u|}\}} \dd u\\
  &\lesssim \frac{\Vert w_{\tilde{\theta}}\alpha\nabla_x f\Vert_\infty^2}{w_{\tilde{\theta}/2}(v)|v|^2\min\{\frac{\alpha(x,v)}{|v|},\frac{\alpha(y,v)}{|v|}\}^{1+\beta}} ,
\end{align*}
where we have used~\eqref{min: nxb} in the first line,~\eqref{k_theta} in the second line and~\eqref{est:nonlocal_wo_e} in Lemma \ref{Lemma: NLN} with $\frac{\alpha(x,v)}{|v|}\lesssim 1$ in the last line.

For~\eqref{G gamma-gamma3}, since $|G|\lesssim 1$, from~\eqref{bound for gamma-gamma3} and~\eqref{inte Gamma-Gamma 2} the contribution of~\eqref{Gamma-Gamma2}\eqref{Gamma-Gamma3} is already included in~\eqref{G gamma- G gamma}. Then we consider the contribution of \eqref{Gamma-Gamma1}, which reads
\begin{align*}
&\Vert wf\Vert_\infty G(x')\frac{\nabla_x f(x',v)-\nabla_x f(y',v) }{|x-y|^\beta} \\
   &+ \Vert wf\Vert_\infty\int_{\mathbb{R}^3}\mathbf{k}_\varrho(v,u)\frac{G(x')|\nabla_x f(x',u)-\nabla_x f(y',u)|}{|x-y|^\beta}  .
\end{align*}
Since $w_{\tilde{\theta}}^{-1}(v)|v|\lesssim w_{\tilde{\theta}/2}^{-1}(v)$, we rewrite
\begin{align*}
   &G(x')\frac{\nabla_x f(x',v)-\nabla_x f(y',v) }{|x-y|^\beta}
     \\
   & =\frac{\nabla_\parallel f(x',v)-\nabla_\parallel f(y',v)}{|x-y|^\beta}+\frac{\Big[G(x')-G(y')\Big]\nabla_x f(y',v)}{|x-y|^\beta}\\
   &\lesssim \frac{\nabla_\parallel f(x',v)-\nabla_\parallel f(y',v)}{|x-y|^\beta}+\frac{w^{-1}_{\tilde{\theta}}(v)|v|\Vert w_{\tilde{\theta}} \alpha\nabla_x f\Vert_\infty}{|v|\alpha(y',v)}\\
   &\lesssim \frac{[\nabla_\parallel f(\cdot ,v)]_{C_{x,1+\beta}^{0,\beta}}}{w_{\tilde{\theta}/2}(v)|v|^2\min\{\frac{\alpha(x',v)}{|v|},\frac{\alpha(y',v)}{|v|}\}^{1+\beta}}+\frac{\Vert w_{\tilde{\theta}} \alpha\nabla_x f\Vert_\infty}{w_{\tilde{\theta}/2}(v)|v|^2\min\{\frac{\alpha(x',v)}{|v|},\frac{\alpha(y',v)}{|v|}\}^{1+\beta}}.
\end{align*}
Thus applying~\eqref{k_theta} and~\eqref{integrate alpha beta} in Lemma \ref{Lemma: NLN} with $p=1+\beta$ we obtain
\begin{align*}
 \int^t_0 e^{-\nu(v)(t-s)}G(x')\eqref{Gamma-Gamma1} &\lesssim \Vert wf\Vert_\infty\frac{[\nabla_\parallel f(\cdot ,v)]_{C_{x,1+\beta}^{0,\beta}}+\Vert w_{\tilde{\theta}}\alpha\nabla_x f \Vert_\infty}{w_{\tilde{\theta}/2}(v)|v|^2\min\{\frac{\alpha(x,v)}{|v|},\frac{\alpha(y,v)}{|v|}\}^{1+\beta}} \\
  & +\frac{\Vert wf\Vert_\infty}{w_{\tilde{\theta}/2}(v)} \int^t_0 e^{-\nu(v)(t-s)}\int_{\mathbb{R}^3}w_{\tilde{\theta}/2}(v)\mathbf{k}_\varrho(v,u)\frac{[\nabla_\parallel f(\cdot ,v)]_{C_{x,1+\beta}^{0,\beta}}+\Vert w_{\tilde{\theta}}\alpha\nabla_x f \Vert_\infty}{w_{\tilde{\theta}/2}(u)|u|^2\min\{\frac{\alpha(x',u)}{|u|},\frac{\alpha(y',u)}{|u|}\}^{1+\beta}} \\
  & \lesssim \Vert wf\Vert_\infty\frac{[\nabla_\parallel f(\cdot ,v)]_{C_{x,1+\beta}^{0,\beta}}+\Vert w_{\tilde{\theta}}\alpha\nabla_x f \Vert_\infty}{w_{\tilde{\theta}/2}(v)|v|^2\min\{\frac{\alpha(x,v)}{|v|},\frac{\alpha(y,v)}{|v|}\}^{1+\beta}}.
\end{align*}

Finally by $\Vert wf\Vert_\infty\ll 1$ from the existence theorem we conclude the lemma.

\end{proof}

We need one more lemma before the proof. In the following lemma we express the difference quotient in~\eqref{C1beta norm},\eqref{C1beta tang norm} along the trajectory. This lemma will significantly simplify the proof of Proposition \ref{Prop: C1beta}.

\begin{lemma}\label{Lemma: intermediate estimate}
Suppose $f$ solves inhomogenenous steady transport equation with the diffuse BC. Then
\begin{align}
& \frac{\partial_{x_i} f(x,v)-\partial_{x_i} f(y,v)}{|x-y|^\beta} \notag\\
&=\frac{O(1)\Vert w_{\tilde{\theta}}\alpha\nabla_x f\Vert_\infty+o(1)[\nabla_{x}f(\cdot,v)]_{C^{0,\beta}_{x;2+\beta}}}{w_{\tilde{\theta}}(v)|v|^2\min\{\frac{\alpha(x,v)}{|v|},\frac{\alpha(y,v)}{|v|}\}^{2+\beta}} \label{C1beta inter:1} \\ &+\frac{O(1)}{w_{\tilde{\theta}}(v)|v|^2\min\{\frac{\alpha(x,v)}{|v|},\frac{\alpha(y,v)}{|v|}\}}\frac{|\xb(x,v)-\xb(y,v)|^\beta}{|x-y|^\beta}\notag\\
&\times w_{\tilde{\theta}}(v)|v|^2\frac{|\partial_{\mathbf{x}_{p^1(x),i}^1} f(\xb(x,v),v)-\partial_{\mathbf{x}_{p^1(y),i}^1}f(\xb(y,v),v)|}{|\xb(x,v)-\xb(y,v)|^\beta}   \label{C1beta inter:4}\\
& +\int_{\max\{0,t-\min\{\tb(x,v),\tb(y,v)\}\}}^t \frac{e^{-\nu (t-s)}[\partial_{x_i} h(x-(t-s)v,v)-\partial_{x_i} h(y-(t-s)v,v)]}{|x-y|^\beta}.\label{C1beta inter:3}
\end{align}

We denote $[G(x)\nabla_x f(x,v)]_i$ as the $i$-th element of $G(x)\nabla_x f$, then
\begin{align}
& \frac{[G(x) \nabla_x f(x,v)]_i-[G(y) \nabla_x f(y,v)]_i}{|x-y|^\beta} \notag\\
&=\frac{\tilde{\alpha}(x,v)\times \eqref{C1beta inter:3}}{|v|} \label{C1betatang inter:0}\\
&+\frac{O(1)\Vert w_{\tilde{\theta}}\alpha\nabla_x f\Vert_\infty+o(1)[\nabla_{x_\parallel}f(\cdot,v)]_{C^{0,\beta}_{x;1+\beta}}+o(1)[\nabla_{x}f(\cdot,v)]_{C^{0,\beta}_{x;2+\beta}}}{w_{\tilde{\theta}/2}(v)|v|^2\min\{\frac{\alpha(x,v)}{|v|},\frac{\alpha(y,v)}{|v|}\}^{1+\beta}} \label{C1betatang inter:1}\\
 & +\int_{\max\{0,t-\min\{\tb(x,v),\tb(y,v)\}\}}^t \notag\\
 &\frac{e^{-\nu (t-s)}\big[[G(x-(t-s)v)\nabla_x h(x-(t-s)v,v)]_i-[G(y-(t-s)v)\nabla_x h(y-(t-s)v,v)]_i\big]}{|x-y|^\beta}\label{C1betatang inter:3}\\
& +O(1)\frac{|\xb(x,v)-\xb(y,v)|^\beta}{|x-y|^\beta}\frac{|\partial_{\mathbf{x}_{p^1(x),i}^1} f(\xb(x,v),v)-\partial_{\mathbf{x}_{p^1(y),i}^1}f(\xb(y,v),v)|}{|\xb(x,v)-\xb(y,v)|^\beta}   \label{C1betatang inter:4}.
\end{align}

\end{lemma}

\begin{proof}

For $f$ satisfying~\eqref{inhomo_transport}, different to~\eqref{fx_x_1}-\eqref{fx_x_5}, we use $\min\{\tb(x,v),\tb(y,v)\}$ to split the cases. For simplicity, denote $t_m(v)=\min\{\tb(x,v),\tb(y,v)\}$. We express $\nabla_x f(x,v)$ along the trajectory as:
\begin{align}
 \partial_{x_i}
f( x,v)
  & = \mathbf{1}_{t\geq t_m(v)}  \sum_{i=1,2}e^{- \nu \tb}   \partial_x \mathbf{x}_{p^1(x),i}^1 \partial_{\mathbf{x}_{p^1(x),i}^1}  f(\xb(x,v), v)  \label{Hd:fx1_x1_1}\\
 &- \mathbf{1}_{t\geq t_m(v)}  \nu \nabla_x \tb(x,v) e^{- \nu \tb(x,v) } f(\xb(x,v),v)
\label{Hd:fx1_x1_2}\\
 &+
  \mathbf{1}_{t < t_m(v)}
 e^{- \nu t } \partial_{x_i}   f
 (x-tv, v
 )\label{Hd:fx1_x1_3}\\
  &+ \mathbf{1}_{t< t_m(v)}\int^{t}_{0}
 e^{-\nu  (t - s)} \partial_{x_i}
 [h(x-(t-s)v, v
 ) ]\dd s\label{Hd:fx1_x1_4}\\
 &+\mathbf{1}_{t\geq t_m(v)}\int^t_{t-\tb(x,v)}e^{-\nu  (t - s)} \partial_{x_i}
 [h(x-(t-s)v, v
 ) ]\dd s\label{Hd:fx1_x1_5}  \\
 &+\mathbf{1}_{t\geq t_m(v)}\partial_{x_i} t_m(v) e^{-\nu t_m(v)} h(x-t_m(v) v,v)\dd s .\label{Hd:fx1_x1_6}
 \end{align}
Taking the difference of $\partial_{x_i}f(x,v)$ and $\partial_{y_i}f(y,v)$ using~\eqref{Hd:fx1_x1_1}-\eqref{Hd:fx1_x1_6} we obtain
\begin{align}
   &\frac{\partial_{x_i} f(x,v)-\partial_{x_i} f(y,v)}{|x-y|^\beta} \notag\\
   &=   \mathbf{1}_{t\geq t_m(v)}\frac{O(1)\nu |\partial_{x_i} \tb(x,v)-\partial_{x_i} \tb(y,v)|e^{-\nu \tb(x,v)}f(\xb(x,v),v)}{|x-y|^\beta}\label{nablatbx-nablatby}\\
    &+ \mathbf{1}_{t\geq t_m(v)}\frac{\partial_{x_i} \tb(y,v)\nu \big[ e^{-\nu \tb}f(\xb(x,v),v)-e^{-\nu \tb}f(\xb(y,v),v)  \big]}{|\xb(x,v)-\xb(y,v)|}\frac{|\xb(x,v)-\xb(y,v)|}{|x-y|^\beta}\label{gx-gy}\\
   &+  \mathbf{1}_{t\geq t_m(v)}\frac{\partial_{x_i} \tb(y,v)\nu \big[ e^{-\nu \tb}f(\xb(y,v),v)-e^{-\nu \tb(y,v)}f(\xb(y,v),v)  \big]}{|x-y|^\beta}\label{tbx-tby}  \\
   & +\mathbf{1}_{t\geq t_m(v)}\sum_{i=1,2}\frac{\partial_{x_i} \mathbf{x}_{p^1(x),i}^1-\partial_{x_i} \mathbf{x}_{p^1(y),i}^1}{|x-y|^\beta}e^{-\nu(v)\tb(x,v)}\partial_{\mathbf{x}_{p^1(x),i}^1} f(\xb(x,v),v)\label{nablaxbx-nablaxby}\\
   &+ \mathbf{1}_{t\geq t_m(v)}\sum_{i=1,2}\partial_{x_i} \mathbf{x}_{p^1(y),i}^1  \frac{e^{-\nu \tb(x,v)}-e^{-\nu\tb(y,v)}}{|x-y|^\beta}\partial_{\mathbf{x}_{p^1(y),i}^1}f(\xb(y,v),v)     \label{tbx-tby nablag}  \\
   &+\mathbf{1}_{t< t_m(v)}\frac{e^{-\nu t}[\partial_{x_i} f(x-tv,v)-\partial_{x_i} f(y-tv,v)]}{|x-tv-(y-tv)|^\beta}\label{nu t}\\
   &+ \mathbf{1}_{t\geq t_m(v)}\sum_{i=1,2}e^{-\nu(v)\tb(x,v)}|\partial_{x_i} \mathbf{x}_{p^1(x),i}^1| \frac{|\xb(x,v)-\xb(y,v)|^\beta}{|x-y|^\beta}\frac{|\partial_{\mathbf{x}_{p^1,i}^1} f(\xb(x,v),v)-\partial_{\mathbf{x}_{p^1,i}^1}f(\xb(y,v),v)|}{|\xb(x,v)-\xb(y,v)|^\beta} \label{nablagx-nablagy}\\
   &+   \mathbf{1}_{t< t_m(v)}  \int_{0}^t \frac{e^{-\nu(t-s)} \big[\partial_{x_i} h(x-(t-s)v,v)-\partial_{x_i} h(y-(t-s)v,v) \big]}{|x-y|^\beta}   \label{h-h deri}                \\
   &+ \mathbf{1}_{t\geq t_m(v)}  \int_{t-t_m(v)}^t \frac{e^{-\nu(t-s)} \big[\partial_{x_i} h(x-(t-s)v,v)-\partial_{x_i} h(y-(t-s)v,v) \big]}{|x-y|^\beta}    \label{h-h deri tb}\\
   &+\mathbf{1}_{t\geq t_m(v)}   \int^{t-t_m(v)}_{\min\{t-\tb(x,v),t-\tb(y,v)\}} \Big[  \mathbf{1}_{\tb(x,v)>\tb(y,v)}\frac{e^{-\nu(t-s)}\partial_{x_i} h(x-(t-s)v,v)}{|x-y|^\beta}\notag\\
   &+\mathbf{1}_{\tb(y,v)>\tb(x,v)}\frac{e^{-\nu(t-s)}\partial_{x_i} h(y-(t-s)v,v)}{|x-y|^\beta}\Big]\label{h itself}\\
   &  +\mathbf{1}_{t\geq t_m(v)}\partial_{x_i} \tb(x,v)e^{-\nu(v)\tb(x,v)}\frac{h(x-\tb(x,v)v,v)-h(y-\tb(y,v)v,v)}{|x-y|^\beta}
     \label{partial h-h}\\
   &+ \mathbf{1}_{t\geq t_m(v)}\frac{\partial_{x_i} \tb(x,v)-\partial_{x_i} \tb(y,v)}{|x-y|^\beta} e^{-\nu(v)\tb(x,v)} h(y-\tb(y,v)v,v)\label{partial nabla-nabla}\\
   &+\mathbf{1}_{t\geq t_m(v)}\frac{e^{-\nu(v)\tb(x,v)}-e^{-\nu(v)\tb(y,v)}}{|x-y|^\beta}\partial_{x_i} \tb(y,v)h(y-\tb(y,v),v)   .  \label{partial e-e}
\end{align}

First we estimate~\eqref{nablatbx-nablatby}-\eqref{nu t}. For ~\eqref{nablatbx-nablatby} we apply~\eqref{min: nabla tb}; for~\eqref{gx-gy} we apply \eqref{nabla_tbxb},\eqref{min: f},\eqref{min: xb} and \eqref{simplicity}; for~\eqref{tbx-tby} we apply \eqref{nabla_tbxb} and \eqref{min: tb}, then we obtain
\begin{equation}\label{Hd: bound for nablatbx-nablatby}
\eqref{nablatbx-nablatby}= \frac{O(1)w_{\tilde{\theta}}(v)|v|f(\xb(x,v),v)}{w_{\tilde{\theta}}|v|^2\min\{\frac{\alpha(x,v)}{|v|},\frac{\alpha(y,v)}{|v|}\}^{2+\beta}}=\frac{O(1)\Vert wf\Vert_\infty}{w_{\tilde{\theta}}(v)|v|^2\min\{\frac{\alpha(x,v)}{|v|},\frac{\alpha(y,v)}{|v|}\}^{2+\beta}},
\end{equation}

\begin{equation}\label{Hd: bound of gx-gy}
  \begin{split}
\eqref{gx-gy}   &=O(1)\partial_{x_i} \tb(x,v)\frac{f(\xb(x,v),v)-f(\xb(y,v),v)}{|\xb(x,v)-\xb(y,v)|}\frac{|\xb(x,v)-\xb(y,v)|}{|x-y|^\beta}\\
& = \frac{O(1)}{\min\{\alpha(x,v),\alpha(y,v)\}} \frac{\Vert w_{\tilde{\theta}}(v) \alpha \nabla_x f\Vert_\infty}{w_{\tilde{\theta}}(v)\min\{\alpha(x,v),\alpha(y,v)\}} \frac{1}{\min\{\frac{\alpha(x,v)}{|v|},\frac{\alpha(y,v)}{|v|}\}^{\beta}}   \\
&= \frac{O(1)\Vert \alpha\nabla_x f\Vert_\infty}{w_{\tilde{\theta}}(v)|v|^2\min\{\frac{\alpha(x,v)}{|v|},\frac{\alpha(y,v)}{|v|}\}^{2+\beta}},
  \end{split}
\end{equation}

\begin{align}
 \eqref{tbx-tby}& =     \frac{O(1)  \Vert w_{\tilde{\theta}}f\Vert_\infty}{w_{\tilde{\theta}}(v)\alpha(x,v)}                 \frac{1}{|v|\min\{\frac{\alpha(x,v)}{|v|},\frac{\alpha(y,v)}{|v|}\}^{\beta}}=\frac{O(1)\Vert wf\Vert_\infty}{w_{\tilde{\theta}}(v)|v|^2\min\{\frac{\alpha(x,v)}{|v|},\frac{\alpha(y,v)}{|v|}\}^{1+\beta}}
  \label{Hd: bound for tbx-tby inte} \\
   &   =\frac{O(1)\Vert wf\Vert_\infty}{w_{\tilde{\theta}}(v)|v|^2\min\{\frac{\alpha(x,v)}{|v|},\frac{\alpha(y,v)}{|v|}\}^{2+\beta}}
   \label{Hd: bound for tbx-tby}    .
\end{align}

From section 4,
\[ \partial_{\mathbf{x}_{p^1(x),i}^1} f(\eta_{p^1}(\mathbf{x}_{p^1(x)}^1),v)=\eqref{fBD_x}+\eqref{fBD_xr}=O(1)\frac{M_W(\eta_{p^1(x)}(\mathbf{x}_{p^{1}(x)}^{1}),v)}{\sqrt{\mu(v)}}\Vert \alpha\nabla_x f\Vert_\infty, \]
thus with $\tilde{\theta}\ll 1$,
\[\max\{\big||v|\partial_{\mathbf{x}_{p^1(x),i}^1} f(\eta_{p^1(x)}(\mathbf{x}_{p^1(x)}^1),v)\big|, \big||v|^2\partial_{\mathbf{x}_{p^1(x),i}^1} f(\eta_{p^1(x)}(\mathbf{x}_{p^1(x)}^1),v)\big| \}=O(1) \frac{\Vert \alpha\nabla_x f\Vert_\infty}{w_{\tilde{\theta}}(v)}.\]

Then for~\eqref{nablaxbx-nablaxby} we apply~\eqref{min: nabla xb}; for~\eqref{tbx-tby nablag} we apply~\eqref{nabla_tbxb} and \eqref{min: tb}, then we obtain

\begin{equation}\label{Hd: bound for nablaxbx-nablaxby}
\eqref{nablaxbx-nablaxby}= \frac{O(1) |v|^2 \partial_{\mathbf{x}_{p^1,i}^1}f(\xb(x,v),v)}{|v|^2\min\{\frac{\alpha(x,v)}{|v|},\frac{\alpha(y,v)}{|v|}\}^{2+\beta}}=\frac{O(1)\Vert \alpha\nabla_x f\Vert_\infty}{w_{\tilde{\theta}}(v)|v|^2\min\{\frac{\alpha(x,v)}{|v|},\frac{\alpha(y,v)}{|v|}\}^{2+\beta}},
\end{equation}

\begin{align}
\eqref{tbx-tby nablag}&=    \frac{O(1)|v|\partial_{\mathbf{x}_{p^1,i}^1}f(\xb(x,v),v)}{|v|^2\min\{\frac{\alpha(x,v)}{|v|},\frac{\alpha(x,v)}{|v|}\}^{1+\beta}}=\frac{O(1)\Vert \alpha\nabla_x f\Vert_\infty}{w_{\tilde{\theta}}(v)|v|^2\min\{\frac{\alpha(x,v)}{|v|},\frac{\alpha(x,v)}{|v|}\}^{1+\beta}} \label{Hd: bound for tbx-tby nablag inte}\\
&=\frac{O(1)\Vert \alpha\nabla_x f\Vert_\infty}{w_{\tilde{\theta}}(v)|v|^2\min\{\frac{\alpha(x,v)}{|v|},\frac{\alpha(x,v)}{|v|}\}^{2+\beta}}
 \label{Hd: bound for tbx-tby nablag}  ,
\end{align}

\begin{equation}\label{Hd: bound for nu t}
\eqref{nu t}= O(1)\frac{e^{-\nu t}}{w_{\tilde{\theta}}(v)|v|^2\min\{\frac{\alpha(x,v)}{|v|},\frac{\alpha(y,v)}{|v|}\}^{2+\beta}}[\nabla_x f(\cdot ,v)]_{C_{x,2+\beta}^{0,\beta}}.
\end{equation}

Therefore, from~\eqref{simplicity} and $t\gg 1$, we conclude
\[\eqref{nablatbx-nablatby}+\cdots+\eqref{nu t}=\eqref{C1beta inter:1}.\]

For~\eqref{nablagx-nablagy}, from~\eqref{xi deri xbp}, such contribution is included in~\eqref{C1beta inter:4}.

The contribution of~\eqref{h-h deri} and~\eqref{h-h deri tb} are already included in~\eqref{C1beta inter:3}.

Then we estimate~\eqref{h itself}-\eqref{partial e-e}. We apply~\eqref{nabla h bounded} to~\eqref{h itself} to have
\begin{align}
  \eqref{h itself} & =   O(1)\frac{\Vert w_{\tilde{\theta}}\alpha\nabla_x f\Vert_\infty}{w_{\tilde{\theta}}(v)}\int_{\min\{t-\tb(x,v),t-\tb(y,v)\}}^{\max\{t-\tb(x,v),t-\tb(y,v)\}} \int_{\mathbb{R}^3} \frac{e^{-\nu(t-s)}\mathbf{k}_{\tilde{\varrho}}(v,u)}{\alpha(x-(t-s)v,u)|x-y|^\beta} \notag           \\
   & =   \frac{O(1)\Vert w_{\tilde{\theta}}\alpha\nabla_x f\Vert_\infty}{w_{\tilde{\theta}}(v)\alpha(x,v)} \frac{|e^{-C_1\nu\tb(x,v)}-e^{-C_1\nu\tb(y,v)}|^\beta}{|x-y|^\beta}  =    \frac{O(1)\Vert w_{\tilde{\theta}}\alpha\nabla_x f\Vert_\infty}{w_{\tilde{\theta}}(v)|v|^2\min\{\frac{\alpha(x,v)}{|v|},\frac{\alpha(y,v)}{|v|}\}^{1+\beta}}
   \label{Hd: bound for h itself inte}\\
   & = \frac{O(1)\Vert w_{\tilde{\theta}}\alpha\nabla_x f\Vert_\infty}{w_{\tilde{\theta}}(v)|v|^2\min\{\frac{\alpha(x,v)}{|v|},\frac{\alpha(y,v)}{|v|}\}^{2+\beta}}, \label{Hd: bound for h itself}
\end{align}
where we have applied~\eqref{NLN general} of Lemma \ref{Lemma: NLN} in the first equality of the second line, \eqref{min: tb} in the second equality.

Then for~\eqref{partial h-h} we apply~\eqref{nabla_tbxb},\eqref{min: xb} and \eqref{h-h bounded}; for~\eqref{partial nabla-nabla} we apply~\eqref{min: nabla tb} and \eqref{h bounded}; for~\eqref{partial e-e} we apply~\eqref{min: tb} and \eqref{h bounded}, then we obtain

\begin{align}
  \eqref{partial h-h} &=O(1) \frac{1}{\alpha(x,v)}\frac{h(\xb(x,v),v)-h(\xb(y,v),v)}{|\xb(x,v)-\xb(y,v)|^\beta}\frac{|\xb(x,v)-\xb(y,v)|^\beta}{|x-y|^\beta}  \notag\\
   & = \frac{O(1)\Vert w_{\tilde{\theta}}\alpha\nabla_x f\Vert_\infty}{w_{\tilde{\theta}}(v)\min\{\alpha(x,v),\alpha(y,v)\}^{2}}\frac{1}{\min\{\frac{\alpha(x,v)}{|v|},\frac{\alpha(y,v)}{|v|}\}^\beta}= \frac{O(1)\Vert w_{\tilde{\theta}}\alpha\nabla_x f\Vert_\infty}{w_{\tilde{\theta}}(v)|v|^2\min\{\alpha(x,v),\alpha(y,v)\}^{2+\beta}} ,\label{Hd: bound for partial h-h}
\end{align}

\begin{align}
  \eqref{partial nabla-nabla} &=O(1)\frac{\Vert w_{\tilde{\theta}}|v|h\Vert_\infty}{w_{\tilde{\theta}}(v)|v|} \frac{\nabla_x \tb(x,v)-\nabla_x \tb(y,v)}{|x-y|^\beta}= \frac{O(1)\Vert wf\Vert_\infty}{w_{\tilde{\theta}}(v)|v|^2\min\{\alpha(x,v),\alpha(y,v)\}^{2+\beta}} ,\label{Hd: bound for partial nabla-nabla}
\end{align}

\begin{align}
  \eqref{partial e-e} &=\frac{O(1)\Vert w_{\tilde{\theta}}h\Vert_\infty}{w_{\tilde{\theta}}(v)\alpha(y,v)} \frac{e^{-\nu \tb(x,v)}-e^{-\nu \tb(y,v)}}{|x-y|^\beta}= \frac{O(1)\Vert wf\Vert_\infty}{w_{\tilde{\theta}}(v)|v|^2\min\{\alpha(x,v),\alpha(y,v)\}^{1+\beta}}\label{Hd: bound for partial e-e inte} \\
  &=\frac{O(1)\Vert wf\Vert_\infty}{w_{\tilde{\theta}}(v)|v|^2\min\{\alpha(x,v),\alpha(y,v)\}^{2+\beta}}.
  \label{Hd: bound for partial e-e}
\end{align}

Therefore, by~\eqref{simplicity}, collecting~\eqref{Hd: bound for h itself},\eqref{Hd: bound for partial h-h},\eqref{Hd: bound for partial nabla-nabla},\eqref{Hd: bound for partial e-e} we conclude
\[\eqref{h itself}+\cdots+\eqref{partial e-e}=\eqref{C1beta inter:1}.\]

Then we prove the estimate~\eqref{C1betatang inter:0}-\eqref{C1betatang inter:4}.

We rewrite
\begin{align}
    \frac{    G(x) \nabla_x f(x,v)-G(y)\nabla_x f(y,v)   }{|x-y|^\beta} &=  \frac{G(y)-G(x)}{|x-y|^\beta} \nabla_x f(x,v) \label{tang: nx-ny}\\
   &  + G(y) \frac{\nabla_x f(x,v)-\nabla_x f(y,v)}{|x-y|^\beta} \label{tang: n f-f}.
\end{align}
By~\eqref{min: nxb} we conclude that
\begin{align}
\eqref{tang: nx-ny}_i &= O(1)  \frac{\Vert w_{\tilde{\theta}}\alpha \nabla_x f\Vert_\infty}{w_{\tilde{\theta}}(v)\alpha(x,v)}.
  \label{tang: bound for 1}
\end{align}

Then we consider~\eqref{tang: n f-f}. Note that $\eqref{nablatbx-nablatby}-\eqref{partial e-e}$ represent the $i$-th component of $\nabla_x f(x,v)-\nabla_x f(y,v)$, for convenience we define a notation thats represent the vector consists of the element $\eqref{nablatbx-nablatby}:$
\[[\eqref{nablatbx-nablatby}]=[\eqref{nablatbx-nablatby}_{i=1},\eqref{nablatbx-nablatby}_{i=2},\eqref{nablatbx-nablatby}_{i=3}].\]
Similarly we can define the same notation for $\eqref{gx-gy}-\eqref{partial e-e}$. We can use this representation to express $G(y)\frac{\nabla_x f(x,v)-\nabla_x f(y,v)}{|x-y|^\beta}$.

Then for $G(y)[\eqref{nablatbx-nablatby}]$ we apply~\eqref{min: I nabla tb}; for $G(y)[\eqref{gx-gy}]$ we apply \eqref{tang nabla tb},\eqref{min: xb},\eqref{min: f} and \eqref{simplicity}; for $G(y)[\eqref{tbx-tby}]$ we apply $|G|=O(1)$; for $G(y)[\eqref{nablaxbx-nablaxby}]$ we apply \eqref{min: I nabla xb} and \eqref{equivalent}; for $G(y)[\eqref{tbx-tby nablag}]$ we apply $|G|=O(1)$; for $G(y)[\eqref{nablagx-nablagy}]$ we apply \eqref{tang nabla xb}, then we obtain
\begin{equation}\label{tang: bound for nablatbx-nablatby}
[G(y)\eqref{nablatbx-nablatby}]_i=O(1) \frac{\Vert w_{\tilde{\theta}}|v|f\Vert_\infty }{w_{\tilde{\theta}}(v)|v|^2\min\{\frac{\alpha(x,v)}{|v|},\frac{\alpha(y,v)}{|v|}\}^{1+\beta}}=\frac{O(1)\Vert wf\Vert_\infty}{w_{\tilde{\theta}}(v)|v|^2\min\{\frac{\alpha(x,v)}{|v|},\frac{\alpha(y,v)}{|v|}\}^{1+\beta}},
\end{equation}

\begin{equation}\label{tang: bound for gx-gy}
[G(y)[\eqref{gx-gy}]]_i= O(1) \frac{\Vert w_{\tilde{\theta}}\alpha \nabla_x f\Vert_\infty}{w_{\tilde{\theta}}(v)|v|^2 \min\{\frac{\alpha(x,v)}{|v|},\frac{\alpha(y,v)}{|v|}\}^{1+\beta}},
\end{equation}

\begin{equation}\label{tang: bound for tbx-tby}
[G(y)[\eqref{tbx-tby}]]_i=O(1)\eqref{Hd: bound for tbx-tby inte} =O(1) \frac{\Vert wf\Vert_\infty}{w_{\tilde{\theta}}(v)|v|^2 \min\{\frac{\alpha(x,v)}{|v|},\frac{\alpha(y,v)}{|v|}\}^{1+\beta}},
\end{equation}

\begin{equation}\label{tang: bound for nablaxbx-nablaxby}
 [G(y)[\eqref{nablaxbx-nablaxby}]]_i= O(1)\frac{\Vert w_{\tilde{\theta}}(v)|v|^2 \partial_{\mathbf{x}_{p^1,i}^1}f \Vert_\infty}{|v|^2\min\{\frac{\alpha(x,v)}{|v|},\frac{\alpha(y,v)}{|v|}\}^{1+\beta}}= O(1)\frac{\Vert wf\Vert_\infty}{w_{\tilde{\theta}}(v)|v|^2\min\{\frac{\alpha(x,v)}{|v|},\frac{\alpha(y,v)}{|v|}\}^{1+\beta}},
\end{equation}

\begin{equation}\label{tang: bound for tbx-tby nablag}
 [G(y)[\eqref{tbx-tby nablag}]]_i=O(1)\eqref{Hd: bound for tbx-tby nablag inte}=O(1) \frac{\Vert wf\Vert_\infty}{w_{\tilde{\theta}}(v)|v|^2\min\{\frac{\alpha(x,v)}{|v|},\frac{\alpha(y,v)}{|v|^2}\}^{1+\beta}},
\end{equation}

\[[G(y) [\eqref{nablagx-nablagy}]]_i=\eqref{C1betatang inter:3}.\]

For $[G(y)[\eqref{nu t}]]_i$, we rewrite
\begin{align}
& G(y)\frac{\nabla_x f(x-tv,v)-\nabla_x f(y-tv,v)}{|x-y|^\beta} \notag\\
&= \underbrace{\frac{G(x)\nabla_x f(x-tv,v)-G(y)\nabla_x f(y-tv,v)}{|x-y|^\beta}}_{\eqref{rewrite Gnablaxf-nablayf}_1}+\underbrace{\frac{\big[G(y)-G(x)\big]\nabla_x f(x-tv,v)}{|x-y|^\beta}}_{\eqref{rewrite Gnablaxf-nablayf}_2} \label{rewrite Gnablaxf-nablayf}.
\end{align}

We apply~\eqref{min: nxb} to have
\[[\eqref{rewrite Gnablaxf-nablayf}_2]_i=O(1)\frac{\Vert |v|^2 w_{\tilde{\theta}/2}\alpha\nabla_x f\Vert_\infty}{w_{\tilde{\theta}/2}(v)|v|^2 }=\frac{O(1)\Vert w_{\tilde{\theta}}\alpha\nabla_x f\Vert_\infty}{w_{\tilde{\theta}/2}(v)|v|^2\min\{\frac{\alpha(x,v)}{|v|^2},\frac{\alpha(y,v)}{|v|}\}^{1+\beta}}.\]

For$~\eqref{rewrite Gnablaxf-nablayf}_1$ we apply $\eqref{Gf-Gf 1}$ and conclude
\begin{align*}
  |\eqref{rewrite Gnablaxf-nablayf}_1| & =O(1) \frac{[\nabla_{x_\parallel}f(\cdot,v)]_{C^{0,\beta}_{x;1+\beta}}}{w_{\tilde{\theta}/2}(v)|v|^2\min\{\frac{\alpha(x,v)}{|v|},\frac{\alpha(y,v)}{|v|}\}^{1+\beta}}     +\frac{\frac{\tilde{\alpha}(x,v)}{|v|}[\nabla_{x}f(\cdot,v)]_{C^{0,\beta}_{x;2+\beta}}}{w_{\tilde{\theta}}(v)|v|^2\min\{\frac{\alpha(x,v)}{|v|},\frac{\alpha(y,v)}{|v|}\}^{2+\beta}}          \\
   &+\frac{|v|O(1)\Vert w_{\tilde{\theta}}\alpha\nabla_x f\Vert_\infty}{w_{\tilde{\theta}}(v)|v|^2 \min\{\frac{\alpha(x,v)}{|v|},\frac{\alpha(y,v)}{|v|}\}^{1+\beta}}   \\
   & =O(1)\frac{[\nabla_{x_\parallel}f(\cdot,v)]_{C^{0,\beta}_{x;1+\beta}}+[\nabla_{x}f(\cdot,v)]_{C^{0,\beta}_{x;2+\beta}}+\Vert w_{\tilde{\theta}} \alpha\nabla_x f\Vert_\infty}{w_{\tilde{\theta}/2}(v)|v|^2\min\{\frac{\alpha(x,v)}{|v|},\frac{\alpha(y,v)}{|v|}\}^{1+\beta}},
\end{align*}
where we have used $\tilde{\alpha}(v)w^{-1}_{\tilde{\theta}}(v)\lesssim \alpha(v)w^{-1}_{\tilde{\theta}/2}(v), |v|w^{-1}_{\tilde{\theta}}(v)\lesssim w^{-1}_{\tilde{\theta}/2}(v)$. Thus with $e^{-\nu t}\ll 1$ when $t\gg 1$,
\begin{equation}\label{tang: bound for nu t}
 [G(y) \eqref{nu t}]_i= o(1)\frac{\Vert w_{\tilde{\theta}}\alpha\nabla_x f\Vert_\infty+[\nabla_{x_\parallel}f(\cdot,v)]_{C^{0,\beta}_{x;1+\beta}}+[\nabla_{x}f(\cdot,v)]_{C^{0,\beta}_{x;2+\beta}}}{w_{\tilde{\theta}/2}(v)|v|^2\min\{\frac{\alpha(x,v)}{|v|},\frac{\alpha(y,v)}{|v|}\}^{1+\beta}},
\end{equation}
which is already included in~\eqref{C1betatang inter:1}.

Then we estimate $[G(y)\eqref{h-h deri}]_i$. We rewrite
\begin{align}
&\frac{G(y)[\nabla_x h(x-(t-s)v)-\nabla_x h(y-(t-s)v)]}{|x-y|^\beta} \notag\\
&=\underbrace{\frac{G(x)\nabla_x h(x-(t-s)v)-G(y)\nabla_x h(y-(t-s)v)}{|x-y|^\beta}}_{\eqref{Gxh-Gyh}_1}+\underbrace{\frac{G(y)-G(x)}{|x-y|^\beta}\nabla_x h(x-(t-s)v)}_{\eqref{Gxh-Gyh}_2}. \label{Gxh-Gyh}
\end{align}

We bound $w^{-1}_{\tilde{\theta}}(v)|v|\lesssim w_{\tilde{\theta}/2}^{-1}(v).$ From~\eqref{min: nxb} and~\eqref{nabla h bounded} the contribution of$~\eqref{Gxh-Gyh}_2$ in $[G(y)[\eqref{h-h deri}]]_i$ is
\begin{align*}
   &  O(1)\Vert \xi\Vert_{C^2} \frac{\Vert w_{\tilde{\theta}}\alpha\nabla_x f\Vert_\infty}{w_{\tilde{\theta}}(v)} \int_0^t \int_{\mathbb{R}^3} \frac{e^{-\nu(t-s)}\mathbf{k}_{\tilde{\varrho}}(v,u)}{\alpha(x-(t-s)v,u)}\\
   & =O(1)\Vert \xi\Vert_{C^2}\frac{\Vert w_{\tilde{\theta}}\alpha\nabla_x f\Vert_\infty |v|w^{-1}_{\tilde{\theta}}(v) }{|v|\alpha(x,v)}=O(1)\Vert \xi\Vert_{C^2}\frac{\Vert w_{\tilde{\theta}}\alpha\nabla_x f\Vert_\infty}{w_{\tilde{\theta}/2}(v)|v|^2\min\{\frac{\alpha(x,v)}{|v|},\frac{\alpha(y,v)}{|v|}\}^{1+\beta}},
\end{align*}
which is included in~\eqref{C1betatang inter:1}.

For$~\eqref{Gxh-Gyh}_1$ we apply~\eqref{Gf-Gf 1} in Lemma \ref{Lemma: Gf-Gf}. Then such contribution in $[G(y)[\eqref{h-h deri}]]_i$ equals to
\begin{align*}
& \frac{[\nabla_\parallel h(x-(t-s)v,u)]_i-[\nabla_\parallel h(y-(t-s)v,u)]_i}{|x-y|^\beta}\\
& +\frac{\tilde{\alpha}(x,v)}{|v|}\frac{\partial_{x_i} h(x-(t-s)v,v)-\partial_{x_i} h(y-(t-s)v,v)}{|x-y|^\beta}+\frac{\Vert w_{\tilde{\theta}} \alpha\nabla_x f\Vert_\infty}{w_{\tilde{\theta}/2}(v)|v|\min\{\frac{\alpha(x,v)}{|v|},\frac{\alpha(y,v)}{|v|}\}^{1+\beta}},
\end{align*}
which are included in~\eqref{C1betatang inter:0},\eqref{C1betatang inter:1} and~\eqref{C1betatang inter:4} respectively.

Last we estimate~\eqref{h itself}-\eqref{partial e-e}. For~\eqref{h itself} we apply $|G(y)|=O(1)$;
for~\eqref{partial h-h} we apply \eqref{min: xb},\eqref{h-h bounded} and \eqref{min: I nabla tb}; for~\eqref{partial nabla-nabla} we apply \eqref{h bounded},\eqref{nabla_tbxb} and \eqref{min: I nabla tb} ; for~\eqref{partial e-e} we apply $|G|=O(1)$, then we obtain

\begin{equation}\label{tang: bound for h itself}
[G(y) [\eqref{h itself}]]_{i}=O(1)        \eqref{Hd: bound for h itself inte}=\frac{O(1)\Vert w_{\tilde{\theta}}\alpha\nabla_x f\Vert_\infty}{w_{\tilde{\theta}/2}(v)|v|^2\min\{\frac{\alpha(x,v)}{|v|},\frac{\alpha(y,v)}{|v|}\}^{1+\beta}},
\end{equation}

\begin{align}
& [ G(y)[\eqref{partial h-h}] ]_i =\frac{O(1)}{|v|} \frac{h(\xb(x,v),v)-h(\xb(y,v),v)}{|\xb(x,v)-\xb(y,v)|^\beta}\frac{|\xb(x,v)-\xb(y,v)|^\beta}{|x-y|^\beta} \notag \\
  & \lesssim \frac{\Vert w_{\tilde{\theta}}\alpha\nabla_x f\Vert_\infty}{w_{\tilde{\theta}}(v)|v|\min\{\alpha(x,v),\alpha(y,v)\}} \frac{1}{\min\{\frac{\alpha(x,v)}{|v|},\frac{\alpha(y,v)}{|v|}\}^{\beta}}\lesssim \frac{1}{w_{\tilde{\theta}/2}(v)|v|^2\min\{\frac{\alpha(x,v)}{|v|},\frac{\alpha(y,v)}{|v|}\}^{1+\beta}},
   \label{tang: bound for partial h-h}
\end{align}

\begin{align}
  [G(y)[\eqref{partial nabla-nabla}]]_i & = \frac{O(1)|v| \Vert h\Vert_\infty }{w_{\tilde{\theta}}(v)|v|^2\min\{\frac{\alpha(x,v)}{|v|},\frac{\alpha(y,v)}{|v|}\}^{1+\beta}}=\frac{O(1)\Vert wf\Vert_\infty  }{w_{\tilde{\theta}/2}(v)|v|^2\min\{\frac{\alpha(x,v)}{|v|},\frac{\alpha(y,v)}{|v|}\}^{1+\beta}},
  \label{tang: bound for partial nabla-nabla}
\end{align}

\begin{align}
  [G(y)[\eqref{partial e-e}]]_i & =O(1) \eqref{Hd: bound for partial e-e inte} =O(1) \frac{\Vert wf\Vert_\infty}{w_{\tilde{\theta}/2}(v)|v|^2\min\{\frac{\alpha(x,v)}{|v|},\frac{\alpha(y,v)}{|v|}\}^{1+\beta}}.
     \label{tang: bound for partial e-e}
\end{align}

These four estimates are all included in~\eqref{C1betatang inter:1}.

We conclude the lemma.

\end{proof}

Now are ready to prove Proposition \ref{Prop: C1beta}.
\begin{proof}[\textbf{Proof of Proposition \ref{Prop: C1beta}}]
We will use 3 steps to prove this proposition. Since we already have the expression of the difference quotient from Lemma \ref{Lemma: intermediate estimate}, we mainly estimate~\eqref{C1beta inter:3},\eqref{C1beta inter:4} and~\eqref{C1betatang inter:0},\eqref{C1betatang inter:3},\eqref{C1betatang inter:4}. The estimate of~\eqref{C1beta inter:4} is put in Step 1 and the estimate of \eqref{C1beta inter:3} is put in Step 2. Thus Step 1 and Step 2 together conclude~\eqref{Bound of C1beta}. In Step 3 we estimate~\eqref{C1betatang inter:0},\eqref{C1betatang inter:3},\eqref{C1betatang inter:4} and thus conclude~\eqref{Bound of C1beta tangential}.

Before going into these steps we first list some estimates for the $\alpha$-weight. We will heavily rely on these estimates to make the computation more precise. For $0\leq s\leq \tb(\xb(x,v),v^1)$, we have
\begin{equation}\label{fact 1}
\alpha(\xb(x,v),v^1)\thicksim \alpha(\xb(x,v)-sv^1,v^1)\thicksim \alpha(\xb(\xb(x,v),v^1),v^1),
\end{equation}
and
\begin{equation}\label{fact 2}
\begin{split}
|n(\xb(x,v))\cdot v^1|e^{-C|v^1|^2}&=\mathbf{1}_{|n(\xb(x,v))\cdot v^1|\geq 1}e^{-C|v^1|^2}+\mathbf{1}_{|n(\xb(x,v))\cdot v^1|\leq 1}e^{-C|v^1|^2} \\
    & \lesssim \alpha(\xb(x,v),v^1)\frac{|n(\xb(x,v))\cdot v^1|}{\alpha(\xb(x,v),v^1)}e^{-C|v^1|^2}+\alpha(\xb(x,v),v^1)e^{-C|v^1|^2}\\
    &\lesssim \alpha(\xb(x,v),v^1)e^{-C|v^1|^2/2},
\end{split}
\end{equation}
where we have used~\eqref{kinetic_distance} and~\eqref{Velocity_lemma} in the derivation.

Suppose $\alpha(\xb(x,v))\leq \alpha(\xb(y,v))$. We let $\e\ll 1$ such that $\beta+\e<1$ and $\frac{1-\beta}{1-\beta-\e}<2$. Then we apply the H\"{o}lder inequality with $(\beta+\e)+(1-\beta-\e)=1  $ to have
\begin{align}
&   \int_{n(\xb(x,v))\cdot v^1>0}\frac{e^{-C|v^1|^2}|n(\xb(x,v))\cdot v^1|}{|v^1|^{2}\min\{\frac{\alpha(\xb(x,v),v^1)}{|v^1|},\frac{\alpha(\xb(x,v),v^1)}{|v^1|}\}^{1+\beta}} \dd v^1   \notag \\
&= \int_{n(\xb(x,v))\cdot v^1>0}\frac{e^{-C|v^1|^2}}{|v^1|^{1-\beta}\alpha^\beta(\xb(x,v),v^1)}\dd v^1  &  \notag\\
 & \lesssim \Big(\int   \frac{e^{-C|v^1|^2}}{|v^1|^{\frac{1-\beta}{1-\beta-\e}}}  \Big)^{1-\beta-\e}  \Big(\int   \frac{e^{-C|v^1|^2}}{\alpha^{\frac{\beta}{\beta+\e}}(\xb(x,v),v^1)}  \Big)^{\beta+\e}  \lesssim 1\label{fact 3},
 \end{align}
 where we have used
 \begin{align*}
 \int   \frac{e^{-C|v^1|^2}}{\alpha^{\frac{\beta}{\beta+\e}}(\xb(x,v),v^1)}&\lesssim\int_{n(\xb(x,v))\cdot v^1>0}\mathbf{1}_{\alpha(\xb(x,v),v^1)\geq 1}\frac{e^{-C|v^1|^2}}{\alpha^\frac{\beta}{\beta+\e}(\xb(x,v),v^1)}\dd v^1\notag\\
   &+\int_{n(\xb(x,v))\cdot v^1>0}\mathbf{1}_{\alpha(\xb(x,v),v^1)\leq 1}\frac{e^{-C|v^1|^2}}{|n(\xb(x,v))\cdot v^1|^\frac{\beta}{\beta+\e}}\dd v^1\lesssim 1.
\end{align*}
Then we start the proof.

\textbf{Step 1: estimate of~\eqref{C1beta inter:4}.}

We focus on
\begin{equation}\label{Key term in C1beta 4}
w_{\tilde{\theta}}(v)|v|^2\frac{|\partial_{\mathbf{x}_{p^1(x),i}^1} f(\xb(x,v),v)-\partial_{\mathbf{x}_{p^1(y),i}^1}f(\xb(y,v),v)|}{|\xb(x,v)-\xb(y,v)|^\beta} .
\end{equation}

Applying the diffuse boundary condition~\eqref{diffuse_f intro} we get
\begin{equation}\label{Hd: nabla bdr}
  \begin{split}
& \partial_{\mathbf{x}_{p^1(x),i}^1}f(\xb(x,v),v) \notag\\
=  &   \frac{M_W(\xb(x,v),v)}{\sqrt{\mu(v)}}  \int_{\mathbf{v}_{p^1(x),3}^1>0} \bigg[\underbrace{\Big(\partial_{\mathbf{x}_{p^1(x),i}^1}T^t_{\mathbf{x}_{p^1(x)}^1 \mathbf{v}_{p^1}^1}\Big)\cdot \nabla_v f(\eta_{p^1(x)}(\mathbf{x}_{p^1(x)}^1),T^t_{\mathbf{x}_{p^1(x)}^1}\mathbf{v}_{p^1(x)}^1)}_{
      \eqref{Hd: nabla bdr}_1}  \\
     & +     \underbrace{\partial_{\mathbf{x}_{p^1(x),j}^1}f(\eta_{p^1(x)}(\mathbf{x}_{p^1(x)}^1),T^t_{\mathbf{x}_{p^1(x)}^1}\mathbf{v}_{p^1(x)}^1) }_{\eqref{Hd: nabla bdr}_2}  \bigg]      \sqrt{\mu(\mathbf{v}_{p^1(x)}^1)}\mathbf{v}_{p^1(x),3}^1 \dd \mathbf{v}_{p^1(x)}^1    \\
     & +\frac{\partial_{\mathbf{x}_{p^1(x),i}^1} M_W(\xb(x,v),v)}{\sqrt{\mu(v)}}\int_{n(\xb(x,v))\cdot v^1>0}f(\xb(x,v),v^1)\sqrt{\mu(v^1)} \{n(\xb(x,v))\cdot v^1\}\dd v^1 \\
      &+ \partial_{\mathbf{x}_{p^1(x),i}^1}r(\xb(x,v),v)  .
  \end{split}
\end{equation}
From Lemma \ref{Lemma: bc estimate} the contribution of the last two terms of~\eqref{Hd: nabla bdr} in~\eqref{Key term in C1beta 4} is bounded by
\[\Vert T_W-T_0\Vert_{C^2}\Vert \alpha\nabla_x f\Vert_\infty.\]

\textit{Velocity derivative:} first we consider the contribution of$~\eqref{Hd: nabla bdr}_1$ in~\eqref{C1beta inter:4}. From~\eqref{v_under_v}, we rewrite$~\eqref{Hd: nabla bdr}_1$ as
\begin{align}
&\int_{\mathbf{v}_{p^1,3}^1>0} \sum_{m,n} \eqref{v_under_v_mn}_{mn,k+1\to 1}(x) \mathbf{v}_{p^1,m}^1 \partial_{\mathbf{v}_{p^1,n}^1}[f(\eta_{p^1(x)}(\mathbf{x}_{p^1(x)}^1),T^t_{\mathbf{x}_{p^1(x)}^1}\mathbf{v}_{p^1}^1)]\sqrt{\mu(\mathbf{v}_{p^1}^1)}\mathbf{v}_{p^1,3}^1 \dd \mathbf{v}_{p^1}^1 \notag\\
&= \int_{\mathbf{v}_{p^1,3}^1>0} \sum_{m,n} \eqref{v_under_v_mn}_{mn,k+1\to 1}(x) f(\eta_{p^1}(\mathbf{x}_{p^1(x)}^1),T^t_{\mathbf{x}_{p^1(x)}^1}\mathbf{v}_{p^1}^1)\partial_{\mathbf{v}_{p^1,n}^1}\big[\sqrt{\mu(\mathbf{v}_{p^1}^1)}\mathbf{v}_{p^1,3}^1 \mathbf{v}_{p^1,m}^1 \big]\dd \mathbf{v}_{p^1}^1\label{nablav in Hd}  .
\end{align}
Here we dropped the $x$ dependence in $\mathbf{v}_{p^1(x)}^1$ since it becomes a dummy variable.

From~\eqref{IBP_v} we have
\begin{equation}\label{Hd: IBVbound}
\eqref{nablav in Hd}=O(1)\Vert \eta\Vert_{C^2}\Vert wf\Vert_\infty.
\end{equation}

Then the contribution of$~\eqref{Hd: nabla bdr}_1$ in~\eqref{C1beta inter:4} can be written as
\begin{align}
 & \frac{w_{\tilde{\theta}}(v)|v|^2}{|\xb(x,v)-\xb(y,v)|^\beta}\bigg[ \frac{M_W(\xb(x,v),v)-M_W(\xb(y,v),v)}{\sqrt{\mu(v)}}\times \eqref{nablav in Hd} \label{M-M}\\
   & +\frac{M_W(\xb(y,v),v)}{\sqrt{\mu(v)}}\int_{\mathbf{v}_{p^1,3}^1>0}  \big[f(\eta_{p^1}(\mathbf{x}_{p^1(x)}^1),T^t_{\mathbf{x}_{p^1(x)}^1}\mathbf{v}_{p^1}^1) -f(\eta_{p^1}(\mathbf{x}_{p^1(y)}^1),T^t_{\mathbf{x}_{p^1(y)}^1}\mathbf{v}_{p^1}^1) \big] \notag\\
  & \times  \sum_{m,n} \eqref{v_under_v_mn}_{mn,k+1\to 1}(x) \partial_{\mathbf{v}_{p^1,n}^1}\big[\sqrt{\mu(\mathbf{v}_{p^1}^1)}\mathbf{v}_{p^1,3}^1 \mathbf{v}_{p^1,m}^1 \big]  \dd \mathbf{v}_{p^1}^1\label{f-f in nablav}\\
 &  +\frac{M_W(\xb(y,v),v)}{\sqrt{\mu(v)}}\int_{\mathbf{v}_{p^1,3}^1>0}      \sum_{m,n}  \partial_{\mathbf{v}_{p^1,n}^1}\big[\sqrt{\mu(\mathbf{v}_{p^1}^1)}\mathbf{v}_{p^1,3}^1 \mathbf{v}_{p^1,m}^1 \big]    \notag  \\
 &         \times f(\eta_{p^1(y)}(\mathbf{x}_{p^1(y)}^1),T^t_{\mathbf{x}_{p^1(y)}^1}\mathbf{v}_{p^1}^1)\Big[\eqref{v_under_v_mn}_{mn,k+1\to 1}(x)-\eqref{v_under_v_mn}_{mn,k+1\to 1}(y) \Big]\dd v^1  \bigg].             \label{mn-mn}
\end{align}

For~\eqref{M-M}, since $\tilde{\theta}\ll 1$, from~\eqref{Hd: IBVbound} we derive that
\begin{align}
   &\eqref{M-M}\lesssim \Vert \eta\Vert_{C^2}\Vert wf\Vert_\infty \frac{w_{\tilde{\theta}}(v)|v|^2[M_W(\xb(x,v),v)-M_W(\xb(y,v),v)]}{\sqrt{\mu(v)}|\xb(x,v)-\xb(y,v)|^\beta}  \lesssim       \Vert \eta\Vert_{C^2}\Vert wf\Vert_\infty \Vert T_W\Vert_{C^1} .\label{g: bound for M-M}
\end{align}

For~\eqref{f-f in nablav}, from~\eqref{min: f} with~\eqref{simplicity} and \eqref{min: xb} we compute
\begin{align*}
   &\frac{f(\eta_{p^1(x)}(\mathbf{x}_{p^1(x)}^1),T^t_{\mathbf{x}_{p^1(x)}^1}\mathbf{v}_{p^1}^1) -f(\eta_{p^1(y)}(\mathbf{x}_{p^1(y)}^1),T^t_{\mathbf{x}_{p^1(y)}^1}\mathbf{v}_{p^1}^1)}{|\xb(x,v)-\xb(y,v)|^\beta}  \\
   & \lesssim \frac{f(\eta_{p^1(x)}(\mathbf{x}_{p^1(x)}^1),T^t_{\mathbf{x}_{p^1(x)}^1}\mathbf{v}_{p^1}^1) -f(\eta_{p^1(y)}(\mathbf{x}_{p^1(y)}^1),T^t_{\mathbf{x}_{p^1(x)}^1}\mathbf{v}_{p^1}^1)}{|\eta_{p^1(x)}(\mathbf{x}_{p^1(x)}^1)-\eta_{p^1(y)}(\mathbf{x}_{p^1(y)}^1)|^\beta}\\
   & +\frac{f(\eta_{p^1(y)}(\mathbf{x}_{p^1(y)}^1),T^t_{\mathbf{x}_{p^1(x)}^1}\mathbf{v}_{p^1}^1) -f(\eta_{p^1(y)}(\mathbf{x}_{p^1(y)}^1),T^t_{\mathbf{x}_{p^1(y)}^1}\mathbf{v}_{p^1}^1)}{|T^t_{\mathbf{x}_{p^1(x)}^1}\mathbf{v}_{p^1}^1-T^t_{\mathbf{x}_{p^1(y)}^1}\mathbf{v}_{p^1}^1|}
\frac{|T^t_{\mathbf{x}_{p^1(x)}^1}\mathbf{v}_{p^1}^1-T^t_{\mathbf{x}_{p^1(y)}^1}\mathbf{v}_{p^1}^1|}{|\eta_{p^1(x)}(\mathbf{x}_{p^1(x)}^1)-\eta_{p^1(y)}(\mathbf{x}_{p^1(y)}^1)|}\\
   &\lesssim     \frac{\Vert \alpha\nabla_x f\Vert_\infty}{\min\{\alpha(\eta_{p^1(x)}(\mathbf{x}_{p^1(x)}^1),T^t_{\mathbf{x}_{p^1(x)}^1}\mathbf{v}_{p^1}^1),\alpha(\eta_{p^1(y)}(\mathbf{x}_{p^1(y)}^1),T^t_{\mathbf{x}_{p^1(y)}^1}\mathbf{v}_{p^1}^1)\}^\beta}     +     \frac{\Vert |v|^2\nabla_vf\Vert_\infty \Vert \eta\Vert_{C^2} |\mathbf{v}_{p^1}^1|}{|\mathbf{v}_{p^1}^1|^2}             ,
\end{align*}
where we have used the definition of $T_{\mathbf{x}_p}$~\eqref{T} and the mean value theorem regarding $\nabla_v f$ in the last line.

Since $\sum_{m,n} \eqref{v_under_v_mn}_{mn,k+1\to 1}(x) \partial_{\mathbf{v}_{p^1,n}^1}\big[\sqrt{\mu(\mathbf{v}_{p^1}^1)}\mathbf{v}_{p^1,3}^1 \mathbf{v}_{p^1,m}^1 \big] \lesssim 1$, we have
\begin{align}
 \eqref{f-f in nablav}  & \lesssim \Vert\alpha\nabla_x f\Vert_\infty \frac{w_{\tilde{\theta}}(v)|v|^2M_W(\xb(y,v),v)}{\sqrt{\mu(v)}}  \notag\\ & \times \int_{\mathbf{v}_{p^1,3}^1>0}\frac{\mu^{1/4}(\mathbf{v}_{p^1}^1)}{\min\{\alpha(\eta_{p^1(x)}(\mathbf{x}_{p^1(x)}^1),T_{\mathbf{x}^1_{p^1(x)}}^t\mathbf{v}_{p^1}^1),\alpha(\eta_{p^1(y)}(\mathbf{x}_{p^1(y)}^1),T_{\mathbf{x}^1_{p^1(y)}}^t\mathbf{v}_{p^1}^1)\}^\beta} \dd \mathbf{v}_{p^1}^1\notag \\
 & + \Vert \eta\Vert_{C^2}\Vert |v|^2 \nabla_v f\Vert_\infty \frac{w_{\tilde{\theta}}(v)|v|^2M_W(\xb(y,v),v)}{\sqrt{\mu(v)}}   \int_{\mathbf{v}_{p^1,3}^1>0} \frac{\mu^{1/4}(\mathbf{v}_{p^1}^1)}{|\mathbf{v}_{p^1}^1|} \dd \mathbf{v}_{p^1}^1   \notag\\
   &\lesssim \big(\Vert\alpha\nabla_x f\Vert_\infty+\Vert |v|^2\nabla_v f\Vert_\infty\big)\frac{|v|^2M_W(\xb(y,v),v)}{\sqrt{\mu(v)}}\notag\\
   &\times \Big[ \int_{\mathbb{R}^3} \big( \frac{\mu^{1/4}(\mathbf{v}_{p^1}^1)}{|n(\xb(x,v))\cdot T_{\mathbf{x}^1_{p^1(x)}}^t\mathbf{v}_{p^1}^1|^\beta}+\frac{\mu^{1/4}(\mathbf{v}_{p^1}^1)}{|n(\xb(y,v))\cdot T_{\mathbf{x}^1_{p^1(y)}}^t\mathbf{v}_{p^1}^1|^\beta}\dd \mathbf{v}_{p^1}^1\big)+\int_{\mathbb{R}^3} \frac{\mu^{1/4}(\mathbf{v}_{p^1}^1)}{|\mathbf{v}_{p^1}^1|} \mathbf{v}_{p^1}^1\Big]\notag\\
   &\lesssim \Vert \alpha\nabla_x f\Vert_\infty+\Vert |v|^2\nabla_v f\Vert_\infty,\label{g: bound for f-f in nablav}
\end{align}
where we have used $\beta<1$ in the last line.

Last we estimate~\eqref{mn-mn}. From~\eqref{v_under_v_mn} and~\eqref{nabla_tbxb} we compute
\begin{align*}
   & \frac{\eqref{v_under_v_mn}_{mn}(x)-\eqref{v_under_v_mn}_{mn}(y)}{|\eta_{p^1(x)}(\mathbf{x}_{p^1(x)}^1)-\eta_{p^1(y)}(\mathbf{x}_{p^1(y)}^1)|^\beta} \\
   &     \lesssim \Big[\frac{\p }{ \p{\mathbf{x}^{1}_{p^{1}(x), j}}  }\left(
  \frac{\p_m \eta_{p^{1}(x),l} (\mathbf{x}^{1}_{p^{1}(x)})}{\sqrt{g_{p^{1}(x),mm}(\mathbf{x}^{1}_{p^{1}(x)})} }\right) \frac{  \frac{\p_n \eta_{p^{1}(x),l} (\mathbf{x}^{1}_{p^{1}(x)})}{\sqrt{g_{p^{1}(x),nn}(\mathbf{x}^{1}_{p^{1}(x)})} }-  \frac{\p_n \eta_{p^{1}(y),l} (\mathbf{x}^{1}_{p^{1}(y)})}{\sqrt{g_{p^{1}(y),nn}(\mathbf{x}^{1}_{p^{1}(y)})} }}{|\eta_{p^1(x)}(\mathbf{x}_{p^1(x)}^1)-\eta_{p^1(y)}(\mathbf{x}_{p^1}^1(y))|^\beta}\\
  &+   \frac{\p_n \eta_{p^{1}(y),l} (\mathbf{x}^{1}_{p^{1}(y)})}{\sqrt{g_{p^{1}(y),nn}(\mathbf{x}^{1}_{p^{1}(y)})} }\frac{\frac{\p }{ \p{\mathbf{x}^{1}_{p^{1}(x), j}}  }\left(
  \frac{\p_m \eta_{p^{1}(x),l} (\mathbf{x}^{1}_{p^{1}(x)})}{\sqrt{g_{p^{1}(x),mm}(\mathbf{x}^{1}_{p^{1}(x)})} }\right)-\frac{\p }{ \p{\mathbf{x}^{1}_{p^{1}(y), j}}  }\left(
  \frac{\p_m \eta_{p^{1}(y),l} (\mathbf{x}^{1}_{p^{1}(y)})}{\sqrt{g_{p^{1}(y),mm}(\mathbf{x}^{1}_{p^{1}(y)})} }\right)}{|\eta_{p^1(x)}(\mathbf{x}_{p^1(x)}^1)-\eta_{p^1(y)}(\mathbf{x}_{p^1(y)}^1)|^\beta}\Big]\lesssim \Vert\eta\Vert_{C^3},
\end{align*}
where we have used $\eta\in C^3$ and mean value theorem in the last line. Thus we conclude
\begin{equation}\label{g: bound for mn-mn}
\eqref{mn-mn}\lesssim \Vert \eta\Vert_{C^3} \Vert wf\Vert_\infty .
\end{equation}

Combining~\eqref{g: bound for M-M},\eqref{g: bound for f-f    in nablav} and \eqref{g: bound for mn-mn}, we conclude that the contribution of the velocity derivative$~\eqref{Hd: nabla bdr}_1$ in~\eqref{Key term in C1beta 4} has an upper bound
\begin{equation}\label{Hd: bound for nablagx-nablagy 1}
|\eqref{Key term in C1beta 4}_{\eqref{Hd: nabla bdr}_1}|\lesssim \Vert\eta\Vert_{C^3}\Vert T_W\Vert_{C^1}\Vert \alpha\nabla_x f\Vert_\infty ,
\end{equation}
where we used~\eqref{simplicity}.


\textit{Spatial derivative:}
 we consider the contribution of the spatial derivative$~\eqref{Hd: nabla bdr}_2$ in~\eqref{C1beta inter:4}. We rewrite the $\mathbf{v}_{p^1}^1$-integration using $v^1$ integration and get
\begin{equation}\label{Hd: spatial deri}
  \begin{split}
  \eqref{Hd: nabla bdr}_2 &=  e^{-\nu \tb(x,v)} \frac{M_W(\xb(x,v),v)}{\sqrt{\mu(v)}}  \\
  & \times\int_{n(\xb(x,v))\cdot v^1>0}\underbrace{\partial_{\mathbf{x}_{p^1(x),i}^1}\big[f(\eta_{p^1(x)}(\mathbf{x}_{p^1(x)}^1),v^1) \big] \sqrt{\mu(v^1)}|n(\xb(x,v))\cdot v^1|}_{\eqref{Hd: spatial deri}_*} \dd v^1.
  \end{split}
\end{equation}
Then the contribution of$~\eqref{Hd: nabla bdr}_2$ in~\eqref{C1beta inter:4} can be written as
\begin{align}
&\Bigg[\frac{w_{\tilde{\theta}}(v)|v|^2[M_W(\xb(x,v),v)-M_W(\xb(y,v),v)]}{|\xb(x,v)-\xb(y,v)|^\beta\sqrt{\mu(v)}}\times \eqref{Hd: spatial deri} \label{spatial:M-M}\\
   & +\frac{|v|^2M_W(\xb(y,v),v)}{\sqrt{\mu(v)}|\xb(x,v)-\xb(y,v)|^\beta}\Big(\int_{n(\xb(x,v))\cdot v^1>0} \frac{\partial}{\partial \mathbf{x}_{p^1(x),i}^1} f ( \eta_{p^{1}(x)} ( \mathbf{x}_{p^{1}(x) }^{1})  , v^1) \sqrt{\mu(v^1)}|n(\xb(x,v))\cdot v^1| \dd v^1\notag \\
  &    -\int_{n(\xb(y,v))\cdot v^1>0}\frac{\partial}{\partial \mathbf{x}_{p^1(y),i}^1} f ( \eta_{p^{1}} ( \mathbf{x}_{p^{1}(y) }^{1}  ), v^1)    \sqrt{\mu(v^1)}|n(\xb(y,v))\cdot v^1|        \dd v^1 \Big)\Bigg] \label{spatial:f-f in nablav}.
\end{align}


From~\eqref{equivalent} in Lemma \ref{Lemma: equivalent},
\[\eqref{Hd: spatial deri}\lesssim \Vert |v|\nabla_\parallel f\Vert_\infty \int \sqrt{\mu(v^1)} \frac{|n(\xb(x,v))\cdot v^1|}{|v^1|} \dd v^1\lesssim \Vert |v|\nabla_\parallel f\Vert_\infty.     \]

Thus applying~\eqref{min: xb} and~\eqref{min: M_w} we derive that

\begin{equation}\label{g: bound for spatial M-M}
\eqref{spatial:M-M}\lesssim     \Vert \eta\Vert_{C^2}\Vert |v|\nabla_\parallel f\Vert_\infty    .
\end{equation}

For~\eqref{spatial:f-f in nablav} we express $\sum_{i=1,2}\partial_{\mathbf{x}_{p^1(x),i}^1}f(\eta_{p^1(x)}(\mathbf{x}_{p^1(x)}^1),v^1)$ and $\sum_{i=1,2}\partial_{\mathbf{x}_{p^1(y),i}^1}f(\eta_{p^1(y)}(\mathbf{x}_{p^1(y)}^1),v^1)$ \\ using$~\eqref{ffx1_x1_1}-\eqref{ffx1_x1_5}$ with the notation~\eqref{second backward}:
\begin{align}
   &\mathbf{1}_{t^1\geq \min\{\tb^2(x),\tb^2(y)\}} e^{-\nu \tb^2(x)}  \sum_{i=1,2}\partial_{\mathbf{x}_{p^1(x),i}^1}\big[f(\xb(\eta_{p^1(x)}(\mathbf{x}_{p^1(x)}^1),v^1),v^1) \big]   \label{ffx1_x1_1}\\
   &-\mathbf{1}_{t^1\geq \min\{\tb^2(x),\tb^2(y)\}} \nu \sum_{i=1,2}\partial_{\mathbf{x}_{p^1(x),i}^1} \tb^2(x) e^{-\nu \tb^2(x)} f(\xb^2(x),v^1)   \label{ffx1_x1_2}\\
   & +\mathbf{1}_{t^1\leq \min\{\tb^2(x),\tb^2(y)\}} e^{-\nu t^1} \sum_{i=1,2} \partial_{\mathbf{x}_{p^1(x),i}^1}\big[f(\eta_{p^1(x)}(\mathbf{x}_{p^1(x)}^1)-t^1v^1,v^1) \big]     \label{ffx1_x1_3}\\
   & +\mathbf{1}_{t^1\leq \min\{\tb^2(x),\tb^2(y)\}} \int^{t^1}_{0} e^{-\nu(t^1-s^1)}   \sum_{i=1,2} \partial_{\mathbf{x}_{p^1(x),i}^1}\big[h(\eta_{p^1(x)}(\mathbf{x}_{p^1(x)}^1)-(t^1-s^1)v^1,v^1) \big]\dd s^1      \label{ffx1_x1_4}\\
   & +\mathbf{1}_{t^1\geq \min\{\tb^2(x),\tb^2(y)\}} \int^{t^1}_{t^1-\tb^2(x)} e^{-\nu(t^1-s^1)}   \sum_{i=1,2}   \partial_{\mathbf{x}_{p^1(x),i}^1}\big[h(\eta_{p^1(x)}(\mathbf{x}_{p^1(x)}^1)-(t^1-s^1)v^1,v^1) \big]\dd s^1         \label{ffx1_x1_42}\\
   &+ \mathbf{1}_{t^1\geq \min\{\tb^2(x),\tb^2(y)\}}   \sum_{i=1,2}\partial_{\mathbf{x}_{p^1(x),i}^1}\tb^2(x) e^{-\nu \tb^2(x)}h(\xb^2(x),v^1).\label{ffx1_x1_5}
\end{align}

We first estimate the boundary term~\eqref{ffx1_x1_1}. We split~\eqref{ffx1_x1_1} into two cases using \\ $\min\{\alpha(\xb(x,v),v^1),\alpha(\xb(y,v),v^1)\}$. We put the discussion for
$\min\{\alpha(\xb(x,v),v^1),\alpha(\xb(y,v),v^1)\}\leq \e,\text{ or }|v|\geq \e$ together with the estimate of \eqref{ffx1_x1_2}-\eqref{ffx1_x1_5}. Here we discuss the case that \\ $\min\{\alpha(\xb(x,v),v^1),\alpha(\xb(y,v),v^1)\}\geq \e \text{ and }|v|\leq \e^{-1}$.

For this case the difference quotient of $\eqref{ffx1_x1_1}$ reads
\begin{equation}\label{bdr: bdr-bdr}
 \begin{split}
   &\frac{\mathbf{1}_{\{\min\{n(\xb(x,v))\cdot v^1,n(\xb(y,v))\cdot v^1\} \geq \e,|v|\leq \e^{-1}\}}}{|\xb(x,v)-\xb(y,v)|^\beta}\Big( \int_{n(\xb(x,v))\cdot v^1>0}  \eqref{ffx1_x1_1}(x) \sqrt{\mu(v^1)}|n(\xb(x,v))\cdot v^1| \dd v^1 \\
   & -\int_{n(\xb(y,v))\cdot v^1>0} \eqref{ffx1_x1_1}(y) \sqrt{\mu(v^1)}|n(\xb(y,v))\cdot v^1| \dd v^1\Big).
\end{split}
\end{equation}

We perform the change of variable~\eqref{map_v_to_xbtb} and use~\eqref{p_xf_total1_under} to rewrite
\begin{equation}\label{dv to dx}
\begin{split}
   & \int_{n(\xb(x,v))\cdot v^1>0} \frac{\partial}{\partial \mathbf{x}_{p^1,i}^1} [f ( \eta_{p^{2}} ( \mathbf{x}_{p^{2} }^{2})  , v^1)] \sqrt{\mu(v^1)}|n(\xb(x,v))\cdot v^1| \dd v^1  \\
=&
\sum_{p^{2} \in \mathcal{P}}\iint _{|\mathbf{x}_{p^{2}}^{2}|< \delta_1}
\int^{t-\tb(x,v)}_0
e^{- \nu(v^{1}) \tb^1}
\iota_{p^{2}} (
\eta_{p^{2}} (\mathbf{x}_{p^{2}}^{2} )
)\\
& \times
\sum_{j^\prime=1,2}
\frac{\p \mathbf{x}^{2}_{p^{2},j^\prime}}{\p{\mathbf{x}^{1}_{p^{1}(x),j}}} \p_{\mathbf{x}_{p^{2},j^\prime}^{2}} [  f( \eta_{p^{2}} ( \mathbf{x}^{2}_{p^{2} } ),
 {v}^{1}
 )]\\
&
\times\frac{n_{p^{1}(x)} (\mathbf{x}_{p^{1}(x)} ^{1}) \cdot  (\xb(x,v) -
 \eta_{p^{2}} (\mathbf{x}_{p^{2}} ^{2})
 )
}{\tb^{2}}
 \frac{ n_{p^{2}}(\mathbf{x}^{2}_{p^{2}}) \cdot (\xb(x,v)-
 \eta_{p^{2}} (\mathbf{x}_{p^{2}} ^{2})
 ) }{|\tb^{2}|^4}\\
 &  \times  e^{-\frac{|\xb(x,v)-\eta_{p^2}(\mathbf{x}_{p^2}^2)|}{4|\tb^2|}}
 \dd \tb^{2}
\sqrt{g_{p^{2},11}g_{p^{2},22}  }
 \dd \mathbf{x}^{2}_{p^{2},1}\dd \mathbf{x}^{2}_{p^{2},2}.
\end{split}
\end{equation}
Here we dropped the $x$ dependence on $p^2(x)$ since $\mathbf{x}_{p^2}^2$ becomes dummy variable after the change of variable.

In~\eqref{dv to dx} the variables that depend on $x$ are $\tb(x,v),\xb(x,v),\mathbf{x}_{p^1(x)}^1$. Thus we have
\begin{align}
  \eqref{bdr: bdr-bdr} &=\frac{\eqref{dv to dx}(x)-\eqref{dv to dx}(y)}{|\xb(x,v)-\xb(y,v)|^\beta}\notag  \\
   & = \frac{1}{|\xb(x,v)-\xb(y,v)|^\beta}\bigg[     \sum_{p^{2} \in \mathcal{P}}\iint _{|\mathbf{x}_{p^{2}}^{2}|< \delta_1}
\int^{t-\min\{\tb(x,v),\tb(y,v)\}}_{t-\max\{\tb(x,v),\tb(y,v)\}} \cdots   \label{dv to dx 1}\\
&+ \sum_{p^{2} \in \mathcal{P}}\iint _{|\mathbf{x}_{p^{2}}^{2}|< \delta_1}
\int^{t-\max\{\tb(x,v),\tb(y,v)\}}_{0}  \Big[\frac{\p \mathbf{x}^{2}_{p^{2},j^\prime}}{\p{\mathbf{x}^{1  }_{p^{1 }(x),j}}}-\frac{\p \mathbf{x}^{2}_{p^{2},j^\prime}}{\p{\mathbf{x}^{1  }_{p^{1 }(y),j}}}   \Big]\cdots \label{dv to dx 2}\\
&+   \Big[\frac{n_{p^{1}(x)} (\mathbf{x}_{p^{1}(x)} ^{1}) \cdot  (\xb(x,v) -
 \eta_{p^{2}} (\mathbf{x}_{p^{2}} ^{2})
 )
-n_{p^{1}(y)} (\mathbf{x}_{p^{1}(y)}^{1}) \cdot  (\xb(y,v) -
 \eta_{p^{2}} (\mathbf{x}_{p^{2}} ^{2})
 )
}{\tb^2}\Big]\cdots \label{dv to dx 3}\\
&  +\Big[\frac{ n_{p^{2}}(\mathbf{x}^{2}_{p^{2}}) \cdot (\xb(x,v) -
 \eta_{p^{2}} (\mathbf{x}_{p^{2}} ^{2})
 ) }{|\tb^{2}|^4}-\frac{ n_{p^{2}}(\mathbf{x}^{2}_{p^{2}}) \cdot (\xb(y,v) -
 \eta_{p^{2}} (\mathbf{x}_{p^{2}} ^{2})
 ) }{|\tb^{2}|^4} \Big]\cdots    \label{dv to dx 4}\\
 &+ \Big[ e^{-\frac{|\xb(x,v)-\eta_{p^2}(\mathbf{x}_{p^2}^2)|}{4|\tb^2|}}- e^{-\frac{|\xb(y,v)-\eta_{p^2}(\mathbf{x}_{p^2}^2)|}{4|\tb^2|}} \Big] \bigg]\cdots \label{dv to dx 5}.
\end{align}

Since $\min\{\alpha(\xb(x,v),v^1),\alpha(\xb(y,v),v^1)\}\geq \e$, from~\eqref{n geq alpha}, clearly we have $|n(\xb(x,v))\cdot v^1|\gtrsim \e$. Moreover, due to $|v^1|\leq \e^{-1}$, we have a lower bound for $\tb^2$ from~\eqref{tb bounded}:
\[\tb^2\gtrsim \frac{1}{|v^1|}\frac{\min\{|n(\xb(x,v))\cdot v^1|,|n(\xb(y,v))\cdot v^1|\}}{|v^1|}=O(\e^3).\]
From~\eqref{equivalent} in Lemma \ref{Lemma: equivalent} and Lemma \ref{Lemma: min max} we obtain the following estimate:
\begin{align*}
   &\Big| \partial_{\mathbf{x}_{p^2,j'}^2}[f(\eta_{p^2}(\mathbf{x}_{p^2}^2),v^1)]\Big|\leq \Vert |v|\nabla_\parallel f\Vert_\infty<\infty,\quad  \Big|\frac{\partial \mathbf{x}_{p^2,j'}^2}{\partial \mathbf{x}_{p^1,j}^1} \Big|\lesssim 1, \\
   & \Big| \frac{n_{p^{1}(x)} (\mathbf{x}_{p^{1}(x)} ^{1}) \cdot  (\xb(x,v) -
 \eta_{p^{2}} (\mathbf{x}_{p^{2}} ^{2})
 )
}{\tb^{2}}\Big| \lesssim_\Omega O(\e^{-3})    ,\\
   &  \Big|\frac{ n_{p^{2}}(\mathbf{x}^{2}_{p^{2}}) \cdot (\xb(x,v)-
 \eta_{p^{2}} (\mathbf{x}_{p^{2}} ^{2})
 ) }{|\tb^{2}|^4} \Big|\lesssim_\Omega      O(\e^{-12}).
\end{align*}

Now we estimate~\eqref{dv to dx 1}-\eqref{dv to dx 5}. By~\eqref{min: tb} we compute
\begin{align}
   \frac{\eqref{dv to dx 1}}{|\xb(x,v)-\xb(y,v)|^\beta}&\lesssim \frac{O(\e^{-15})}{|\xb(x,v)-\xb(y,v)|^\beta}\iint \int_{t-\max\{\tb(x,v),\tb(y,v)\}}^{t-\min\{\tb(x,v),\tb(y,v)\}} e^{-\nu \tb^2} \notag \\
   &\lesssim O(\e^{-15}) \frac{   \big|e^{-\nu \tb(x,v)}-e^{-\nu \tb(y,v)}\big|}{|\xb(x,v)-\xb(y,v)|^\beta}\notag\\
&=O(\e^{-15}) \frac{   \big|e^{-\nu \tb(x,v)}-e^{-\nu \tb(y,v)}\big|}{|x-y|^\beta}\frac{|x-y|^\beta}{|\xb(x,v)-\xb(y,v)|^\beta} \notag\\
&\lesssim \frac{O(\e^{-15})}{\min\{\alpha(\xb(x,v),v^1),\alpha(\xb(y,v),v^1)\}} \frac{|x-y|^\beta}{|\xb(x,v)-\xb(y,v)|^\beta}\notag\\
&\lesssim  O(\e^{-16})\frac{|x-y|^\beta}{|\xb(x,v)-\xb(y,v)|^\beta}. \label{C1beta bdr cov 1}
\end{align}
The extra term $\frac{|x-y|^\beta}{|\xb(x,v)-\xb(y,v)|^\beta}$ will be cancelled by $\frac{|\xb(x,v)-\xb(y,v)|^\beta}{{|x-y|^\beta}}$ in~\eqref{C1beta inter:4}.

Then we estimate~\eqref{dv to dx 2}. By~\eqref{min: partial xip1 xip2} we have
\begin{align*}
   &\frac{\frac{\p \mathbf{x}^{2}_{p^{2},j^\prime}}{\p{\mathbf{x}^{1  }_{p^{1 }(x),j}}}-\frac{\p \mathbf{x}^{2}_{p^{2},j^\prime}}{\p{\mathbf{x}^{1  }_{p^{1 }(y),j}}}   }{|\xb(x,v)-\xb(y,v)|^\beta} \mathbf{1}_{\min\{\alpha(\xb(x,v),v^1),\alpha(\xb(y,v),v^1)\}\geq \e}  \lesssim \frac{|v^1|^3}{\min\{\alpha(\xb(x,v),v^1),\alpha(\xb(y,v),v^1)\}^3}\lesssim O(\e^{-6}).
\end{align*}
Thus
\begin{align}
 & \frac{\eqref{dv to dx 2}}{|\xb(x,v)-\xb(y,v)|^\beta}\lesssim O(\e^{-15})\int_0^\infty e^{-\nu_0 \tb^2}\frac{\frac{\p \mathbf{x}^{2}_{p^{2},j^\prime}}{\p{\mathbf{x}^{1  }_{p^{1}(x),j}}}-\frac{\p \mathbf{x}^{2}_{p^{2},j^\prime}}{\p{\mathbf{x}^{1  }_{p^{1}(y),j}}}   }{|\xb(x,v)-\xb(y,v)|^\beta}    \lesssim O(\e^{-21}). \label{C1beta bdr cov 2}
\end{align}

Then we estimate~\eqref{dv to dx 3}. By~\eqref{min: nxb} we compute
\begin{align*}
   & \frac{n_{p^{1}(x)} (\mathbf{x}_{p^{1}(x)} ^{1}) \cdot  (\xb(x,v) -
 \eta_{p^{2}} (\mathbf{x}_{p^{2}} ^{2})
 )
-n_{p^{1}(y)} (\mathbf{x}_{p^{1}(y)} ^{1}) \cdot  (\xb(y,v) -
 \eta_{p^{2}} (\mathbf{x}_{p^{2}} ^{2})
 )
}{\tb^2 |\xb(x,v)-\xb(y,v)|^\beta} \\
   &  \lesssim O(\e^{-3})\Big[ \frac{|n_{p^{1}(x)} (\mathbf{x}_{p^{1}(x)}^{1})-n_{p^{1}(y)} (\mathbf{x}_{p^{1}(y)} ^{1})||\xb(x,v)-\eta_{p^2}(\mathbf{x}_{p^2}^2)|}{|\xb(x,v)-\xb(y,v)|^\beta}+\frac{|\xb(x,v)-\xb(y,v)|}{|\xb(x,v)-\xb(y,v)|^\beta}\Big]  \lesssim O(\e^{-3}).
\end{align*}
Thus
\begin{align}
\frac{\eqref{dv to dx 3}}{|\xb(x,v)-\xb(y,v)|^\beta}   &\lesssim O(\e^{-15}). \label{C1beta bdr cov 3}
\end{align}

For~\eqref{dv to dx 4} we compute the difference as
\begin{align*}
   &\frac{ n_{p^{2}}(\mathbf{x}^{2}_{p^{2}}) \cdot (\xb(x,v) -
 \eta_{p^{2}} (\mathbf{x}_{p^{2}} ^{2})
 )- n_{p^{2}}(\mathbf{x}^{2}_{p^{2}}) \cdot (\xb(y,v) -
 \eta_{p^{2}} (\mathbf{x}_{p^{2}} ^{2})
 )  }{|\tb^{2}|^4|\xb(x,v)-\xb(y,v)|^\beta} \\
   & \lesssim O(\e^{-12})\frac{ |\xb(x,v)-\xb(y,v)|  }{|\xb(x,v)-\xb(y,v)|^\beta} \lesssim O(\e^{-12}).
\end{align*}
Thus
\begin{align}
\frac{\eqref{dv to dx 4}}{|\xb(x,v)-\xb(y,v)|^\beta}   & \lesssim  O(\e^{-12}). \label{C1beta bdr cov 4}
\end{align}

Last we estimate~\eqref{dv to dx 5}. By mean value theorem,
\begin{align*}
   & \frac{e^{-\frac{|\xb(x,v)-\eta_{p^2}(\mathbf{x}_{p^2}^2)|}{4|\tb^2|}}- e^{-\frac{|\xb(y,v)-\eta_{p^2}(\mathbf{x}_{p^2}^2)|}{4|\tb^2|}}}{|\xb(x,v)-\xb(y,v)|^\beta} \\
   & \lesssim \frac{e^{-\frac{|\xb(x,v)-\eta_{p^2}(\mathbf{x}_{p^2}^2)|}{4|\tb^2|}}- e^{-\frac{|\xb(y,v)-\eta_{p^2}(\mathbf{x}_{p^2}^2)|}{4|\tb^2|}}}{|\xb(x,v)-\eta_{p^2}(\mathbf{x}_{p^2}^2)-\xb(y,v)+\eta_{p^2}(\mathbf{x}_{p^2}^2)|^\beta/|\tb^2|^\beta}\frac{1}{|\tb^2|^\beta}\lesssim O(\e^{3\beta}).
\end{align*}
Thus
\begin{align}
\frac{\eqref{dv to dx 5}}{|\xb(x,v)-\xb(y,v)|^\beta}   & \lesssim  O(\e^{-15-3\beta}). \label{C1beta bdr cov 5}
\end{align}

Therefore, from~\eqref{simplicity}, we collect~\eqref{C1beta bdr cov 1},\eqref{C1beta bdr cov 2},\eqref{C1beta bdr cov 3},\eqref{C1beta bdr cov 4} and~\eqref{C1beta bdr cov 5} to conclude
\begin{equation}\label{Hd: bound for fx11}
\eqref{bdr: bdr-bdr}\lesssim \big[O(\e^{-21})+O(\e^{-16})\frac{|x-y|^\beta}{|\xb(x,v)-\xb(y,v)|^\beta}\big]\Vert w_{\tilde{\theta}}\alpha\nabla_x f\Vert_\infty^2.
\end{equation}

Then we estimate the rest terms in~\eqref{ffx1_x1_1}-\eqref{ffx1_x1_5}. First we rewrite the contribution of these term in~\eqref{spatial:f-f in nablav} into
\begin{align}
   & \frac{\int_{n(\xb(x,v))\cdot v^1>0} \eqref{ffx1_x1_1}\mathbf{1}_{\cdots}(x)+\cdots+\eqref{ffx1_x1_5}(x)- \int_{n(\xb(y,v))\cdot v^1>0} \eqref{ffx1_x1_1}\mathbf{1}_{\cdots}(y)+\cdots+\eqref{ffx1_x1_5}(y)}{|\xb(x,v)-\xb(y,v)|^\beta}\notag\\
   &    = \frac{\int_{|n(\xb(x,v))-n(\xb(y,v))|\geq \frac{n(\xb(x,v)\cdot v^1)}{|v^1|}>0} \cdots}{|\xb(x,v)-\xb(y,v)|^\beta}-\frac{\int_{|n(\xb(x,v))-n(\xb(y,v))|\geq \frac{n(\xb(y,v)\cdot v^1)}{|v^1|}>0} \cdots        }{|\xb(x,v)-\xb(y,v)|^\beta}             \label{int-int} \\
   &+ \int_{n(\xb(x,v))\cdot v^1>0,n(\xb(y,v))\cdot v^1>0}\frac{|n(\xb(x,v))\cdot v^1|-|n(\xb(y,v))\cdot v^1|}{|\xb(x,v)-\xb(y,v)|^\beta}\sqrt{\mu(v^1)}\big[\eqref{ffx1_x1_1}
   \mathbf{1}_{\cdots}(x)+\cdots+\eqref{ffx1_x1_5}(x)\big]\label{n-n in integral}\\
   & +\int_{n(\xb(x,v))\cdot v^1>0,n(\xb(y,v))\cdot v^1>0}|n(\xb(y,v))\cdot v^1|\sqrt{\mu(v^1)}\notag\\
    &\times \frac{\eqref{ffx1_x1_1}\mathbf{1}_{\cdots}(x)+\cdots+\eqref{ffx1_x1_5}(x)- \eqref{ffx1_x1_1}\mathbf{1}_{\cdots}(y) -\cdots-\eqref{ffx1_x1_5}(y)}{|\xb(x,v)-\xb(y,v)|^\beta}.\label{f-f}
\end{align}

By~\eqref{equivalent} in Lemma \ref{Lemma: equivalent},
\[|\eqref{ffx1_x1_1}+\cdots+\eqref{ffx1_x1_5}|\lesssim \eqref{tang: nabla fxb}+\cdots \eqref{tang: nabla h}\lesssim \eqref{nablaparallel f bound} \lesssim \frac{\Vert |v|\nabla_\parallel f\Vert_\infty+\Vert \alpha\nabla_x f\Vert_\infty}{|v^1|}.\]
For~\eqref{int-int}, from~\eqref{min: nxb} and~\eqref{simplicity} we have
\begin{align*}
   &  \frac{1}{|\xb(x,v)-\xb(y,v)|^\beta}\int_{|n(\xb(x,v))-n(\xb(y,v))|\geq \frac{n(\xb(x,v))\cdot v^1}{|v^1|}>0} |\eqref{ffx1_x1_1}+\cdots+\eqref{ffx1_x1_5}|\\
   & \lesssim\frac{\Vert \eta\Vert_{C^2} \Vert \alpha\nabla_x f\Vert_\infty}{|\xb(x,v)-\xb(y,v)|^\beta}\int_{|n(\xb(x,v))-n(\xb(y,v))|\geq \frac{n(\xb(x,v))\cdot v^1}{|v^1|}>0} \frac{|n(\xb(x,v))\cdot v^1|}{|v|^1}  \sqrt{\mu(v^1)}\\
   &\lesssim \frac{\Vert \eta\Vert_{C^2} \Vert \alpha\nabla_x f\Vert_\infty|n(\xb(x,v))-n(\xb(y,v))|}{|\xb(x,v)-\xb(y,v)|^\beta}\lesssim \Vert \eta\Vert_{C^2} \Vert \alpha\nabla_x f\Vert_\infty.
\end{align*}

Similarly
\[\frac{1}{|\xb(x,v)-\xb(y,v)|^\beta}\int_{|n(\xb(x,v))-n(\xb(y,v))|\geq \frac{n(\xb(y,v))\cdot v^1}{|v^1|}>0}\cdots \lesssim \Vert \eta\Vert_{C^2} \Vert \alpha\nabla_x f\Vert_\infty.\]
Thus
\begin{equation}\label{Hd: bound for int-int}
\eqref{int-int}\lesssim \Vert \eta\Vert_{C^2}\Vert \alpha\nabla_x f\Vert_\infty.
\end{equation}

For~\eqref{n-n    in integral}, applying~\eqref{min: nxb} we have
\begin{align}
 \eqref{n-n    in integral}  &\lesssim \int \frac{|n(\xb(x,v))\cdot v^1|-|n(\xb(y,v))\cdot v^1|}{|\xb(x,v)-\xb(y,v)|^\beta}\sqrt{\mu(v^1)} \cdots  \notag\\
   & \lesssim \Vert \eta\Vert_{C^2}\Vert \alpha\nabla_x f\Vert_\infty. \label{Hd: bound for n-n}
\end{align}

Then we focus~\eqref{f-f}, this estimate is the most delicate one. First of all we bound
\begin{equation}\label{comparable}
|n(\xb(y,v)\cdot v^1)|\leq \underbrace{|n(\xb(x,v))-n(\xb(y,v))||v^1|}_{\eqref{comparable}_1}+\underbrace{\min\{n(\xb(x,v))\cdot v^1,n(\xb(y,v))\cdot v^1\}}_{\eqref{comparable}_2}.
\end{equation}

By~\eqref{min: nxb} the contribution of$~\eqref{comparable}_1$ in~\eqref{f-f} is bounded by
\begin{align}
&\frac{\eqref{f-f}_{\eqref{comparable}_1}}{|\xb(x,v)-\xb(y,v)|^\beta} \notag\\
& \lesssim\int_{n(\xb(x,v))\cdot v^1>0,n(\xb(y,v))\cdot v^1>0}   \frac{|n(\xb(x,v))\cdot v^1|-|n(\xb(y,v))\cdot v^1|}{|\xb(x,v)-\xb(y,v)|^\beta} \sqrt{\mu(v^1)} \cdots \dd v^1 \notag\\
   &\lesssim \Vert \eta \Vert_{C^2}\Vert \alpha\nabla_x f\Vert_\infty \int_{n(\xb(x,v))\cdot v^1>0,n(\xb(y,v))\cdot v^1>0} \sqrt{\mu(v^1)}\lesssim \Vert \eta\Vert_{C^2}\Vert \alpha\nabla_x f\Vert_\infty. \label{Hd: bound for f-f1}
\end{align}

We focus on the contribution of$~\eqref{comparable}_2$ in~\eqref{f-f}. Then without loss generality, we can assume
\begin{equation}\label{wlog}
|n(\xb(y,v))\cdot v^1|=\min\{|n(\xb(x,v))\cdot v^1|,|n(\xb(y,v))\cdot v^1|\}.
\end{equation}
In result we can replace $|n(\xb(y,v))\cdot v^1|$ or $|n(\xb(y,v))\cdot v^1|$ by $\min\{|n(\xb(x,v))\cdot v^1|,|n(\xb(y,v))\cdot v^1|\}$.

Note that from~\eqref{equivalent} in Lemma \ref{Lemma: equivalent},
\[\sum_{i=1,2}\partial_{\mathbf{x}_{p^1(x),i}^1}f(\eta_{p^1(x)}(\mathbf{x}_{p^1(x)}^1),v^1)\thicksim G(\xb(x,v))\nabla_x f(\xb(x,v),v^1),\]
and we have an expression of $G(\xb(x,v))\nabla_x f(\xb(x,v),v^1)-G(\xb(y,v))\nabla_x f(\xb(y,v),v^1)$ from Lemma \ref{Lemma: intermediate estimate}. Thus the contribution of~\eqref{ffx1_x1_1}-\eqref{ffx1_x1_5} in~\eqref{f-f} can be expressed using~\eqref{C1betatang inter:0}-\eqref{C1betatang inter:4}, with replacing $x\to \xb(x,v)$, $y\to \xb(y,v)$, $v\to v^1$, $\mathbf{x}_{p^1,i}^1\to \mathbf{x}_{p^2,i}^i$, $\xb(x,v)\to \eta_{p^2(x)}(\mathbf{x}_{p^2(x)}^2)$, $\xb(y,v)\to \eta_{p^2(y)}(\mathbf{x}_{p^2(y)}^2)$.

From~\eqref{fact 3} we derive that the contribution of~\eqref{C1betatang inter:1} is bounded by
\begin{align}
   & \int_{n(\xb(x,v))\cdot v^1>0,n(\xb(y,v))\cdot v^1>0} \notag\\
   &\frac{|n(\xb(y,v))\cdot v^1|\sqrt{\mu(v^1)}[\Vert w_{\tilde{\theta}}\alpha\nabla_x f\Vert_\infty+o(1)[\nabla_x f(\cdot,v)]_{C^{0,\beta}_{x;2+\beta}}+o(1)[\nabla_\parallel f(\cdot,v)]_{C^{0,\beta}_{x;1+\beta}}]}{|v^1|^2\min\{\frac{\alpha(\xb(x,v),v^1)}{|v^1|},\frac{\alpha(\xb(y,v),v^1)}{|v^1|}\}^{1+\beta}}\notag \\
   & \lesssim \Vert  w_{\tilde{\theta}}\alpha\nabla_x f\Vert_\infty+o(1)[\nabla_x f(\cdot,v)]_{C^{0,\beta}_{x;2+\beta}}+o(1)[\nabla_\parallel f(\cdot,v)]_{C^{0,\beta}_{x;1+\beta}}, \label{Hd: bound for f-f2: 1}
\end{align}
where we have used~\eqref{fact 3}.

Then we estimate the contribution of~\eqref{C1betatang inter:3} and~\eqref{C1betatang inter:4} and~\eqref{C1betatang inter:0}.

We begin with~\eqref{C1betatang inter:4}. Note that we only need to consider the case $\min\{\alpha(\xb(x,v),v^1),\alpha(\xb(y,v),v^1)\}\leq \e, \text{ or }|v^1|\geq \e^{-1} $. We derive
\begin{align}
&\int \mathbf{1}_{\min\{\alpha(\xb(x,v),v^1),\alpha(\xb(y,v),v^1)\}\leq \e, \text{ or }|v^1|\geq \e^{-1}} \sqrt{\mu(v^1)}|n(\xb(y,v)\cdot v^1)|\notag\\
   & \sum_{j=1,2} \big[  e^{-\nu \tb^2(x)}   \frac{|\xb^2(x)-\xb^2(y)|^\beta}{|\xb(x,v)-\xb(y,v)|^\beta}\notag\\
     &\times\frac{\partial_{\mathbf{x}_{p^2(x),j}^2}f(\eta_{p^2(x)}(\mathbf{x}_{p^2(x)}^2),v^1)-\partial_{\mathbf{x}_{p^2(y),j}^2}f(\eta_{p^2(y)}(\mathbf{x}_{p^2(y)}^2),v^1)}{|\eta_{p^2(x)}(\mathbf{x}_{p^2(x)}^2)-\eta_{p^2(y)}(\mathbf{x}_{p^2(y)}^2)|^\beta}  \big]\notag\\
     &\lesssim \int_{\e>n(\xb(x,v))\cdot v^1>0,\e>n(\xb(y,v))\cdot v^1>0,\text{ or }|v^1|\geq \e^{-1}}\sqrt{\mu(v^1)}|n(\xb(x,v))\cdot v^1| \notag \\
   &\times    \frac{[\nabla_{x_\parallel}f(\cdot,v)]_{C^{0,\beta}_{x;1+\beta}}}{|v^1|^2\min\{\frac{\alpha(\xb(x,v),v^1)}{|v^1|},\frac{\alpha(\xb(y,v),v^1)}{|v^1|}\}^{1+\beta}} \notag\\
 &\lesssim O(\e)[\nabla_{x_\parallel}f(\cdot,v)]_{C^{0,\beta}_{x;1+\beta}}  \label{Hd: bound for f-f2: 4},
   \end{align}
where we have applied~\eqref{equivalent} in Lemma \ref{Lemma: equivalent} to the third line, \eqref{min: xb tang} to the second line, \eqref{wlog} and \eqref{fact 3} to the integral in the fourth line.

Then we focus on the contribution of~\eqref{C1beta inter:3}. First we consider $h=K(f)$. Denote
\[x^s=\xb(x,v)-(t^1-s^1)v^1,\quad y^s=\xb(y,v)-(t^1-s^1)v^1,\]
we need to compute
\begin{equation}\label{h-h deri in bdr}
\int \sqrt{\mu(v^1)}|n(\xb(y,v))\cdot v^1|\int^{t^1}_0 \dd s^1 e^{-\nu(v^1)(t^1-s^1)} \int_{\mathbb{R}^3}\dd u \mathbf{k}(v^1,u)\underbrace{\frac{G(x^s)\nabla_x f(x^s,u)-G(y^s)\nabla_x f(y^s,u)}{|\xb(x,v)-\xb(y,v)|^\beta}}_{\eqref{h-h deri in bdr}_*}.
\end{equation}

We use the decomposition~\eqref{px_f_5_split} for the $\dd s^1$ integral. When $t^1-s^1\leq \e,$ we apply~\eqref{integrate alpha beta small} in
Lemma \ref{Lemma: NLN} with $p=1+\beta$ and~\eqref{wlog} to conclude that
\begin{align*}
   \eqref{h-h deri in bdr}\mathbf{1}_{t^1-s^1\leq \e} & \lesssim
  [\nabla_{x_\parallel}f(\cdot,v)]_{C^{0,\beta}_{x;1+\beta}}\int \sqrt{\mu(v^1)}|n(\xb(y,v))\cdot v^1|\dd v^1      \\
   & \times \int^{t^1}_{t^1-\e}\dd s^1 e^{-\nu(v^1)(t^1-s^1)}\int_{\mathbb{R}^3}\dd u \mathbf{k}(v^1,u)\frac{1}{|u|^2\min\{\frac{\alpha(x^s,u)}{|u|},\frac{\alpha(y^s,u)}{|u|}\}^{1+\beta}}\\
   &\lesssim  O(\e)
   [\nabla_{x_\parallel}f(\cdot,v)]_{C^{0,\beta}_{x;1+\beta}} \int \sqrt{\mu(v^1)}\frac{|n(\xb(y,v),v^1)|}{|v^1|^2\min\{\frac{\alpha(\xb(x,v),v^1)}{|v|^1},\frac{\alpha(\xb(y,v),v^1)}{|v^1|}\}^{1+\beta}}\dd v^1\\
   &\lesssim
   O(\e)[\nabla_{x_\parallel}f(\cdot,v)]_{C^{0,\beta}_{x;1+\beta}}.
\end{align*}

When $t^1-s^1\geq \e$. We rewrite
\begin{equation}\label{h-h deri*}
\eqref{h-h deri in bdr}_*= \underbrace{\frac{\Big[G(x^s)-G(y^s)\Big]\nabla_x f(y^s,u)}{|\xb(x,v)-\xb(y,v)|^\beta}}_{\eqref{h-h deri*}_1}+\underbrace{\frac{G(y^s)\Big[\nabla_x f(x^s,u)-\nabla_x f(y^s,u)\Big] }{|\xb(x,v)-\xb(y,v)|^\beta}}_{\eqref{h-h deri*}_2}.
\end{equation}
By~\eqref{min: nxb} we have $\eqref{h-h deri*}_1\lesssim \frac{\Vert \alpha\nabla_x f\Vert_\infty}{\alpha(y^s,u)}$. Thus such contribution in~\eqref{h-h deri in bdr} is bounded by
\begin{align*}
  & \Vert \alpha\nabla_x f\Vert_\infty \int \sqrt{\mu(v)^1}|n(\xb(y,v))\cdot v^1|\dd v^1 \int^{t^1}_{0} \dd s^1 e^{-\nu(v^1)}(t^1-s^1)\int_{\mathbb{R}^3} \dd u \frac{\mathbf{k}(v^1,u)}{\alpha(y^s,u)}\\
   &\lesssim \Vert \alpha\nabla_x f\Vert_\infty\int \sqrt{\mu(v^1)}\frac{|n(\xb(y,v))\cdot v^1|}{\min\{\alpha(\xb(x,v),v^1),\alpha(\xb(y,v),v^1)\}} \dd v^1 \lesssim \Vert \alpha\nabla_x f\Vert_\infty,\end{align*}
where we have used~\eqref{est:nonlocal_wo_e} in Lemma \ref{Lemma: NLN} and~\eqref{comparable}.

Then we focus on the contribution of$~\eqref{h-h deri*}_2$. We exchange $\nabla_x$ for $\nabla_{v^1}$:
\[\nabla_x f(x^s,u)=\nabla_x f(\xb(x,v)-(t^1-s^1)v^1,u)=\frac{\nabla_{v^1} f(\xb(x,v)-(t^1-s^1)v^1,u)}{-(t^1-s^1)}.\]
Then we perform an integration by parts for $\dd v^1$. The $\dd v^1$ integral in$~\eqref{h-h deri in bdr}\mathbf{1}_{t^1-s^1\geq \e}$ becomes
\begin{align}
 & \int \nabla_{v^1}\Big[|n(\xb(y,v))\cdot v^1|\sqrt{\mu(v^1)}\int^{t^1}_{\max\{t^1-\tb^2(x),t^1-\tb^2(y)\}}\frac{e^{-\nu(v^1)(t^1-s^1)}}{-(t^1-s^1)}\mathbf{k}_\varrho(v^1,u)G(y^s)\Big] \dd u   \notag \\
&\times \frac{f(x^s,u)-f(y^s,u)}{|\xb(x,v)-\xb(y,v)|^\beta} \notag\\
&\lesssim \int \int^{t^1}_{\max\{t^1-\tb^2(x),t^1-\tb^2(y)\}}\nabla_{v^1}\big[|n(\xb(y,v))\cdot v^1|\sqrt{\mu(v^1)}e^{-\nu(v^1)(t^1-s^1)}\big]\cdots
 \label{h-h deri*2 1}\\
 &+\int  \int^{t^1}_{\max\{t^1-\tb^2(x),t^1-\tb^2(y)\}}\nabla_{v^1} \mathbf{k}_{\varrho}(v^1,u)\cdots\label{h-h deri*2 3}\\
&+\int \int^{t^1}_{\max\{t^1-\tb^2(x),t^1-\tb^2(y)\}} \nabla_{v^1}G(\xb(y,v)-(t^1-s^1)v^1)\cdots \label{h-h deri*2 3p} \\
&+\int  \nabla_{v^1}\min\{\tb^2(x),\tb^2(y)\} \frac{e^{-\nu(v^1)\min\{\tb^2(x),\tb^2(y)\}}}{\min\{\tb^2(x),\tb^2(y)\}} \cdots                       .\label{h-h deri*2 4}
\end{align}

For~\eqref{h-h deri*2 1}, since $\nabla_{v^1}|n(\xb(y,v)\cdot v^1)\sqrt{\mu(v^1)}e^{-\nu(v^1)(t^1-s^1)}|\lesssim \mu^{1/4}(v^1)e^{-\nu(v^1)(t^1-s^1)/2}$, by \eqref{min: f} with~\eqref{simplicity} and~\eqref{integrate beta<1} we have
\begin{align*}
  \eqref{h-h deri*2 1} &  \lesssim    O(\e^{-1})\Vert \alpha\nabla_x f\Vert_\infty\int \mu^{1/4}(v^1)\int e^{-\nu(v^1)(t^1-s^1)/2}\dd s^1\int_{\mathbb{R}^3}\frac{\mathbf{k}_\varrho(v^1,u)}{\min\{\alpha(x^s,v^1),\alpha(y^s,v^1)\}^\beta}\dd u\\
   & \lesssim O(\e^{-1})\Vert \alpha\nabla_x f\Vert_\infty \int \mu^{1/4}(v^1) \lesssim O(\e^{-1})\Vert \alpha\nabla_x f\Vert_\infty.
\end{align*}

For~\eqref{h-h deri*2 3} from~\eqref{k_varrho}, we have $\nabla_{v^1}\mathbf{k}(v^1,u)\lesssim \frac{\langle v^1\rangle\mathbf{k}(v^1,u)}{|v^1-u|}$. Then by~\eqref{integrate k/v-u} in Lemma \ref{Lemma: NLN} we have
\begin{align*}
  \eqref{h-h deri*2 3} &  \lesssim    O(\e^{-1})\Vert \alpha\nabla_x f\Vert_\infty\int |n(\xb(y,v))\cdot v^1|\sqrt{\mu(v^1)} \langle v^1\rangle\dd v^1\\
  &\times\int e^{-\nu(v^1)(t^1-s^1)}\dd s^1\int_{\mathbb{R}^3}\frac{\mathbf{k}_\varrho(v^1,u)}{|v^1-u||\min\{\alpha(x^s,v^1),\alpha(y^s,v^1)\}^\beta}\dd u\\
   & \lesssim O(\e^{-1})\Vert \alpha\nabla_x f\Vert_\infty \int |n(\xb(y,v))\cdot v^1|\mu^{1/4}(v^1) \frac{1}{\min\{\alpha(\xb(x,v),v^1),\alpha(\xb(y,v),v^1)\}^\beta}\\
   &\lesssim O(\e^{-1})\Vert \alpha\nabla_x f\Vert_\infty.
\end{align*}

For~\eqref{h-h deri*2 3p}, since $|\nabla_{v^1}G(\xb(y,v)-(t^1-s^1)v^1)|\lesssim \Vert \xi\Vert_{C^2}(t^1-s^1) $, and $(t^1-s^1)e^{-\nu(v^1)(t^1-s^1)}\lesssim e^{-\nu(v^1)(t^1-s^1)/2}$, we have
\[\eqref{h-h deri*2 3p} \lesssim O(\e^{-1}) \Vert \alpha\nabla_x f\Vert_\infty  \int \mu^{1/4}(v^1)\dd v^1\lesssim O(\e^{-1})\Vert \xi\Vert_{C^2}\Vert \alpha\nabla_x f\Vert_\infty.  \]

For~\eqref{h-h deri*2 4}, since we consider $t^1-s^1\geq \e$, $\min\{\tb^2(x),\tb^2(y)\}\geq \e$. From~\eqref{nabla_tbxb} we have
\[\nabla_{v^1} \min\{\tb^2(x),\tb^2(y)\} \frac{e^{-\nu(v^1)\min\{\tb^2(x),\tb^2(y)\}}}{\min\{\tb^2(x),\tb^2(y)\}}\lesssim O(\e^{-1}).\]
Denote
\[x^b=\xb(x,v)-\min\{\tb^2(x),\tb^2(y)\}v^1,\quad y^b=\xb(y,v)-\min\{\tb^2(x),\tb^2(y)\}v^1.\]
Using~\eqref{nabla_tbxb} and from~\eqref{integrate beta<1} in Lemma \ref{Lemma: NLN} we have
\begin{align*}
  \eqref{h-h deri*2 4} & \lesssim O(\e^{-1})\int   \frac{|n(\xb(y,v))\cdot v^1|}{|n(\xb(y,v))\cdot v^1|}\sqrt{\mu(v^1)}\int_{\mathbb{R}^3}\frac{\mathbf{k}_\varrho(v^1,u)}{\min\{\alpha(x^b,u),\alpha(y^b,u)\}^\beta}  \\
   &\lesssim \Vert \alpha\nabla_x f\Vert_\infty \int   \sqrt{\mu(v^1)}  \lesssim \Vert \alpha\nabla_x f\Vert_\infty.
\end{align*}

Thus the contribution of$~\eqref{h-h deri*}_2$ in~\eqref{h-h deri in bdr} is bounded by
\begin{equation}\label{contri of ibp in h bdr}
O(\e^{-1})\Vert \alpha\nabla_x f\Vert_\infty.
\end{equation}

Then we obtain
\begin{equation}\label{Hd: bound for f-f2: 3K}
\eqref{h-h deri in bdr}\lesssim O(\e^{-1})\Vert \alpha\nabla_x f\Vert_\infty+O(\e)[\nabla_{x_\parallel}f(\cdot,v)]_{C^{0,\beta}_{x;1+\beta}}.
\end{equation}

Then we consider $h=\Gamma(f,f)$. We use~\eqref{G gamma- G    gamma} in Lemma \ref{Lemma: gamma-gamma} and~\eqref{wlog} and~\eqref{fact 1} to obtain
\begin{align}
    &\int |n(\xb(y,v))\cdot v^1|\sqrt{\mu(v^1)}\int_{\max\{t^x,t^y\}}^{t^1}\dd s^1 e^{-\nu(v^1)(t^1-s^1)}\notag\\
    &\times\frac{ G(\xb(x,v))\nabla_x \Gamma(f,f)(x^s,v^1) -G(\xb(y,v))\nabla_{x} \Gamma(f,f)(y^s,v^1)}{|\xb(x,v)-\xb(y,v)|^\beta} \notag\\
 &\lesssim \big(\Vert\alpha\nabla_x f\Vert_\infty^2+\Vert wf\Vert_\infty[\nabla_{x_\parallel}f(\cdot,v)]_{C^{0,\beta}_{x;1+\beta}}\big) \notag\\
&   \times \int \sqrt{\mu(v^1)}\frac{|n(\xb(y,v))\cdot v^1|}{|v^1|^2\min\{\alpha(\xb(x,v),v^1),\alpha(\xb(y,v),v^1)\}^{1+\beta}}\dd v^1  \notag\\
   & \lesssim  \Vert\alpha\nabla_x f\Vert_\infty^2+o(1)[\nabla_{x_\parallel}f(\cdot,v)]_{C^{0,\beta}_{x;1+\beta}}.\label{Hd: bound for f-f2: 3G}
\end{align}

The last term is~\eqref{C1betatang inter:0}. This estimate is very similar to the contribution of~\eqref{C1betatang inter:3}. Note that $\sqrt{\mu(v^1)}\tilde{\alpha}(\xb(x,v),v^1)\lesssim \mu^{1/4}(v^1)\alpha(\xb(x,v),v^1),$ we need to compute
\begin{equation}\label{h-h in bdr 0}
\int \mu^{1/4}(v^1)|n(\xb(y,v))\cdot v^1|\int^{t^1}_0 \dd s^1 e^{-\nu(v^1)(t^1-s^1)} \int_{\mathbb{R}^3}\dd u \frac{\alpha(\xb(x,v),v^1)}{|v^1|} \mathbf{k}(v^1,u)\frac{\nabla_x f(x^s,u)-\nabla_x f(y^s,u)}{|\xb(x,v)-\xb(y,v)|^\beta}.
\end{equation}

Again we first consider $t^1-s^1\leq \e$. We apply~\eqref{integrate alpha beta small} in Lemma \ref{Lemma: NLN} with $p=2+\beta$ and~\eqref{wlog} to obtain
\begin{align}
\eqref{h-h in bdr 0}\mathbf{1}_{t^1-s^1\leq \e} &\lesssim [\nabla_x f(\cdot ,v)]_{C_{x;2+\beta}^{0,\beta}}\int \mu^{1/4}(v^1)|n(\xb(y,v),v^1)|\frac{\alpha(\xb(x,v),v^1)}{|v^1|}  \notag\\
&\times \int_{t^1-\e}^{t^1}\dd s^1 e^{-\nu(v^1)(t^1-s^1)}\int_{\mathbb{R}^3}\dd u \mathbf{k}(v^1,u)\frac{1}{|u|^2\min\{\frac{\alpha(x^s,u)}{|u|},\frac{\alpha(y^s,u)}{|u|}\}^{2+\beta}}\notag\\
&\lesssim O(\e)[\nabla_x f(\cdot ,v)]_{C_{x;2+\beta}^{0,\beta}} \int \mu^{1/4}(v^1) \frac{|n(\xb(y,v),v^1)| }{|v^1|^2\min\{\frac{\alpha(\xb(x,v),v^1)}{|v^1|},\frac{\alpha(\xb(y,v),v^1)}{|v^1|}\}^{1+\beta}}\notag\\
&\lesssim O(\e)[\nabla_x f(\cdot ,v)]_{C_{x;2+\beta}^{0,\beta}}.\notag
\end{align}

For $t^1-s^1\geq \e$, we apply the same integration by parts technique as in~\eqref{h-h deri*2 1}-\eqref{h-h deri*2 4}. The only difference is we do not have an extra term $G(y^s)$ here. But this term doesn't apply a role in the estimate for~\eqref{h-h deri*2 1},\eqref{h-h  deri*2 3} and~\eqref{h-h deri*2 4}. Thus for this case we have the same upper bound as~\eqref{contri of ibp in h bdr}.

Combining~\eqref{Hd: bound for fx11},\eqref{Hd: bound for int-int},\eqref{Hd: bound for n-n},\eqref{Hd: bound for f-f1},\eqref{Hd: bound for f-f2: 1},\eqref{Hd: bound for f-f2: 3K},\eqref{Hd: bound for f-f2:    4} and~\eqref{Hd: bound for    f-f2: 3G},~\eqref{h-h in bdr 0} we conclude that
\begin{equation}\label{g: bound for f-f}
\eqref{spatial:f-f in nablav}\lesssim o(1)\Big[[\nabla_x f(\cdot ,v)]_{C_{x,2+\beta}^{0,\beta}}+[\nabla_\parallel f(\cdot , v)]_{C_{x,1+\beta}^{0,\beta}}\Big]+\big[(\e^{-21})+O(\e^{-16})\frac{|x-y|^\beta}{|\xb(x,v)-\xb(y,v)|^\beta}\big]\Vert w_{\tilde{\theta}}\alpha\nabla_x f\Vert_\infty^2.
\end{equation}

This, together with~\eqref{g: bound for spatial M-M} and \eqref{Hd: bound for nablagx-nablagy 1}, lead to the conclusion:
\begin{equation}\label{estimate for C1beta4}
\eqref{C1beta inter:4}\lesssim \Vert T_W-T_0\Vert_{C^2}\frac{o(1)\Big[[\nabla_x f(\cdot ,v)]_{C_{x,2+\beta}^{0,\beta}}+[\nabla_\parallel f(\cdot , v)]_{C_{x,1+\beta}^{0,\beta}}\Big]+O(\e^{-21})\Vert w_{\tilde{\theta}}\alpha\nabla_x f\Vert_\infty^2}{w_{\tilde{\theta}}(v)|v|^2\min\{\frac{\alpha(x,v)}{|v|},\frac{\alpha(y,v)}{|v|}\}^{1+\beta}},
\end{equation}
where we have applied~\eqref{min: xb} to $\frac{|\xb(x,v)-\xb(y,v)|^\beta}{|x-y|^\beta}$.

\textbf{Step 2: estimate of~\eqref{C1beta inter:3}.}

Now we estimate the contribution of the collision operator. First we consider $h=\Gamma(f,f)$. Applying~\eqref{int gamma-gamma} in Lemma \ref{Lemma: gamma-gamma} we have
\begin{align}
\eqref{C1beta inter:3}\mathbf{1}_{h=\Gamma}  & \lesssim \frac{o(1)[\nabla_x f(\cdot,v)]_{C_{x,2+\beta}^{0,\beta}}+\Vert w_{\tilde{\theta}}\alpha \nabla_x f\Vert_\infty^2}{w_{\tilde{\theta}}(v)|v|^2\min\{\frac{\alpha(x,v)}{|v|},\frac{\alpha(y,v)}{|v|}\}^{2+\beta}}.
 \label{Hd: bound for h-h deri Gamma}
\end{align}

Now we focus on the estimate for $h=K(f)$, which is
\begin{equation}\label{Est: h-h deri}
 \int_{t-t_m(v)}^t \dd s e^{-\nu(t-s)}\int_{\mathbb{R}^3}\dd u \mathbf{k}(v,u) \underbrace{\frac{\partial_{x_i} f(x-(t-s)v,u)-\partial_{x_i} f(y-(t-s)v,u)}{|x-y|^\beta}}_{\eqref{Est: h-h deri}_*}.
\end{equation}
Since $|x-y|=|x-(t-s)v-[y-(t-s)v]|$, we express$~\eqref{Est: h-h deri}_*$ by~\eqref{C1beta inter:1}-\eqref{C1beta inter:3}. The contribution of~\eqref{C1beta inter:1} in~\eqref{Est: h-h deri} is bound by
\begin{align}
&\frac{\big[o(1)[\nabla_{x_\parallel}f(\cdot,v)]_{C^{0,\beta}_{x;1+\beta}}+\Vert w_{\tilde{\theta}}\alpha\nabla_x f\Vert_\infty\big]}{w_{\tilde{\theta}}(v)}\notag\\
   & \times  \int_{t-t_m(v)}^t \dd s e^{-\nu(t-s)}\int_{\mathbb{R}^3}\dd u \frac{w_{\tilde{\theta}}(v)\mathbf{k}(v,u)}{w_{\tilde{\theta}}(u)|u|^2\min\{\frac{\alpha(x-(t-s)v,u)}{|u|},\frac{\alpha(x-(t-s)v,u)}{|u|}\}^{2+\beta}} \notag\\
   & \lesssim \frac{o(1)[\nabla_{x_\parallel}f(\cdot,v)]_{C^{0,\beta}_{x;1+\beta}}+\Vert w_{\tilde{\theta}}\alpha\nabla_x f\Vert_\infty}{w_{\tilde{\theta}}(v)|v|^2\min\{\frac{\alpha(x,v)}{|v|},\frac{\alpha(y,v)}{|v|}\}^{2+\beta}},\label{k: inter1}
\end{align}
where we have applied Lemma \ref{Lemma: NLN}.

Then we consider the contribution of~\eqref{C1beta inter:4} in~\eqref{Est: h-h deri}.
By~\eqref{estimate for C1beta4} and~\eqref{k_theta} and~\eqref{integrate alpha beta}, such contribution is bounded by
\begin{equation}\label{k: bound for nablagx-nablagy}
\begin{split}
  &\Vert T_W-T_0\Vert_{C^2}\frac{ o(1)\Big[[\nabla_x f(\cdot ,v)]_{C_{x,2+\beta}^{0,\beta}}+[\nabla_\parallel f(\cdot , v)]_{C_{x,1+\beta}^{0,\beta}}\Big]+O(\e^{-21})\Vert \alpha\nabla_x f\Vert_\infty^2}{w_{\tilde{\theta}}(v)}\\
    &\times \int^t_0 \dd s e^{-\nu(t-s)}\int_{\mathbb{R}^3}\dd u \frac{ w_{\tilde{\theta}}(v) \mathbf{k}(v,u)}{w_{\tilde{\theta}}(u)|u|^2\min\{\frac{\alpha(x-(t-s)v,u)}{|u|},\frac{\alpha(y-(t-s)v,u)}{|u|}\}^{1+\beta}}\\
   &\lesssim \Vert T_W-T_0\Vert_{C^2}\frac{o(1)\Big[[\nabla_x f(\cdot ,v)]_{C_{x,2+\beta}^{0,\beta}}+[\nabla_\parallel f(\cdot , v)]_{C_{x,`+\beta}^{0,\beta}}\Big]+O(\e^{-21})\Vert \alpha\nabla_x f\Vert_\infty^2}{w_{\tilde{\theta}}(v)|v|^2\min\{\frac{\alpha(x,v)}{|v|},\frac{\alpha(y,v)}{|v|}\}^{1+\beta}}.
\end{split}
\end{equation}

Then we focus the contribution of the double collision operator, the~\eqref{C1beta inter:3} in~\eqref{Est: h-h deri}. We first estimate $h=\Gamma(f,f)$. By Lemma \ref{Lemma: Gf-Gf}, such contribution in~\eqref{Est: h-h deri} is bounded by
\begin{align}
  &\frac{1}{w_{\tilde{\theta}}(v)}\int^t_{t-t_m(v)}\dd s e^{-\nu(t-s)}\int_{\mathbb{R}^3}\dd u  w_{\tilde{\theta}}(v)\mathbf{k}(v,u)\frac{o(1)[\nabla_x f(\cdot,v)]_{C_{x,2+\beta}^{0,\beta}}+\Vert w_{\tilde{\theta}}\alpha\nabla_x f \Vert_\infty^2}{w_{\tilde{\theta}}(u)|u|^2\min\{\frac{\alpha(x,u)}{|u|},\frac{\alpha(y,u)}{|u|}\}^{2+\beta}} \notag\\
  & \lesssim \frac{o(1)[\nabla_x f(\cdot,v)]_{C_{x,2+\beta}^{0,\beta}}+\Vert w_{\tilde{\theta}}\alpha\nabla_x f \Vert_\infty^2}{w_{\tilde{\theta}}(v)|v|^2\min\{\frac{\alpha(x,v)}{|v|},\frac{\alpha(y,v)}{|v|}\}^{2+\beta}},\label{k: gamma-gamma}
\end{align}
where we have used Lemma \ref{Lemma: NLN} and Lemma \ref{Lemma: k tilde}.

Then we estimate $h=K(f)$, which is the most delicate one. We denote $t^s_m(u)=\min\{\tb(x-(t-s)v,u),\tb(y-s(t-s)v,u)\}$. We need to compute
\begin{equation}\label{k: h-h deri}
\begin{split}
    &  \int_{t-t_m(v)}^t \dd s e^{-\nu(t-s)}\int_{\mathbb{R}^3} \dd u \mathbf{k}(v,u)\int^s_{s-t^s_m(u)} \dd s' e^{-\nu(s-s')} \int_{\mathbb{R}^3}\dd u' \\
     & \times\mathbf{k}(u,u')\frac{\nabla_x f(x-(t-s)v-(s-s')u,u')-\nabla_x f(y-(t-s)v-(s-s')u,u')}{|x-y|^\beta}.
\end{split}
\end{equation}
We first decompose the $s'$ integration as
\begin{equation}\label{decompose s'}
\underbrace{\int^s_{s-\e}\dd s'}_{\eqref{decompose s'}_1}+\underbrace{\int_0^{s-\e} \dd s'}_{\eqref{decompose s'}_2}.
\end{equation}

Applying~\eqref{integrate alpha beta small} in Lemma \ref{Lemma: NLN} with $p=2+\beta$ we conclude that the contribution of$~\eqref{decompose s'}_1$ in~\eqref{k: h-h deri} is bounded by
  \begin{align*}
   &\frac{[\nabla_x f(\cdot,v)]_{C_{x,2+\beta}^{0,\beta}}}{ w_{\tilde{\theta}}(v)} \int_{t_m(s)}^t \dd s e^{-\nu(v)(t-s)}\int_{\mathbb{R}^3} \dd u \frac{\mathbf{k}(v,u)w_{\tilde{\theta}}(v)}{w_{\tilde{\theta}}(u)}  \int_{s-\e}^s \dd s' e^{-\nu(u)(s-s')}\\
   &\times \int_{\mathbb{R}^3}\dd u' \frac{\mathbf{k}(u,u')w_{\tilde{\theta}}(u)}{w_{\tilde{\theta}}(u')|u'|^2\min\{\frac{\alpha(x-(t-s)v-(s-s')u,u')}{|u'|},\frac{\alpha(y-(t-s)v-(s-s')u,u')}{|u'|}\}^{2+\beta}} \\
    &\lesssim  \frac{O(\e)[\nabla_x f(\cdot ,v)]_{C_{x,2+\beta}^{0,\beta}}}{w_{\tilde{\theta}}(v)|v|^2\min\{\frac{\alpha(x,v)}{|v|},\frac{\alpha(y,v)}{|v|}\}^{2+\beta}}.
\end{align*}

Then we consider contribution of$~\eqref{decompose s'}_2$. For simplicity we denote
\begin{equation}\label{x'}
x''=x-(t-s)v-(s-s')u,\quad y''=y-(t-s)v-(s-s')u,\quad x''-y''=x-y.
\end{equation}
We exchange $\nabla_x$ for $\nabla_u$:
\[\nabla_x f(x-(t-s)v-(s-s')u,u')-\nabla_x f(y-(t-s)v-(s-s')u,u')=\nabla_u [f(x'',u')-f(y'',u')]\frac{-1}{s-s'}.\]
Since $s-s'\geq \e$ the contribution of$~\eqref{decompose s'}_2$ in~\eqref{k: h-h deri} is
\begin{equation}\label{Hd:kk du}
\begin{split}
   &  \int_{t-t_m(v)}^t \dd s e^{-\nu(v)(t-s)}\int_{\mathbb{R}^3} \dd u \mathbf{k}(v,u)  \int_{s-t_m^s(u)}^{s-\e} e^{-\nu(u)(s-s')}\dd s' \mathbf{1}_{s-s'\geq \e}\\
   & \times \int_{\mathbb{R}^3}\dd u' \mathbf{k}(u,u')\frac{\nabla_u [f(x'',u')-f(y'',u')]}{|x-y|^\beta}\frac{-1}{s-s'}.
\end{split}
\end{equation}
Then we integrate by part for $\dd u$ to have
\begin{align}
  \eqref{Hd:kk du} &=\int_{t-t_m(v)}^t \dd s e^{-\nu(v)(t-s)}\int_{\mathbb{R}^3} \dd u   \mathbf{1}_{s-s'\geq \e}\notag  \\
   & \times \Big[ \nabla_u [\mathbf{k}(v,u)\mathbf{k}(u,u')]\int_{s-t_m^s(u)}^{s-\e}\frac{e^{-\nu(u)(s-s')}}{s-s'}\dd s' \int_{\mathbb{R}^3}\dd u'    \frac{f(x'',u')-f(y'',u')}{|x-y|^\beta}    \label{Hd: IBP1}\\
   &+        \mathbf{k}(v,u)\int_{s-t_m^s(u)}^{s-\e}\frac{\nabla_u e^{-\nu(u)(s-s')}}{s-s'}\dd s' \int_{\mathbb{R}^3}\dd u'   \mathbf{k}(u,u') \frac{f(x'',u')-f(y'',u')}{|x-y|^\beta} \label{Hd: IBP2} \\
  & + \mathbf{k}(v,u) \nabla_u t_m^s(u) \frac{e^{-\nu t_m^s(u)}}{t_m^s(u)} \int_{\mathbb{R}^3}\dd u' \mathbf{k}(u,u') \frac{f(x^b,u')-f(y^b,u')}{|x-y|^\beta} \Big]
  \label{Hd: IBP3}.
\end{align}
Here we denoted
\begin{equation}\label{x^b}
x^b=x-(t-s)v-t_m^s(u)u,\quad y^b=y-(t-s)v-t_m^s(u) u.
\end{equation}

First we estimate~\eqref{Hd: IBP1}. We begin with $\nabla_u \mathbf{k}(u,u')$. Since $w^{-1}_{\tilde{\theta}}(u')\langle u'\rangle |u'|^2\lesssim 1$, from~\eqref{min: f} with~\eqref{simplicity} and~\eqref{k_varrho} we have
\begin{align}
 & \eqref{Hd: IBP1}\notag\\
  &\lesssim O(\e^{-1})\frac{\Vert w_{\tilde{\theta}}\alpha\nabla_x f\Vert_\infty}{w_{\tilde{\theta}}(v)}  \int_{t-t_m(v)}^t \dd se^{-\nu(v)(t-s)}\int_{\mathbb{R}^3}\dd u  \frac{\mathbf{k}(v,u) w_{\tilde{\theta}}(v)}{ w_{\tilde{\theta}}(u)}   \notag\\
   & \times     \int_{s-t_m^s(u)}^{s-\e}e^{-\nu(u)(s-s')}\dd s' \int_{\mathbb{R}^3}\dd u' \langle u'\rangle |u'|^2 w^{-1}_{\tilde{\theta}}(u')   \frac{ \mathbf{k}(u,u') w_{\tilde{\theta}}(u)}{|u-u'| w_{\tilde{\theta}}(u')}    \frac{1 }{|u'|^2\min\{\alpha(x'',u'),\alpha(y'',u')\}^{\beta}}                      \notag\\
   &  \lesssim \frac{O(\e^{-1})\Vert w_{\tilde{\theta}}\alpha\nabla_x f\Vert_\infty}{w_{\tilde{\theta}}(v)}\int_{t-t_m(v)}^t \dd se^{-\nu(v)(t-s)}\int_{\mathbb{R}^3} \frac{\dd u  \mathbf{k}_{\tilde{\varrho}}(v,u)}{|u|^2\min\{\alpha(x-(t-s)v,u),\alpha(y-(t-s)v,u)\}^{\beta}}           \notag  \\
   &\lesssim   \frac{O(\e^{-1}) \Vert w_{\tilde{\theta}}\alpha\nabla_x f\Vert_\infty}{w_{\tilde{\theta}}(v)|v|^2}\lesssim \frac{O(\e^{-1}) \Vert w_{\tilde{\theta}}\alpha\nabla_x f\Vert_\infty}{w_{\tilde{\theta}}(v)|v|^2 \min\{\frac{\alpha(x,v)}{|v|},\frac{\alpha(y,v)}{|v|}\}^{2+\beta}},
   \label{Hd: bound for IBP1}
\end{align}
where we have used~\eqref{integrate k/v-u} and~\eqref{k_theta} in the third line,~\eqref{integrate beta<1} in the last line.

The term with $\nabla_u \mathbf{k}(v,u)$ can be similarly bounded by
\begin{align}
&\frac{O(\e^{-1}) \Vert w_{\tilde{\theta}}\alpha\nabla_x f\Vert_\infty}{w_{\tilde{\theta}}(v)|v|^2 \min\{\frac{\alpha(x,v)}{|v|},\frac{\alpha(y,v)}{|v|}\}^{2+\beta}} .
   \label{Hd: bound for IBP11}
\end{align}

Then we estimate~\eqref{Hd: IBP2}. From~\eqref{min: f} with~\eqref{simplicity} we have
\begin{align}
  \eqref{Hd: IBP2} & \lesssim   \frac{\Vert w_{\tilde{\theta}}\alpha\nabla_x f\Vert_\infty}{w_{\tilde{\theta}}(v)} \int_0^t \dd s e^{-\nu(v)(t-s)}\int_{\mathbb{R}^3}\dd u \frac{w_{\tilde{\theta}}(u)\mathbf{k}(v,u)}{w_{\tilde{\theta}}(v)}\int_0^{s-\e} e^{-\nu(u)(s-s')}\dd s' \notag\\
   & \times \int_{\mathbb{R}^3}\dd u'  |u'|^2w^{-1}_{\tilde{\theta}}(u')   \frac{w_{\tilde{\theta}}(u)\mathbf{k}(u,u')}{w_{\tilde{\theta}}(u')|u'|^2\min\{\alpha(x'',u'),\alpha(y'',u')\}^\beta}\notag\\
   &\lesssim \frac{ \Vert w_{\tilde{\theta}}\alpha\nabla_x f\Vert_\infty}{w_{\tilde{\theta}}(v)} \int_0^t \dd s e^{-\nu(v)(t-s)}\int_{\mathbb{R}^3}\dd u \frac{1}{|u|^2} \mathbf{k}_{\tilde{\varrho}}(v,u) \lesssim \frac{\Vert w_{\tilde{\theta}}\alpha\nabla_x f\Vert_\infty}{w_{\tilde{\theta}}(v)|v|^2 \min\{\frac{\alpha(x,v)}{|v|},\frac{\alpha(y,v)}{|v|}\}^{2+\beta}},
   \label{Hd: bound for IBP2}
\end{align}
where we have used Lemma \ref{Lemma: k tilde} in the second line and~\eqref{integrate beta<1} in Lemma \ref{Lemma: NLN} in the last line.

Last we estimate~\eqref{Hd: IBP3}. Since we are considering $s-s'\geq \e$, $\tb(x-(t-s)v,u)\geq \e$. From~\eqref{min: f} with~\eqref{simplicity} and~\eqref{nabla_tbxb}, we have
\begin{align}
   \eqref{Hd: IBP3}& \lesssim \frac{\Vert w_{\tilde{\theta}}\alpha\nabla_x f\Vert_\infty}{w_{\tilde{\theta}}(v)}\int_{t-t_m(v)}^t \dd s e^{-\nu(v)(t-s)}\int_{\mathbb{R}^3} \frac{w_{\tilde{\theta}}(u)\mathbf{k}(v,u)}{w_{\tilde{\theta}}(v)} \dd u \frac{e^{-\nu(u)t_m^s(u) }}{t_m^s(u)}\notag \\
   & \times \nabla_u t_m^s(u) \int_{\mathbb{R}^3}\dd u'   |u'| w^{-1}_{\tilde{\theta}}(u') \frac{w_{\tilde{\theta}}(u)\mathbf{k}(u,u')}{|u'|w_{\tilde{\theta}}(u')\min\{\alpha(x^b,u'),\alpha(y^b,u')\}^\beta} \notag\\
   & \lesssim O(\e^{-1})\frac{\Vert w_{\tilde{\theta}}\alpha \nabla_x f\Vert_\infty}{w_{\tilde{\theta}}(v)} \int_{t-t_m(v)}^t \dd s e^{-\nu(v)(t-s)}\int_{\mathbb{R}^3} \frac{\mathbf{k}_{\tilde{\varrho}}(v,u)}{|u|^2\min\{\frac{\alpha(x-(t-s)v,u)}{|u|},\frac{\alpha(y-(t-s)v,u)}{|u|}\}} \dd u \notag\\
   &\lesssim \frac{O(\e^{-1})\Vert w_{\tilde{\theta}}\alpha\nabla_x f\Vert_\infty}{w_{\tilde{\theta}}(v)|v|^2\min\{\frac{\alpha(x,v)}{|v|},\frac{\alpha(y,v)}{|v|}\}}\lesssim \frac{O(\e^{-1})\Vert w_{\tilde{\theta}}  \alpha\nabla_x f\Vert_\infty}{w_{\tilde{\theta}}(v)|v|^2\min\{\frac{\alpha(x,v)}{|v|},\frac{\alpha(y,v)}{|v|}\}^{2+\beta}},\label{Hd: bound for IBP3}
\end{align}
where we have used Lemma \ref{Lemma: k tilde} in the third line and~\eqref{integrate beta<1} in Lemma \ref{Lemma: NLN} in the last line.

Then combining~\eqref{Hd: bound for IBP1},\eqref{Hd: bound for IBP11},\eqref{Hd: bound for IBP2} and~\eqref{Hd: bound for IBP3} we conclude
\begin{equation}\label{k: bound for h-h deri}
  \eqref{k: h-h deri}\lesssim    \frac{o(1)[\nabla_x f(\cdot ,v)]_{C_{x,2+\beta}^{0,\beta}}+O(\e^{-1})\Vert w_{\tilde{\theta}}\alpha\nabla_x f\Vert_\infty}{w_{\tilde{\theta}}(v)|v|^2\min\{\frac{\alpha(x,v)}{|v|},\frac{\alpha(y,v)}{|v|}\}^{2+\beta}}.
\end{equation}

Combining~\eqref{k: bound for h-h deri},~\eqref{k:   gamma-gamma},~\eqref{k: bound for nablagx-nablagy} ,~\eqref{k:    inter1} and~\eqref{Hd: bound for h-h deri Gamma} we conclude that
\begin{equation}\label{estimate of C1beta 3}
\eqref{C1beta inter:3}\lesssim  \Vert T_W-T_0\Vert_{C^2}\frac{o(1)\Big[[\nabla_x f(\cdot ,v)]_{C_{x,2+\beta}^{0,\beta}}+[\nabla_\parallel f(\cdot , v)]_{C_{x,1+\beta}^{0,\beta}}\Big]+O(\e^{-21})\Vert w_{\tilde{\theta}}\alpha\nabla_x f\Vert_\infty^2}{w_{\tilde{\theta}}(v)|v|^2\min\{\frac{\alpha(x,v)}{|v|},\frac{\alpha(y,v)}{|v|}\}^{2+\beta}}.
\end{equation}

Finally from~\eqref{C1beta inter:1}-\eqref{C1beta inter:3} and the estimate~\eqref{estimate of C1beta 3},\eqref{estimate for C1beta4}, we conclude the proof of~\eqref{Bound of C1beta}.

\textbf{Step 3: proof of~\eqref{Bound of C1beta tangential}.}

Now we prove~\eqref{Bound of C1beta tangential}. From Lemma \ref{Lemma: intermediate estimate},~\eqref{C1beta inter:4} is already bounded from~\eqref{estimate for C1beta4}.

For~\eqref{C1betatang inter:0}, since $w_{\tilde{\theta}}^{-1}(v)\tilde{\alpha}(x,v)\lesssim w_{\tilde{\theta}/2}^{-1}(v)\alpha(x,v)$, by~\eqref{estimate of C1beta 3} we conclude
\begin{equation}\label{tang: bound for C1betang: inter:0}
\eqref{C1betatang inter:0}\lesssim \frac{\tilde{\alpha}(x,v)}{|v|}\times \eqref{estimate of C1beta 3}= \frac{o(1)\Big[[\nabla_x f(\cdot ,v)]_{C_{x,2+\beta}^{0,\beta}}+[\nabla_\parallel f(\cdot , v)]_{C_{x,1+\beta}^{0,\beta}}\Big]+O(\e^{-21})\Vert w_{\tilde{\theta}}\alpha\nabla_x f\Vert_\infty^2}{w_{\tilde{\theta}/2}(v)|v|^2\min\{\frac{\alpha(x,v)}{|v|},\frac{\alpha(y,v)}{|v|}\}^{1+\beta}}.
\end{equation}

Then we only need to estimate~\eqref{C1betatang inter:3}. First we consider $h=\Gamma(f,f)$. Such contribution is directly bounded using~\eqref{G gamma- G gamma} in Lemma \ref{Lemma: gamma-gamma}, thus
\begin{equation}\label{estimate of C1betatang 3}
\eqref{C1betatang inter:3}_{h=\Gamma}\lesssim \frac{o(1)[\nabla_\parallel f(\cdot , v)]_{C_{x,1+\beta}^{0,\beta}}+\Vert w_{\tilde{\theta}}\alpha\nabla_x f\Vert_\infty^2}{w_{\tilde{\theta}/2}(v)|v|^2\min\{\frac{\alpha(x,v)}{|v|},\frac{\alpha(y,v)}{|v|}\}^{1+\beta}}.
\end{equation}

Then we consider $h=K(f)$, which reads
\begin{align}
   &\int^t_{t-t_m(v)} \dd s e^{-\nu(t-s)} \int_{\mathbb{R}^3} \dd u \mathbf{k}(v,u)     \notag     \\
   &\times \frac{G(x-(t-s)v)\nabla_x f(x-(t-s)v,u)-G(y-(t-s)v)\nabla_x f(y-(t-s)v,u)}{|x-y|^\beta} .  \label{k in Gf-Gf}
\end{align}

We express~\eqref{k in Gf-Gf} by~\eqref{C1betatang inter:0}-\eqref{C1betatang inter:4} along $u$.

Note that
\[\eqref{C1betatang inter:0}\lesssim \eqref{tang: bound for C1betang: inter:0},\quad \eqref{C1betatang inter:4} \lesssim \eqref{estimate for C1beta4}, \]
we conclude that the contribution of~\eqref{C1betatang inter:0},\eqref{C1betatang inter:1} and~\eqref{C1betatang inter:4} in~\eqref{k in Gf-Gf} are bounded by
\begin{align}
& \frac{1}{w_{\tilde{\theta}/2}(v)}\int^t_{t-t_m(v)} \dd s e^{-\nu(t-s)} \int_{\mathbb{R}^3} \dd u w_{\tilde{\theta}/2}(v)\mathbf{k}(v,u) \notag\\
&\times \frac{o(1)\Big[[\nabla_x f(\cdot ,v)]_{C_{x,2+\beta}^{0,\beta}}+[\nabla_\parallel f(\cdot , v)]_{C_{x,1+\beta}^{0,\beta}}\Big]+O(\e^{-21})\Vert w_{\tilde{\theta}}\alpha\nabla_x f\Vert_\infty^2}{w_{\tilde{\theta}/2}(u)|u|^2\min\{\frac{\alpha(x-(t-s)v,u)}{|u|},\frac{\alpha(y-(t-s)v,u)}{|u|}\}^{1+\beta}} \notag\\
& \lesssim \frac{o(1)\Big[[\nabla_x f(\cdot ,v)]_{C_{x,2+\beta}^{0,\beta}}+[\nabla_\parallel f(\cdot , v)]_{C_{x,1+\beta}^{0,\beta}}\Big]+O(\e^{-21})\Vert w_{\tilde{\theta}}\alpha\nabla_x f\Vert_\infty^2}{w_{\tilde{\theta}/2}(v)|v|^2\min\{\frac{\alpha(x,v)}{|v|},\frac{\alpha(y,v)}{|v|}\}^{1+\beta}}, \label{bound for 014}
\end{align}
where we have used~\eqref{integrate alpha beta} in Lemma \ref{Lemma: NLN} with $p=1+\beta$.

Then we focus on the contribution of the double collision operator~\eqref{C1betatang inter:3}. By Lemma \ref{Lemma: gamma-gamma} the contribution of $h=\Gamma$ is bounded by
\begin{align}
   &\frac{1}{w_{\tilde{\theta}/2}(v)}\int^t_{t-t_m(v)} \dd s e^{-\nu(t-s)} \int_{\mathbb{R}^3} \dd u w_{\tilde{\theta}/2}(v)\mathbf{k}(v,u)   \frac{o(1)[\nabla_\parallel f(\cdot ,v)]_{C_{x,1+\beta}^{0,\beta}}+\Vert w_{\tilde{\theta}}\alpha\nabla_x f\Vert_\infty^2}{w_{\tilde{\theta}/2}(u)|u|^2\min\{\frac{\alpha(x-(t-s)v,u)}{|u|},\frac{\alpha(y-(t-s)v,u)}{|u|}\}^{1+\beta}} \notag  \\
   & \lesssim \frac{o(1)[\nabla_\parallel f(\cdot ,v)]_{C_{x,1+\beta}^{0,\beta}}+\Vert w_{\tilde{\theta}}\alpha\nabla_x f\Vert_\infty^2}{w_{\tilde{\theta}/2}(v)|v|^2\min\{\frac{\alpha(x,v)}{|v|},\frac{\alpha(y,v)}{|v|}\}^{1+\beta}}.\label{Gf-Gf: k:gamma}
\end{align}

Last we focus on the contribution of $h=K(f)$. Recall the notation $x'',y''$ in~\eqref{x'}. We need to compute
\begin{align}
&  \int_{t-t_m(v)}^t \dd s e^{-\nu(t-s)}\int_{\mathbb{R}^3} \dd u \mathbf{k}(v,u)\int^s_{s-t_m^s(u)} \dd s' e^{-\nu(s-s')} \int_{\mathbb{R}^3}\dd u' \mathbf{k}(u,u')\notag\\
& \times \underbrace{\frac{G(x'')\nabla_x f(x'',u')-G(y'')\nabla_x f(y'',u')}{|x-y|^\beta}}_{\eqref{tang: k: h-h deri 2}_*}. \label{tang: k: h-h deri 2}
\end{align}
We apply the decomposition~\eqref{decompose s'} for $\dd s'$.

When $s-s'\leq \e$, by~\eqref{integrate alpha beta small} in Lemma \ref{Lemma: NLN} with $p=1+\beta$ we have
\begin{equation}\label{tang:k: bound for h-h deri 1}
\begin{split}
\eqref{tang: k: h-h deri 2}\mathbf{1}_{s-s'\leq \e}   & \frac{1}{w_{\tilde{\theta}/2}(v)}\lesssim \int_{t-t_m(v)}^t \dd s e^{-\nu(t-s)}\int_{\mathbb{R}^3} \dd u \frac{w_{\tilde{\theta}/2}(v)\mathbf{k}(v,u)}{w_{\tilde{\theta}/2}(u)}\int^s_{s-\e} \dd s' e^{-\nu(s-s')} \\
    & \times  \int_{\mathbb{R}^3}\dd u' w_{\tilde{\theta}/2}(u)\mathbf{k}(u,u')      \frac{[\nabla_\parallel f(\cdot ,v)]_{C^{0,\beta}_{x,1+\beta}}}{w_{\tilde{\theta}/2}(u')|u'|^2\min\{\alpha(x'',u'),\alpha(y'',u')\}^{1+\beta}}     \\
    &\lesssim O(\e)\frac{[\nabla_\parallel f(\cdot ,v)]_{C^{0,\beta}_{x,1+\beta}}}{w_{\tilde{\theta}/2}(v)|v|^2\min\{\frac{\alpha(x,v)}{|v|},\frac{\alpha(y,v)}{|v|}\}^{1+\beta}}.
\end{split}
\end{equation}

When $s-s'\geq \e$, we rewrite
\begin{align}
  \eqref{tang: k: h-h deri 2}_*  &= \frac{[G(x'')-G(y'')]\nabla_x f(x'',u')}{|x-y|^\beta}  \label{tang: h-h 1}\\
   & +\frac{G(y'')[\nabla_x f(x'',u')-\nabla_x f(y'',u')]}{|x-y|^\beta}.\label{tang: h-h 2}
\end{align}

For~\eqref{tang: h-h 1}, since $|u'|w^{-1}_{\tilde{\theta}}(u')\lesssim w^{-1}_{\tilde{\theta}/2}(u')$, we apply~\eqref{min: nxb} to conclude
\begin{align}
  \eqref{tang: h-h 1} & \lesssim \frac{\Vert w_{\tilde{\theta}}\alpha\nabla_x f\Vert_\infty}{w_{\tilde{\theta}/2}(v)}\int_{t-t_m(v)}^t \dd s e^{-\nu(t-s)}\int_{\mathbb{R}^3} \dd u \frac{w_{\tilde{\theta}/2}(v)\mathbf{k}(v,u)}{w_{\tilde{\theta}/2}(u)}\notag \\
  &\times \int^s_{s-t_m^s(u)} \dd s' e^{-\nu(s-s')} \int_{\mathbb{R}^3}\dd u' |u'|\frac{ w_{\tilde{\theta}/2}(u)\mathbf{k}(u,u')}{w_{\tilde{\theta}/2}(u')|u'|\alpha(x'',u')}\notag \\
   & \lesssim  \frac{\Vert w_{\tilde{\theta}}\alpha\nabla_x f \Vert_\infty}{w_{\tilde{\theta}/2}(v)|v|^2\min\{\frac{\alpha(x,v)}{|v|},\frac{\alpha(y,v)}{|v|}\}^{1+\beta}},
   \label{bound for tang h-h 1}
\end{align}
where we have applied Lemma~\ref{Lemma: k tilde} and~\eqref{est:nonlocal_wo_e} in Lemma \ref{Lemma: NLN} to the last inequality.

For~\eqref{tang: h-h 2}, we exchange $\nabla_x f(x'',u')=\frac{\nabla_{u} f(x'',u')}{-(s-s')}$ and perform an integration by parts to $\dd u$. Since $|G(y'')|\lesssim 1$, the contribution of~\eqref{tang: h-h 2} in~\eqref{tang: k: h-h deri 2} is bounded by $\eqref{Hd: IBP1},\eqref{Hd: IBP2},\eqref{Hd: IBP3}$ and with an extra term that corresponds to the derivative of $G(y'')$:
\begin{align}
 &\int_{t-t_m(v)}^t \dd s e^{-\nu(v)(t-s)}\int_{\mathbb{R}^3} \dd u   \mathbf{1}_{s-s'\geq \e}\notag  \\
   &\times        \mathbf{k}(v,u)\int_{s-t_m^s(u)}^{s-\e}\nabla_u G(y'')  e^{-\nu(u)(s-s')}\dd s' \int_{\mathbb{R}^3}\dd u'   \mathbf{k}(u,u') \frac{f(x'',u')-f(y'',u')}{|x-y|^\beta} \notag\\
   & \lesssim  \frac{ \Vert w_{\tilde{\theta}}\alpha\nabla_x f\Vert_\infty}{w_{\tilde{\theta}/2}(v)} \int_0^t \dd s e^{-\nu(v)(t-s)}\int_{\mathbb{R}^3}\dd u \frac{w_{\tilde{\theta}/2}(u)\mathbf{k}(v,u)}{w_{\tilde{\theta}/2}(v)}\int_0^{s-\e} e^{-\nu(u)(s-s')}(s-s')\Vert \xi\Vert_{C^2} \dd s' \notag\\
   &\times \int_{\mathbb{R}^3}\dd u' w_{\tilde{\theta}/2}(u)\mathbf{k}(u,u')        \frac{|u'|^2}{w_{\tilde{\theta}}(u')|u'|^2\min\{\alpha(x'',u'),\alpha(y'',u')\}^\beta}\notag\\
    &\lesssim \frac{ \Vert w_{\tilde{\theta}}\alpha\nabla_x f\Vert_\infty}{w_{\tilde{\theta}/2}(v)} \int_0^t \dd s e^{-\nu(v)(t-s)}\int_{\mathbb{R}^3}\dd u  \frac{ \mathbf{k}_{\tilde{\varrho}}(v,u)}{|u|^2}\int_{s-t^s_m(u)}^{s-\e} e^{-\nu(u)(s-s')/2}\dd s'\notag\\
    &\lesssim \frac{\Vert w_{\tilde{\theta}}\alpha\nabla_x f\Vert_\infty}{w_{\tilde{\theta}/2}(v)|v|^2 \min\{\frac{\alpha(x,v)}{|v|},\frac{\alpha(y,v)}{|v|}\}^{1+\beta}}. \label{IBP G}
\end{align}
Here we have used~\eqref{min: f} with~\eqref{simplicity} in the fourth line and Lemma \ref{Lemma: k tilde} in the fifth line.

Thus the contribution of~\eqref{tang: h-h 2} in~\eqref{tang: k: h-h deri 2} is bounded by
\begin{align}
&   \eqref{Hd: IBP1}+\eqref{Hd: IBP2}+\eqref{Hd: IBP3}+\eqref{IBP G} \notag  \\
   & \lesssim \eqref{Hd: bound for IBP1}+\eqref{Hd: bound for IBP2}+\eqref{Hd: bound for IBP3}+\eqref{IBP G}\lesssim \frac{O(\e^{-1})\Vert w_{\tilde{\theta}}\alpha\nabla_x f \Vert_\infty}{w_{\tilde{\theta}/2}(v)|v|^2\min\{\frac{\alpha(x,v)}{|v|},\frac{\alpha(y,v)}{|v|}\}^{1+\beta}} .\label{tang:k: bound for h-h deri 2}
\end{align}

This, together with~\eqref{bound for tang h-h 1} and~\eqref{tang:k: bound for h-h deri 1}, lead to the conclusion:
\begin{equation}\label{tang:k: bound for h-h deri}
\eqref{tang: k: h-h deri 2}\lesssim \frac{o(1)\Big[[\nabla_x f(\cdot ,v)]_{C_{x,2+\beta}^{0,\beta}}+[\nabla_\parallel f(\cdot , v)]_{C_{x,1+\beta}^{0,\beta}}\Big]+O(\e^{-21})\Vert w_{\tilde{\theta}}\alpha\nabla_x f\Vert_\infty^2}{w_{\tilde{\theta}/2}(v)|v|^2\min\{\frac{\alpha(x,v)}{|v|},\frac{\alpha(y,v)}{|v|}\}^{1+\beta}}.
\end{equation}

Finally collecting~\eqref{tang:k: bound for h-h deri},\eqref{Gf-Gf:    k:gamma},\eqref{bound for 014},\eqref{estimate of C1betatang 3},\eqref{tang: bound for C1betang: inter:0} and Lemma \ref{Lemma: intermediate estimate} we conclude the proof of~\eqref{Bound of C1beta tangential}.

\end{proof}

\bigskip

\noindent\textbf{Acknowledgements}: The authors thank Yan Guo for his interest. C.K. thanks Ikun Chen for earlier discussions on the same subject and a kind hospitality of National Taiwan University during his visit in 2014. This research is supported in part by National Science Foundation under Grant No. $1501031$, $1750488$, and $1900923$, and the Wisconsin Alumni Research Foundation and UW-Madison Data Science Initiative.

\end{document}